\documentclass[12pt,oneside]{amsart} 

\usepackage{amsthm,amsbsy,amsmath,amssymb,amsfonts,array,mathrsfs,verbatim,enumitem,xypic,epic,eepic,setspace,color,graphicx,scalefnt}

\setenumerate{listparindent=\parindent}

\usepackage[all]{xy}
\xyoption{arc}

\theoremstyle{plain}
\newtheorem*{unnumberedthm}{Theorem}
\newtheorem*{fcc}{First Chart Criterion}
\newtheorem*{scc}{Second Chart Criterion}
\newtheorem*{fc}{Fiberwise Criterion} 
\newtheorem{thm}{Theorem}[subsection]
\newtheorem{prop}[thm]{Proposition}     
\newtheorem{lem}[thm]{Lemma}
\newtheorem{cor}[thm]{Corollary}

\theoremstyle{definition}

\newtheorem*{unnumbereddefn}{Definition}
\newtheorem{defn}[thm]{Definition}

\newtheorem{example}[thm]{Example} 
\newtheorem{rem}[thm]{Remark}

\DeclareFontFamily{OT1}{rsfs}{}
\DeclareFontShape{OT1}{rsfs}{n}{it}{<-> rsfs10}{}
\DeclareMathAlphabet{\curly}{OT1}{rsfs}{n}{it}

\makeatletter \@addtoreset{equation}{subsection} \makeatother  
\makeatletter \@addtoreset{thm}{subsection} \makeatother  

\makeatletter \@addtoreset{equation}{subsection} \makeatother  
\makeatletter \@addtoreset{thm}{subsection} \makeatother  

\newcommand{\X}{\mathscr{X}}

\newcommand{\A}{{\bf A}}
\newcommand{\B}{{\bf B}}
\newcommand{\C}{{\bf C}}
\newcommand{\D}{{\bf D}}
\newcommand{\E}{{\bf E}}
\newcommand{\Fib}{{\bf Fib}}
\newcommand{\LQuot}{{\bf LQuot}}
\newcommand{\Quot}{{\bf Quot}}
\newcommand{\LogStr}{{\bf LogStr}}
\newcommand{\Desc}{{\bf Desc}}
\newcommand{\An}{{\bf An}}
\newcommand{\Sets}{{\bf Sets}}

\newcommand{\Mod}{{\bf Mod}}
\newcommand{\Qco}{{\bf Qco}}
\newcommand{\Coh}{{\bf Coh}}

\newcommand{\GrAn}{{\bf GrAn}}

\newcommand{\LRS}{{\bf LRS}}
\newcommand{\Log}{\curly{Log}}   
\newcommand{\LogSch}{{\bf LogSch}} 
\newcommand{\Mon}{{\bf Mon}}

\newcommand{\Ab}{{\bf Ab}}
\newcommand{\CFG}{{\bf CFG}}
\newcommand{\Sch}{{\bf Sch}} 
\newcommand{\Esp}{{\bf Esp}} 
\newcommand{\AAA}{\curly{A}} 

\newcommand{\dirlim}{\displaystyle \lim_{ \longrightarrow } \,} 
\newcommand{\invlim}{\displaystyle \lim_{ \longleftarrow } \,} 

\renewcommand{\AA}{\mathbb{A}}  
\newcommand{\GG}{\mathbb{G}}  
\newcommand{\NN}{\mathbb{N}}  
\newcommand{\ZZ}{\mathbb{Z}} 
\newcommand{\QQ}{\mathbb{Q}} 
\newcommand{\CC}{\mathbb{C}} 
\newcommand{\PP}{\mathbb{P}} 

\newcommand{\F}{\curly{F}} 

\newcommand{\M}{\mathcal{M}} 
\newcommand{\N}{\mathcal{N}} 
\renewcommand{\O}{\mathcal{O}} 

\renewcommand{\L}{\mathcal{L}}

\newcommand{\m}{\mathfrak{m}}

\newcommand{\p}{\mathfrak{p}}  
\renewcommand{\P}{\mathfrak{P}}  
\newcommand{\n}{\mathfrak{n}}  

\newcommand{\ov}{\overline}
\renewcommand{\u}{\underline}
\newcommand{\into}{\hookrightarrow}
\newcommand{\be}{\begin{eqnarray*}}
\newcommand{\ee}{\end{eqnarray*}}
\newcommand{\bne}[1]{\begin{eqnarray} \label{#1} }
\newcommand{\ene}{\end{eqnarray}}
\newcommand{\xym}{\xymatrix}
\newcommand{\bp}{\begin{pmatrix}}
\newcommand{\ep}{\end{pmatrix}}
\newcommand{\slot}{ \hspace{0.05in} {\rm \_} \hspace{0.05in} } 

\newcommand{\Ann}{\operatorname{Ann}}
\renewcommand{\H}{\operatorname{H}}
\renewcommand{\Im}{\operatorname{Im}}

\newcommand{\Hom}{\operatorname{Hom}}   
\newcommand{\bHom}{\operatorname{{\bf Hom}}}   
\newcommand{\Bilin}{\operatorname{Bilin}}   
\newcommand{\Ker}{\operatorname{Ker}}   
\newcommand{\Cok}{\operatorname{Cok}}  
   
\newcommand{\Ext}{\operatorname{Ext}}

\newcommand{\Aut}{\operatorname{Aut}}
\newcommand{\Id}{\operatorname{Id}}
\newcommand{\GL}{\operatorname{GL}}

\newcommand{\Spec}{\operatorname{Spec}}
\newcommand{\Tor}{\operatorname{Tor}}
\newcommand{\Supp}{\operatorname{Supp}}
\newcommand{\Hilb}{\operatorname{Hilb}}
\newcommand{\Pic}{\operatorname{Pic}}

\begin{document}

\author{W.~D.~Gillam}
\address{Department of Mathematics, Brown University, Providence, RI 02912}
\email{wgillam@math.brown.edu}
\date{\today}
\title{Logarithmic Flatness}

\begin{abstract} A map of fine log schemes $X \to Y$ induces a map from the scheme underlying $X$ to Olsson's algebraic stack of strict morphisms of fine log schemes over $Y$.  A sheaf on $X$ is called \emph{log flat over} $Y$ iff it is flat over this algebraic stack.   This paper is a study of log flatness and the related notions of flatness for maps of monoids and graded rings.  It is shown that log flatness is equivalent to a more general notion of ``formal log flatness" that makes sense for an arbitrary map of log ringed topoi.  Concrete log flatness criteria are given for many $X \to Y$ that occur ``in nature," such as toric varieties, nodal curves, and the like.  For very simple $X \to Y$ it turns out that log flatness is equivalent to previously extant notions of ``perfection," thus it provides a generalization for more complicated $X \to Y$ useful for studying moduli of sheaves via degeneration techniques. \end{abstract}

\maketitle

\arraycolsep=2pt 
\def\arraystretch{1.2}

\setcounter{tocdepth}{2}
\tableofcontents

\section{Introduction}

Let $Y$ be a fine log scheme with underlying scheme $\u{Y}$.  Consider the category $\Log(Y)$ whose objects are fine log schemes over $Y$ and whose maps are strict maps of fine log schemes over $Y$.  The forgetful functor \be \Log(Y) & \to & \Sch/\u{Y} \\ (f : X \to Y) & \mapsto & (\u{f} : \u{X} \to \u{Y}) \ee is clearly a groupoid fibration over $\u{Y}$-schemes.  Olsson showed in \cite[Theorem~1.1]{Ols} that $\Log(Y)$ is an algebraic stack of locally finite presentation over $\u{Y}$ with representable, finitely presented, and locally separated diagonal.  In particular, for any map of algebraic stacks $\u{X} \to \Log(Y)$, it makes sense to ask whether a sheaf (``sheaf" means ``quasi-coherent sheaf" for now) on $\u{X}$ is flat over $\Log(Y)$ (see \S\ref{section:flatnessoverstacks}).

A morphism of fine log schemes $f : X \to Y$ induces a morphism of (algebraic) stacks $\Log_f : \u{X} \to \Log(Y)$ by regarding a scheme over $\u{X}$ as a fine log scheme over $Y$ by pulling the log data back from $X$ (see \S\ref{section:tautologicalsection}).  We can then make the following:

\begin{unnumbereddefn} For a morphism of fine log schemes $X \to Y$, a sheaf $M$ on $X$ is called \emph{log flat over} $Y$ iff $M$ is flat over $\Log(Y)$ via $\Log_f : \u{X} \to \Log(Y)$.  If $Y$ is the spectrum of a field with trivial log structure we often say ``log flat" instead of ``log flat over $Y$." \end{unnumbereddefn}

\noindent {\bf Acknowledgement:}  The idea that such a notion might be of interest, particularly in degeneration methods for sheaf-theoretic curve counting, arose in conversation with Davesh Maulik. 

The purpose of this paper is to study log flatness.  We first observe: \begin{enumerate}[label=(\roman*), ref=\roman*] \item \label{kernelandextension} A kernel or extension of log flat sheaves is log flat. \item When $f : X \to Y$ is strict ($f^\dagger : f^* \M_Y \to \M_X$ is an isomorphism), log flatness is equivalent to flatness over $\u{Y}$ in the usual sense (Lemma~\ref{lem:strictlogflatness}). \item \label{basechangestability} Log flatness is stable under strict base change and \item is strict-fppf local on $f : X \to Y$ (Lemma~\ref{lem:fppflocal}). \end{enumerate}

Statement \eqref{kernelandextension} is immediate from the definition.  Regarding \eqref{basechangestability}, we eventually prove (Theorem~\ref{thm:logflatnessandbasechange}) that log flatness is in fact stable under \emph{arbitrary base change}.  Unlike the proof of \eqref{basechangestability}, which is formal, the proof of the latter statement seems to require almost the full arsenal of techniques developed here, as well as Kato's theory of \emph{neat charts}.  It would be interesting to find a more elementary proof of this fact.

\subsection{The first chart criterion and formal log flatness} In particular, log flatness is strict-\'etale local in nature, so it is in many ways sufficient to study log flatness for a map $f : X \to Y$ where $\u{f} = \Spec ( A \to C)$ is a map of affine schemes and there is a global chart \bne{globalchart} & \xym{ P \ar[r] & \M_X(X) \ar[r]^-{\alpha_X} & \O_X(X)=C \\ Q \ar[u]^h \ar[r] & \M_Y(Y) \ar[u]^{f^\dagger} \ar[r]^-{\alpha_Y} & \O_Y(Y)=A \ar[u] } \ene for $f : X \to Y$.  Our first results are the \emph{Chart Criteria} for log flatness (\S\ref{section:chartcriteriatheorem}), which give criteria for log flatness in terms of the chart \eqref{globalchart}.   These should be viewed as analogs of Kato's Chart Criterion for Log Smoothness \cite[3.5]{KK}.

Let $b : P \to C$, $t : Q \to A$ denote the compositions in \eqref{globalchart}, let $A(h,t)$ denote the quotient of the ring $A[Q^{\rm gp} \oplus P]$ by the ideal generated by the expressions \bne{relations} t(q)[q,0] - [0,h(q)] & \quad \quad & (q \in Q), \ene and let $A(h,t) \to C[P^{\rm gp}]$ be the unique $A$-algebra map with $[q,p]  \mapsto  b(p)[h(q)+p].$

\begin{fcc} A $C$-module $M$ is log flat over $Y$ (i.e.\ the quasi-coherent sheaf $M^{\sim}$ on $X$ is log flat over $Y$) iff the $C[P^{\rm gp}]$-module $M[P^{\rm gp}] :=  M \otimes_{\ZZ} \ZZ[P^{\rm gp}]$ is flat over $A(h,t)$. \end{fcc}

This criterion is little more than an unraveling of definitions using Olsson's \'etale cover of $\Log(Y)$ (\S\ref{section:convenientetalecover}) and can also be viewed as an unraveling of Olsson's general ``Chart Criterion for Weak Log ${\bf P}$" \cite[5.31]{Ols}.  In \S\ref{section:independenceofcharts} we prove directly (i.e.\ without making any use of Olsson's stacks) that the above criterion for log flatness is independent of the chosen chart.  In fact, we prove independence of the chosen chart assuming only that $Q$ and $P$ are integral (but not necessary finitely generated).  This motivates our definition of \emph{formal log flatness} (\S\ref{section:formallogflatness}) which makes sense for an arbitrary map of log ringed topoi, makes no mention of stacks, and which reduces to the usual log flatness for morphisms of fine log schemes (Theorem~\ref{thm:formallogflatness}).  Formal log flatness enjoys many of the same formal properties as log flatness and may be of use, for example, in the study of the ``relatively coherent" log schemes of Ogus et al (see the definition in \cite[3.6]{NO} or Ogus' forthcoming work \cite{ORC}), where the usual theory of log flatness is not available.

This is the most ``theoretical" part of the paper.  The basic point is that one can make the construction of $A(h,t) \to C[P^{\rm gp}]$ in the \'etale topos of $X$ itself, replacing $A$ with $f^{-1} \O_Y$, $C$ with $\O_X$, and $h$ with $f^\dagger : f^{-1} \M_Y \to \M_X$ (so there is no longer any choice of chart) then define an $\O_X$-module to be formally log flat iff it satisfies the conclusion of the First Chart Criterion.  It takes a certain amount of work to show that this notion is well-behaved.

Though it is well-suited for theoretical purposes, the First Chart Criterion is not always easy to check in practice, so we next search for a simpler criterion under some additional hypotheses on the map of fine monoids $h : Q \to P$ used in the chart for $f$.    

\subsection{The second chart criterion and graded flatness}  Our \emph{Second Chart Criterion} is really a combination of two flatness criteria.  To explain the first of these criteria, we need a digression on \emph{graded flatness}.  Suppose $B$ is a ring graded by an abelian group $G$ and $M$ is a $B$-module in the \emph{usual ungraded sense}.  If we have a \emph{graded} $B$-module $N$, then we can forget that $N$ is graded and form the usual tensor product $M \otimes_A N$.  This defines a right exact map of abelian categories \bne{Motimesslot} M \otimes_B \slot : \Mod(G,B) & \to & \Mod(B) \ene from graded $B$-modules to $B$-modules.  We say that the $B$-module $M$ is \emph{graded flat over} $(G,B)$ iff \eqref{Motimesslot} is exact.  In \S\ref{section:gradedmodules} we give a general treatment of graded rings and modules, studying graded flatness in particular in \S\ref{section:gradedflatness}.

The chart \eqref{globalchart} determines a ring $B := A \otimes_{\ZZ[Q]} \ZZ[P]$ graded by the abelian group $G := (P/Q)^{\rm gp}$ and a map of rings (in the usual ungraded sense) $B \to C$, hence we can regard a $C$-module $M$ as a $B$-module (in the usual ungraded sense) via restriction of scalars along this map.  Then:

\begin{scc} Suppose the map $h$ in the chart \eqref{globalchart} is an injective map of fine monoids.  Then a $C$-module $M$ is log flat over $Y$ iff $M$ is graded flat over $(G,B)$. \end{scc}

The proof of this theorem relies on a technical result (Theorem~\ref{thm:LhtoLogAQetale}) that says the natural map \be [ \Spec \ZZ[P] \, / \, \Spec \ZZ[(P/Q)^{\rm gp}] ] & \to & \Log(\Spec (Q \to \ZZ[Q])) \ee is representable \'etale whenever $Q \into P$ is an injective map of fine monoids.

The assumption that $h$ be monic in the chart \eqref{globalchart} is not particularly restrictive: Kato's theory of \emph{neat charts} (Theorem~\ref{thm:neatcharts}) says, in particular, that we can always find such a chart fppf locally (or even \'etale locally in characteristic zero). 

Having reformulated log flatness in terms of graded flatness, we next look for a practical criterion for graded flatness of modules over $B = A \otimes_{\ZZ[Q]} \ZZ[P]$.  Our best result along these lines is Theorem~\ref{thm:gradedflatness}, which gives a practical criterion for such flatness when $h : Q \to P$ is \emph{free} in the sense that there is a subset $S \subseteq P$ (called a \emph{basis}) such that the map $Q \times S \to P$ given by $(q,s) \mapsto h(q)+s$ is bijective. 

Freeness is really a concept for \emph{modules} over a monoid---the subject of \S\ref{section:modules}.  In analogy with modules over rings, we define a module over a monoid to be \emph{flat} iff it is a filtered direct limit of free modules.  This notion of flatness enjoys formal properties similar to those of flat modules over rings.  We show in Theorem~\ref{thm:integralflatness} that this abstract notion of flatness coincides with Kato's notion of ``integrality" \cite[4.1(ii),(v)]{KK} on their common domain of definition.

\subsection{Examples and applications} The simplest example of log flatness is the following:

\begin{example} \label{example:smoothdivisor} Let $X$ be a smooth variety equipped with the log structure \be \M_X & := & \{ f \in \O_X : f|_{X \setminus D} \in \O_{X \setminus D}^* \} \ee from a smooth divisor $D \subseteq X$.  In this case an $\O_X$-module $M$ is log flat iff any local defining equation for $D$ in $X$ is $M$-regular.  More generally, the same log flatness criterion holds when $D$ is a Cartier divisor in a variety $\u{X}$ and $X$ is given the log structure $\M_X \subseteq \O_X$ generated by the ideal sheaf $\O_X(-D)$ of $D$ in $\u{X}$. \end{example} 

Here is a mild generalization:

\begin{unnumberedthm} Let $k$ be a field, $P$ a fine monoid, $X := \Spec(P \to k[P])$.  A $k[P]$-module $M$ is log flat iff \be \Tor_1^{k[P]}(M,k[P]/k[I]) & = & 0 \ee for every prime ideal $I \subseteq P$. \end{unnumberedthm}

A geometric version of the above theorem is the following:

\begin{unnumberedthm} Let $T \subseteq X$ be a toric variety over a field $k$ endowed with the usual log structure $ \M_X  =  \{ f \in \O_X : f|_T \in \O_T^* \} $ making it log smooth over $\Spec k$ with the trivial log structure.  A coherent sheaf $M$ on $X$ is log flat iff \be \Tor_1^{X}(M,\O_Z) & = & 0 \ee for every $T$-invariant subvariety $Z \subseteq X$.  \end{unnumberedthm}

It is desirable---especially in applications to moduli problems---to have criteria for log flatness that can be checked on, say, geometric points of $\u{Y}$.  To this end, we establish the following ``\emph{crit\`ere de platitude logarithmique par fibres}" (Theorem~\ref{thm:geometricpoints}):

\begin{fc} Suppose $f : X \to Y$ is an integral morphism of fine log schemes of locally finite presentation and $M$ is a locally finitely presented $\O_X$-module.  Then $M$ is log flat over $Y$ iff $M$ is flat over $\u{Y}$ and $M|X_y$ is log flat over $y$ for each geometric point $y$ of $Y$ (with $y$ given the log structure from $Y$). \end{fc}

\begin{example} \label{example:logcurve} Suppose $\u{f} : \u{X} \to \u{Y}$ is a nodal curve with marking sections $s_i$.  Then F.~Kato \cite{FK} constructed ``canonical" log structures on $\u{X}$ and $\u{Y}$ and a lifting $f : X \to Y$ of $\u{f}$ to a log smooth map of fine (in fact fs) log schemes.  For a reasonable sheaf $M$ on $X$, the Fiberwise Criterion says that log flatness of $M$ is equivalent to flatness of $M$ over $Y$ in the usual sense, plus log flatness after restricting to fibers of $f$ over (strict) geometric points of $Y$.  When $\u{Y} = \Spec k$, $k$ an algebraically closed field, log flatness of $M$ is equivalent to saying $M$ is locally free near the marked points and nodes of $\u{X}$.  \end{example}

In principal, our techniques yield a solid understanding of log flatness whenever $f : X \to Y$ is a \emph{semi-stable degeneration} (\S\ref{section:semistabledegenerations}).  However, for simplicity of exposition in the present paper, we have chosen to focus on the following special case: A \emph{nodal degeneration} is a log smooth morphism $f : X \to Y$ such that at any point of $X$, there is a chart witnessing log smoothness where the map of monoids is either i) an isomorphism, ii) a pushout of $\Delta : \NN \to \NN^2$, or iii) a pushout of $0 \to \NN$.  These nodal degenerations include the log curves of Example~\ref{example:logcurve}, as well as all the morphisms used in ``classical" degeneration theory---i.e.\ the ``expanded degenerations" and ``expanded pairs" introduced by J.~Li \cite{Li} and studied further in \cite{ACFW}.

Corollary~\ref{cor:nodaldegeneration} gives a simple criterion for log flatness when $f : X \to Y$ is a nodal degeneration.  For expanded degenerations/pairs, it turns out that log flatness is equivalent to the notions of ``perfect along..." and ``relative to..." as defined by Wu \cite[2.1, 2.10]{Wu}.

In \S\ref{section:logquotientspace} we introduce the space of \emph{log quotients} of a reasonable sheaf $M$ on $X$ for a map $f : X \to Y$ of fine log schemes (or stacks).  This space is nothing but the (open) locus in the usual relative quotient scheme where the universal quotient sheaf is log flat.  This space has already manifested itself in several special cases:  For example, our Theorem~\ref{thm:stablequotients} says that the space of \emph{stable quotients} introduced in \cite{MOP} is nothing but the (stable locus in) the space of log quotients of $\O^n_{\mathcal{C}}$ for the universal curve $\mathcal{C} \to \mathcal{M}$ over the moduli stack $\mathcal{M}$ of all marked nodal curves (with F.~Kato's log structure discussed in Example~\ref{example:logcurve}).  Wu's ``relative Hilbert scheme" of a smooth pair $(X,D)$ is nothing but the (stable locus in the) log Hilbert scheme of the universal expansion $\mathcal{X} \to \mathcal{T}$ of $(X,D)$.

We prove a fairly general gluing formula (\S\ref{section:gluing}) for spaces of log quotients on nodal degenerations.  This gluing formula can be used to recover all of the usual degeneration formulas in sheaf-theoretic curve-counting theory, though we will not fully explain the details in the present paper.  What is more interesting is the possibility of generalizing this gluing formula to, say, all semi-stable degenerations.  We will return to this point in future work \cite{G2}.

\subsection{Notation} \label{section:notation} We usually write $X$ for a log scheme and $\u{X}$ for its underlying scheme, though we often write $X$ when we mean $\u{X}$, especially when referring to structures that have nothing to do with log geometry (e.g.\ we say ``point of $X$" and ``$\O_X$-module" instead of ``point of $\u{X}$" and ``$\O_{\u{X}}$-module").  For a fine monoid $P$, we let $\AA(P)$ denote the log scheme $\Spec(P \to \ZZ[P])$, $\u{\AA}(P)$ its underlying scheme, $\GG(P) = \GG(P^{\rm gp})$ the group scheme $\Spec \ZZ[P^{\rm gp}]$ (sometimes viewed as a log group scheme with trivial log structure).  There is a tautological action of $\GG(P)$ on $\AA(P)$ (as log schemes).  We let $\AAA(P)$ denote the algebraic stack $[ \u{\AA}(P) / \GG(P) ]$.  (The log algebraic stack $[ \AA(P) / \GG(P) ]$ will not play any significant role.)  All of these constructions are contravariantly functorial in $P$.

\subsection{Leitfaden}  One should probably read \S\S\ref{section:Log}-\ref{section:applications} in order, though \S\S\ref{section:gluingscholium}-\ref{section:descentscholium} can be read independently from anything else.  \S\S\ref{section:modules}-\ref{section:stacks} can be read independently in any order, or can be refered to only as necessary.  The results of \S\ref{section:stacks} are logically necessary for results in \S\S\ref{section:Log}-\ref{section:applications}---but these results are rather technical in nature; one can simply accept them as true, referring to the proofs only as desired.  Some of the graded flatness criteria in \S\ref{section:gradedflatnesscriteria} are used frequently elsewhere in the text; one could probably follow the proofs of these results after a light skimming of \S\ref{section:gradedmodules}.

\section{Stacks of log schemes} \label{section:Log}  In this section we recall some basic facts about Olsson's algebraic stacks $\Log(Y)$ and explain how they are relevant to logarithmic flatness.

\subsection{Tautological section} \label{section:tautologicalsection} The structure map $\Log(Y) \to \u{Y}$ comes with a tautological section $\u{Y} \to \Log(Y)$ taking a $\u{Y}$-scheme $\u{f} : \u{X} \to \u{Y}$ to the $Y$-log scheme $f = (\u{f},\Id) : (\u{X},f^*\M_Y) \to Y$.  The map $\u{Y} \to \Log(Y)$ is in fact representable by open embeddings \cite[3.19]{Ols}---this boils down to the standard fact that the ``strict locus" of a morphism of fine log schemes $X \to Y$ is an open subspace of $X$.

A morphism of fine log schemes $f : X \to Y$ induces a morphism of stacks \be \Log(f) : \Log(X) & \to & \Log(Y)  \\ (g : U \to X) & \mapsto & (fg : U \to Y). \ee  This is a morphism of algebraic stacks which is clearly faithful as a functor, hence it is representable (c.f.\ \S\ref{section:representability}, \cite[8.1.2]{LM}).  We can precompose $\Log(f)$ with the tautological section $\u{X} \to \Log(X)$ to obtain a (representable) morphism of algebraic stacks \bne{Logsubf} \Log_f : \u{X} & \to & \Log(Y) \ene which is, in a sense, the central object of study in this paper.  The diagram \bne{tautologicalsectionsectiondiagram} & \xym{ \u{X} \ar[r]^{\u{f}} \ar[d]_{} & \u{Y} \ar[d] \\ \Log(X) \ar[r] & \Log(Y) } \ene does not commute in general.  Instead we have:

\begin{lem} \label{lem:strictcriteria} For a map of fine log schemes $f : X \to Y$, the following are equivalent: \begin{enumerate} \item $f$ is strict.  \item The diagram \eqref{tautologicalsectionsectiondiagram} is $2$-cartesian. \item The diagram \eqref{tautologicalsectionsectiondiagram} is $2$-commutative.  \end{enumerate} \end{lem}

\begin{proof} Exercise with the definitions. \end{proof}

\subsection{Translation to stacks} It is quite standard to interpret properties of a map of fine log schemes $f : X \to Y$ in terms of corresponding properties of the maps $\Log(f) : \Log(X) \to \Log(Y)$ and $\Log_f : \u{X} \to \Log(Y)$ of algebraic stacks.  

For example, Olsson \cite[4.6]{Ols} showed that $f$ is log smooth (resp.\ formally log smooth) iff $\Log(f)$ is log smooth (resp.\ \dots) iff $\Log_f$ is log smooth (resp.\ \dots).  (Many of these implications are formal exercises.)  The general context for this kind of translation is the following definition \cite[4.1]{Ols}:

\begin{defn} \label{defn:logP} Let {\bf P} be a property of representable morphisms of algebraic stacks.  Then we say that a map of fine log schemes $f : X \to Y$ has property \emph{log} {\bf P} (resp.\ \emph{weak log} {\bf P}) iff $\Log(f)$ (resp.\ $\Log_f$) has property {\bf P}.  \end{defn}

In analogy with Definition~\ref{defn:logP}, the property we are calling ``log flatness" might actually be called ``weak log flatness" and one might reserve ``log flat" to mean the pullback of $M$ to $\Log(X)$ along the structure map $\Log(X) \to X$ is flat over $\Log(Y)$.  However, we will never consider the latter notion in this paper, though Olsson does consider this notion for the case where $M$ is the structure sheaf in \cite[4.6]{Ols}, where he shows that it is equivalence to weak log flatness.

\subsection{Limit preservation} \label{section:limitpreservation} It is clear from the definition of $\Log(Y)$ that for a strict morphism $f : X \to Y$ of fine log schemes, the diagram of stacks $$ \xym@C+20pt{ \Log(X) \ar[d] \ar[r]^-{\Log(f)} & \Log(Y) \ar[d] \\ \u{X} \ar[r]^{\u{f}} & \u{Y} } $$ is $2$-cartesian.  More generally it is not hard to see that $\Log( \slot )$ ``preserves inverse limits" in the sense that when the left diagram below is a cartesian diagram of fine log schemes,\footnote{The underlying diagram of schemes may not be cartesian in this situation.} the right diagram below is a $2$-cartesian diagram of stacks \cite[3.20]{Ols}. \be \xym{ W \ar[r] \ar[d] & X_1 \ar[d] \\ X_2 \ar[r] & Y } & \quad \quad & \xym{ \Log(W) \ar[r] \ar[d] & \Log(X_1) \ar[d] \\ \Log(X_2) \ar[r] & \Log(Y) } \ee 

\subsection{Convenient \'etale cover} \label{section:convenientetalecover} The stack $\Log(Y)$ has a very concrete \'etale cover which is useful for ``computations".  Suppose $Q$ is a fine monoid and $b : Q \to \M_Y(Y)$ is a chart for a fine log scheme $Y$.  In this situation we often write $t : Q \to \O_Y(Y)$ for the composition of $b$ and $\alpha_Y : \M_Y(Y) \to \O_Y(Y)$.  The map $t$ is the same thing as a map of schemes $\u{Y} \to \u{\AA}(Q)$ (notation of \S\ref{section:notation}), which we can compose with the projection $\u{\AA}(Q) \to \AAA(Q)$ to obtain a map of algebraic stacks $\u{Y} \to \AAA(Q)$.  Let $h : Q \to P$ be a map of fine monoids.  Then there is a natural map of algebraic stacks \be \M(\u{Y},h,t) := \u{Y} \times_{\AAA(Q)} \AAA(P) & \to & \Log(Y) \ee explained carefully in \S\ref{section:ontheetalecover}.  If $U \to Y$ is a strict \'etale map and $b : Q \to \M_Y(U)$ is a chart for $\M_Y|_U$, then we can compose the analogous map for $U$ with the natural map $\Log(U) \to \Log(Y)$ to obtain a map \bne{coveringmaps} U \times_{\AAA(Q)} \AAA(P) & \to & \Log(Y) . \ene  Olsson \cite[5.25]{Ols} proved:

\begin{thm} \label{thm:etalecover} {\bf (Olsson)} Let $Y$ be a fine log scheme.  The disjoint union \be \coprod_{U/Y, \, b, \, h} \u{U} \times_{\AAA(Q)} \AAA(P) & \to & \Log(Y) \ee of the maps \eqref{coveringmaps} over all triples $(U/Y,b,h)$ consisting of a strict \'etale map $U \to Y$, a chart $b : Q \to \M_Y(U)$ for $\M_Y|_U$, and a map of fine monoids $h : Q \to P$ is a representable \'etale cover of $\Log(Y)$. \end{thm}

\subsection{Stack of integral morphisms}  \label{section:opensubstacks}  There are various substacks inside $\Log(Y)$ parameterizing fine log schemes over $Y$ with various additional properties.  Since an \'etale map is open on spaces, we can define, for example, the open substack \bne{Logflat} \Log^{\rm int}(Y) & \subseteq & \Log(Y) \ene as the image of all the maps in Theorem~\ref{thm:etalecover} where the map of fine monoids $h : Q \to P$ is flat (equivalently, where $\ZZ[h]: \ZZ[Q] \to \ZZ[P]$ is flat---see Corollary~\ref{cor:flatiffintegral}).  

Whenever $h : Q \to P$ is a flat map of monoids, we have a $2$-commutative diagram of (sufficiently) algebraic stacks $$ \xym{ \u{\AA}(P) \ar[d] \ar[r] & \u{\AA}(Q) \ar[d] \\ \AAA(P) \ar[r] & \AAA(Q) } $$ where the vertical maps are (affine) fpqc covers (because the left map, for example, is essentially by definition an \'etale-locally-trivial principal $\GG(P)$ bundle and $\GG(P)$ is an (affine) fpqc cover of $\Spec \ZZ$) and the top horizontal map is flat, hence the bottom map is also flat.

Since the maps in Theorem~\ref{thm:etalecover} are \'etale, and the stacks $\u{U} \times_{\AAA(Q)} \AAA(P)$ are flat over $\u{U}$ when $h$ is flat (by the discussion above and stability of flatness under base change), the following result is clear (using Lemma~\ref{lem:etaleflatness}):

\begin{lem} \label{lem:LogintflatoverY} The structure map $\Log(Y) \to \u{Y}$ is flat on the open substack $\Log^{\rm int}(Y)$. \end{lem}

\begin{lem} \label{lem:integral} Let $f : X \to Y$ be a map of fine log schemes, $x$ an \'etale point of $X$.  The following are equivalent: \begin{enumerate} \item \label{mf1} $\M_{Y,f(x)} \to \M_{X,x}$ is a flat map of monoids. \item \label{mf2} $\ov{\M}_{Y,f(x)} \to \ov{\M}_{X,x}$ is a flat map of monoids. \item \label{mf3} After possibly replacing $f$ with an \'etale neighborhood of $x$ in $f$, there is a chart $$ \xym{P \ar[r] & \M_X(X) \\ Q \ar[u]^h \ar[r] & \M_Y(Y) \ar[u]  } $$ for $f$ where $h$ is a flat map of fine monoids. \item \label{mf4} There is an \'etale neighborhood $\u{U}$ of $x$ in $\u{X}$ such that the composition $\u{U} \to \u{X} \to \Log(Y)$ factors through the open substack $\Log^{\rm int}(Y)$. \end{enumerate} \end{lem}

\begin{proof} The map in \eqref{mf2} is the sharpening of the map in \eqref{mf1} and the monoids in question are integral, so \eqref{mf1} and \eqref{mf2} are equivalent by Lemma~\ref{lem:sharpeningflatness}.  For \eqref{mf2} implies \eqref{mf3} we start by producing an arbitrary chart on a neighborhood of $x$ in $f$ as indicated, then we replace $Q$ and $P$ with their localizations at the preimage of $\O_{X,x}^*$.  After possibly shrinking to a smaller neighborhood we thus obtain a chart as indicated where the induced maps $\ov{P} \to \ov{\M}_{X,x}$ and $\ov{Q} \to \ov{\M}_{Y,f(x)}$ are isomorphisms, hence $h$ is flat as desired by Lemma~\ref{lem:sharpeningflatness} and the hypothesis in \eqref{mf2}.  This same argument works to prove \eqref{mf3} implies \eqref{mf2} because the aforementioned localization of a map $h$ that was flat to begin with will also be flat.  \eqref{mf3} implies \eqref{mf4} because a chart as in \eqref{mf3} yields a factorization of $\u{X} \to \Log(Y)$ through the corresponding \'etale map to $\Log(Y)$ in Theorem~\ref{thm:etalecover}.  Similarly \eqref{mf4} implies \eqref{mf3} because a map to $\Log(Y)$ which factors through the open substack $\Log^{\rm int}(Y)$ factors \'etale locally through one of the maps in Theorem~\ref{thm:etalecover} where $h$ is flat, and because of the way those maps are defined (see \S\ref{section:ontheetalecover}) it is clear that for any map of schemes $\u{X} \to \Log(Y)$ factoring through such a map, the corresponding map of fine log schemes $X \to Y$ \'etale locally admits a chart using $h$. \end{proof}

\begin{defn} \label{defn:integral} A map of fine log schemes $f :X \to Y$ will be called \emph{integral at} $x$ (resp.\ \emph{integral}) iff it satisfies the conditions of Lemma~\ref{lem:integral} at $x$ (resp.\ at every $x$). \end{defn}

\begin{rem} There seems to be some ambiguous usage of the term ``integral" for a map of integral monoids $Q \to P$ or integral log schemes $X \to Y$, which is one reason I tried to avoid this terminology in \S\ref{section:modules}.  Some of the confusion results from a small error in \cite[4.1(2)]{KK}: one really does need to break those five conditions into two separate ``equivalence classes" as in \cite[4.1(1)]{KK}.  Any map of sharp, fine monoids $h : Q \to P$ with $h^{-1}(0) = \{ 0 \}$ arises as the map $\ov{f}^\dagger_x : \ov{\M}_{X,x} \to \ov{\M}_{Y,f(x)}$ for a map of fine log schemes $X \to Y$ and an \'etale point $x$ of $X$ (and conversely $\ov{f}^\dagger_x$ is always such a map).  For example, take $X \to Y$ to be $\Spec (\slot \to \slot)$ of the diagram of prelog rings: \bne{prelogrings} & \xym{ P \ar[r]^0 & \CC \ar@{=}[d] \\ Q \ar[u]^h \ar[r]^0 & \CC } \ene In particular, the addition map $\NN^2 \to \NN$ will arise.  But this map (or any surjection of integral monoids which is not an isomorphism) clearly satisfies condition \cite[4.1.(iv)]{KK}, though it is certainly not flat (that is, it does \emph{not} satisfy the equivalent conditions (ii), (v) in \cite[4.1]{KK}.)  In \cite[Definition~4.3]{KK}, Kato should take ``integral" to mean integral in the sense of Definition~\ref{defn:integral} above.  (If you only impose the weaker condition in \cite[4.1]{KK} in your definition of ``integral" then it won't even be true that a log smooth integral morphism is flat on underlying schemes.) \end{rem}

In particular, Lemma~\ref{lem:LogintflatoverY} implies that the substack $\Log^{\rm int}(Y) \subseteq \Log(Y)$ is identified with the substack of $\Log(Y)$ consisting of those fine log schemes $f : X \to Y$ over $Y$ where $f$ is integral.

\subsection{Applications to log flatness}

\begin{lem} \label{lem:strictlogflatness} Let $f : X \to Y$ be a strict map of fine log schemes.  Then an $\O_X$-module $M$ is log flat over $Y$ iff $M$ is flat over $\u{Y}$. \end{lem}

\begin{proof} When $f$ is strict we saw in Lemma~\ref{lem:strictcriteria} that $\Log_f : \u{X} \to \Log(Y)$ factors as $\u{f} : \u{X} \to \u{Y}$ followed by the open embedding $\u{Y} \to \Log(Y)$, hence $M$ is flat over $\Log(Y)$ iff it is flat over $\u{Y}$. \end{proof}

\begin{lem} \label{lem:fppflocal} Consider a commutative diagram of log schemes \bne{logschemediagram} \xym{ X' \ar[d]_g \ar[r]^{f'} & Y' \ar[d]^h \\ X \ar[r]^f & Y } \ene with $g$ and $h$ strict and an $\O_X$-module $M$.  If the diagram is cartesian and $M$ is log flat over $Y$, then the $\O_{X'}$-module $g^*M$ is log flat over $Y'$.  If $\u{g}$ is an fppf cover, $\u{h}$ is flat, and $g^*M$ is log flat over $Y'$, then $M$ is log flat over $Y$ (regardless of whether the diagram is cartesian). \end{lem}

\begin{proof}  Since $g$ and $h$ are strict, we have a commutative diagram of algebraic stacks \bne{fppflocaldiagram} & \xym{ \u{X}' \ar@/^2pc/[rr]^{\Log_f} \ar[d]_{\u{g}} \ar[r] & \Log(X') \ar[d] \ar[r] & \Log(Y') \ar[r] \ar[d]_{\Log(h)} & \u{Y}' \ar[d]^{\u{h}} \\ \u{X} \ar@/_2pc/[rr]_{\Log_{f'}} \ar[r] & \Log(X) \ar[r] & \Log(Y) \ar[r] & \u{Y}  } \ene  where the left (Lemma~\ref{lem:strictcriteria}) and right (\S\ref{section:limitpreservation}) squares are cartesian.  For the first statement:  Since \eqref{logschemediagram} is cartesian, the inverse-limit-preservation property of $\Log$ (\S\ref{section:limitpreservation}) implies the middle square in \eqref{fppflocaldiagram} is cartesian so the statement we want is now just the usual stability of flatness under base change for the cartesian diagram obtained by composing the left two squares of \eqref{fppflocaldiagram}.

For the second statement: The hypothesis on $\u{h}$ implies $\Log(h)$ is flat since the right square in \eqref{fppflocaldiagram} is cartesian.  By appropriately factoring the tensor product / pullback functors, the conclusion follows formally from the hypotheses on $\u{g}$ and $g^*M$ (c.f.\ Lemma~\ref{lem:flatness2}). \end{proof}

Now we can derive the \emph{crit\`ere de platitude logarithmique par fibres}, the \emph{openness of the log flat locus}, and the \emph{representability of the log flat locus}.

\begin{thm} \label{thm:geometricpoints} Let $f : X \to Y$ be an integral morphism of fine log schemes, $M$ an $\O_X$-module.  Assume that at least one of the following finiteness hypotheses holds:  \begin{enumerate} \item $X$ and $Y$ are locally noetherian and $M$ is coherent. \item $\u{f}$ and $M$ are of locally finite presentation. \end{enumerate}  Then the following are equivalent: \begin{enumerate} \item $M$ is log flat over $Y$. \item $M$ is flat over $\u{Y}$ and for each strict geometric point $y$ of $Y$, the restriction $M|_{X_y}$ to the fiber $X_y$ over $y$ is log flat over $y$. \end{enumerate} \end{thm}

\begin{proof} Since $f$ is integral, $\Log_f : \u{X} \to \Log(Y)$ factors through the open substack $\Log^{\rm int}(Y) \subseteq \Log^{\rm int}(Y)$, hence log flatness for $M$ is the same as flatness over $\Log(Y)^{\rm int}$.  Since the structure map $\Log^{\rm int}(Y) \to \u{Y}$ is flat (Lemma~\ref{lem:LogintflatoverY}), the conclusion follows by applying the \emph{crit\`ere de platitude par fibres} (Theorem~\ref{thm:fiberwiseflatness}, Remark~\ref{rem:fiberwiseflatness}) to the diagram of algebraic stacks $\u{X} \to \Log^{\rm int}(Y) \to \u{Y}$, noting that, by the strict base change property of log stacks (\S\ref{section:limitpreservation}), the base change of this diagram along a geometric point $\u{y}$ of $\u{Y}$ is the diagram $\u{X_y} \to \Log^{\rm int}(y) \to \u{y}$, where $y$ is $\u{y}$ with the log structure inherited from $Y$.\end{proof}

\begin{thm} \label{thm:openness} Let $f : X \to Y$ be an integral morphism of fine log schemes such that the underlying morphism of schemes $\u{f} : \u{X} \to \u{Y}$ is of locally finite presentation, $M$ a quasi-coherent $\O_X$-module of locally finite presentation.  Then the set $U$ of $x \in X$ for which $M$ is log flat over $Y$ at $x$ is open in $X$.  If we assume that $M$ is flat over $Y$ in the usual sense and has proper support over $Y$, then $V := Y \setminus f(X \setminus U)$ is an open subspace of $Y$ which (when endowed with the log structure inherited from $Y$) is the terminal object in the category of fine log schemes $Y'$ over $Y$ with strict structure map $Y' \to Y$ for which the pullback $M'$ of $M$ to $X' := X \times_Y Y'$ is log flat over $Y'$. \end{thm}

\begin{proof} We apply Theorem~\ref{thm:fiberwiseflatness} to $\u{X} \to \Log^{\rm int}(Y) \to \u{Y}$, arguing much as in the previous proof.    A strict map of fine log schemes $Y' \to Y$ is the same thing as a map of schemes $\u{Y}' \to \u{Y}$ and for such a map the diagram $$ \xym{ \Log(Y') \ar[d] \ar[r] & \Log(Y) \ar[d] \\ \u{Y}' \ar[r] & \u{Y} } $$ is cartesian (\S\ref{section:limitpreservation}).  This translates the conclusion of Theorem~\ref{thm:fiberwiseflatness} into the conclusion of the present theorem. \end{proof}

\section{Log Flatness} \label{section:logflatness}  This section is the heart of the paper.  We begin in \S\S\ref{section:setup}-\ref{section:chartcriteriatheorem} by setting up and proving the Chart Criteria.  To do this, we make heavy use of some results on stacks from \S\ref{section:stacks}.  We use one of our chart criteria to prove the stability of log flatness under base change in \S\ref{section:logflatnessandbasechange}.  In \S\ref{section:formallogflatness} we then discuss formal log flatness.

\subsection{Setup for the chart criteria} \label{section:setup} In this section and the next, we frequently consider a solid diagram of monoids \bne{inputdiagram} & \xym{ P \ar@{.>}[r]^b & C \\ Q \ar[u]^h \ar[r]^t & A \ar@{.>}[u]_f } \ene where $A = \O_Y(Y)$ is the ring of global sections of a scheme $\u{Y}$.  This solid diagram will often be completed as indicated with $f : A \to C$ equal to the map $f^* : \O_Y(Y) \to \O_X(X)$ between rings of global sections induced by a map of schemes $\u{f} : \u{X} \to \u{Y}$.  (There will be abusive uses of the notation ``$f$").  From the solid diagram alone, we can make the following constructions: \bne{constructions} I(h,t) & \subseteq & A[Q^{\rm gp} \oplus P] \\ \nonumber A(h,t) & := & A[Q^{\rm gp} \oplus P]/I(h,t) \\ \nonumber \u{Y}(h,t) & := & \Spec_Y \O_Y[Q^{\rm gp} \oplus P]/I(h,t) \\ \nonumber \L(\u{Y},h,t) & := & [\u{Y} \times_{\AA(Q)} \AA(P) / \GG(P/Q) ] \\ \nonumber \M(\u{Y},h,t) & := & \u{Y} \times_{\AAA(Q)} \AAA(P). \ene Here $I(h,t)$ is the ideal of $A[Q^{\rm gp} \oplus P]$ generated by the elements \bne{idealgenerators} t(q)[q,0] - [0,h(q)] & \quad \quad & (q \in Q) \ene and it is also abuse of notation for the ideal sheaf of $\O_Y[Q^{\rm gp} \oplus P]$ generated by the global sections \eqref{idealgenerators}.  We also have a map of algebraic stacks \bne{LtoM} \L(\u{Y},h,t) & \to & \M(\u{Y},h,t) \ene discussed in \S\ref{section:equivalence}.  

If we also have a completion as indicated, then there is an $A$-algebra map \bne{AhttoCP} A(h,t) & \to & C[P^{\rm gp}] \\ \nonumber [q,p] & \mapsto & b(p)[h(q)+p] \ene and a corresponding map of $\u{Y}$-schemes \bne{XPtoYht} \u{X} \times \GG(P) & \to & \u{Y}(h,t), \ene as well as a map of algebraic stacks \bne{XtoLYht} \u{X} & \to & \L(\u{Y},h,t). \ene  All of these constructions are functorial in the diagram \eqref{inputdiagram}.  

\subsection{The chart criteria} \label{section:chartcriteriatheorem} Suppose now that the completed diagram \eqref{inputdiagram} of \S\ref{section:setup} is obtained from a map $f : X \to Y$ of fine log schemes equipped with a (fine) global chart: \bne{globalchart2} & \xym{ P \ar[r] \ar@/^2pc/[rr]^-b & \M_X(X) \ar[r]^{\alpha_X} & \O_X(X) \\ Q \ar@/_2pc/[rr]_t \ar[u]^h \ar[r] & \M_Y(Y) \ar[u]^{f^\dagger} \ar[r]^{\alpha_Y} & \O_Y(Y) \ar[u]_{f^*} } \ene  Then we also have Olsson's representable \'etale map of algebraic stacks \bne{MYhttoLogY} \M(\u{Y},h,t) & \to & \Log(Y) \ene as in Theorem~\ref{thm:etalecover}, and a factorization of the all-important map $\Log_f : \u{X} \to \Log(Y)$ as below: \bne{factorizationofLogf} & \xym@C+20pt{ \u{X} \ar@/_1pc/[rdd]_{\Log_f} \ar[r]^-{\eqref{XtoLYht}} & \L(\u{Y},h,t) \ar[d]^{\eqref{LtoM}} \\ & \M(\u{Y},h,t) \ar[d]^{\eqref{MYhttoLogY}} \\ & \Log(Y) } \ene

\begin{thm} \label{thm:chartcriteria}  Let $f : X \to Y$ be a map of fine log schemes equipped with a global chart \eqref{globalchart2}.  Referring to the diagram \eqref{factorizationofLogf}, the following are equivalent for an $\O_X$-module $M$: \begin{enumerate} \item \label{cc1} $M$ is log flat over $Y$. \item \label{cc2} $M$ is flat over $\Log(Y)$. \item \label{cc3} $M$ is flat over $\M(\u{Y},h,t)$. \item \label{cc4} $\pi_1^*M$ is flat over $\u{Y}(h,t)$ via tha map of schemes \eqref{XPtoYht}. \end{enumerate} If $h$ is monic these conditions are also equivalent to: \begin{enumerate} \setcounter{enumi}{4} \item \label{cc5} $M$ is flat over $\L(\u{Y},h,t)$. \end{enumerate} \end{thm}

\begin{proof} \eqref{cc1} and \eqref{cc2} are equivalent by definition of ``log flat."\   \eqref{cc2} and \eqref{cc3} (resp.\ \eqref{cc2} and \eqref{cc5}) are equivalent (resp.\ when $h$ is monic) by Lemma~\ref{lem:etaleflatness} applied in the diagram \eqref{factorizationofLogf} because $\M(\u{Y},h,t) \to \Log(Y)$ is representable \'etale by Theorem~\ref{thm:etalecover} (resp.\ $\L(\u{Y},h,t) \to \Log(Y)$ is representable \'etale since it is the base change of the representable \'etale map $\L(h) \to \Log(\AA(Q))$ of Theorem~\ref{thm:LhtoLogAQetale} along $t : \u{Y} \to \u{\AA}(Q)$, as discussed in \S\ref{section:equivalence}).

To prove the equivalence of \eqref{cc3} and \eqref{cc4} we digress momentarily.  Suppose $a : G \times Z \to Z$ is an action of a reasonable group scheme $G$ on a scheme $Z$ and $t : Y \to Z$, $h : W \to Z$ are maps of schemes.  Then by looking at the $2$-cartesian diagram of algebraic stacks $$ \xym{ Y \times_{[Z/G]} W \ar[r] \ar[d] & W \times G \ar[r] \ar[d]^{h \times \Id} & W \ar[d]^h \\ Y \times G \ar[r]^{t \times \Id} \ar[d]_{\pi_1} & Z \times G \ar[d]_{\pi_1} \ar[r]^-{a} & Z \ar[d] \\ Y \ar[r]^t & Z \ar[r] & [Z/G] } $$ we see that $Y \times_{[Z/G]} W$ is the scheme representing the presheaf taking a scheme $U$ to the set $$ \{ (y,g,w) \in Y(U) \times G(U) \times W(U) : g \cdot ty = hw \}. $$ 

By applying the discussion of the previous paragraph with $$(h : W \to Z,\, t : Y \to Z,\, a : G \times Z \to Z)$$ given by the data $$(\u{\AA}(h) : \u{\AA}(P) \to \u{\AA}(Q),\, t : \u{Y} \to \u{\AA}(Q),\, \GG(Q) \times \u{\AA}(Q) \to \u{\AA}(Q)),$$ associated to our Setup, we see that $\u{Y} \times_{\AAA(Q)} \AA(P)$ is the scheme representing the presheaf which takes $\u{U}$ to the set of triples $(f,g,p)$ consisting of a map of schemes $f : \u{U} \to \u{Y}$, a group homomorphism $g : Q^{\rm gp} \to \O_U^*(U)$, and a monoid homomorphism $p : P \to \O_U(U)$ such that $g \cdot f^* t = ph$.  It is elementary to see that this presheaf is represented by the closed subscheme of \be \u{Y} \times \GG(Q) \times \u{\AA}(P) & = & \Spec_Y \O_Y[Q^{\rm gp} \oplus P] \ee defined by the ideal generated by the global sections \bne{idealgeneratorsglobal} t(q)[q,0] - [0,h(q)] & \quad \quad & (q \in Q). \ene  We then obtain a $2$-cartesian diagram of algebraic stacks \bne{2cart} & \xym{ \u{X} \times \GG(P) \ar[d]_{\pi_1} \ar[r] & \u{Y} \times_{\AAA(Q)} \u{\AA}(P) = \u{Y}(h,t) \ar[d]^{\pi} \\ \u{X} \ar[r] \ar@{.>}[ru]^-{(\u{f},\, b)} & \u{Y} \times_{\AAA(Q)} \AAA(P) = \M(\u{Y},h,t) } \ene where the top horizontal arrow corresponds, under the description of $\u{Y}(h,t)$ in the previous paragraph and formula \eqref{constructions}, to the map \eqref{XPtoYht}.  The map $\pi$ here is a base change of the natural map $\u{\AA}(P) \to \AAA(P)$, which is an \'etale-locally-trivial $\GG(P)$ bundle, and is hence, in particular, an fppf cover.\  The base change $\pi_1$ of this $\GG(P)$ bundle to $\u{X}$ is trivial because there is a lift as indicated in \eqref{2cart}. \eqref{cc3} and \eqref{cc4} are hence equivalent because flatness is stable under base change and fppf local on the base.  \end{proof}

When $\u{X}$ and $\u{Y}$ are affine, we can rephrase Theorem~\ref{thm:chartcriteria} in a way that makes no mention of stacks.  If $B$ is a ring graded by an abelian group $G$ and $M$ is an $B$-module in the usual ungraded sense, ``recall" (Definition~\ref{defn:gradedflat3}) that $M$ is called \emph{graded flat over} $(G,B)$ iff the ``usual tensor product" functor \be \slot \otimes_B M : \Mod(G,B) & \to & \Mod(B) \ee from $G$-graded $B$-modules to ungraded $B$-modules is exact.

\begin{cor} \label{cor:chartcriteria} Consider a commutative diagram of monoids \bne{inputdiagram2} & \xym{ P \ar[r]^b & C \\ Q \ar[u]^h \ar[r]^t & A \ar[u]_f } \ene where $h : Q \to P$ is a map of fine monoids and $f : A \to C$ is a map of rings.  Let \be f : X \to Y & := & \Spec(b : P \to C) \to \Spec(t : Q \to A) \ee be the associated map of affine fine log schemes.  Let $B := A \otimes_{\ZZ[Q]} \ZZ[P]$, graded by $G := (P/Q)^{\rm gp}$ in the evident manner and let $B \to C$ be the ring map obtained from \eqref{inputdiagram2}.  The following are equivalent for a $C$-module $M$: \begin{enumerate} \item \label{affinecc1} The quasi-coherent sheaf $M^{\sim}$ on $X$ is log flat over $Y$. \item  \label{affinecc2} The $C[P^{\rm gp}]$-module $M[P^{\rm gp}]$ is flat over $A(h,t)$ via the ring map \eqref{AhttoCP}. \end{enumerate} If $h$ is monic, these conditions are equivalent to \begin{enumerate} \setcounter{enumi}{2} \item \label{affinecc3} $M$, viewed as an ungraded $B$-module via restriction of scalars along $B \to C$, is graded flat over $(G,B)$. \end{enumerate} \end{cor}

\begin{proof}  We apply the theorem to $f : X \to Y$ using the global chart for $f$ obtained tautologically from \eqref{inputdiagram2}.  To see that condition \eqref{cc5} in that theorem is equivalent to condition \eqref{affinecc3} here, we let $g : B \to C$ be natural map; we view this as a map of graded rings $(0,g) : (G,B) \to (0,C)$, giving $C$ the trivial grading.  We note that \be \L(\u{Y},h,t) & = & [ \Spec B / \Spec \ZZ[G] ] \\ & = & \Spec (B/G) \ee in the notation of \S\ref{section:stacksperspective} and the map $\u{X}=\Spec C \to \L(\u{Y},h,t)$ is the map denoted \be \Spec (g/0) : \Spec(C/0) & \to & \Spec(B/G) \ee in Proposition~\ref{prop:gradedflatiffstacksflat}.  According to that proposition, $M^\sim$ is flat over $\L(\u{Y},h,t)$ iff $M$ is graded flat over $(B,G)$ in the sense that \be \slot \otimes_B M : \Mod(G,B) & \to & \Mod(0,C) = \Mod(C) \ee is exact.  This is the same thing as $M$ being graded flat as defined just above (c.f.\ Lemma~\ref{lem:gradedflatnessdefinition}). \end{proof}

\subsection{Log flatness and base change} \label{section:logflatnessandbasechange} In Corollary~\ref{cor:chartcriteria}\eqref{affinecc3} we gave a useful criterion for log flatness under the assumption of the existence of a chart where the map of fine monoids $h : Q \to P$ is injective.  In fact Kato's \emph{Neat Charts Theorem} (Theorem~\ref{thm:neatcharts}) implies, rather surprisingly, that \emph{every} map of fine log schemes admits such a chart, at least fppf locally.  Combining Corollary~\ref{cor:chartcriteria}\eqref{affinecc3} and the latter fact about charts allows us to prove that log flatness is stable under arbitrary base change (Theorem~\ref{thm:logflatnessandbasechange}).

\begin{defn} Let $X$ be a log scheme, $x$ a geometric point of $X$.  A chart $P \to \M_X(X)$ is called a \emph{characteristic chart at} $x$ (or \emph{good at} $x$ in the terminology of Ogus \cite[Def.~2.2.8]{Og}) iff $P \to \ov{\M}_{X,x}$ is an isomorphism. \end{defn}

\begin{lem} \label{lem:characteristiccharts} Let $X$ be a fine log scheme, $x$ a geometric point of $X$.  After possibly replacing $X$ with an fppf neighborhood of $x$, $X$ admits a characteristic chart at $x$.  If we assume the order of the torsion part of $\ov{\M}_{X,x}^{\rm gp}$ is invertible in the \'etale local ring $\O_{X,x}$ (which certainly holds if we work in characteristic zero) then we can replace ``fppf" with ``\'etale." \end{lem}

\begin{proof} This is standard (c.f.\ \cite[2.2.15]{Og} or \cite[Proposition~2.1]{Ols}).  Since $X$ is fine, producing such a chart is the same thing as producing a section of the quotient map $\M_{X,x} \to \ov{\M}_{X,x}$.  Since the kernel of this quotient map is the group $\O_{X,x}^*$, this is equivalent to saying that the \emph{characteristic sequence} $$0 \to \O_{X,x}^* \to \M_{X,x}^{\rm gp} \to \ov{\M}_{X,x}^{\rm gp} \to 0 $$ splits (i.e.\ the class of this extension in $\Ext^1(\ov{\M}_{X,x}^{\rm gp},\O_{X,x}^*)$ is zero).  This is true under the invertibility assumption, which implies that $\O_{X,x}^*$ is divisible by the order of the torsion part of the finitely generated abelian group $\ov{\M}_{X,x}^{\rm gp}$, and it is certainly true if we work with the fppf local ring, for then $\O_{X,x}^*$ is a divisible (i.e.\ injective) abelian group. \end{proof}

\begin{defn} \label{defn:neatcharts}  Let $f : X \to Y$ be a map of fine log schemes, $x$ a geometric point of $X$.  A fine chart $$ \xym{ P \ar[r] & \M_X(X) \\ Q \ar[u]^h \ar[r] & \M_Y(Y) \ar[u]_{f^\dagger} } $$ for $f$ is called \emph{neat} at $x$ iff the following are satisfied: \begin{enumerate} \item $h$ is injective. \item The induced map $P^{\rm gp} / Q^{\rm gp} \to \ov{\M}_{X/Y,x}^{\rm gp}$ is an isomorphism. \item The induced map $\ov{P} \to \ov{\M}_{X,x}$ is an isomorphism. \end{enumerate} \end{defn}

\begin{thm} \label{thm:neatcharts} {\bf (Kato)} Let $f : X \to Y$ be a map of fine log schemes, $x$ a geometric point of $X$, $Q \to \M_Y(Y)$ a fine chart for $Y$.  After possibly replacing $X$ by an fppf neighborhood of $x$ we can extend $Q \to \M_Y(Y)$ to a fine chart for $f$ neat at $x$ (as in Definition~\ref{defn:neatcharts}).  If $Q \to \M_Y(Y)$ is a characteristic chart at $f(x)$, then we can also arrange that $P \to \M_X(X)$ is characteristic at $x$.  If we assume $\QQ \subseteq \O_{X,x}$ then we can replace ``fppf" by ``\'etale" without changing the conclusions. \end{thm}

\begin{proof} (c.f.\ \cite[2.2.18]{Og})  Write $\O_{X,x}$ for the fppf local ring of $X$ at $x$ (or the \'etale local ring of $X$ at $x$ under the $\QQ \subseteq \O_{X,x}$ assumption), so that $\O_{X,x}^*$ is a divisible (i.e.\ injective) abelian group.  Set $N := \Im((f^* \M_Y)_x \to \M_{X,x})$, $S := \Im(Q \to \M_{X,x})$ so we have a surjection $Q \to S$ and an injection $S \subseteq N$.  We are going to construct a commutative diagram of abelian groups (all finitely generated, except possibly $N^{\rm gp}$ and $\M_{X,x}^{\rm gp}$) with exact rows as below. \bne{neatchartsdiagram} & \xym{ 0 \ar[r] & Q^{\rm gp} \ar[d] \ar[r] & L \ar[d] \ar[r] & \ov{\M}_{X/Y,x}^{\rm gp} \ar@{=}[d] \ar[r] & 0 \\ 0 \ar[r] & S^{\rm gp} \ar[d] \ar[r] & E \ar[r] \ar[d] & \ov{\M}_{X/Y,x}^{\rm gp} \ar@{=}[d] \ar[r] & 0 \\ 0 \ar[r] & N^{\rm gp} \ar[d] \ar[r] & \M_{X,x}^{\rm gp} \ar[r] \ar[d] & \ov{\M}_{X/Y,x}^{\rm gp} \ar@{=}[d] \ar[r] & 0 \\ 0 \ar[r] & \ov{\M}_{Y,f(x)} \ar[r] & \ov{\M}_{X,x}^{\rm gp} \ar[r] & \ov{\M}_{X/Y,x}^{\rm gp} \ar[r] & 0 } \ene  The bottom two rows (and the map between them) and the left column are obtained from the definitions.

Since $Q \to \M_Y(Y)$ is a chart, the induced map $Q^{\rm gp} \to \ov{\M}_{Y,y}^{\rm gp}$ is surjective and we see that $N^{\rm gp} \subseteq \M_{X,x}^{\rm gp}$ is generated by $S^{\rm gp}$ and $\O_{X,x}^*$.  Hence we have a surjection $\O_{X,x}^* \to N^{\rm gp} / S^{\rm gp}$, hence $N^{\rm gp} / S^{\rm gp}$ is divisible (a quotient of a divisible group is divisible) and therefore the map \bne{ext1map} \Ext^1(\ov{\M}_{X/Y,x}^{\rm gp} , S^{\rm gp}) & \to & \Ext^1(\ov{\M}_{X/Y,x}^{\rm gp}, N^{\rm gp}) \ene induced by the inclusion $S^{\rm gp} \subseteq N^{\rm gp}$ is surjective.  Using the Yoneda Ext description of \eqref{ext1map}, a lifting of the (isomorphism class of the) third row of \eqref{neatchartsdiagram} under \eqref{ext1map} yields the second row and the map from the second to the third row in \eqref{neatchartsdiagram}.  

Since $\Ext^2=0$ in the category of abelian groups, the surjection $Q^{\rm gp} \to S^{\rm gp}$ induces a surjection \bne{secondext1map} \Ext^1(\ov{\M}_{X/Y,x},S^{\rm gp}) & \to & \Ext^1(\ov{\M}_{X/Y,x},Q^{\rm gp}). \ene  Choosing a lifting of the second row of \eqref{neatchartsdiagram} yields the first row of \eqref{neatchartsdiagram} and the map from it to the second row of \eqref{neatchartsdiagram}.  This completes the construction of \eqref{neatchartsdiagram}.

Since the map $Q^{\rm gp} \to \ov{\M}_{Y,y}^{\rm gp}$ in \eqref{neatchartsdiagram} is surjective (because $Q \to \M_Y(Y)$ is a chart), the Snake Lemma applied to the top and bottom rows of \eqref{neatchartsdiagram} shows that $L \to \ov{\M}_{X,x}^{\rm gp}$ is surjective.  Let $P \subseteq L$ be the preimage of $\M_{X,x} \subseteq \M_{X,x}^{\rm gp}$ under $L \to \M_{X,x}^{\rm gp}$.  It is a standard exercise (\cite{KK} or \cite[2.2.11]{Og}) to show that $P$ is fine and $P \to \M_{X,x}$ is a chart (hence lifts to a chart $P \to \M_X(X)$ after replacing $X$ with an fppf / \'etale neighborhood of $x$ as appropriate).  It is also straightforward to check that the injection $h : Q \to P$ induced by $Q^{\rm gp} \into L$ will serve as the $h$ for a neat chart for $f$ as in Definition~\ref{defn:neatcharts}.  From the definition of $P$, it is clear that $P \to \ov{\M}_{X,x}$ is surjective; if $Q \to \M_Y(Y)$ is a characteristic chart, then $Q^{\rm gp} \to \ov{\M}_{Y,y}^{\rm gp}$ is an isomorphism, hence $L^{\rm gp} \to \ov{\M}_{X,x}^{\rm gp}$ is an isomorphism by the Five Lemma applied to the top and bottom rows of \eqref{neatchartsdiagram}, hence $P \to \ov{\M}_{X,x}$ is an isomorphism because injectivity can be checked after groupifying. \end{proof}

\begin{cor} \label{cor:neatcharts} Let $f : X \to Y$ be a map of fine log schemes, $x$ a geometric point of $X$.  After possibly replacing $f$ with an fppf neighborhood of $x$ in $f$ (or an \'etale neighborhood if we assume $\QQ \subseteq \O_{Y,f(x)}$), there is a neat chart for $f$ as in Definition~\ref{defn:neatcharts} where $Q \to \M_Y(Y)$ (resp.\ $P \to \M_X(X)$) is a characteristic chart at $f(x)$ (resp.\ $x$). \end{cor}

\begin{proof} First apply Lemma~\ref{lem:characteristiccharts} to build a characteristic chart $Q \to \M_Y(Y)$ at $f(x)$ (shrink $Y$ if necessary), then lift it to the desired chart for $f$ by Theorem~\ref{thm:neatcharts}. \end{proof}

\begin{thm}  \label{thm:logflatnessandbasechange} Let $f : X \to Y$ be a map of fine log schemes, $\F$ a quasi-coherent sheaf on $X$ log flat over $Y$.  For any map $Y' \to Y$ of fine log schemes, the quasi-coherent sheaf $\F' = \pi_1^* \F$ on $X \times_Y Y'$ is log flat over $Y'$.   \end{thm}

\begin{proof}  We can factor $Y' \to Y$ as $Y' \to Y'' \to Y$ where $Y'' \to Y$ is strict and $Y' \to Y''$ is the identity on underlying schemes.  By Lemma~\ref{lem:fppflocal} we already know the theorem holds when $Y' \to Y$ is strict.  We thus reduce to the case where $Y' \to Y$ is the identity on underlying schemes.  Lemma~\ref{lem:fppflocal} also says that log flatness is strict fppf local in nature, so by Corollary~\ref{cor:neatcharts} we can reduce to the case where $\u{X} = \Spec C$ and $\u{Y}=\u{Y}' = \Spec A$ are affine (so $\F = M^\sim$ for some $C$-module $M$) and we have neat global charts  \bne{neatglobalsharts} \xym{ P \ar[r] & \M_X(X) \ar[r] & C = \O_X(X) \\ Q \ar[u]^h \ar[r] & \M_Y(Y) \ar[u] \ar[r] & A = \O_Y(Y) \ar[u] } &  & \xym{ Q' \ar[r] & \M_{Y'}(Y) \ar[r] & A = \O_Y(Y) \\ Q \ar[u] \ar[r] & \M_Y(Y) \ar[u] \ar[r] & A = \O_Y(Y) \ar@{=}[u] } \ene for $f$ and $Y' \to Y$.  In particular, $h$ is an injective map of fine monoids, so by Corollary~\ref{cor:chartcriteria}\eqref{affinecc3}, log flatness of $\F = M^\sim$ over $Y$ is equivalent to graded flatness of $M \in \Mod(C)$ over $(G,B)$, where $B := A \otimes_{\ZZ[Q]} \ZZ[P]$ and $G := (P/Q)^{\rm gp}$ (and $M$ is viewed as an (ungraded) $B$-module via the natural map $B \to C$).  Set $P' := P \oplus_Q Q'$.  We have $P'/Q' = P/Q$ and since $\ZZ[ \slot ]$ preserves direct limits, we see that $B = A \otimes_{\ZZ[Q']} \ZZ[P']$ (unambiguously graded by $G = (P/Q)^{\rm gp} = (P'/Q')^{\rm gp}$).   

It is understood in the statement of the theorem that the fibered product $X'$ is the one taken in fine (or, equivalently, integral) log schemes, not the one taken in arbitrary (coherent) log schemes.  The underlying scheme of $X'$ is given by $\Spec$ of the ring $C' := C \otimes_{\ZZ[P']} \ZZ[(P')^{\rm int}]$ and the natural map $h' : Q' \to (P')^{\rm int}$ serves as (the fine monoid map in) a global chart for $X' \to Y'$.  The map $h'$ is also injective: It suffices to check that $(h')^{\rm gp}$ is injective since $Q'$ and $(P')^{\rm int}$ are integral; but $P^{\rm gp} = (P')^{\rm gp}$ so $(h')^{\rm gp}$ is just the groupification of $Q' \to P'$, which is injective since it is a pushout of the groupification of $h$.  Applying Corollary~\ref{cor:chartcriteria}\eqref{affinecc3} again, we see that log flatness of $\F' = (M')^\sim$ (where $M' := M \otimes_C C'$) is equivalent to graded flatness of $M'$ over $(G,B')$, where $B' = A \otimes_{\ZZ[Q']} \ZZ[(P')^{\rm int}]$ and $G = (P')^{\rm gp} / (Q')^{\rm gp}$ is the same $G$ we've been dealing with all along.  Since $$ \xym{ (G,B) \ar[r] \ar[d] & (G,B') \ar[d] \\ (0,C) \ar[r] & (0,C')  } $$ is a pushout diagram of graded rings (\S\ref{section:gradedflatnessandbasechange}), graded flatness of $M$ over $(G,B)$ implies graded flatness of $M'$ over $(G,B')$ by the (rather limited) stability of graded flatness under base change (Proposition~\ref{prop:gradedflatnessandbasechange}).   \end{proof}

\begin{rem} \label{rem:logflatnessandbasechange} In the above proof, we encountered a variant of the following situation, which arises frequently.  Suppose $$ \xym{ P \ar[r] & P' \ar[r] & C \\ Q \ar[r] \ar[u]^h & Q' \ar[u]_h \ar[r] & A \ar[u] } $$ is a commutative diagram of monoids where $A \to C$ is a map of rings.  Assume $h$ is a monomorphism of fine monoids, $Q'$ is a fine monoid, and the left square is a pushout diagram of monoids.  Assume furthermore that the finitely generated monoid $P'$ is actually fine (this is automatic if $h$ is flat (i.e.\ integral)), so the left square is also a pushout in the category of integral monoids.   This implies that $h'$ is also monic (since this can be checked on groupifications and monomorphisms of abelian groups are stable under pushout).  Since the left square is a pushout, it stays a pushout after applying $\ZZ[ \slot ]$, and the natural map \be A \otimes_{\ZZ[Q]} \ZZ[P] & \to & A \otimes_{\ZZ[Q']} \ZZ[P'] \ee is an isomorphism of rings graded by $G := (P/Q)^{\rm gp} = (P'/Q')^{\rm gp}$.   In particular, the two resulting notions of graded flatness for a $C$-module $M$ are the same.  Since Corollary~\ref{cor:chartcriteria}\eqref{affinecc3} says these notions of graded flatness are equivalent to log flatness of $M^\sim$ for the two maps of log schemes \be X := \Spec(P \to C) & \to & \Spec (Q \to A) =: Y \\ X' := \Spec (P' \to C) & \to & \Spec (Q' \to A) =: Y' \ee lying over the same map of schemes \be (\u{X} \to \u{Y}) & = & \Spec(A \to C), \ee we see that log flatness of $M^\sim$ over $Y$ is the same as log flatness of $M^\sim$ over $Y'$.  This is a sense in which log flatness is local for certain \emph{non-strict} base changes (which in this case are the identity on underlying schemes). \end{rem}

\subsection{Independence of charts} \label{section:independenceofcharts}  The purpose of this section is to prove (directly, without making any use of the log stacks $\Log(Y)$) that the criterion for log flatness in Corollary~\ref{cor:chartcriteria}\eqref{affinecc2} is independent of the chosen chart.  In fact we will prove a more general statement which motivates the definition of formal log flatness in \S\ref{section:formallogflatness}.

\begin{lem} \label{lem:strictflat} For a strict map $h : Q \to P$ of integral monoids, the following are equivalent: \begin{enumerate} \item $h$ is free. \item $h$ is flat. \item $h$ is injective. \end{enumerate} \end{lem}

\begin{proof} The only nontrivial implication is that injective implies free.  By Lemma~\ref{lem:strict}, $h$ strict implies $h$ is a pushout of $h^* : Q^* \to P^*$, which is an injective, hence free (Example~\ref{example:groupflatness}) map of groups when $h$ is injective.  Free maps are stable under pushout (Lemma~\ref{lem:basechange} or Lemma~\ref{lem:freemorphisms}).  \end{proof} 

\begin{lem} \label{lem:independenceofcharts} Suppose $P$, $P'$, $Q$, and $Q'$ are fine monoids, $A$ is a ring, and \bne{comparisondiagram} & \xym{ P \ar[r]^b & P' \ar[r]^{s'} & C \\ Q  \ar[u]^h \ar[r]^a & Q' \ar[u]_{h'} \ar[r]^{t'} & A \ar[u]_f } \ene is a commutative diagram of monoids with $a$ and $b$ strict.  Set $s := s'b$, $t := t'a$, $\u{Y} := \Spec A$.  Then the induced map \bne{inducedetalemap} \M(\u{Y},h',t')  & \to & \M(\u{Y},h,t) \ene of algebraic stacks is representable \'etale.  If, furthermore, $b :P \to P'$ is flat (c.f.\ Lemma~\ref{lem:strictflat}), then the induced map of schemes \bne{secondinducedmap} \u{Y} \times_{\AAA(Q')} \u{\AA}(P') & \to & \u{Y} \times_{\AAA(Q)} \u{\AA}(P) \\ = \nonumber \Spec ( A(h,t) & \to & A(h',t') ) \ene is flat. \end{lem}

\begin{proof} We consider a big $2$-cartesian diagram which defines $F$, $F'$, $F''$ (only the bottom three rows are relevant for the first statement): \bne{hugecartesiandiagram} & \xym@C-10pt@R-10pt{ \u{Y} \times_{\AAA(Q')} \u{\AA}(P) \ar[rr] \ar[d]_{flat} & & \u{\AA}(P') \ar[d]^{flat} \\ \u{Y} \times_{\AAA(Q')} \u{\AA}(P) \ar[dd]_{f.f.} \ar[rr] \ar[rd] & & F'' \ar[dd]|\hole_<<<<{f.f.} \ar[rd]^{\acute{e}t} \\ & \u{Y} \times_{\AAA(Q)} \u{\AA}(P) \ar[dd]_<<<<{f.f.} \ar[rr] & & F' \ar[r]^-{\acute{e}t} \ar[dd]^{f.f.} & \u{\AA}(P) \ar[dd]^{f.f.} \\ \M(\u{Y},h',t') \ar[rr]|(.557)\hole \ar[rd]_{\acute{e}t} & & \AAA(P') \ar[rd]^{\acute{e}t} \\ & \M(\u{Y},h,t) \ar[rr] \ar[d] & & F \ar[r]^-{\acute{e}t} \ar[d] & \AAA(P) \ar[d] \\ & \u{Y} \ar[rr] & & \AAA(Q') \ar[r]^-{\acute{e}t} & \AAA(Q) } \ene where the indicated maps are (representable) \'etale, flat, or faithfully flat.  This requires some justification: First, the map $\AAA(Q') \to \AAA(Q)$ is (representable) \'etale by Theorem~\ref{thm:AAAhetale} because we assume $a$ is strict; the two maps above it are then \'etale by base change.  The map $\AAA(P') \to F$ is (representable) \'etale by the two-out-of-three property because the composition $\AAA(P') \to \AAA(P)$ is \'etale for the same reason that $\AAA(Q') \to \AAA(Q)$ is \'etale; then \eqref{inducedetalemap} and $F'' \to F'$ are \'etale by base change.  For the ``furthermore," the assumption that $b$ is flat means $\u{\AA}(b) : \u{\AA}(P') \to \u{\AA}(P)$ is flat (because $b$ flat means $\ZZ[P] \to \ZZ[P']$ is flat by Theorem~\ref{thm:freetofree}).  Since the flat map $\u{\AA}(b)$ factors as $\u{\AA}(P') \to F''$, followed by the \'etale map $F'' \to \u{\AA}(P)$, the map $\u{\AA}(P') \to F''$ is flat by Lemma~\ref{lem:etaleflatness}.  The desired result is now immediate from stability of flatness under composition and base change.     \end{proof}

\begin{lem} \label{lem:fineapproximation} Given a commutative square of integral monoids \bne{square} & \xym{ P \ar[r]^b & P' \\ Q  \ar[u]^h \ar[r]^a & Q' \ar[u]_{h'}  } \ene where $a$ and $b$ are strict, there exists a filtered partially ordered set $I$ and diagrams \bne{squarei} & \xym{ P_i \ar[r]^{b_i} & P'_i \\ Q_i  \ar[u]^{h_i} \ar[r]^{a_i} & Q'_i \ar[u]_{h'_i}  } \ene of fine monoids lying over \eqref{square}, natural in $i \in I$, such that: \begin{enumerate} \item The maps $a_i$ and $b_i$ are strict for every $i \in I$.  \item The maps $P_i \to P$, $Q_i \to Q$, $P_i' \to P_i$, and $Q_i' \to Q_i$ are injective for every $i \in I$.  \item The direct limit of the squares \eqref{squarei} is \eqref{square}. \end{enumerate}  Furthermore:  \begin{enumerate} \item \label{furthermore1} If $Q$ and $P$ are fine, then we can take $Q_i=Q$ and $P_i=P$ for all $i$.  \item \label{furthermore2} If $b$ is injective, we can take each $b_i$ injective and hence flat (c.f.\ Lemma~\ref{lem:strictflat}). \end{enumerate}  \end{lem}

\begin{proof} Let $I$ be the set of quadruples $(Q_i,P_i,G_i,H_i)$ where $Q_i \subseteq Q$ is a finitely generated (equivalently fine) submonoid of $Q$, $P_i \subseteq P$ is a fine submonoid of $P$, $G_i$ is a finitely generated subgroup of $(Q')^*$ containing $h(Q_i^*)$, and $H_i$ is a finitely generated subgroup of $(P')^*$ containing $h'(G_i)$ and $b(P_i^*)$.  For such a quadruple $i=(Q_i,P_i,G_i,H_i)$, we set \be Q_i' & := & Q_i \oplus_{Q_i^*} G_i \\ P_i' & := & P_i \oplus_{P_i^*} H_i. \ee  The natural maps yield a diagram \eqref{squarei} lying over \eqref{square} where $a_i$ and $b_i$ are strict by Lemma~\ref{lem:strict} because they are pushouts of the group homomorphisms $Q_i^* \to G_i$ and $P_i^* \to H_i$ respectively.  If we order $I$ by coordinate-wise inclusion, then it is clear that $I$ is filtered and the diagrams \eqref{squarei} are natural in $i \in I$.  Clearly the limit of all our quadruples is $(Q,P,(Q')^*,(P')^*)$, so the limit of the diagrams \eqref{squarei} is \eqref{square} because we have $Q' = Q \oplus_{Q^*} (Q')^*$ and $P' = P \oplus_{P^*} (P')^*$ by strictness of $a$ and $b$ (Lemma~\ref{lem:strict}).  For the first furthermore, just note that if $Q$ and $P$ are already fine, we can replace $I$ with the cofinal subset consisting of those $i = (Q_i,P_i,G_i,H_i)$ where $Q_i=Q$ and $P_i-P$.  For the second furthermore, note that by construction, the map $b_i$ is a pushout of the injective (hence flat) map of groups $b : P_i^* \to H_i^*$.  \end{proof}

\begin{thm} \label{thm:independenceofcharts} Suppose $P$, $P'$, $Q$, and $Q'$ are integral (but not necessarily finitely generated) monoids, $f : A \to C$ is a ring homomorphism, and we have a commutative diagram \eqref{comparisondiagram} as in Lemma~\ref{lem:independenceofcharts}.  Assume that $a$ (resp.\ $b$) induces an isomorphism $Q_t^a \cong (Q')_{t'}^a$ (resp.\ $P_s^a \cong (P')_{s'}^a$) on associated log structures on $A$ (resp.\ $B$).  Let \bne{ringdiagram} & \xym{ A(h',t') \ar[r] & C[(P')^{\rm gp}] \\ A(h,t) \ar[u] \ar[r] & C[P] \ar[u] } \ene be the resulting commutative diagram of rings, where the horizontal arrows are the maps \eqref{AhttoCP}.  Let $M$ be a $C$-module.  Consider the following conditions: \newline \quad {\bf P}: \; $M[P^{\rm gp}]$ is flat over $A(h,t)$. \newline \quad {\bf P}$'$: \; $M[(P')^{\rm gp}]$ is flat over $A(h',t')$.  \newline Then: \begin{enumerate} \item \label{st1} If $Q' = (t^{-1}A^*)^{-1} Q$ and $P' = (s^{-1}C^*)^{-1} P$,  then we have ${\bf P}$ iff ${\bf P}'$.  \item \label{st2} If $Q$, $Q'$, $P$, $P'$ are fine, then we have ${\bf P}$ iff ${\bf P}'$. \item \label{st3} If $Q$ and $P$ are fine, then we have ${\bf P}$ implies ${\bf P}'$ and we have the converse if, furthermore, $b$ is injective. \end{enumerate} \end{thm}

\begin{proof} The submonoid $F := t^{-1} A^*$ (resp.\ $G := s^{-1} C^*$) of $Q$ (resp.\ $P$) is a face (a submonoid whose complement is a prime ideal).  We note for usage in the proof of \eqref{st2} that this implies $F$ and $G$ are fine when $Q$ and $P$ are fine,\footnote{If $P$ is any monoid, $F$ is any face, and $\Sigma \subseteq P$ generates $P$, then $\Sigma \cap F$ generates $F$.} and hence that $F^{-1}P$ and $G^{-1} Q$ are fine.  We also note that $Q/F = \ov{F^{-1}Q}$ is the characteristic monoid of the log structure $Q_t^a$ associated to $a : Q \to A$.

\noindent \eqref{st1}: Here we suppress notation for the maps $a$ and $b$ and abusively write $t'=t$ and $h'=h$.  In this case we have $Q^{\rm gp} = (Q')^{\rm gp}$ and $P^{\rm gp} = (P')^{\rm gp}$.  Consider the commutative diagram \bne{ff} & \xym{ A[Q^{\rm gp} \oplus P] \ar[r] \ar[d] & A[(Q')^{\rm gp} \oplus P'] \ar[d] \\ A(h,t) \ar[r] & A(h',t') } \ene where the vertical maps are the quotients by the ideals $I(h,t)$ and $I(h',t')$ which define $A(h,t)$, $A(h',t')$ (\S\ref{section:setup}).  

I claim that \eqref{ff} is a a pushout.  We need to show that the ideal $I(h',q')$ is generated by the image of $I(h,q)$.  Certainly the image of $I(h,q)$ is contained in $I(h',q')$, so we need to show that for $q' \in Q'$, the equality \bne{fff0} t(q')[q',0] & = & [0,h(q')] \ene in $A(h',t')$ already holds in $A[(Q')^{\rm gp} \oplus P']$ modulo (the image of) $I(h,t)$.   We can write $q' = q-f$ for some $q \in Q$, $f \in F \subseteq Q$.  By definition of $I(h,t)$, we have \bne{fff1} t(q)[q,0] & = & [0,h(q)] \\ \label{fff2} t(f)[f,0] & = & [0,h(f)] \ene in $A[(Q')^{\rm gp} \oplus P']$ modulo $I(h,t)$.  But since $f \in F$, both sides of \eqref{fff2} are units in $A[(Q')^{\rm gp} \oplus P']$, so we also have \bne{fff3} t(f)^{-1}[-f,0] & = & [0,-h(f)] \ene in  $A[(Q')^{\rm gp} \oplus P']$ modulo $I(h,t)$.  But the product of \eqref{fff3} and \eqref{fff1} is \eqref{fff0}, so we have \eqref{fff0} modulo $I(h,t)$ as desired.  

Since $P \to P'$ is a localization (and $Q^{\rm gp} = (Q')^{\rm gp}$), the top horizontal arrow in \eqref{ff} is flat, hence so is the bottom horizontal arrow because the square \eqref{ff} is a pushout.  This shows that the left vertical arrow in \eqref{ringdiagram} is flat, and the right vertical arrow there is an isomorphism (hence in particular faithfully flat) so {\bf P} and {\bf P}$'$ are equivalent by Lemma~\ref{lem:flatness2}.

\noindent \eqref{st2}: By replacing $Q$, $P$ (resp.\ $Q'$, $P'$) with $F^{-1}Q$, $G^{-1}P$ (resp.\ $(F')^{-1}Q'$, $(G')^{-1} P'$) and using the result of \eqref{st1}, we can assume that $a$ and $b$ are strict because the condition that $a$ and $b$ induce isomorphisms on associated log structures is equivalent to the condition that $a$ and $b$ induces isomorphisms on characteristic monoids of those log structures, but once we localize as described, those characteristic monoids are just the usual sharpenings and the induced maps between them are just $\ov{a}$ and $\ov{b}$.  Note that the discussion before \eqref{st1} shows that this localization trick doesn't destroy the hypothesis that $Q$, $P$, $Q'$, $P'$ are fine.  We saw in the course of proving \eqref{thm:chartcriteria} that {\bf P} (resp.\ {\bf P}$'$) is equivalent to saying that the quasi-coherent sheaf $M^{\sim}$ on $\Spec C$ is flat over $\M(\u{Y},h,t)$ (resp.\ $\M(\u{Y},h',t')$), where $\u{Y} := \Spec A$.  But the map \bne{inducedetalemap2} \M(\u{Y},h',t')  & \to & \M(\u{Y},h,t) \ene is \'etale by Lemma~\ref{lem:independenceofcharts} because $a$ and $b$ are strict, so these latter two conditions are equivalent by Lemma~\ref{lem:etaleflatness}.

\noindent \eqref{st3} By the result of \eqref{st1}, we can reduce to the case where $a$ and $b$ are strict, as we did in Step 2 (the localization trick won't destroy the hypothesis that $b$ is injective).  By the first ``furthermore" in Lemma~\ref{lem:fineapproximation} we can then find a filtered partially ordered set $I$ and diagrams of integral monoids \bne{comparisondiagrami} & \xym{ P \ar@/^2pc/[rr]^b \ar[r]^{b_i} & P_i \ar[r] & P' \\ Q \ar@/_2pc/[rr]_a \ar[u]^{h} \ar[r]^{a_i} & Q_i \ar[u]_{h_i} \ar[r] & Q' \ar[u]_{h'} } \ene natural in $i \in I$ such that the monoids $Q_i$, $P_i$ are fine, the maps $a_i$ and $b_i$ are strict, the direct limit of the $Q_i$ is $Q'$ and the direct limit of the $P_i$ is $P'$.  Let $t_i := t'|_{Q_i} : Q_i \to A$ be the induced map.  The diagram \eqref{ringdiagram} factors as \bne{ringdiagrami} & \xym{ A(h',t') \ar[r] & C[(P')^{\rm gp}] \\ A(h_i,t_i) \ar[u] \ar[r] & C[P_i^{\rm gp}] \ar[u]  \\ A(h,t) \ar[u] \ar[r] & C[P^{\rm gp}] \ar[u] } \ene naturally in $i \in I$.  The relevant functors ($A( \slot, \slot)$, $C[\slot]$, and $M[ \slot ]$) clearly commute with filtered direct limits, so the top row of \eqref{ringdiagrami} is the filtered direct limit of the middle rows and $M[(P')^{\rm gp}]$ is the filtered direct limit of the $M[P_i^{\rm gp}]$. 

To see that ${\bf P}$ implies ${\bf P}'$, note that for each $i \in I$, part \eqref{st2} says that the condition ${\bf P}$ implies the following condition ${\bf P}_i$:  The $C[P_i^{\rm gp}]$-module $M[P_i^{\rm gp}]$ is flat over $A(h_i,t_i)$.  But then we have ${\bf P}'$ because a filtered direct limit of flats is flat (note that we are using the variant of this fact where the ring is also varying in the filtered direct limit system).

To see that ${\bf P}'$ implies ${\bf P}$ when $b$ is injective, we can assume (by the second ``furthermore" in Lemma~\ref{lem:fineapproximation}) the $b_i$ were taken flat / injective.  The maps $A(h,t) \to A(h_i,t_i)$ are then flat by Lemma~\ref{lem:independenceofcharts}, hence their filtered limit $A(h,t) \to A(h',t')$ is also flat.  But $b$ injective implies $b^{\rm gp} : P^{\rm gp} \to (P')^{\rm gp}$ injective, which implies $C[P^{\rm gp}] \to C[(P')^{\rm gp}]$ faithfully flat (Corollary~\ref{cor:AGtoAHflat}), so we can conclude ${\bf P}'$ implies ${\bf P}$ using Lemma~\ref{lem:flatness2} because \be M[(P')^{\rm gp}] & = & M[P^{\rm gp}] \otimes_{C[P^{\rm gp}]} C[(P')^{\rm gp}]. \ee \end{proof}

\begin{cor} \label{cor:independenceofcharts} Fix a map of fine log rings \bne{integrallogrings} & \xym{ \M_C \ar[r] & C \\ \M_A \ar[u] \ar[r] & A \ar[u] } \ene and a $C$-module $M$.  For a chart $K$ for this map of fine log rings as below \bne{CHART}  \xym{ P(K) \ar[r] & \M_C \ar[r] & C \\ Q(K) \ar[u]^h \ar[r] & \M_A \ar[u] \ar[r] & A \ar[u] } \ene with $Q(K)$ and $P(K)$ integral but not necessarily fine, consider the condition:

${\bf P}(K)$ \; := \; $M[P(K)^{\rm gp}]$ is flat over $A(h,t)$.

\noindent Then ${\bf P}(K)$ holds for one such chart iff it holds for all such charts (including the terminal chart $T$ with $Q(T)=\M_A$ and $P(T) = \M_C$). \end{cor}

\begin{proof} Suppose ${\bf P}(K)$ holds for some chart $K$.  Fix some other chart $K'$ and let us prove that ${\bf P}(K')$ holds.   Since the log structures are fine, by standard chart production techniques \cite[2.8-2.10]{KK}, we can find charts $L$ and $L'$ with $Q(L)$, $P(L)$, $Q(L')$, $P(L')$ fine and maps of charts $L \to K$ and $L' \to K'$.  Replacing $Q(L)$ (resp.\ $P(L)$) with its image in $Q(K)$ (resp.\ $P(K)$) (this won't destroy the property of being a chart) if necessary, we can assume $Q(L) \to Q(K)$ and $P(L) \to P(K)$ are injective and similarly for the primed charts.  Then ${\bf P}(K)$ implies ${\bf P}(L)$ (Theorem~\ref{thm:independenceofcharts}\eqref{st3}), which implies ${\bf P}(L')$ (Theorem~\ref{thm:independenceofcharts}\eqref{st2}), which implies ${\bf P}(K')$ (Theorem~\ref{thm:independenceofcharts}\eqref{st3}). \end{proof}

\begin{rem} It may be that the conditions ${\bf P}$ and ${\bf P}'$ in Theorem~\ref{thm:independenceofcharts} are always equivalent without imposing any of the finiteness conditions on the integral monoids, and it may very well be that ``fine" can be relaxed to ``integral" in Corollary~\ref{cor:independenceofcharts}, but I couldn't prove these statements, mostly because the finiteness hypotheses are so deeply entrenched that it was all I could do to extract the crucial statement Theorem~\ref{thm:independenceofcharts}\eqref{st3}. \end{rem}

\subsection{Formal log flatness}  \label{section:formallogflatness}  One thing that is a little confusing about the notion of log flatness is that log structures live on the \emph{\'etale} site of a scheme and our definition of log flatness is in terms of modules on the \emph{Zariski} site of a scheme.  Let us begin by reformulating the log flatness criterion of Theorem~\ref{thm:chartcriteria}\eqref{cc4} solely in terms of the \'etale site.  First of all, the flatness for the map of schemes there can be checked at points of $X$ and is equivalent to saying $M_x[P^{\rm gp}]$ is flat over $\O_{Y,f(x)}(h,t)$ for every $x \in X$.  But the pullback from the Zariski to the the \'etale site is faithfully flat, so we could instead check the condition at all the \emph{\'etale} points of $X$, which is equivalent to saying $M_{\ov{x}}[P^{\rm gp}]$ is flat over $\O_{Y_{\acute{e}t},f(\ov{x})}(h,t)$ for each \'etale point $\ov{x}$ of $X$.  If $x \in X$ is the corresponding Zariski point, note that \be M_{\ov{x}} & = & M_{x} \otimes_{\O_{X,x}} \O_{X_{\acute{e}t},\ov{x}} \\ & = & (M_{\acute{e}t})_{\ov{x}} \ee is the stalk of the pullback $M_{\acute{e}t}$ of $M$ from the Zariski to the \'etale site of $X$.  By Corollary~\ref{cor:independenceofcharts} we know that this latter condition does not depend on the chosen chart, and we even know that we do not have to pick a chart at all: we can just use the monoids $\M_{X,\ov{x}}$ and $\M_{Y,f(\ov{x})}$ themselves.  We summarize this discussion as:

\begin{prop} \label{prop:formallogflatness} Let $f : X \to Y$ be a map of fine log schemes, $M \in \Mod(\O_X)$.  Then $M$ is log flat over $Y$ iff $(M_{\acute{e}t})_{\ov{x}} \in \Mod(\O_{X_{\acute{e}t,\ov{x}}})$ satisfies the equivalent conditions ${\bf P}(K)$ of Corollary~\ref{cor:independenceofcharts} for the map of fine log rings $$ \xym{ \M_{X,\ov{x}} \ar[r] & \O_{X_{\acute{e}t,\ov{x}}} \\ \M_{Y,\ov{y}} \ar[u] \ar[r] & \O_{Y_{\acute{e}t},\ov{y}} \ar[u] } $$ for each \'etale point $\ov{x}$ of $X$ with image \'etale point $\ov{y} = f(\ov{x})$. \end{prop}

\begin{defn} \label{defn:formallylogflat} Consider a map of integral log rings in a topos $X$ as below. \bne{logringobject} & \xym{ \M_C \ar[r]^-{\alpha_C} & C \\ \M_A \ar[r]^{\alpha_A} \ar[u]^{f^\dagger} & A \ar[u]_f} \ene Let $A(f^\dagger,\alpha_A) \to C[\M_C^{\rm gp}]$ be the map of ring objects constructed as in \S\ref{section:setup}.  Then a $C$-module $M$ is called \emph{formally log flat} over $\alpha_A : \M_A \to A$ iff $M[\M_C^{\rm gp}]$ is flat as an $A(f^\dagger,\alpha_A)$-module. \end{defn}

\begin{rem} \label{rem:formallylogflat1} One could make the definition without putting in the word ``integral," but I have no idea whether this would be a good definition or whether it would be of any use. \end{rem}

\begin{rem} \label{rem:formallylogflat2} If the topos $X$ has enough points, then we can check the flatness condition in Definition~\ref{defn:formallylogflat} at points.  Furthermore, the relevant constructions $$A(\slot,\slot) \to C[ \slot ]$$ and $M[ \slot ]$ commute with filtered direct limits, so on stalks at $x$, this flatness condition becomes the condition ${\bf P}(T)$ of Corollary~\ref{cor:independenceofcharts} for the terminal chart $T$ for the map of log rings \bne{logringobjectstalks} & \xym{ \M_{C,x} \ar[r]^-{\alpha_C} & C_x \\ \M_{A,x} \ar[r]^{\alpha_A} \ar[u]^{f^\dagger_x} & A_x \ar[u]_f} \ene given by the stalk of \eqref{logringobject} at $x$. If \eqref{logringobjectstalks} is a map of fine log rings, then Corollary~\ref{cor:independenceofcharts} says that we can check this condition ${\bf P}(T)$ by instead checking ${\bf P}(K)$ for some chart $K$ which we may choose as we see fit. \end{rem}

Suppose now that $f : X \to Y$ is a map of (integral) log schemes.  This determines a map of log ring objects \bne{schemelogringobject} & \xym@C+20pt{ \M_X \ar[r]^{\alpha_X} & \O_{X_{\acute{e}t}} \\ f^{-1}\M_{Y} \ar[u]^{f^\dagger} \ar[r]^{f^{-1} \alpha_Y} & f^{-1} \O_{Y_{\acute{e}t}} \ar[u] } \ene in the \'etale topos of $X$.  The inverse image functor $f^{-1}$ here is of course the one for the \'etale topoi.

\begin{defn} \label{defn:formallylogflatschemes}  For a map $f : X \to Y$ of (integral) log schemes and an $\O_{X_{\acute{e}t}}$-module $M$, we say that $M$ is \emph{formally log flat over} $Y$ iff $M$ is formally log flat over $f^{-1} \alpha_Y : f^{-1}\M_{Y} \to f^{-1} \O_{Y_{\acute{e}t}}$ (Definition~\ref{defn:formallylogflat}).  For an $\O_X$-module $M$, we say that $M$ is \emph{formally log flat over} $Y$ iff the pullback $M_{\acute{e}t}$ of $M$ to the \'etale topos of $X$ is formally log flat over $Y$ in the aforementioned sense. \end{defn}

\begin{thm} \label{thm:formallogflatness} For a map of fine log schemes $f : X \to Y$, an $\O_X$-module $M$ is log flat over $Y$ iff it is formally log flat over $Y$ (Definition~\ref{defn:formallylogflatschemes}). \end{thm}

\begin{proof}  Since the \'etale topos has enough points, this is clear from Remark~\ref{rem:formallylogflat2} and Proposition~\ref{prop:formallogflatness}.  \end{proof}

\begin{rem} Theorem~\ref{thm:formallogflatness} says that formal log flatness coincides with log flatness whenever the latter is defined.  Formal log flatness is not a (completely) idle generalization of log flatness.  For example, in log geometry one often looks at the log structure \be \M_{U/X} & := & \{ f \in \O_X : f|_U \in \O_U^* \} \ee associated to an open subvariety $U$ in a variety $X$ with closed complement $Z$.  (Here $\O_X$ and $\O_U$ are meant to be the \'etale structure sheaves.)  Even if $(U,X,Z)$ looks \'etale locally like $(T,X,Z)$ for a toric variety $(T,X)$ and a $T$-invariant Cartier divisor $Z \subseteq X$, the log structure $\M_{U/X}$ need not be fine.  But I believe in this situation that all its stalks will be fine, so formal log flatness yields a fairly well-behaved log flatness notion for such ``mildly incoherent" log structures. \end{rem}

\section{Examples and Applications} \label{section:applications}  In this section we will give some examples of maps $f : X \to Y$ where we have a good understanding of log flatness (\S\S\ref{section:semistabledegenerations}-\ref{section:examplesofnodaldegenerations}).  We then define the \emph{log quotient space} (\S\ref{section:logquotientspace}), give some examples (\S\ref{section:examplesoflogquotientspaces}), and explain how log flatness plays a natural role in the theory of descent and gluing morphisms for spaces of log quotients (\S\S\ref{section:descentscholium}-\ref{section:gluing}).  In the present paper, our intention is mostly to recast some known gluing constructions in the language of log flatness, though many of the results here are certainly new and greatly generalize the known gluing constructions.  In future work, we will give a more general treatment of moduli spaces of log quotients and the gluing maps relating them \cite{G2}. 

\subsection{Semistable degenerations}  \label{section:semistabledegenerations} In practice, the most important/useful types of maps of fine log schemes are those defined below.

\begin{defn} A \emph{semistable degeneration} (resp.\ \emph{semistable degeneration with boundary}) is a log smooth map $f : X \to Y$ of fine log schemes such that at each \'etale point $x$ of $X$, the map $\ov{\M}_{Y,f(x)} \to \ov{\M}_{X,x}$ of fine monoids is a partition morphism (\S\ref{section:partitionmorphisms}) (resp.\ \emph{partition morphism with boundary}).  \end{defn}

This is a slightly more general version of Olsson's notion of \emph{essentially semi-stable} \cite[2.1]{O2}.

\begin{defn} A \emph{nodal degeneration} is a log smooth map $f : X \to Y$ of fine log schemes such that at each \'etale point $x$ of $X$, the map $\ov{\M}_{Y,f(x)} \to \ov{\M}_{X,x}$ of fine monoids is \begin{enumerate} \item an isomorphism, \item a pushout of $\NN \to \NN^2$, or \item a pushout of $0 \to \NN$. \end{enumerate}  The point $x$ is then called a \emph{smooth point}, \emph{nodal point}, or \emph{boundary point} (respectively) according to which of these possibilities occurs.  The locus $\u{D} \subseteq \u{X}$ of boundary points is called the \emph{(relative) boundary} of $f$.  (We will see in Remark~\ref{rem:relativeboundary} that $\u{D} \subseteq \u{X}$ has a natural closed subscheme structure.)  We view $D$ as a log scheme by pulling back the log structure from $Y$ (not from $X$), so that $f|D : D \to Y$ is a strict map of log schemes.  If $D = \emptyset$, then $f$ is called a \emph{nodal degeneration without boundary}. \end{defn}  

\begin{rem} The relative boundary of a nodal degeneration is a special case of a general notion of relative boundary of an arbitrary map of fine log schemes (c.f.\ \cite[2.17]{GM}, \cite{G2}).  With our definition of the log structure on $D$, the closed embedding $\u{D} \into \u{X}$ does \emph{not} lift to a map of log schemes $D \to X$. \end{rem}

Evidently nodal degenerations are a special case of semistable degenerations.  We now work out the local structure of a nodal degeneration.  These arguments are known to the experts but difficult to find in print.

\begin{lem} \label{lem:groupsmooth} Let $G \into H$ be an injective map of finitely generated abelian groups.  Let $N$ be a positive integer such that multiplication by $N$ annihilates the torsion subgroup of $H/G$.  Then $\ZZ[1/N][G] \to \ZZ[1/N][H]$ is a smooth ring map. \end{lem}

\begin{proof} Let $G'$ be the subgroup of $H$ consisting of those $h \in H$ such that $nh \in G$ for some positive integer $n$.  The smallest such $n$ always divides $N$, hence $\ZZ[1/N][G] \to \ZZ[1/N][G']$ is an \'etale cover (it can be presented by adjoining various $n^{\rm th}$ roots of units $[g] \in \ZZ[1/N][G]^*$ with $n$ invertible in $\ZZ[1/N]$).  The map in question factors as the latter \'etale cover followed by $\ZZ[1/N][G'] \to \ZZ[1/N][H]$, so it suffices to prove the latter is smooth.  But this is clear because we can choose a spliting $H \cong G' \oplus \ZZ^m$ since $H/G'$ is torsion-free. \end{proof}

\begin{prop} \label{prop:nodaldegenerations}  Let $f : X \to Y$ be a nodal degeneration, $x$ an \'etale point of $X$.  \begin{enumerate} \item \label{smoothpoint} If $x$ is a smooth point, then there is a (Zariski) neighborhood of $x$ in $X$ on which $f$ is strict and $\u{f}$ is smooth. \item \label{nodalpoint} If $x$ is a nodal point and $a : Q \to \M_Y(U)$ is a fixed fine chart for $Y$ on an (\'etale) neighborhood $U$ of $f(x)$, then, after possibly replacing $f$ with a neighborhood of $x$ in $f^{-1}(U) \to U$, there is a diagram $$ \xym{ \NN^2   \ar[r] & P \ar[r] & \M_X(X) \\ \NN  \ar[u]^{\Delta} \ar[r] & Q \ar[u] \ar[r]^-a &  \M_Y(Y) \ar[u]_{f^\dagger} } $$ where the left square is a pushout diagram of fine monoids and the right square is a fine chart for $f$ such that the induced map \be \u{X} & \to & \u{Y} \times_{\u{\AA}(Q)} \u{\AA}(P) =  \u{Y} \times_{\u{\AA}(\NN)} \u{\AA}(\NN^2) \ee is smooth.  \item \label{boundarypoint} If $x$ is a boundary point and $a : Q \to \M_Y(U)$ is a fixed fine chart for $Y$ on an (\'etale) neighborhood $U$ of $f(x)$, then, after possibly replacing $f$ with a neighborhood of $x$ in $f^{-1}(U) \to U$, there is a fine chart for $f$ of the form $$ \xym@C+25pt{ Q \oplus \NN \ar[r]^-{(f^\dagger a, \, b) } &  \M_X(X) \\ Q \ar[r]^-a \ar[u]^{(\Id,0)} & \M_Y(Y) \ar[u]_{f^\dagger} } $$ such that the induced map \be \u{X} & \to & \u{Y} \times_{\u{\AA}(Q)} \u{\AA}(Q \oplus \NN) = \u{Y} \times \u{\AA}^1 \ee is smooth. \item \label{nodaldegenflat} The map $\u{f}$ is flat and smooth away from the locus of nodal points. \end{enumerate}  \end{prop}

\begin{proof} For \eqref{smoothpoint}, use the fact that the strict locus of a map of fine log schemes is open together with the fact that a strict map $f$ is log smooth iff $\u{f}$ is smooth in the usual sense. 

For \eqref{nodalpoint}, we first use Kato's Chart Criterion for Log Smoothness to find, after possibly shrinking, a fine chart \bne{affinechart} \xym{ T \ar[r] & \M_X(X)  \\ Q \ar[u]^h \ar[r]^-a & \M_Y(Y) \ar[u]  } \ene for $f$ such that: \begin{enumerate} \item \label{chartsmoothmap} The induced map $\u{X} \to \u{Y} \times_{\u{\AA}(Q)} \u{\AA}(T)$ is smooth. \item The map $h$ is injective.  \item The torsion subgroup of $T^{\rm gp}/Q^{\rm gp}$ is annihilated by a positive integer $N$ invertible on $X$.  \end{enumerate}  Set $y := f(x)$.  Since $\O_{Y,y} \to \O_{X,x}$ is local and $N \in \O_{X,x}^*$, $N \in \O_{Y,y}^*$, so we can assume, after possibly shrinking again, that $N$ is also invertible on $Y$.  Let $F \subseteq T$ be the preimage of $\O_{X,x}^*$ in $T$ under the map $T \to \O_{X,x}$ obtained from \eqref{affinechart}.  This $F$ is a face of $T$, hence is also fine, and hence $\u{\AA}(F^{-1}T) \to \u{\AA}(T)$ is an open embedding.  Since $T \to \M_{X}(X)$ is a chart, $T/F = \ov{F^{-1}T} \to \ov{\M}_{X,x}$ is an isomorphism.   Since $F$ is finitely generated, we can assume, after shrinking to a smaller neighborhood of $x$ in $X$ that all elements of $F$ map into $\O_X^*(X) \subseteq \M_X(X)$, so that we can replace $T$ by $F^{-1} T$ in our chart \eqref{affinechart}.  We can similarly replace $Q$ by the analogous localization.  All of the properties listed above continue to hold for the new chart (we just pass to an open locus in the smooth map in \eqref{chartsmoothmap}, and we don't change the groups $Q^{\rm gp}$, $T^{\rm gp}$).  We can thus assume in the rest of the argument that the maps $\ov{T} \to \ov{\M}_{X,x}$ and $\ov{Q} \to \ov{\M}_{Y,y}$ induced by \eqref{affinechart} are isomorphisms.  

By definition of ``nodal point," we can find a pushout diagram \bne{pushdiagram} & \xym{ \NN^2 \ar[r]^-{m,n} & \ov{\M}_{X,x} \ar@{=}[r] & \ov{T} \\ \NN \ar[u]^{\Delta} \ar[r]^-t & \ov{\M}_{Y,y} \ar[u] \ar@{=}[r] & \ov{Q} \ar[u]_{\ov{h}} } \ene  Choose a lift $q \in Q$ of $t \in \ov{\M}_{Y,y} = \ov{Q}$ and lifts $a,b \in T$ of $m,n \in \ov{\M}_{X,x} = \ov{T}$.  We know $h(q) = a + b$ after sharpening, so after adjusting our choice of $b$ by some unit of $T$ if necessary, we can assume $h(q) = a + b$.  We thus obtain a commutative diagram \bne{kpl} & \xym{ \NN^2 \ar[r] & P \ar[r] & T \\ \NN \ar[u]^{\Delta} \ar[r] & Q \ar[u] \ar@{=}[r] & Q \ar[u]_h } \ene where the fine monoid $P$ is defined by making the left square a pushout.  Since the big square of \eqref{kpl} becomes a pushout on sharpening, the map $\ov{P} \to \ov{T}$ is an isomorphism, hence so is the induced map $\ov{P} \to \ov{\M}_{X,x}$; this implies that $P \to \M_X(X)$ is a chart (after possibly shrinking).  It remains only to prove that the induced map \be \u{X} & \to & \u{Y} \times_{\u{\AA}(Q)} \u{\AA}(P) \ee is smooth.  Since we already know \eqref{chartsmoothmap} is smooth, we can prove the claim by showing that \bne{ll} \u{Y} \times_{\u{\AA}(Q)}  \u{\AA}(T) & \to & \u{Y} \times_{\u{\AA}(Q)}  \u{\AA}(P)  \ene is smooth.  Since $N$ is invertible in $Y$, \eqref{ll} is a basechange of $\Spec$ of the ring map \bne{lll} \ZZ[1/N][P] & \to & \ZZ[1/N][T], \ene so it suffices to show that \eqref{lll} is a smooth ring map.  Since $\ov{P} \to \ov{T}$ is an isomorphism, $P \to T$ is a pushout of $P^* \to T^*$ (Lemma~\ref{lem:strict}), so it suffices to show \be \ZZ[1/N][P^*] & \to & \ZZ[1/N][T^*] \ee is smooth.  By Lemma~\ref{lem:groupsmooth}, it suffices to show that $P^* \to T^*$ is injective and that $T^*/P^*$ is annihilated by $N$.  

For injectivity, it suffices to prove $P^{\rm gp} \to T^{\rm gp}$ is injective.  From the pushout definition of $P$, we know \be P/Q & = & \NN^2 / \Delta(\NN) \\ & = & \ZZ e_1 \\ & = & P^{\rm gp} / Q^{\rm gp}. \ee  The injectivity of $P^{\rm gp} \to T^{\rm gp}$ will follow from the Snake Lemma in $$ \xym{ 0 \ar[r] & Q^{\rm gp} \ar@{=}[d] \ar[r] & P^{\rm gp} \ar[r] \ar[d] & \ZZ e_1 \ar[d]^{e_1 \mapsto m} \ar[r] & 0 \\ 0 \ar[r] & Q^{\rm gp} \ar[r] & T^{\rm gp} \ar[r] & T^{\rm gp}/Q^{\rm gp} \ar[r] & 0 } $$ provided $\ZZ e_1 \to T^{\rm gp}/Q^{\rm gp}$ is injective (i.e.\ $m$ generates a non-zero free abelian subgroup of $T^{\rm gp}/Q^{\rm gp}$).  We also know that \be \ov{Q}^{\rm gp} / \ov{T}^{\rm gp} & = & (\ov{\M}_{X,x} / \ov{\M}_{Y,y})^{\rm gp} \\ & = & \ZZ m. \ee  From the Snake Lemma applied to $$ \xym{ 0 \ar[r] & Q^* \ar[r] \ar[d] & Q^{\rm gp} \ar[r] \ar[d] & \ov{Q}^{\rm gp} \ar[d] \ar[r] & 0 \\ 0 \ar[r] & T^* \ar[r] & T^{\rm gp} \ar[r] & \ov{T}^{\rm gp} \ar[r] & 0 } $$ we obtain a short exact sequence \bne{ty} & \xym{ 0 \ar[r] & T^*/Q^* \ar[r] & T^{\rm gp}/Q^{\rm gp} \ar[r] & \ZZ m \ar[r] & 0.} \ene  In particular, we see that $m \in T^{\rm gp}/Q^{\rm gp}$ generates a free abelian subgroup mapping isomorphically onto $\ov{Q}^{\rm gp} / \ov{T}^{\rm gp}$.  We also have a splitting $P^{\rm gp} = Q^{\rm gp} \oplus \ZZ e_1$, from which we obtain an identifiction \be T^{\rm gp} / P^{\rm gp} & = & T^{\rm gp} / (Q^{\rm gp} \oplus \ZZ e_1) \\ & = & T^* / Q^*. \ee  We conclude from \eqref{ty} that $N$ annihilates the torsion subgroup of $T^{\rm gp} / P^{\rm gp}$.  Finally, the Snake Lemma applied to $$ \xym{ 0 \ar[r] & P^* \ar[r] \ar[d] & P^{\rm gp} \ar[d] \ar[r] & \ov{P}^{\rm gp} \ar[d]^{\cong} \ar[r] & 0 \\ 0 \ar[r] & T^* \ar[r] & T^{\rm gp} \ar[r] & \ov{T}^{\rm gp} \ar[r] & 0 } $$ yields an isomorphism \be T^{\rm gp}/P^{\rm gp} & = & T^*/P^*. \ee

The proof of \eqref{boundarypoint} point is like the proof of \eqref{nodalpoint}, but easier, and will be left to the reader. 

The smoothness statement in \eqref{nodaldegenflat} is clear from the \'etale local nature of smoothness, \eqref{smoothpoint}, and \eqref{boundarypoint}.  The flatness statement in \eqref{nodaldegenflat} follows from the standard fact that $\u{f}$ is flat whenever $f$ is log smooth and integral.  Actually, one can see the flatness directly: since $\u{f}$ is smooth away from the nodal points, the only issue is to prove flatness of $\u{f}$ near a nodal point.  Near such a point, \eqref{nodalpoint} says that $\u{f}$ factors as a smooth map followed by a base change of \be \u{\AA}(\Delta) : \u{\AA}(\NN^2) & \to & \u{\AA}(\NN). \ee  The latter map is flat since it is $\Spec$ of the ring map $\ZZ[t] \to \ZZ[x,y]$ with $t \mapsto xy$, which makes $\ZZ[x,y]$ a free $\ZZ[t]$-module with basis $1,x,y,x^2,y^2,\dots .$   \end{proof}

\begin{rem} \label{rem:relativeboundary} In the chart of Proposition~\ref{prop:nodaldegenerations}\eqref{boundarypoint}, any two choices of $b \in \M_X(X)$ yielding such a chart must differ by a unit, so the closed subscheme of $\u{X}$ cut out by $\alpha_X b \in \O_X(X)$ does not depend on this choice of $b$.  Given the form of that chart, the corresponding closed subset $\u{D}$ of $\u{X}$ is clearly the relative boundary (the locus on which the relative characteristic monoid of $f^\dagger$ is $\NN$.)  Thus we can define a closed subscheme structure\footnote{Since we only have such a chart \'etale locally, we implicitly appeal here to \'etale descent for closed subschemes, the fact the smoothness is \'etale local, and the fact that being a Cartier divisor is \'etale local.} on $\u{D} \subseteq \u{X}$ (making $\u{D}$ a Cartier divisor in $\u{X}$) so that the diagram $$ \xym@C+20pt{ \u{D} \ar[r] \ar[d] & \u{X} \ar[d] \\ \u{Y} \ar[r]^-{(\Id,0)} & \u{Y} \times \AA^1 } $$ is cartesian, where the right vertical arrow is the smooth map of Proposition~\ref{prop:nodaldegenerations}\eqref{boundarypoint}.  In particular, this shows that $\u{D} \to \u{Y}$ is smooth.  Since $D \to Y$ is strict by definition, $D \to Y$ is log smooth. \end{rem}

\begin{thm} \label{thm:nodaldegeneration} Suppose $f : X \to Y$ is a nodal degeneration with $\u{Y} = \Spec k$ for an algebraically closed field $k$, where $Y=\Spec k$ is equipped with an arbitrary fine log structure.  \begin{enumerate} \item The non-strict locus $\u{Z}$ of $f$ (with the reduced scheme structure from $\u{X}$) is the disjoint union of the singular locus of $\u{f}$ (with the reduced scheme structure from $\u{X}$) and the relative boundary $\u{D}$ (the scheme structure on $\u{D}$ discussed in Remark~\ref{rem:relativeboundary} coincides in this situation with the reduced induced structure from $\u{X}$).  \item A quasi-coherent sheaf $\F$ on $X$ is log flat over $Y$ iff $\Tor_1^X(\F,\O_Z)=0$. \end{enumerate} \end{thm}

\begin{proof} We can work (strict) \'etale locally near a given \'etale point of $X$.  Near a smooth point Proposition~\ref{prop:nodaldegenerations} says that $f$ is strict.  But then, by Lemma~\ref{lem:strictlogflatness}, log flatness is the same thing as usual flatness over $Y$, which holds trivially because $k$ is a field.  The vanishing $\Tor_1^X(\F,\O_Z)=0$ also holds trivially because $f$ is strict, so the non-strict locus $Z$ is empty.

A fine log structure on the algebraically closed field $k$ is necessarily of the form \be \Spec (0 : Q \to k) & = & Q \oplus k^* \to k \ee for some fine, sharp monoid $Q$.

Proposition~\ref{prop:nodaldegenerations} says that any nodal point has, say, an affine \'etale neighborhood $U=\Spec C$ where we have a commutative diagram  $$ \xym{ \NN^2 \ar[r] & P \ar[r] &  \M_X(U) \ar[r] & C \\ \NN  \ar[u]^{\Delta} \ar[r] & Q \ar[u] \ar[r] & \M_Y(Y) \ar[u]  \ar[r] &  k \ar[u] } $$ (note that the composition $Q \to k$ here is zero) where the left square is a pushout square of fine monoids, and the middle square is a chart for $f$, and the induced ring map \bne{inducedsm} k \otimes_{\ZZ[\NN]} \ZZ[\NN^2] = k[x,y]/(xy) & \to & C \ene is smooth (in particular it is flat).  On $U$ we can write $\F = M^\sim$ for a $C$-module $M$.  By Corollary~\ref{cor:chartcriteria}\eqref{affinecc3}, log flatness of $\F$ is equivalent to graded flatness of $M$ over \be B & := & k \otimes_{\ZZ[Q]} \ZZ[P] \\ & = & k \otimes_{\ZZ[\NN]} \ZZ[\NN^2] \\ & = & k[x,y]/(xy), \ee graded by $(P/Q)^{\rm gp} = (\NN^2/\NN)^{\rm gp} = \ZZ$ with $|x|=1$, $|y|=-1$ as usual.  According to Corollary~\ref{cor:nodalflatness} this graded flatness is equivalent to $\Tor_1^B(M,B/(x,y))=0$.  But $C$ is flat (even smooth) over $B$, so this is equivalent to \be \Tor_1^C(M,C/(x,y)) & = & 0. \ee But is is clear from the properties of the diagram above that the non-strict locus of $f$ is cut out by the images of $x,y$ in $C$ (the images of the standard generators of $\NN^2$ in $C$), so this latter vanishing is exactly the affine translation of $\Tor_1^X(\F,\O_Z)=0.$  It is also clear from smoothness of \eqref{inducedsm} that the images of $x,y$ in $C$ cut out the singular locus of $C$ with the reduced induced structure.

The case of a boundary point $x$ is very similar.  One uses Corollary~\ref{cor:gradedflatnessoverA1} instead of Corollary~\ref{cor:nodalflatness}.  \end{proof}

\begin{cor} \label{cor:nodaldegeneration} Suppose $f : X \to Y$ is a nodal degeneration with relative boundary $D \to Y$ and $\F$ is a quasi-coherent sheaf on $X$ of locally finite presentation.  Then $\F$ is log flat over $Y$ iff $\F$ is flat over $Y$ in the usual sense and \be \Tor_1^{X_y}(\F|X_y,\O_{Z}) & = & 0 \ee for each geometric point $y$ of $Y$ (here $Z \subseteq X_y$ is the non-strict locus of $X_y \to \{ y \}$).  If $\F$ is log flat over $Y$, then $\F|D$ is flat over $Y$. \end{cor}

\begin{proof} For the first statement, combine the fiberwise log flatness criterion (Theorem~\ref{thm:geometricpoints}; note that $f$ is of locally finite presentation by the definitions) and the theorem. 

For the second statement:  Since $\u{D} \subseteq \u{X}$ is a Cartier divisor, its ideal sheaf $\O_X(-D)$ is invertible, so $\F(-D)$ is also of loc.\ fin.\ pres.\ so the natural map $\F(-D) \to \F$ is a map of loc.\ fin.\ pres.\ $\O_X$-modules with codomain flat over $Y$ (using part of the first statement), hence, by \cite[IV.11.3.7]{EGA}, $\F(-D) \to \F$ will be injective and its quotient $\F|D$ will be flat over $Y$ provided we can show it is injective on each geometric fiber of $f$.  The map $\F(-D) \to \F$ restricts, on a geometric fiber $X_y$, to the analogous map $(\F|X_y)(-D_y) \to \F|X_y,$ where $D_y$ is the boundary of $X_y$.  We know this is injective because the theorem says that log flatness of $\F|X_y$ implies $\Tor_1^{X_y}(\F|X_y,\O_{D_y})=0$. \end{proof}

\subsection{Examples of nodal degenerations} \label{section:examplesofnodaldegenerations} Several examples of nodal degenerations (\S\ref{section:semistabledegenerations}) occur in nature.

\begin{example} \label{example:expansions} {\bf (Expansions and pairs)}  The \emph{expanded pairs} and \emph{expanded degenerations} used in relative Gromov-Witten theory and its cousin theories (DT theory and Stable Pairs theory) are important examples of nodal degenerations.  The basic input for expanded pairs is a \emph{smooth pair} $(X,D)$ consisting of a smooth variety $X$ and a smooth divisor $D \subseteq X$ (we assume $D$ is connected to simplify the exposition).  Let $\Delta := \PP(N_{D/X},\O_X)$, so $\Delta$ is a $\PP^1$-bundle over $D$ with two obvious sections, one with normal bundle $N_{D/X}$ and one with normal bundle $N_{D/X}^\lor$.  One can construct a log algebraic stack $\mathcal{T}$ of ``targets" and a diagram of log algebraic stacks \bne{universalexpansion} & \xym{ \mathcal{T} \times D \ar[r]^-i \ar[rd]_{\pi_1} & \mathcal{X} \ar[d]^{\pi} \ar[r]^-c & \mathcal{T} \times X \ar[ld]^{\pi_1} \\ & \mathcal{T} } \ene where $c$ is representable and (at least locally) projective (hence proper) and $\pi$ is a (representable) nodal degeneration.  One often abbreviates \eqref{universalexpansion} by $\pi : \mathcal{X} \to \mathcal{T}$.  

At each geometric point $t$ of $\mathcal{T}$, the diagram \eqref{universalexpansion} takes the form $$ \xym{ D \ar[r]^-i  & X[n]_0 \ar[r]^-c & X } $$ (for some $n$), where \be X[n]_0 & = & X \coprod_{D_1} \Delta_1 \coprod_{D_2} \cdots \coprod_{D_n} \Delta_n \ee is an ``accordian" obtained by gluing copies $\Delta_i$ of $\Delta$ ``end-to-end" along copies $D_i$ of $D$.  The inclusion $i$ includes $D$ as the ``other" copy $D_{n+1}$ of $D$ in $\Delta_n$ (the one with normal bundle $N_{D/X}$, as opposed to $D_n \subseteq \Delta_n$, which has normal bundle $N_{D/X}^\lor$).   The map $c$ contracts all of the $\Delta_i$ back onto $D \subseteq X$ via the projection for the $\PP^1$ bundle $\Delta \to D$.  The log structure on $X[n]_0$ has relative characteristic monoid $\ZZ$ along the singular loci $D_1,\dots,D_n$ and relative characteristic monoid $\NN$ along $i(D) \subseteq \Delta_n \subseteq X[n]_0$.  The log structure on the base $t$ has characteristic monoid $\NN^n$ (same $n$ as the one in $X[n]_0$).  

Up to this point, we have given $D$ the log structure inherited from $X$, and we have defined the log structure on $\mathcal{T} \times D$ so that $i$ is strict.  However, if we think of $X$ as a nodal degeneration, the ``right" log structure to put on $D$ is the trivial one, and if we think of $\pi$ as a nodal degeneration, the ``right" log structure to put on its boundary $\mathcal{T} \times D$ is the one pulled back from $\mathcal{T}$ (note that $i$ can no longer be viewed as a map of log schemes).  The singular locus of (the scheme underlying) $X[n]_0$ is given by $D_1 \cup \cdots \cup D_n$ and its boundary (in the sense of nodal degenerations) is given by $i(D)=D_{n+1}$ (with the log structure pulled back from $t$).  The non-strict locus of $X[n]_0$ over $t$ is given by \be Z & = & D_1 \coprod \cdots \coprod D_{n+1}. \ee

For expanded degenerations, one starts with a nodal degeneration $W \to \AA^1$ (smooth in the usual sense away from $0$) with central fiber $W_0 = X_1 \coprod_D X_2$ a union of smooth varieties $X_1$, $X_2$ along a common codimension one smooth subvariety $D$.  One then makes a new space of ``targets" $\mathcal{T}$ (over $\AA^1$) and a universal target $\mathcal{W} \to \mathcal{T}$ which is a representable nodal degeneration whose geometric fibers have a similar ``accordian" form \be W[n]_0 & = & X_1 \coprod_{D_1} \Delta_1 \coprod_{D_2} \cdots \coprod_{D_{n}} \Delta_n \coprod_{D_{n+1}} X_2. \ee

By Theorem~\ref{thm:nodaldegeneration}, a quasi-coherent sheaf $\F$ on $Y = X[n]_0$ or $Y = W[n]_0$ is log flat iff \be \Tor_1^Y(\F,\O_{D_i}) & = & 0 \ee for $i=1,\dots,n+1$. \end{example}

\begin{example} \label{example:nodalcurves} {\bf (Nodal curves)}  A nodal curve $\pi : C \to Y$ with marking sections $s_1,\dots,s_n : Y \to C$ becomes a nodal degeneration when endowed with the canonical log structure of F.~Kato \cite{FK} (and in fact all liftings of $\u{\pi}$ to a log smooth map of fs log schemes are pulled back from the canonical one).  The relative boundary of $\pi$ consists of $n$ copies of $Y$ (with the same log structure $Y$ has as the base of $\pi$).  If $\u{Y} = \Spec k$ for an algebraically closed field $k$, the non-strict locus of $\pi$ is the disjoint union of the singular locus of $\u{C}$ and the marked points of $\u{C}$.  By arguing exactly as in the proof of Theorem~\ref{thm:nodaldegeneration}, and making use of Corollary~\ref{cor:nodalflatness}\eqref{Mlocallyfree} and Corollary~\ref{cor:gradedflatnessoverA1}\eqref{Mlocallyfree2}, one sees that a coherent sheaf $\F$ on the marked nodal curve $C$ over is log flat over such a $Y$ iff $\F$ is locally free near the marked points and nodes of $\u{C}$. \end{example}

\subsection{Log quotient space} \label{section:logquotientspace}  The main purpose of this section is to define the \emph{log quotient space} $\LQuot(M/X/Y)$ associated to a (reasonable) map of fine log schemes $f : X \to Y$ and a (reasonable) quasi-coherent sheaf $M$ on $X$.  As long as $M$ and $f$ are reasonable, $\LQuot(M/X/Y)$ will be an algebraic space, separated over $Y$ and of locally finite presentation over $Y$.  Indeed, it will simply be defined to be the log flat locus of the universal quotient sheaf on the usual quotient space, so, for example, it will be quasi-projective over $\u{Y}$ whenever $\u{f}$ is quasi-projective.  For many of the applications we have in mind, it is necessary to define $\LQuot(M/X/Y)$ for a map of fine log \emph{algebraic stacks} $f : X \to Y$, so we will have to make some general nonsense definitions to work at this greater level of generality.

The point is that the theory of log flatness makes sense not only for maps of fine log schemes, but also for maps of fine log algebraic stacks.  First we have to define the categories of quasi-coherent sheaves and (fine) log structures on a stack.  In fact we can define both of these categories for an arbitrary groupoid fibration in the same way, as follows.

Fix a base scheme $S$ (often $S = \Spec \CC$) and let $\Sch$ denote the category of schemes over $S$ and morphisms of $S$-schemes (which we will refer to simply as ``schemes" and ``morphisms").  Let $\Qco$ (resp.\ $\LogSch$) be the category whose objects are pairs $(U,M)$ consisting of a scheme $U$ and a quasi-coherent sheaf (resp.\ log structure) $M$ on $U$ and whose morphisms $$(f,g) : (U,M) \to (V,N)$$ are pairs $(f,g)$ consisting of a morphism of schemes $f : U \to V$ and a morphism $g : f^*N \to M$ of quasi-coherent sheaves (resp.\ log structures) on $U$.  The forgetful functor $(U,M) \mapsto U$ makes $\Qco$ (resp.\ $\LogSch$) a fibered category over $\Sch$ in the sense of \cite[3.1]{Vis}.  The ``cartesian arrows" in $\Qco$ and $\LogSch$ are those $(f,g)$ for which $g$ is an isomorphism (for $\LogSch$ these are the strict maps).  The restriction of the aforementioned forgetful functor to the subcategory $\Qco^{\rm cart}$ (resp.\ $\LogSch^{\rm cart}$) of cartesian arrows is a groupoid fibration (as is the case for any fibered category).  Let $\Fib / \Sch$ denote the $2$-category of fibered categories over $\Sch$ and let $\CFG / \Sch$ denote the full sub-$2$-category of $\Fib / \Sch$ consisting of groupoid fibrations over $\Sch$.  

\begin{defn} \label{defn:QcoX} For an object $X$ of $\CFG / \Sch$, we define the category $\Qco(X)$ of \emph{quasi-coherent} sheaves (resp.\ the category $\LogStr(X)$ of \emph{log structures}) on $X$ to be the full subcategory of $$\bHom_{\Fib / \Sch}(X,\Qco) \; \; {\rm (resp.} \bHom_{\Fib / \Sch}(X,\LogSch))$$ whose objects are those of $$\bHom_{\CFG / \Sch}(X,\Qco) \; \; {\rm (resp.} \bHom_{\CFG / \Sch}(X,\LogSch)).$$ \end{defn}

This simply makes precise the idea that ``to give a quasi-coherent sheaf or log structure on $X$ is to give its pullback along any map from a scheme to $X$."  To explain this, we first note that the ``$2$-Yoneda Lemma" is the statement that for a scheme $U$ with corresponding groupoid fibration $\Sch/U \to \Sch$, the functor \bne{Yoneda} \Qco(U) & \to & \Qco(\Sch/U) \\ \nonumber M & \mapsto & ((f : U' \to U) \mapsto (U',f^*M)) \\ \nonumber g : M \to N & \mapsto & (f : U' \to U) \mapsto ((\Id,f^*g) : (U',M) \to (U',f^*N)) \ene is an equivalence of categories from the usual category of quasi-coherent sheaves on $U$ (the fiber category of $\Qco \to \Sch$ over $U$) to the category of quasi-coherent sheaves on the groupoid fibration $\Sch/U \to \Sch$ as defined above (the inverse equivalence is given by $M \mapsto M(\Id : U \to U) \in \Qco(U)$).  We constantly suppress this equivalence in what follows.  We next note that a $\CFG / \Sch$ morphism $f : X \to Y$ yields a pullback functor \bne{fupperstar} f^* : \Qco(Y) & \to & \Qco(X) \\ \nonumber M & \mapsto & f^*M \ene where $(f^*M)(x) := M(f(x))$ for each object $x$ of $X$.  Now suppose $X \in \CFG / \Sch$, $M \in \Qco(X)$, and $x$ is an object of $X$ lying over $U \in \Sch$.  Then we can think of the quasi-coherent sheaf $M(x)$ on $U$ as $x^* M \in \Qco(U)$---indeed, if we think of $x$ as a $\CFG / \Sch$-morphism $x : \Sch / U \to X$ via the Yoneda Lemma, and we identity $\Qco(\Sch/U)$ with $\Qco(U)$ via \eqref{Yoneda}, then it is indeed the case that $M(x) = x^* M$.

\begin{defn} \label{defn:Qco2} A quasi-coherent sheaf $M$ is called \emph{flat} (resp.\ \emph{locally finitely presented}, \dots) iff, for each object $x$ of $X$ with image $U \in \Sch$, the quasi-coherent sheaf $M(x) \in \Qco(U)$ on $U$ is flat (resp.\ locally finitely presented, \dots).  Similarly, a log structure $\M_X$ on $X$ is called \emph{integral} (resp.\ \emph{fine}, \dots) iff $\M_X(x)$ is an integral (resp.\ fine, \dots) log structure on the scheme $U$ for all objects $x$ of $X$. \end{defn}

\begin{rem} Suppose $X$ is an algebraic stack.  Since quasi-coherent sheaves and fine log structures \cite[Appendix A]{Ols} satisfy fppf descent, one could alternatively define a quasi-coherent sheaf or fine log structure on $X$---with the aid of an fppf cover from a scheme $U \to X$---in terms of descent.  This notion of quasi-coherent sheaf or fine log structure coincides with that of Definitions~\ref{defn:QcoX} and \ref{defn:Qco2}.  \end{rem}

\begin{rem} It is tautological from the definitions that flat (resp.\ locally finitely presented, \dots) quasi-coherent sheaves are stable along under pullback along an arbitrary $\CFG / \Sch$ morphism.  Indeed, any property of log structures or quasi-coherent sheaves on schemes that is stable under pullback thus gives rise to a corresponding property for log structures or quasi-coherent sheaves on objects of $\CFG / \Sch$ which is also stable under pullback. \end{rem}

We need some ``algebraicity" assumption on a $\CFG / \Sch$ morphism $f : X \to Y$ in order to say when a quasi-coherent sheaf $M \in \Qco(X)$ is \emph{flat over} $Y$ and when such a sheaf has \emph{universally proper support over} $Y$.   Although we could certainly get away with less, let us now assume that $f : X \to Y$ is relatively DM, as this will be sufficient for our applications (in fact even the case where $f$ is relatively an algebraic space would suffice).  We then declare $M$ to be \emph{flat over} $Y$ (resp.\ to have \emph{universally proper support over} $Y$) iff, for each $\CFG / \Sch$-morphism $y : U \to Y$ with $U$ a scheme, the quasi-coherent sheaf $\pi_2^* M \in \Qco(U \times_Y X)$ is flat over $U$ (resp.\ has proper support over $U$) via $\pi_1 : U \times_Y X \to U$ (these latter concepts can be defined in any number of reasonable and equivalent ways because $\pi_1$ is a DM stack over the scheme $U$).  It is tautological to check that these notions of relative flatness and relatively proper support are stable under base change.

Now suppose we have a relatively DM $\CFG / \Sch$-morphism $f : X \to Y$ and a quasi-coherent sheaf $M \in \Qco(X)$, which we will assume is of locally finite presentation (Definition~\ref{defn:Qco2}) .  Then we define a presheaf $\Quot(M/X/Y)$, called the \emph{presheaf of quotients} of $M$, on the category $Y$ by taking an object $y$ of $Y$ (letting $U \in \Sch$ be the image of $y$, we view $y \in Y(U)$ as a $\CFG / \Sch$ morphism $y : U \to Y$ as usual) to the set of quotients $q : \pi_2^* M \to N$ (in the abelian category $\Qco(U \times_Y X)$ of quasi-coherent sheaves on the DM stack $U \times_Y X$) such that \begin{enumerate} \item \label{quotflatness} $N$ flat over $U$ via $\pi_1 : U \times_Y X \to U$. \item \label{quotpropersupport} $N$ has proper support over $U$ via $\pi_1 : U \times_Y X \to U$. \item \label{quotfinpres} $N$ is of locally finite presentation. \end{enumerate}  

\begin{rem} Assuming \eqref{quotflatness} and \eqref{quotfinpres}, the formation of the support of $N$ commutes with base change along maps $U' \to U$, so one could replace ``proper support" in \eqref{quotpropersupport} with ``universally proper support" without altering the definition of $\Quot(M/X/Y)$. \end{rem}

Since these conditions on $N$ are stable under base change, we can define the restriction maps for this presheaf simply by pulling back.  We can think of the presheaf $\Quot(M/X/Y)$ over $Y$ as a $\CFG / \Sch$-morphism $\Quot(M/X/Y) \to Y$ which is formally representable (Definition~\ref{defn:formallyrepresentable}).  Forming this presheaf ``commutes with pullback" in the sense that for any $\CFG / \Sch$-morphism $Y' \to Y$, there is a $2$-cartesian diagram in $\CFG / \Sch$ as below.  \bne{Quotcart} & \xym{ \Quot(\pi_2^*M / Y' \times_{Y} X / Y') \ar[d] \ar[r] & \Quot(M/X/Y) \ar[d] \\ Y' \ar[r] & Y } \ene 

We will need the following reformulation of a result of Olsson and Starr:

\begin{prop} \label{prop:OlssonStarr} Let $f : X \to Y$ be a $\CFG / \Sch$-morphism representable by relatively separated, locally finitely presented DM stacks and let $M$ be a quasi-coherent sheaf on $X$ of locally finite presentation.  Then $\Quot(M/X/Y) \to Y$ is representable by relatively separated algebraic spaces of locally finite presentation.  If $M$ has proper support over $Y$, then $\Quot(M/X/Y) \to Y$ satisfies the valuative criterion for properness. \end{prop}

\begin{proof} The conclusion means that for any scheme $U$ and any $\CFG / \Sch$ morphism $U=\Sch / U \to Y$, the fibered product $\Quot(M/X/Y) \times_Y U$ (which is \emph{a prioi} only a presheaf on $\Sch / U$) is in fact an algebraic space, separated and of locally finite presentation over $U$.  But the compatibility of the $\Quot$ construction with change of base \cite[Rem.~3.9]{Hilb} ensures that this fibered product is just $\Quot(\pi_2^*M/U \times_Y X,U)$, which is an algebraic space over $U$ with the desired properties by \cite[Theorem~1.1]{OS} because $U \times_Y X$ is a DM-stack and $\pi_1 : U \times_Y X \to U$ is separated and of locally finite presentation (by definition of the hypothesis on $f$) and $\pi_2^*M$ is a quasi-coherent sheaf on $U \times_Y X$ of locally finite presentation (by the hypothesis on $M$).   If $M$ has proper support over $Y$, then $\pi_2^*M$ also has proper support over $U$ (via $\pi_1$), and in this case that same theorem of Olsson and Star says that $\Quot(\pi_2^*M/U \times_Y X,U) \to U$ satisfies the valuative criterion for properness---that is, the base change of $f$ along any map $U \to Y$ with $U$ a scheme satisfies the valuative criterion for properness; this easily implies the final statement of the proposition (which is entirely a statement about maps from schemes to $X$ and $Y$).  \end{proof}

If $Y \in \CFG / \Sch$ and $\M_Y$ is a fine log structure on $Y$, then one can define ``the category of fine log schemes over $Y$ with morphisms given by strict morphisms" $\Log(Y) \in \CFG / \Sch$ in an evident manner;  when $Y$ is an algebraic stack, $\Log(Y)$ is again an algebraic stack, of locally finite presentation over $\u{Y}$, etc.\ \cite[5.9]{Ols}.

Now suppose $f : X \to Y$ is a map of fine log algebraic stacks and $M$ is a quasi-coherent sheaf on $X$ (i.e.\ on the underlying algebraic stack $\u{X}$) such that: \begin{enumerate} \item \label{frepresentable} $\u{f} : \u{X} \to \u{Y}$ is representable by relatively separated, locally finitely presented DM stacks. \item $M$ is of locally finite presentation. \end{enumerate}  By Proposition~\ref{prop:OlssonStarr}, we then have an algebraic stack $\u{Q} := \Quot(M/\u{X}/\u{Y})$ such that the structure morphism $\u{Q} \to \u{Y}$ is representable by algebraic spaces separated and of locally finite presentation over $\u{Y}$.  Furthermore, the universal quotient sheaf $N$ on $\u{Q} \times_{\u{Y}} \times \u{X}$ is tautologically of locally finite presentation and has proper support of $\u{Q}$.  Give $Q$ the log structure pulled back from $Y$.  Note that $\u{Q \times_Y X} = \u{Q} \times_{\u{Y}} \times \u{X}$ because $Q \to Y$ is strict.  We now assume, furthermore, that $f$ is integral.  Then so is its base change $Q \times_Y X \to Q$, hence by Theorem~\ref{thm:openness} there is an open substack \be \LQuot(M/X/Y) & \subseteq & Q = \Quot(M/\u{X}/\u{Y}) \ee of $Q$ which is terminal among strict maps of algebraic stacks $Q' \to Q$ for which the pullback $N'$ of $N$ to $Q' \times_Y X$ is log flat over $Q'$.  After unwinding the definition of $\LQuot(M/X/Y)$ we can summarize this discussion as:

\begin{thm} \label{thm:logquotientspace}  Let $f : X \to Y$ be an integral map of fine log algebraic stacks such that $\u{f}$ is representable by relatively separated, locally finitely presented DM stacks.  Let $M$ be a quasi-coherent sheaf on $X$ of locally finite presentation.  \begin{enumerate} \item The groupoid fibration $\LQuot(M/X/Y)$ whose objects over a scheme $\u{U}$ are pairs $(y,q)$ consisting of a map $y : \u{U} \to \u{Y}$ and a quotient $q : \pi_2^*M \to N$ on $U \times_Y X$ ($U$ is given the log structure pulled back from $Y$) such that $N$ is log flat over $U$ and satisfies the habitual conditions\footnote{i.e. $N$ is loc.\ fin.\ pres.\ and has proper support over $U$} is representable by an algebraic stack---in fact an open substack of $\Quot(M/X/Y)$. \item  The structure map $\LQuot(M/X/Y) \to \u{Y}$ is representable by relatively separated algebraic spaces of locally finite presentation. \end{enumerate} \end{thm}

\subsection{Examples of log quotient spaces} \label{section:examplesoflogquotientspaces}  In the present paper we are not going to construct any new moduli spaces.  We will content ourselves by simply observing that many moduli spaces already studied in the literature are in fact nothing but special cases of the log quotient spaces of \S\ref{section:logquotientspace}.

\begin{thm}  Let $X$ be a smooth projective threefold over an algebraically closed field, $D \subseteq X$ a smooth hypersurface.  The moduli space of stable relative ideal sheaves $\Hilb(X/D)$ of \cite{Wu} coincides with the (stable locus in) the log relative Hilbert scheme ${\bf LHilb}(\mathcal{X}/\mathcal{T})$ of the universal expansion $\mathcal{X} \to \mathcal{T}$. \end{thm}

\begin{proof} By definition (c.f.\ \S\ref{section:logquotientspace}, \cite[3.6]{Wu}), both moduli spaces are moduli spaces of certain quotients $$\O_{\mathcal{X}} \to \O_Z \to 0$$ on expansions $\mathcal{X} \to Y$ of $(X,D)$ over a varying base scheme $Y$ (pullbacks of the universal expansion $\pi : \mathcal{X} \to \mathcal{T}$ of $(X,D)$ discussed in Example~\ref{example:expansions} along maps $Y \to \mathcal{T}$; one views $Y$ as a log scheme by pulling back the log structure on $\mathcal{T}$).  In both cases the ``stable locus" can be taken as the locus where the stabilizer group at each geometric point is finite.  In Wu's space, one places the following conditions on the quotient $\O_Z$: First, it should be flat over $Y$, and, second, it should have certain $\Tor$-vanishing properties at each geometric point of $Y$ \cite[3.6]{Wu}.  By Corollary~\ref{cor:nodaldegeneration} and the discussion in Example~\ref{example:expansions}, these two properties are equivalent to demanding that $\O_Z$ be log flat over $Y$. \end{proof}

Of course one can make a similar statement about moduli of stable perfect ideal sheaves on an expansion $W \to \AA^1$.

\begin{thm} \label{thm:stablequotients} The moduli space of stable quotients of \cite{MOP} coincides with the (stable locus in) the log relative quotient scheme $\LQuot(\mathcal{C}/\mathcal{M},\O^n)$ of the universal marked nodal curve $\mathcal{C} \to \mathcal{M}$, endowed with the canonical log structure of F.~Kato \cite{FK}. \end{thm} 

\begin{proof} Both moduli spaces parameterize certain quotients $$\O_C^n \to N \to 0 $$ where $C \to Y$ is a marked nodal curve over a varying base scheme $Y$ (the map $C \to Y$ is a pullback of the universal such map $\mathcal{C} \to \mathcal{M}$, and we view $Y$ as a log scheme by pulling back the log structure on $\mathcal{M}$).  In both cases, one can take as the stability condition the finiteness of stabilizers at geometric points.  In \cite{MOP}, the quotient $N$ is required to be flat over $Y$ and to be locally free near nodes and marked points on each geometric fiber of $C \to Y$.  These conditions combined are equivalent to log flatness of $\O_Z$ over $Y$ by Corollary~\ref{cor:nodaldegeneration} and the discussion in Example~\ref{example:nodalcurves}. \end{proof}

It is interesting to note that the moduli spaces of stable log quotients in both theorems above are proper, but to achieve this properness in either case one must work with quotients on spaces which are not themselves ``stable".  It would be interesting to establish general results concerning the properness of the log quotient space.

\subsection{Gluing scholium} \label{section:gluingscholium} Since the functor \bne{Specfunctor} \Spec : \An^{\rm op} & \to & \Sch \ene preserves finite inverse limits, it is natural to ask:  To what extent does \eqref{Specfunctor} preserve finite direct limits?  Of course it cannot preserve \emph{all} direct limits, because one can glue affine schemes along affine open subschemes to obtain schemes which are not affine.  It does, however, preserve (finite!) direct sums (coproducts): \bne{Specproduct} \Spec(C_1 \times C_2) & = & (\Spec C_1) \coprod (\Spec C_2). \ene  One might next ask whether \eqref{Specfunctor} preserves pushouts---i.e.\ whether $\Spec$ takes a cartesian diagram of rings $C_\bullet$ as below \bne{pullbackringdiagram} \xym{ C \ar[r]^{p_1} \ar[d]_{p_2} & C_1 \ar[d]^{f_1} \\ C_2 \ar[r]_{f_2} & C_0 } \ene to a pushout diagram of schemes.  In general it will not: 

\begin{example} The diagram of rings $$ \xym@C+20pt{ \ZZ \ar[r] \ar[d] & \ZZ[y] \ar[d]^{y \mapsto t^{-1}} \\ \ZZ[x] \ar[r]_-{x \mapsto t} & \ZZ[t,t^{-1}] } $$ is cartesian, but the corresponding diagram of schemes $$ \xym{ \Spec \ZZ[t,t^{-1}] \ar[d] \ar[r] & \Spec \ZZ[y] \ar[d] \\ \Spec \ZZ[x] \ar[r] &  \Spec \ZZ } $$ is certainly not a pushout (the actual pushout is $\PP^1$). \end{example}

In general it is not even clear whether the direct limit of a finite diagram of affine schemes will exist in the category of schemes.

\begin{rem} Instead of considering pushouts, we might consider coequalizers.  Of course one can convert between the two questions, so we will consider pushouts for the sake of concreteness, leaving it to the reader to formulate and prove the corresponding statements for coequalizers. \end{rem}

Notice that every pushout diagram of schemes mentioned above (i.e.\ ``gluing along a common open subscheme") is also a pullback diagram and is in fact a pushout diagram in ringed spaces (and a pullback diagram in ringed spaces).  (Recall that the direct limit of a functor $i \mapsto X_i$ to ringed spaces is constructed by taking the direct limit $X$ of the $X_i$ in topological spaces and endowing $X$ with the sheaf of rings $\O_X$ given by the inverse limit of the pushforwards of the $\O_{X_i}$.)  

We will now present some general results to the effect that, for \emph{certain} cartesian ring diagrams $C_\bullet$ as in \eqref{pullbackringdiagram} (e.g.\ those where the $f_i$ are surjective), the $\Spec$ functor \eqref{Specfunctor} will take $C_\bullet$ to a pushout diagram of schemes.  We first need some related results from the topological situation.  The following is useful:

\begin{defn} A map $f : X \to Y$ of topological spaces is called a \emph{quotient map} if $f$ is surjective and has the following property:  A subset $U \subseteq Y$ is open in $Y$ iff $f^{-1}(U)$ is open in $X$. \end{defn}

The terminology is explained as follows:  Any $f : X \to Y$ yields an equivalence relation $\sim$ on $X$ given by ``having the same image under $f$."  There is an induced map $X/\sim \to Y$, where $X/\sim$ is of course given the quotient topology.  This induced map is an isomorphism (``$Y$ is the quotient of $X$ by $\sim$") iff $f$ is a quotient map.

\begin{rem} A quotient map is an effective descent morphism of topological spaces.  ``Conversely,"
 a surjective effective descent morphism is a quotient map (the point is that the ``open sets" functor is represented by the Sierpinski space).  An open surjective map is a quotient map.  Being a quotient map can be checked locally on the base.  A locally finite closed cover is a quotient map.  The map on spaces underlying an fppf or fpqc cover of schemes is a quotient map (see \cite[IV.2.4.6]{EGA} and \cite[IV.2.3.12]{EGA}). \end{rem}

In the next few lemmas, we will often consider a commutative diagram $X_\bullet$ of schemes or topological spaces as below.  \bne{topspacediagram} & \xym{ X_0 \ar[r]^-{j_2} \ar[d]_{j_1} & X_2 \ar[d]^{f_2} \\ X_1 \ar[r]_-{f_1} & X } \ene

\begin{lem} \label{lem:pushoutsintop}  Suppose $X_\bullet$ is a \emph{cartesian} diagram of topological spaces as in \eqref{topspacediagram} where the map $f_1$ (hence also $j_2$) is a closed embedding.  Then the following are equivalent: \begin{enumerate} \item The diagram $X_\bullet$ is a pushout diagram of spaces.  \item The map $f_2' : X_2 \setminus X_0 \to X \setminus X_1$ induced by $f_2$ is an isomorphism. \end{enumerate}  These statements hold, for example, if $f_1$ and $f_2$ are closed embeddings with $X = X_1 \cup X_2.$ \end{lem}

\begin{proof}  Suppose the diagram is a pushout.  Then the underlying diagram of sets must be a pushout, which easily implies that $f_2'$ is bijective.  To see that it is an isomorphism, we need to see that it is open.  The point here is that, since the diagram is a pushout, $X$ must have the ``weak topology" where a subset $U \subseteq X$ is open iff each $f_i^{-1}(U)$ is open in $X_i$. 

Now suppose $f_2'$ is an isomorphism.  Let us show that \eqref{topspacediagram} is a pushout.  We first show that \eqref{topspacediagram} is a pushout on the level of sets.  Let $\sim$ be the smallest equivalence relation on $X_1 \coprod X_2$ containing the pairs $(j_1(x),j_2(x))$ for $x \in X_0$.  We need to show that the map $(X_1 \coprod X_2)/\sim \to X$ induced by $f_1$ and $f_2$ is bijective.  Surjectivity is easy and uses only the surjectivity of $f_2'$ (and the fact that \eqref{topspacediagram} is cartesian).  For injectivity, we need to show that $x \sim y$ whenever $x,y \in (X_1 \coprod X_2)$ have the same image in $X$.  If both $x$ and $y$ are in $X_1$, then this is trivial because $f_1$ is injective, so in fact $x=y$.  If, say, $x \in X_1$, and $y \in X_2$, then this follows from the fact that \eqref{topspacediagram} is cartesian.  The only issue is when both $x$ and $y$ are in $X_2$ and $f_2(x)=f_2(y) \in X$.  If the point $z := f_2(x)=f_2(y)$ is not in $X_1$, then \eqref{topspacediagram} cartesian implies that neither $x$ nor $y$ is in $X_0$, so we must have $x=y$ because $f_2'$ is injective.  On the other hand, if $z \in X_1$, then since \eqref{topspacediagram} is cartesian, we see that $z \sim x$ and $z \sim y$, hence $x \sim y$ by transitivity.

Now we need to show that $X$ has the weak topology.  We need to show that a subset $U \subseteq X$ is open in $X$, assuming $f_i^{-1}(U)$ is open in $X_i$ for $i=1,2$.  Since $f_2^{-1}(U)$ is open in $X_2$, $f_2^{-1}(U) \setminus X_0$ is open in $X_2 \setminus X_0$, but we have $f_2^{-1}(U) \setminus X_0 = (f_2')^{-1}(U \setminus X_1)$ since \eqref{topspacediagram} is cartesian, hence $U \setminus X_1$ is open in $X \setminus X_1$ (hence also in $X$ itself) because $f_2'$ is an isomorphism.  Since $f_1 : X_1 \to X$ is a closed embedding and $U \cap X_1$ is open in $X_1$, $X \setminus (X_1 \setminus U)$ is open in $X$.  Since $X = X_1 \cup f_2(X_2)$, we have \be U & = & (X \setminus (X_1 \setminus U)) \cup (U \setminus X_1). \ee  This expresses $U$ as a union of two opens, so $U$ is open as desired. 

The final statement is easy to prove directly:  In this situation, since the diagram is cartesian, we have $X_0 = X_1 \cap X_2$ and since $X = X_1 \cup X_2$, it is clear that $X$ is the set-theoretic pushout.  To show that it has the weak topology, we just observe that if the $U \cap X_i$ are open in $X_i$, then the $X_i \setminus U$ are closed in $X_i$, hence also in $X$, so \be X \setminus U & = & (X_1 \setminus U) \cup (X_2 \setminus U) \ee is closed in $X$---i.e.\ $U$ is open in $X$ as desired. \end{proof}

\begin{lem} \label{lem:pushoutsintop2}  Consider a commutative diagram of topological spaces $X_\bullet$ as in \eqref{topspacediagram} and a map of spaces $f:X' \to X$.  Let $X_\bullet'$ be the commutative diagram of spaces obtained from $X_\bullet$ by base change along $f$.  \begin{enumerate} \item \label{openpullback} If $X' \to X$ is an open embedding and $X_\bullet$ is a pushout diagram, then $X_\bullet'$ is a pushout diagram. \item If $X' \to X$ is a quotient map and $X_\bullet'$ is a pushout diagram, then $X_\bullet$ is a pushout diagram. \end{enumerate} \end{lem}

\begin{proof} This is a straightforward exercise with the construction of pushouts of topological spaces discussed in the previous proof. \end{proof}

\begin{lem} \label{lem:affinepushouts} Consider a commutative diagram of affine schemes \bne{affschdiagram} & \xym{ \Spec C_0 \ar[r] \ar[d] & \Spec C_1 \ar[d] \\ \Spec C_2 \ar[r] & \Spec C } \ene  with the property that the underlying diagram of topological spaces is a pushout.  Then \eqref{affschdiagram} is a pushout diagram in ringed spaces iff the corresponding diagram of rings $C_\bullet$ is cartesian. \end{lem}

\begin{proof} The point is that the diagram of sheaves of rings on $\Spec C$ obtained by pushing forward the structure sheaves will be cartesian iff the corresponding diagram of rings is cartesian.  This is because the former is cartesian iff it is cartesian in the category of (quasi-coherent) sheaves on $\Spec C$ (or even in the category of sheaves of sets on $\Spec C$), but, viewed as such a diagram in $\Qco(C)$, it is obtained by applying the usual equivalence of categories $\Mod(C) \to \Qco(C)$ to $C_\bullet$.  Here we think of $C_\bullet$ as a $C$-module diagram and note that being cartesian as such is the same as being cartesian as a ring diagram. \end{proof}

\begin{lem} \label{lem:pushoutsinRS} Suppose $X_\bullet$ is a commutative diagram of schemes as in \eqref{topspacediagram} which is also a pushout diagram in ringed spaces.  Assume the $f_i$ and the $j_i$ are affine morphisms.  Then $X_\bullet$ is also a pushout diagram in schemes. \end{lem}

\begin{proof} The issue is to show that, when $g_i : X_i \to Y$ ($i=1,2$) are maps of schemes with $g_1j_1=g_2j_2$, the unique map $g : X \to Y$ of ringed spaces with $g_i=gf_i$ is, in fact, a map of schemes (i.e.\ a map of locally ringed spaces).  This can be checked locally near a point $x \in X$ with image $g(x) =: y$ in $Y$.  Pick some affine open neighborhood $V = \Spec B$ of $y$ in $Y$ and an affine open neighborhood $X' = \Spec C$ of $x$ in $g^{-1}(V)$.  It suffices to show that the map $g' := g|U : U \to V$ is a map of (affine) schemes.  I claim that the new commutative diagram $X_\bullet'$ obtained by pulling back $X_\bullet$ along the open embedding $X' \to X$ is also a pushout diagram in ringed spaces.  Indeed, on the level of spaces, this is Lemma~\ref{lem:pushoutsintop2}\eqref{openpullback} and, on the level of sheaves of rings, we just note that ``pushing forward commutes with pulling back along the inclusion of an open subset of the base," so the cartesian property of the original diagram of sheaves of rings on $X$ is also enjoyed by the diagram of pushed forward structure sheaves on $X'$ since the latter is obtained from the former simply by restricting to the open subspace $X'$.  Since the $f_i$ and $j_i$ are affine, $X'_\bullet = \Spec C_\bullet$ is a diagram of affine schemes which is a pushout in ringed spaces, so the corresponding diagram of rings $C_\bullet$ is cartesian by Lemma~\ref{lem:affinepushouts}.  The restrictions $g_i'$ of the $g_i$ to the $X_i'$ correspond to ring maps $h_i : B \to C_i$.  The fact that $g_1j_1=g_2j_2$ implies that these $h_i$ yield a ring map $h : B \to C$ such that $g_i' = (\Spec h)f_i'$.  But we also have $g_i' = g' f_i'$, so, since $X_\bullet'$ is a pushout in ringed spaces, we must have $g' = \Spec h$. \end{proof}

\begin{thm} \label{thm:affinepushouts} Suppose $C_\bullet$ is a commutative diagram of rings such that the diagram $\Spec C_\bullet$ of affine schemes is a pushout on the level of topological spaces.  Then $\Spec C_\bullet$ is a pushout diagram in both schemes and ringed spaces. \end{thm}

\begin{proof} Combine the two previous lemmas. \end{proof}

\begin{thm} \label{thm:affinepushouts2} Suppose $C_\bullet$ is a cartesian diagram of rings as in \eqref{pullbackringdiagram} where $f_1$ is surjective.  Then the corresponding diagram of affine schemes $\Spec C_\bullet$ is both cartesian and cocartesian in both schemes and ringed spaces. \end{thm}

\begin{proof} Set $I_1 := \Ker f_1$.  We view $C$ as the subring \be C & = & \{ (c_1,c_2): f_1(c_1)=f_2(c_2) \in C_0 \} \ee of $C_1 \times C_2$, so that the maps $p_i : C \to C_i$ are the projections.  Since $f_1$ is surjective, the map $p_2 : C \to C_2$ is also surjective, with kernel $I_1 \times \{ 0 \} \subseteq C$.  As a $C$-module, the kernel $I_1 \times \{ 0 \}$ ideal is nothing but the $C_1$-module $I_1$, regarded as a $C$-module via restriction of scalars along $p_1$.  The diagram $C_\bullet$ is also a pushout: \bne{pushoutcalculation} C_1 \otimes_C C_2 = C_1 / \Ker(p_2 : C \to C_2) C_1 = C_1/I_1 = C_0, \ene hence the diagram $\Spec C_\bullet$ is cartesian in schemes and also in ringed spaces because the base change of a closed embedding is the same, whether calculated in schemes or in ringed spaces.  To show that $\Spec C_\bullet$ is a pushout in schemes and ringed spaces, it suffices, by Theorem~\ref{thm:affinepushouts} to show that $\Spec C_\bullet$ is a pushout diagram on the level of topological spaces.  Since $\Spec C_\bullet$ is a cartesian diagram of topological spaces and $\Spec f_1$ is a closed embedding, Lemma~\ref{lem:pushoutsintop} reduces us to proving that the map \bne{needaniso} \Spec C_1 \setminus \Spec C_0 & \to & \Spec C \setminus \Spec C_2 \ene induced by \bne{mapinducingit} \Spec p_1 : \Spec C_1 & \to & \Spec C \ene is a homeomorphism of topological spaces.  (In fact we will show directly that \eqref{needaniso} is an isomorphism of schemes.)

We will treat the case where $f_2$ is also surjective in Proposition~\ref{prop:pushouts} below (we will have a lot more to say in that case).  By factoring $f_2$ as the surjection onto its image followed by the inclusion of its image and looking at the corresponding factorization of $C_\bullet$ as a composition of two cartesian diagrams, we can reduce the whole theorem to the case where $f_2$ is injective, which we now assume.

Since $f_2$ is injective, so is $p_1 : C \into C_1$.  As a subset of $C$, the kernel $I_1 \times \{ 0 \}$ is nothing but $I_1 \cap C$, regarding $C$ as a subset of $C_1$ via $p_1 : C \into C_1$.   

Since $p_2 : C \to C_2$ is surjective with kernel $I_1 \cap C$, we have \be \Spec C \setminus \Spec C_2 & = & \bigcup_{(d_1,d_2) \in I_1 \cap C} \Spec C[(d_1,d_2)^{-1}]. \ee  The preimage of the basic affine open subscheme $\Spec C[(d_1,d_2)^{-1}]$ of $\Spec C$ under \eqref{needaniso} (or, equivalently, under \eqref{mapinducingit}) is the affine open subscheme $\Spec C_1[d_1^{-1}]$ of $\Spec C_1 \setminus \Spec C_0$, and the restriction of \eqref{needaniso} to this open subscheme is $\Spec$ of the ring map \bne{d1d2} C[(d_1,d_2)^{-1}] & \to & C_1[d_1^{-1}] \ene  induced by $p_1 : C \into C_1$.  Since we can check isomorphy for \eqref{needaniso} locally on the base, it suffices to prove that \eqref{d1d2} is an isomorphism for each $(d_1,d_2) \in I_1 \cap C$.  Injectivity of \eqref{d1d2} is immediate from injectivity of $p_1$.  For surjectivity, it suffices to show that for every $c_1 \in C_1$, there is some $m \in \NN$ so that $c_1d_1^m \in C \subseteq C_1$.  In fact we can take $m=1$ because $c_1d_1 \in I_1$ (since $d_1 \in I_1$), so $f_1(c_1d_1)=0$, so $(c_1d_1,0) \in C$ as desired. \end{proof}

The following example will help to explain the above proof.

\begin{example} It is standard ``folklore" that one can ``contract a line in the plane to a point," though the resulting scheme will not be noetherian.  Here is what we mean.  Let us work over a field $k$ (in fact $k$ could be any ring at all, but let us retain this geometric point of view).  Define an $k$-algebra $C$ by the cartesian diagram \bne{pullbackringdiagramexample} \xym{ C \ar[r]^-{p_1} \ar[d]_{p_2} & k[x,y] \ar[d]^{y \mapsto 0} \\ k \ar[r]_{} & k[x] } \ene of $k$-algebras.  Explicitly, $p_1 : C \into k[x,y]$ is the subring consisting of elements of the form \bne{elementofC} c & = & a + \sum_{m = 0}^M \sum_{ n = 1}^N a_{m,n} x^m y^n , \ene where $a,a_{m,n} \in k$.  The kernel $K$ of $p_2$ is just the kernel $y k[x,y]$ of $y \mapsto 0$ intersected with $C$---i.e.\ the set of $c \in C$ as in \eqref{elementofC} with $a=0$. Explicitly, $K = I_1 \cap C$ is generated by \bne{Kgens} & y, xy, x^2 y, \dots, \ene but is not finitely generated (note $x \notin C$)---in particular, the ring $C$ is not noetherian.  Notice that this whole situation is ``toric" and that $C$ is the monoid algebra over $k$ on the submonoid $$ P = \langle (0,1),(1,1),(2,1), \dots \rangle $$ of $\NN^2$ (this $P$ is not finitely generated).   By Theorem~\ref{thm:affinepushouts2}, the corresponding diagram of schemes \bne{linecontract} & \xym{ \AA^1 \ar[r] \ar[d] & \ar[d] \AA^2 \ar[d] \\ \Spec k \ar[r] & \Spec C } \ene is both cartesian and cocartesian in both schemes and ringed spaces.  Recall that that theorem reduces to the statement that the natural map \bne{nattyiso}  \AA^2 \setminus \AA^1 = \Spec k[x,y,y^{-1}] & \to & \Spec C \setminus \Spec k \ene is an isomorphism of schemes.  Recall also that this fact boils down to the following:  For each $c \in K = I_1 \cap C$ and each $c_1 \in k[x,y]$, the element $cc_1 \in k[x,y]$ is in fact in the subring $C$.  This latter statement is obvious (such a $c$ is a polynomial in $x,y$ divisible by $y$).

Here is a slightly different argument that \eqref{nattyiso} is an isomorphism:  It is clear from the description of $C$ that the natural map \be C[y^{-1}] & \to & k[x,y,y^{-1}] \ee is an isomorphism, thus we reduce to showing that $\Spec C[y^{-1}] = \Spec C \setminus \Spec k$---i.e.\ that the ideals $yC$ and $K$ define the same closed subspace of $\Spec C$. Certainly $yC \subseteq K$, so we reduce to showing that the natural surjection $C/yC \to C/K = k$ induces a bijection on topological spaces.  It suffices to show that this surjection has nilpotent kernel.  Indeed, it has square-zero kernel because the product of any two of the generators \eqref{Kgens} for $K$ is in $yC$. \end{example}

For the rest of our study of gluing, we will restrict to the case where the maps $f_1,f_2$ (with kernels $I_1$, $I_2$, say) in the diagram $C_\bullet$ of \eqref{pullbackringdiagram} are surjective, and we will work over some base ring $A$.

\begin{prop} \label{prop:pushouts} Let $C_i \to C_0$ ($i=1,2$) be $A$-algebra surjections with fibered product $C := C_1 \times_{C_0} C_2$. \begin{enumerate} \item  \label{push1} The corresponding diagram of closed embeddings of schemes $$ \xym{ \Spec C_0 \ar[r] \ar[d] & \Spec C_1 \ar[d] \\ \Spec C_2 \ar[r] & \Spec C }$$ is both cartesian and cocartesian in both schemes and ringed spaces.  \item \label{push2} If $A \to C_0$ is flat, then formation of the fibered product ring $C$ commutes with extension of scalars along any ring map $A \to A'$: \be C \otimes_A A' & = & (C_1 \otimes_A A') \times_{C_0 \otimes_A A'} (C_2 \otimes_A A'). \ee \item \label{push3} If $C_0,C_1,C_2$ are flat over $A$ then $C$ is flat over $A$. \item \label{push4} If $C_1$ and $C_2$ are noetherian then $C$ is noetherian. \item \label{push5} If $C_1$ and $C_2$ are finite type over $A$ and at least one of the maps $C_i \to C_0$ has finitely generated kernel then $C$ is finite type over $A$. \end{enumerate} \end{prop}

\begin{proof} \eqref{push1}:  First recall that the diagram is cartesian (in schemes) because we noted in \eqref{pushoutcalculation} above that $C_0 = C_1 \otimes_C C_2$.  The base change of a closed embedding, taken in schemes, is the same as the one taken in ringed spaces, so the diagram is also cartesian in ringed spaces.  On the level of topological spaces, we have $\Spec C_1 \cap \Spec C_2 = \Spec C_0$ (for example, by the fact that the diagram is cartesian in ringed spaces).  By Theorem~\ref{thm:affinepushouts} and Lemma~\ref{lem:pushoutsintop} we reduce to proving that the $\Spec C_i$ cover $\Spec C$---i.e.\ that every prime ideal $\p$ of $C$ contains $I_1 \times 0$ or $0 \times I_2$.  If not, then $(i_1,0),(0,i_2) \notin \p$ for some $i_j \in I_j$, but this contradicts primeness of $\p$ because $(0,0)=(i_1,0)(0,i_2)$ is certainly in $\p$.  

Statement \eqref{push2} is just a fancy way of saying that the exact sequence of $C$-modules $$0 \to C \to C_1 \oplus C_2 \to C_0 \to 0$$ (the right map is the difference of the natural projections) will stay exact after applying $\slot \otimes_A A'$.  Statement \eqref{push3} follows from the same exact sequence because the kernel of a map of flats is flat.

For \eqref{push4}, it suffices to prove that every prime ideal $\p$ of $C$ is finitely generated (this is a famous theorem of I.~S.~Cohen).  We just saw above that $\p$ contains either $I_1 \times 0$ or $0 \times I_2$, so by symmetry we can assume that $I_1 \times 0 \subseteq \p$.  Since $C_1$ is noetherian, $I_1$ is finitely generated as a $C_1$-module, which implies $I_1 \times 0$ is finitely generated as a $C$-module.  Since $C_2$ is noetherian and $\p/(I_1 \times 0)$ is a (prime) ideal of $C_2$, it is finitely generated as a $C_2$-module, hence also as a $C$-module (because $C \to C_2$ is surjective), hence $\p$ is finitely generated because we have an exact sequence $$0 \to I_1 \times 0 \to \p \to \p/(I_1 \times 0) \to 0.$$   

For \eqref{push5}, pick finite sets of $A$-algebra generators $\{ x_i \} \subseteq C_1$ and $\{ y_j \} \subseteq C_2$.  We can assume by symmetry that the kernel of $C_1 \to C_0$ is generated by a finite set $\{ z_k \}$.  Since the $C_i \to C_0$ are surjective, we can find elements $p_i \in C_2$ and $q_j \in C_1$ so that the $(x_i,p_i)$ and $(q_j,y_j)$ are in $C$.  I claim that the latter elements, together with the elements $(z_k,0)$, generate $C$ as an $A$-algebra.  Given an arbitrary $(c_1,c_2) \in C$, find a polynomial $p$ with coefficients in $A$ so that $c_2 = p(\ov{y})$.  Then $p(\ov{(q,y)}) \in C$ has second coordinate $c_2$, so by subtracting it off we reduce to showing that an element of the form $(z,0)$ is in the $A$-subalgebra generated by our guys.  This $z$ is in the kernel of $C_1 \to C$, so we can write $z = \sum_k d_k z_k$ for some $d_k \in C_1$.  Find polynomials $t_k$ with coefficients in $A$ so that $d_k = t_k(\ov{x})$.  Then the formula \be (z,0) & = & \sum_k t_k(\ov{(x,p)})(z_k,0) \ee shows that $(z,0)$ is in our subalgebra as desired. \end{proof}

This pushout construction ``globalizes" as follows:

\begin{thm} \label{thm:pushouts} Suppose $X_0 \into X_i$ ($i=1,2$) are closed embeddings of $Y$-schemes.  \begin{enumerate} \item \label{cartcocartdiagram} The pushout $X := X_1 \coprod_{X_0} X_2$ in ringed spaces is also a scheme and the diagram of closed embeddings $$ \xym{ X_0 \ar[r] \ar[d] & X_1 \ar[d] \\ X_2 \ar[r] & X } $$ is both cartesian and cocartesian in both schemes and ringed spaces.  \item \label{pushouts2} If $X_0$ is flat over $Y$ then the formation of the pushout $X$ commutes with base change along any $Y' \to Y$: \be X \times_Y Y' & = & (X_1 \times_Y Y') \coprod_{X_0 \times_Y Y'} (X_2 \times_Y Y').\ee  \item \label{Xflat} If $X_0,X_1,X_2$ are flat over $Y$ then $X$ is flat over $Y$.  \item \label{noeth} If $X_1$ and $X_2$ are noetherian (or locally noetherian), so is $X$.  \item If $X_1$ and $X_2$ are of locally finite type over $Y$ and at least one of the closed embeddings $X_0 \into X_i$ is of locally finite presentation then $X$ is of locally finite type over $Y$. \item \label{pushouts6} The ideals of the closed embeddings in \eqref{cartcocartdiagram} are related by the formulas \be I_{X_1/X} & = & I_{X_0/X_2} \\ I_{X_2/X} & = & I_{X_0/X_1} \\ I_{X_0/X} & = & I_{X_0/X_1} \oplus I_{X_0/X_2} \ee (dropping notation for pushforward).  \item \label{fppflocalpushout} Suppose only that the diagram in \eqref{cartcocartdiagram} is a commutative diagram of closed embeddings of schemes.  If the diagram is a pushout and $g : X' \to X$ is flat, then the base changed diagram $$ \xym{ X_0 \times_X X' \ar[r] \ar[d] & X_1 \times_X X' \ar[d] \\ X_2 \times_X X' \ar[r] & X' } $$  is also a pushout.  The converse holds if $X' \to X$ is a flat cover (fppf or fpqc, say). \end{enumerate}   \end{thm}

\begin{proof} It is clear from the construction of the ringed space pushout that the maps in \eqref{cartcocartdiagram} are closed embeddings of ringed spaces (this would be true for any closed embeddings $X_0 \into X_i$ of ringed spaces).  If $X$ is a scheme, then all the maps in the diagram of \eqref{cartcocartdiagram} are maps (in fact closed embeddings) of schemes (because they are closed embeddings of ringed spaces, which are maps of locally ringed spaces if the ringed spaces happen to be locally ringed).  For similar reasons (the fact that the maps involved are closed embeddings) the pushout property in ringed spaces implies the one in schemes or locally ringed spaces.  So the only issue in \eqref{cartcocartdiagram} is to prove that $X$ is a scheme.  This is clear away from the ``gluing" locus $X_0 \subseteq X$ because \be X \setminus X_0 & = & (X_1 \setminus X_0) \coprod (X_2 \setminus X_0). \ee  

The issue is to prove that a point $x \in X_0$ has an open neighborhood in $X$ isomorphic to an affine scheme.  I claim that there exist affine open neighborhoods $U_i$ of $x$ in $X_i$ ($i=1,2$) so that $U_1 \cap X_0 = U_2 \cap X_0$.  This common intersection $Z$ is then an affine open subscheme of $X_0$ and it is clear from the construction of the ringed space pushout that the ringed space pushout $U_1 \coprod_Z U_2$ is an open neighborhood of $x$ in $X$.  But $U_1 \coprod_Z U_2$ is an affine scheme by Proposition~\ref{prop:pushouts}.  Parts \eqref{pushouts2}-\eqref{pushouts6} then follow from the analogous statements in Proposition~\ref{prop:pushouts} by working locally.

The proof of the claim is shorter and easier if we ``break the symmetry" a little bit.  Start by picking any affine open neighborhood $U = \Spec A$ of $x$ in $X_1$ and any affine open neighborhood $V = \Spec B$ of $x$ in $X_2$.  Write \be U \cap X_0 & = & \Spec (A/I) \\ V \cap X_0 & =  & \Spec (B/J) \ee for some ideals $I \subseteq A$ and $J \subseteq B$.  Then $U \cap V \cap X_0$ is an open neighborhood of $x$ in $X_0$, though it may not be affine.  However, it is an open neighborhood of $x$ in the affine scheme $U \cap X_0$, so it is contained in one of the usual basic open neighborhoods of $x$.  Thus we can find some $f \in A$ such that $$\Spec (A/I)_f = \Spec (A_f/IA_f) = (\Spec A_f) \cap X_0 $$ is an affine open neighborhood of $x$ in $U \cap V \cap X_0$.  The affine open subscheme $\Spec (A/I)_f$ of the affine scheme $\Spec (B/J)$ need not be principal (of the form $\Spec (B/J)_b$), but it will \emph{contain} such a principal open, so we can find $b \in B$ such that \bne{maininclusion} \Spec (B/J)_b & \subseteq \Spec (A/I)_f . \ene  A moment's thought shows that this means we have an equality \be \Spec (B/J)_b & = & \Spec ((A/I)_f)_b \ee (both are characterized as the largest open subscheme of $\Spec (B/J)$ on which the global section $b$ is invertible, say), where, on the right, $b \in (A/I)_f$ is abuse of notation for the image of $b \in B/J$ under the ring map corresponding to the inclusion \be \Spec (A/I)_f & \subseteq & \Spec (B/J) .\ee  Now write $b \in (A/I)_f$ in the form $b = af^m$ for some $a \in A$ and some integer $m$.  Now observe that \be (B/J)_b & = & (A/I)_{f,b} \\ & = & (A/I)_{a,f} \\ & = & (A/I)_{af}. \ee  It follows easily that $U_1 := \Spec A_{af}$ and $U_2 := \Spec B_b$ will do the job.  

It remains to prove \eqref{fppflocalpushout}.  We first prove that the base changed diagram remains a pushout on the level of topological spaces (this doesn't require $g$ flat).  The base change of a closed embedding is the same as the one calculated in ringed spaces; in particular, on the level of spaces, the base changed diagram is just the obvious diagram of closed embeddings $$ \xym{ g^{-1}(X_0) \ar[r] \ar[d] & g^{-1}(X_2) \ar[d] \\ g^{-1}(X_1) \ar[r] & X' } $$ and it follows easily from Lemma~\ref{lem:pushoutsintop} that this is a pushout diagram.  To check that the base changed diagram is a pushout in ringed spaces, it remains to check that the obvious diagram of sheaves of rings on $X$ is a pullback (equivalently the underlying diagram of quasi-coherent $\O_X$-modules is a pullback), but this is clear because it is obtained from the analogous pullback diagram on $X$ by applying the exact functor $g^*$ (here we use $g$ flat).  For the converse when $g$ is a flat cover:  Since we assume from the outset that the diagram is a diagram of closed embeddings, the question of whether it is a pushout on spaces is set-theoretic in nature (Lemma~\ref{lem:pushoutsintop}) and could be checked after pulling back along any $g$ which is surjective on the level of spaces.  Once it is known that the diagram is a pushout on the level of spaces, all that remains to be checked is that the usual diagram of sheaves on rings on $X$ is a pullback, and this can be checked after pulling back along a flat cover. \end{proof}

\begin{rem} I chose to carefully isolate and prove the claim in the above proof because of the following subtlety:  For a closed embedding of schemes $X \into Y$, the map \be \{ {\rm \, affine \; open \; subschemes \; of \;} Y \, \} & \to & \{ \, {\rm affine \; open \; subschemes \; of \;} X \, \} \\ U & \mapsto & U \cap X \ee need not be surjective.  In fact $X$ itself can fail to be in the image of this map even when $X$ is a smooth affine divisor in a smooth complex variety $Y$. \end{rem}

\begin{rem} Given a diagram as in Theorem~\ref{thm:pushouts}\eqref{cartcocartdiagram}, one has \be \O_X  & = &  \O_{X_1} \times_{\O_{X_0}} \O_{X_2} \ee (dropping notation for pushforwards).  By Theorem~\ref{thm:pushouts}\eqref{fppflocalpushout}, this same equality holds even using the \'etale (or fppf) structure sheaves.  We will use this fact without further comment. \end{rem} 

\begin{rem} Diagrams as in Theorem~\ref{thm:pushouts}\eqref{cartcocartdiagram} are ubiquitous.  Suppose $X$ is a reduced scheme and $X_1,X_2$ are closed subschemes of $X$ covering $X$ on the level of topological spaces.  Then if we define $X_0 := X_1 \cap X_2$ by declaring that diagram to be \emph{cartesian}, then in fact it is also cocartesian \cite[6.2]{GG}. \end{rem}

\subsection{Descent scholium} \label{section:descentscholium}  We continue with the affine setup of \S\ref{section:gluingscholium}: $C_i \to C_0$ ($i=1,2$) are $A$-algebra surjections with fibered product $C = C_1 \times_{C_0} C_1$.  To ease notation, we let \be \Desc & := & \Mod(C_1) \times_{\Mod(C_0)} \Mod(C_1) \ee denoted ``the" $2$-fibered product taken in the $2$-category of abelian categories.  Objects of $\Desc$ are triples $(M_1,M_2,\phi)$ consisting of objects $M_i \in \Mod(C_i)$ ($i=1,2$) and a $\Mod(C_0)$ isomorphism $\phi : M_1/I_1M_1 \to M_2/I_2M_2$, called the \emph{clutching function}.  Morphisms \be (f_1,f_2) :(M_1,M_2,\phi) & \to & (N_1,N_2,\psi) \ee are pairs consisting of $\Mod(C_i)$-morphisms $f_i : M_i \to N_i$ ($i=1,2$) commuting with the clutching functions---i.e.\ making the $\Mod(C_0)$-diagram below commute.  $$ \xym{ M_1/I_1M_1 \ar[r]^{\phi}_{\cong} \ar[d]_{f_1/I_1} & M_2/I_2M_2 \ar[d]^{f_2/I_2} \\ N_1 /I_1N_1 \ar[r]^{\psi}_{\cong} & N_2/I_2N_2 } $$  If there is no chance of confusion we will often drop the $\phi$ from the notation and simply write $(M_1,M_2)$ for an object of $\Desc$.  For $(M_1,M_2) \in \Desc$, \emph{we often set} \be  M_0 & := & M_2/I_2,\ee and we refer to the composition of the natural projection $M_1 \to M_1/I_1$ and $\phi$ as ``the natural projection $M_1 \to M_0$."  

\begin{rem} The category $\Desc$ is nothing but the ``descent category" for the ``cover" of $X=\Spec C$ by \be U & := \Spec C_1 \coprod \Spec C_2.\ee  (There is no ``triple overlaps" condition on $\phi$ because all points of $U \times_X U \times_X U$ are degenerate.)  It would be more familiar if this were an open cover rather than a closed cover. \end{rem}

There is an obvious \emph{pullback functor} \be P : \Mod(C) & \to & \Desc \\ M & \mapsto & P(M) := (M/I_2M, M/I_1M). \ee  By abuse of notation, we are writing, say, $I_2$ instead of the more correct $0 \times I_2$ for the kernel of $C \to C_1$.  The implicit clutching function $\phi$ here is the canonical identification \be (M/I_2M)/I_1 (M/I_2M)  & = &  M / (I_1 \times I_2)M \\ & = &  (M/I_1M)/I_2 (M/I_1M). \ee  One might also use the notation \be P(M) & := & (M \otimes_C C_1, M \otimes_C C_2) \ee so that the implicit $\phi$ would be the canonical isomorphism \be (M \otimes_C C_1) \otimes_{C_1} C_0 & = & M \otimes_C C_0 \\ & = & (M \otimes_C C_2) \otimes_{C_2} C_0. \ee The pullback functor $P$ admits a right adjoint \emph{descent functor} \be D : \Desc & \to & \Mod(C) \\ (M_1,M_2) & \mapsto & D(M_1,M_2) := M_1 \times_{M_0} M_2 . \ee  The implicit maps $M_i \to M_0$ are the natural projections.  This set-theoretic fibered product $D(M_1,M_2)$ is of course an abelian group (the abelian group fibered product) and becomes a $C$-module via the scalar multiplication \be (c_1,c_2) \cdot (m_1,m_2) & := & (c_1m_1,c_2m_2). \ee  The $C$-module $D(M_1,M_2)$ can also be defined by the short exact sequence of $C$-modules \bne{Dsequence} & 0 \to D(M_1,M_2) \to M_1 \oplus M_2 \to M_0 \to 0, \ene where the right map is the difference of the natural projections (the $C_i$-module $M_i$ is of course regarded as a $C$-module by restriction of scalars along $C \to C_i$).

\begin{rem} \label{rem:finitegeneration} It is clear from the above description of $D(M_1,M_2)$ that when $C$ is noetherian and the $M_i \in \Mod(C_i)$ ($i=1,2$) are finitely generated, $D(M_1,M_2)$ is also finitely generated since it is contained in $M_1 \oplus M_2$. \end{rem}

The adjunction isomorphism is given by \be \Hom_{\Mod(C)}(M,D(M_1,M_2)) & \to & \Hom_{\Desc}(P(M),(M_1,M_2)) \\ f & \mapsto & (\ov{m} \mapsto \pi_1 f(m), \ov{m} \mapsto \pi_2 f(m)) \ee and its inverse is given by \be (f_1,f_2) & \mapsto & (m \mapsto (f_1(\ov{m}),f_2(\ov{m}))). \ee  It is clear that these maps are inverse; the issue is to check that they are well-defined, which is also fairly straightforward.  For example, to check that $\pi_1 f(m) \in M_1$ is independent of the chosen lift $m \in M$ of $\ov{m} \in M/(0 \times I_2)M$ for a $C$-module map $f : M \to D(M_1,M_2) = M_1 \times_{M_2/I_2M_2} M_2$, we need to check that when we write $f$ set-theoretically as $f = f_1 \times f_2$ (i.e.\ we set $f_i := \pi_i f$), the map $f_1 : M \to M_1$ kills $(0 \times I_2)M$.  This is an easy calculation with $C$-linearity of $f$ and the $C$-module structure on $D(M_1,M_2)$: \be (f_1((0,i)m),f_2((0,i)m)) & = & f((0,i)m) \\ & = & (0,i) \cdot f(m) \\ & = & (0,i) \cdot (f_1(m),f_2(m)) \\ & = & (0,if_2(m)) . \ee

For $(M_1,M_2) \in \Desc$ we natural $\Mod(C_1)$-isomorphisms \bne{calculation} D(M_1,M_2) / (0 \times I_2) D(M_1,M_2) & = & (M_1 \times_{M_2/I_2M_2} M_2) / (0 \times I_2M_2) \\ \nonumber & = & M_1 \times_{M_2/I_2M_2} M_2/I_2M_2 \\ \nonumber & = & M_1. \ene  There is a similar natural isomorphism of $C_2$-modules \be D(M_1,M_2) / (I_1 \times 0) & = & M_2. \ee  These last two isomorphisms combine to give a natural isomorphism \be PD(M_1,M_2) & \to & (M_1,M_2) \ee for all $(M_1,M_2) \in \Desc$.  Using the explicit formulas for the adjunction isomorphisms, one checks that this natural isomorphism is nothing but the adjunction morphism $PD \to \Id$ evaluated on $(M_1,M_2)$.  In other words, $PD \to \Id$ is an isomorphism of functors.

The other adjunction morphism $\Id \to DP$, however, is not generally an isomorphism; $P$ and $D$ are not generally equivalences.

\begin{example} \label{example:nodaldescent} Consider the case where $C = k[x,y]/(xy) = k[x] \times_k k[y]$, so $C_1 = k[x]$ and $C_2 = k[y]$ in the above setup.  The $C$-module $M := C/(x+y)$ has \be P(M) & = & (M/yM , M/xM) \\ & = & (k[x]/x, k[y]/y) \\ & = & (k,k) \in \Desc \ee (the last equality should be viewed as defining the $k[x]$- and $k[y]$-module structures on $k$).  But if we view $k$ as a $C$-module via the natural isomorphism $C/(x,y) = k$, then we also have $P(k) = (k,k)$ despite the fact that $M$ and $k$ are not isomorphic $C$-modules (they don't even have the same dimension as $k$ vector spaces).  Hence $P$ can't be faithful.  Also, $M$ can't be in the essential image of $D$, for if it were, then since we know $PD \to \Id$ is an isomorphism, $M$ would have to be isomorphic to $D(k,k)$, but $D(k,k)=k$ isn't isomorphic to $M$. \end{example}

In the global situation of a pushout diagram of closed embeddings of schemes as in Theorem~\ref{thm:pushouts}\eqref{cartcocartdiagram}, we again have a \emph{pullback functor} \bne{globalpullback} P : \Mod(X) & \to & \Mod(X_1) \times_{\Mod(X_0)} \Mod(X_2) =: \Desc \ene and a right adjoint \emph{descent functor} \bne{globaldescent} D : \Desc & \to & \Mod(X). \ene  The adjunction map $PD \to \Id$ is again an isomorphism---the constructions and arguments we made for a pullback diagram of rings make perfect sense for the pullback diagram $$ \xym{ \O_X \ar[r] \ar[d] & \O_{X_1} \ar[d] \\ \O_{X_2} \ar[r] & \O_{X_0} } $$ of sheaves of rings on $X$ (dropping notation for pushforwards along the various closed embeddings).  We can replace ``$\Mod$" everywhere by ``$\Qco$" if we wish because the exact sequence \eqref{Dsequence} defining $D(M_1,M_2)$ shows that $D(M_1,M_2)$ is a quasi-coherent $\O_X$-module when the $M_i$ are quasi-coherent $\O_{X_i}$-modules.  If $X_1$ and $X_2$ are locally noetherian (so $X$ is also locally noetherian by Theorem~\ref{thm:pushouts}\eqref{noeth}), we can replace ``$\Qco$" everywhere by ``$\Coh$" (c.f.\ Remark~\ref{rem:finitegeneration}).

\begin{thm} \label{thm:descent}  Consider a pushout diagram of closed embeddings of schemes as in Theorem~\ref{thm:pushouts}\eqref{cartcocartdiagram}.  Recall the pullback and descent functors \eqref{globalpullback} and \eqref{globaldescent}.  \begin{enumerate} \item \label{desc1} If $\Tor_1^X(M,\O_{X_0})=0$ for an $\O_X$-module $M$, then for $i=1,2$ we have \be \Tor_1^X(M,\O_{X_i}) & = & 0 \\ \Tor_1^{X_i}(M|X_i,\O_{X_0}) & = & 0. \ee \item \label{desc2} If $(M_1,M_2) \in \Desc$ has $\Tor_1^{X_i}(M_i,\O_{X_0})=0$ for $i=1,2$, then \be \Tor_1^X(D(M_1,M_2),\O_{X_0}) & = & 0.\ee \item \label{desc3} Let $\Mod(X)^\circ$ denote the full subcategory of $\Mod(X)$ consisting of those $\O_X$-modules $M$ with $\Tor_1^X(M,\O_{X_0})=0$.  Let $\Desc^\circ$ denote the full subcategory of $\Desc$ consisting of those $(M_1,M_2)$ with $\Tor_1^{X_i}(M_i,\O_{X_0})=0$ for $i=1,2$.  By \eqref{desc1} and \eqref{desc2}, $P$ and $D$ restrict to functors \be P : \Mod(X)^\circ & \to & \Desc^{\circ} \\ D : \Desc^{\circ} & \to & \Mod(X)^{\circ}. \ee  These functors are inverse equivalences of categories. \item \label{desc4} For $M,N \in \Mod(X)^\circ$, we have natural isomorphisms of abelian groups \be \Hom_X(M,N) & = & \Hom_{X_1}(M_1,N_1) \times_{\Hom_{X_0}(M_0,N_0)} \Hom_{X_2}(M_2,N_2) \\ \Ext^1_X(M,N) & = & \Ext^1_{X_1}(M_1,N_1) \times_{\Ext^1_{X_0}(M_0,N_0)} \Ext^1_{X_2}(M_2,N_2) \ee where $(M_1,M_2)$ and $(N_1,N_2)$ are the images of $M$ and $N$ under $P$. \end{enumerate} See Remark~\ref{rem:descent} below for variations. \end{thm}

\begin{proof} \eqref{desc1}: If $\Tor_1^X(M,\O_{X_0})=0$, the applying $\slot \otimes_{\O_X} M$ to the exact sequence $$0 \to \O_X \to \O_{X_1} \oplus \O_{X_2} \to \O_{X_0} \to 0 $$ shows that $\Tor_1^X(M,\O_{X_1} \oplus \O_{X_2})=0$; for the second vanishing in \eqref{desc1}, we note that $M \otimes_{\O_{X}} \O_{X_0} = M \otimes_{\O_X} \O_{X_i} \otimes_{\O_{X_i}} \O_{X_0}$, so the first vanishing implies \be \Tor_1^X(M,\O_{X_0}) & = & \Tor_1^{X_i}(M|X_i, \O_{X_0}), \ee which vanishes by hypothesis.  

\eqref{desc2}:  The hypothesized vanishing is equivalent to the exactness of the natural sequences \bne{natseqs} & 0 \to I_{X_0/X_i} \otimes_{\O_{X_i}} M_i \to M_i \to M_i|X_0 \to 0 \ene on $X_i$ for $i=1,2$.  Since $I_{X_0/X} = I_{X_0/X_1} \oplus I_{X_0/X_2}$ (Theorem~\ref{thm:pushouts}\eqref{pushouts6}), we have a short exact sequence $$ 0 \to I_{X_0/X_1} \oplus I_{X_0/X_2} \to \O_X \to \O_{X_0} \to 0 $$ and the vanishing we want to establish is equivalent to the injectivity of the natural map \bne{testermap} (I_{X_0/X_1} \otimes D(M_1,M_2)) \oplus (I_{X_0/X_2} \otimes D(M_1,M_2)) & \to & D(M_1,M_2). \ene But $I_{X_0/X_i}$ is an $\O_{X_i}$ module being regarded as an $\O_X$-module by restriction of scalars, so we have \be I_{X_0/X_i} \otimes D(M_1,M_2) & = & I_{X_0/X_i} \otimes_{\O_{X_i}} D(M_1,M_2)|X_i \\ & = & I_{X_0/X_i} \otimes_{\O_{X_i}} M_i \ee (making use of the computation \eqref{calculation}).  If we trace through these natural isomorphisms, we see that the composition of \eqref{testermap} and the natural inclusion $$D(M_1,M_2) \into M_1 \oplus M_2$$ (c.f. \eqref{Dsequence}) is nothing but the sum of the left maps in the sequences \eqref{natseqs}.  Hence \eqref{testermap} is injective and the desired $\Tor$-vanishing is established. 

For \eqref{desc3} we already know by the general discussion above that the adjunction map $PD \to \Id$ is an isomorphism, so the only issue is to check that the other adjunction map $M \to DPM$ is an isomorphism for $M \in \Mod(X)^{\circ}$.  The adjunction map in question is the natural map \bne{themapinquestion} M & \to & M|X_1 \times_{M|X_0} M|X_1. \ene  Isomorphy for this map is clearly equivalent to exactness of the natural sequence $$ 0 \to M \to M|X_1 \oplus M|X_2 \to M|X_0 \to 0.$$  But this sequence is obtained by tensoring $$0 \to \O_X \to \O_{X_1} \oplus \O_{X_2} \to \O_{X_0} \to 0 $$ with $M$, hence it is exact because $M \in \Mod(X)^{\circ}$ satisfies $\Tor_1^X(M,\O_{X_0})=0$.

The first isomorphism in \eqref{desc4} is immediate from the equivalence of categories in \eqref{desc3}.  For the $\Ext^1$ isomorphisms, we need to first explain the maps involved.  An element of $\Ext^1_X(M,N)$ can be viewed as an isomorphism class of short exact sequences \be \u{E} & = & (0 \to N \to E \to M \to 0) \ee in $\Mod(X)$.  The image of $\u{E}$ under the purported isomorphism will be the pair $(\u{E}|X_1,\u{E}|X_2)$.  Note that the sequences $\u{E}|X_i$ are still exact because of the first $\Tor$-vanishing property of $M$ in \eqref{desc1}.  The restriction maps $$\Ext^1_{X_i}(M_i,N_i) \to \Ext^1_{X_0}(M_0,N_0)$$ are defined similarly.  To define the map in the other direction, suppose we have exact sequences \be \u{E}_i & = & (0 \to N_i \to E_i \to M_i \to 0) \ee in $\Mod(X_i)$ with the same image in $\Ext^1_{X_0}(M_0,N_0)$.  This means we have an isomorphism $\phi : E_1|X_0 \to E_2|X_0$ making $$ \xym{ 0 \ar[r] & M_1|X_0 \ar@{=}[d] \ar[r] & E_1|X_0 \ar[r] \ar[d]^{\phi} & N_1|X_0 \ar@{=}[d] \ar[r] & 0 \\ 0 \ar[r] & M_2|X_0 \ar[r] & E_2|X_0 \ar[r] & N_2|X_0 \ar[r] & 0 } $$ commute.  We can then view $(E_1,E_2)$ as an object of $\Desc$ using this $\phi$.  We then check that \bne{newguy} & 0 \to M \to D(E_1,E_2) \to N \to 0 \ene is exact by applying the Snake Lemma to the diagram \bne{SnakeLemmadiagram} & \xym{ 0 \ar[r] & N_1 \oplus N_2 \ar[d] \ar[r] & E_1 \oplus E_2 \ar[d] \ar[r] &  M_1 \oplus M_2 \ar[r] \ar[d] & 0 \\ 0 \ar[r] & N_0 \ar[r] & E_0 \ar[r] & M_0 \ar[r] & 0 } \ene where the vertical arrows are the differences of the natural projections.  We identify the sequence (exact by the Snake Lemma) of kernels of the vertical surjections in \eqref{SnakeLemmadiagram} with \eqref{newguy} using the $\Tor$-vanishing properties of $M$ and $N$ in \eqref{desc1}.  It follows from the Five Lemma that the two constructions are inverse.  Another way of putting this is that when $M,N \in \Mod(X)^\circ$ we have $E \in \Mod(X)^{\circ}$ for any short exact sequence $\u{E}$, and so the short exact sequence $\u{E}$ is the image of an essentially unique sequence in $\Desc$ by \eqref{desc3}, which one then checks is exact.  \end{proof}

\begin{rem} \label{rem:descent}  There are many possible variations of the equivalence of categories in Theorem~\ref{thm:descent}\eqref{desc3}.  Let $\Mod(X_i)^{\circ} \subseteq \Mod(X_i)$ denote the full subcategory whose objects are $\O_{X_i}$-modules $M_i$ with $\Tor_1^{X_i}(M_i,\O_{X_0})=0$.  Suppose \be \E & \subseteq & \Mod(X)^{\circ} \\ \E_i & \subseteq & \Mod(X_i)^{\circ} \quad \quad (i=1,2) \\ \E_0 & \subseteq & \Mod(X_0) \ee are full subcategories satisfying the following: \begin{enumerate} \item For every $M \in \E$, $M|X_i \in \E_i$ for $i=1,2$. \item For $i=1,2$ and $M_i \in \E_i$, $M_i|X_0 \in \E_0$. \end{enumerate} These hypotheses that the usual restriction $M_i \mapsto M_i|\E_0$ defines functors $\E_i \to \E_0$ and that $P$ restricts to a functor $\E \to \E_1 \times_{\E_0} \E_2$.  If we further assume \begin{enumerate} \setcounter{enumi}{2} \item For any $(M_1,M_2) \in \E_1 \times_{\E_0} \E_2$, the object $D(M_1,M_2) \in \Mod(X)^{\circ}$ defined by the short exact sequence $$0 \to D(M_1,M_2) \to M_1 \oplus M_2 \to M_0 \to 0 $$ is in $\E$ \end{enumerate} then $D$ manifestly restricts to an inverse of $P$. 

Let us agree that a  \emph{thick subcategory} of an abelian category $\A$ is a full, additive subcategory $\B$ of $\A$ closed under extensions and kernels (in $\A$).  If we further assume that the full subcategories $\E$, $\E_i$ are thick, then the aforementioned equivalence of categories also yields an identification of $\Ext$ groups as in Theorem~\ref{thm:descent}\eqref{desc4}.

For example, we can replace ``$\Mod$" everywhere by ``$\Qco$" in Theorem~\ref{thm:descent}.  If $X_1$ and $X_2$ are locally noetherian (hence $X_0$ and $X$ are locally noetherian), then we can replace ``$\Mod$" everywhere by ``$\Coh$".  

As another example, suppose the diagram of closed embeddings is a diagram of $Y$-schemes.  Then we can take $\E \subseteq \Mod(X)^\circ$ to be the (thick!) subcategory of objects $M$ such that $M$ is flat over $Y$ and each $M|X_i$ is flat over $Y$, $\E_i \subseteq \Mod(X_i)^{\circ}$ to be the thick subcategory of objects $M_i$ such that $M_i$ and $M_i|X_0$ are flat over $Y$, and we can take $\E_0 \subseteq \Mod(X_0)$ to be the thick subcategory of $\O_{X_0}$-modules which are flat over $Y$.  The first two hypotheses are satisfied by definition and the third holds since a kernel of flats is flat.   \end{rem}

The equivalence of categories in Theorem~\ref{thm:descent} is also ``compatible with tensor product" in a sense we now make precise.  Notice that the abelian category $\Desc$ is equipped with a tensor product (symmetric monoidal structure compatible with the abelian category structure) $\boxtimes$, defined by \be (M_1,M_2) \boxtimes (N_1,N_2) & := & (M_1 \otimes N_1, M_2 \otimes N_2). \ee  The subcategory $\Mod(X)^\circ \subseteq \Mod(X)$ is not generally closed under the tensor product.  However:

\begin{lem} In the situation of Theorem~\ref{thm:descent}, suppose $M,N \in \Mod(X)^{\circ}$ satisfy \bne{vanishinghypothesis} \Tor_1^{X_0}(M|X_0,N|X_0) & = & 0. \ene  Then $M \otimes N \in \Mod(X)^{\circ}$.  \end{lem}

\begin{proof}  We need to show that $\Tor_1^X(M \otimes N,\O_{X_0}) = 0$ when $M,N \in \Mod(X)^\circ$ satisfy \eqref{vanishinghypothesis}.  Since we can factor the tensor product $M \otimes N \otimes \slot$ as a composition of $N \otimes \slot$ followed by $M \otimes \slot$, it suffices to establish the vanishings \bne{need1} \Tor_1^X(N, \O_{X_0}) & = & 0 \\ \label{need2} \Tor_1^X(M, N \otimes \O_{X_0}) & = & 0. \ene  The vanishing \eqref{need1} holds since $N \in \Mod(X)^{\circ}$.  Since we can factor $\slot \otimes_X (N \otimes_{\O_{X_0}})$ as the composition of $\slot \otimes_X \O_{X_0} = \slot|X_0$ and \be \slot \otimes_{X_0} (N \otimes_X \O_{X_0}) & = & \slot \otimes (N|X_0), \ee the vanishing \eqref{need2} will follow from the vanishings \bne{need3} \Tor_1^X(M, \O_{X_0}) & = & 0 \\ \label{need4} \Tor_1^{X_0}(M|X_0,N|X_0) & = & 0. \ene The vanishing \eqref{need3} holds since $M \in \Mod(X)^{\circ}$ and \eqref{need4} holds by hypothesis. \end{proof}

It is clear that $P$ commutes with tensor products: \bne{Pcompatibility} P(M \otimes N) & = & P(M) \boxtimes P(N) \ene because restriction commutes with tensor porudcts.  (This is true regardless of whether $M,N$ are in $\Mod(X)^{\circ}$, or whether they satisfy the vanishing \eqref{vanishinghypothesis}.)  The corresponding formula \be D((M_1,M_2) \boxtimes D(N_1,N_2)) & = & D(M_1,M_2) \otimes D(N_1,N_2) \ee for $D$ will probably not hold in general.  However, this formula will hold if we assume that $M := D(M_1,M_2)$ and $N := D(N_1,N_2)$ are in $\Mod(X)^{\circ}$ and satisfy the vanishing \eqref{vanishinghypothesis}, or, perhaps more generally, as long as we know that $M$, $N$, and $M \otimes N$ are in $\Mod(X)^{\circ}$.  Indeed, in that case we can just make a ridiculous calculation like \be D(M_1,M_2) \otimes D(N_1,N_2) & = & DP \left ( D(M_1,M_2) \otimes D(N_1,N_2) \right ) \\ & = & D( PD(M_1,M_2) \boxtimes PD(N_1,N_2) ) \\ & = &  D( (M_1,M_2) \boxtimes (N_1,N_2) ), \ee using the known compatibility \eqref{Pcompatibility} of $P$ with tensor products and the fact that the appropriate adjunction maps are isomorphisms.

There are also variants of the equivalence of categories in Theorem~\ref{thm:descent} for subcategories of $\Mod(X)$ which are not full.  For example, for a scheme $X$, let $\u{\Phi}_r(X)$ denote the category of rank $r$ locally free $\O_X$-modules (``rank $r$ vector bundles") whose only morphisms are isomorphisms.  Thus $\u{\Phi}_r(X)$ is a groupoid whose set of isomorphism classes is the set of rank $r$ bundles on $X$ up to isomorphism.  In the situation of Theorem~\ref{thm:descent}, restriction clearly defines a functor \bne{Pr} P : \u{\Phi}_r(X) & \to & \u{\Phi}_r(X_1) \times_{\u{\Phi}_r(X_0)} \u{\Phi}_r(X_2) =: \Desc_r. \ene  Similarly, the descent functor $D$ defines a functor \bne{Dr} D : \Desc_r & \to & \u{\Phi}_r(X). \ene  Indeed, the question of whether $D(M_1,M_2)$ is a rank $r$ bundle when the $M_i$ are rank $r$ bundles is local in nature (as is the formation of $D(M_1,M_2)$), so we can assume there are isomorphisms $M_i \cong  \O_{X_i}^r$ for $r=1,2$.  Using these chosen isomorphisms, the clutching isomorphism $\phi : M_1|X_0 \to M_2|X_0$ can be viewed as a matrix $G \in \GL_r(X_0)$.  Since $X_0 \into X_1$ is a closed embedding, we can lift $G$ to $\ov{G} \in \GL_r(X_1)$, after possibly shrinking.  If we now appropriately adjust our choice of isomorphism $M_1 \cong \O_{X_1}^r$ using $\ov{G}$, then, in the ``new" bases, the clutching map $\phi$ becomes the identity matrix.  We have shown that, at least locally, every object of $\Desc_r$ is isomorphic to the object $(\O_{X_1}^r,\O_{X_2}^r)$ with the trivial clutching function, and it is clear that $D$ applied to this object is just $\O_X^r$.  Since the necessary $\Tor$-vanishing hypotheses in Theorem~\ref{thm:descent} hold trivially for bundles, we see that the functor \eqref{Pr} is an equivalence of categories with inverse \eqref{Dr}.

In particular, when $r=1$, $\u{\Phi}_1$ is the Picard stack and our equivalence (which is compatible with the tensor product) becomes an equivalence of Picard categories \be \u{\Pic}(X) & = & \u{\Pic}(X_1) \times_{\u{\Pic}(X_0)} \u{\Pic}(X_2). \ee  One should not be misled into thinking that this equivalence of categories yields an isomorphism of Picard \emph{groups} \be \Pic(X) & = & \Pic(X_1) \times_{\Pic(X_0)} \Pic(X_2). \ee  You can't ``commute" the formation of the fiber product of Picard categories with passing to their underlying groups of isomorphism classes.  Similar warnings of course apply to our equivalence between \emph{categories} of bundles.

\subsection{Gluing} \label{section:gluing}  In this section we will explain the role of log flatness in the descent constructions of the previous section.  We will use this to define gluing maps for moduli spaces of log quotients (\S\ref{section:logquotientspace}), and to show that these gluing maps are compatible with the natural deformation/obstruction theory for moduli spaces of quotients.

We begin with the following warmup, which is just a special case of Theorem~\ref{thm:descent} rephrased in the language of graded flatness.

\begin{thm} \label{thm:gluing} Let $k$ be a field.  Let $B := k[x,y]/(xy)$ graded by $\ZZ$ so that $|x|=1$, $|y|=-1$.  Let $B \to C$ be a ring homomorphism.  Set $C_1 := C/yC$, $C_2 := C/xC$, $C_0 := C/(x,y)C$.  Suppose that $C$ is graded flat over $(\ZZ,B)$.  Then: \begin{enumerate} \item \label{cartesiancocartesian} The diagram of ring surjections $$ \xym{ C \ar[r] \ar[d] & C_1 \ar[d] \\ C_2 \ar[r] & C_0 } $$ is both cartesian and cocartesian.  \item \label{pullbackrestricts} Let $\Mod(C)^\circ \subseteq \Mod(C)$ be the full subcategory of $C$-modules graded flat over $B$ and let $\Mod(C_i)^\circ \subseteq \Mod(C_i)$ be the full subcategory of $C_i$-modules graded flat over $k[x]$ ($i=1$) or $k[y]$ ($i=2$).  Then the pullback functor \be P : \Mod(C) & \to & \Mod(C_1) \times_{\Mod(C_0)} \Mod(C_2) \\ M & \mapsto & (M/yM,M/xM) \ee and the restriction functor \be D : \Mod(C_1) \times_{\Mod(C_0)} \Mod(C_2) & \to & \Mod(C) \\ (M_1,M_2) & \mapsto & D(M_1,M_2)=M_1 \times_{M_0} M_2 \ee descend to an equivalence of categories \be \Mod(C)^\circ & = & \Mod(C_1)^{\circ} \times_{\Mod(C_0)} \Mod(C_2)^{\circ}. \ee \item \label{defobcompatibility} For $M,N \in \Mod(C)^\circ$ we have natural isomorphisms of abelian groups \be \Hom_C(M,N) & = & \Hom_{C_1}(M_1,N_1) \times_{\Hom_{C_0}(M_0,N_0)} \Hom_{C_2}(M_2,N_2) \\ \Ext^1_C(M,N) & = & \Ext^1_{C_1}(M_1,N_1) \times_{\Ext^1_{C_0}(M_0,N_0)} \Ext^1_{C_2}(M_2,N_2) \ee where we set $M_1 := M/yM$, $M_2 := M/xM$, $M_0 := M/(x,y)M$, and similarly with $M$ replaced by $N$. \end{enumerate} \end{thm}

\begin{proof} For \eqref{cartesiancocartesian}: The diagram is clealy a pushout even without the graded flatness assumption on $C$: $$ C_1 \otimes_C C_2  =  C_1/yC_1 = C_0.$$ To see that the diagram is a pullback we need to check that the natural projections yield an isomorphism \be C & = & \{ (c_1,c_2) \in C_1 \times C_2 : \ov{c}_1 = \ov{c}_2 \in C_0 \}. \ee  This is equivalent to saying that the sequence of $C$-modules $$0 \to C \to C_1 \oplus C_2 \to C_0 \to 0 $$ is exact, where the right map is the difference of the natural projections.  But this sequence is exact since it is obtained by applying $\slot \otimes_B C$ to the analogous sequence \bne{Bseqin1} & 0 \to B \to B/xB \oplus B/yB \to B/(x,y)B \to 0 \ene for $B$ (which is exact since the analogous diagram for $B$ is a pullback) and we have $\Tor_1^B(C,B/(x,y)B)=0$ by graded flatness (Corollary~\ref{cor:nodalflatness}). 

Now that we know \eqref{cartesiancocartesian}, we can prove \eqref{pullbackrestricts} and \eqref{defobcompatibility} by applying Theorem~\ref{thm:descent} to the corresponding pushout diagram of closed embeddings of affine schemes (c.f.\ Proposition~\ref{prop:pushouts}\eqref{push1}).  We need to translate the graded flatness hypotheses in the present theorem into the $\Tor$-vanishing hypotheses of Theorem~\ref{thm:descent}.  For a $C$-module $M$, Corollary~\ref{cor:nodalflatness} says that graded flatness of $M$ over $B$ is equivalent to \bne{star} \Tor_1^B(M,B/(x,y)) & = & 0. \ene Since $C$ is graded flat over $B$ by hypothesis, we have the vanishing in \eqref{star} when $M=C$, as we mentioned above.  This fact and the computation \be M \otimes_B B/(x,y) & = & M \otimes_C (C \otimes_B B/(x,y)) \\ & = & M \otimes_C (C/(x,y)C) \\ & = & M \otimes_C C_0 \ee show that the vanishing \eqref{star} is equivalent to $\Tor_1^C(M,C_0)=0$.  Since we have this vanishing when $M=C$, Theorem~\ref{thm:descent}\eqref{desc1} gives the vanishings \bne{newstar} \Tor_1^{C_i}(C_i,C_0) & = & 0 . \ene  By Corollary~\ref{cor:gradedflatnessoverA1}, graded flatness of a $C_1$-module $M_1$ over $k[x]$ is equivalent to \bne{star2} \Tor_1^{k[x]}(M_1,k[x]/x) & = & 0. \ene  From the vanishing \eqref{newstar} and the computation \be M_1 \otimes_{k[x]} k[x]/x & = & M_1 \otimes_{C_1} (C_1 \otimes_{k[x]} k[x]/x) \\ & = & M_1 \otimes_{C_1} C_0 \ee we see that the vanishing \eqref{star2} is equivalent to \bne{star3} \Tor_1^{C_1}(M_1,C_0) & = & 0. \ene  Similarly, graded flatness of a $C_2$-module $M_2$ over $k[y]$ is equivalent to \bne{star4} \Tor_1^{C_2}(M_2,C_0) & = & 0. \ene  The results now follow Theorem~\ref{thm:descent} because we have shown that the category $\Mod(C)^{\circ}$ defined here is the same as the category that would have been denoted $\Qco(\Spec C)^{\circ}$ in Theorem~\ref{thm:descent} (or, rather, in Remark~\ref{rem:descent}) and that the category $\Mod(C_i)^{\circ}$ here is the same as the category $\Qco(\Spec C_i)^{\circ}$ from Theorem~\ref{thm:descent}. \end{proof}

The rest of this section is devoted to the globalization of Theorem~\ref{thm:gluing}.  

\noindent {\bf Setup:} Suppose $Y$ is a locally noetherian scheme and $X_0 \into X_i$ ($i=1,2$) are closed embeddings of flat, locally finite type $Y$-schemes.  Then $X := X_1 \coprod_{X_0} X_2$ is also flat and of locally finite type over $Y$ (Theorem~\ref{thm:pushouts}); in particular $X$ is noetherian.  For a geometric point $y$ of $Y$, we will use notation such as $X(y)$ to denote the fiber of $X \to Y$ over $y$ and notation such as $M(y)$ to denote the pullback of $M \in \Mod(X)$ to $X(y)$.  We will sometimes write $X_i$ when we mean $\O_{X_i}$ and $X(y)$ when we mean $\O_{X(y)}$.  Let $\E \subseteq \Coh(X)$ denote the full subcategory consisting of coherent sheaves $M$ on $X$ satisfying the following conditions: \begin{enumerate} \item $M$ is flat over $Y$. \item For each geometric point $y$ of $Y$, \be \Tor_1^{X(y)}(M(y),X_0(y)) & = & 0. \ee \end{enumerate} For $i=1,2$, let $\E_i \subseteq \Coh(X_i)$ denote the full subcategory consisting of coherent sheaves $M_i$ on $X_i$ satisfying the following conditions: \begin{enumerate} \item $M_i$ is flat over $Y$. \item For each geometric point $y$ of $Y$, \be \Tor_1^{X_i(y)}(M_i(y),X_0(y)) & = & 0. \ee \end{enumerate}  Let $\E_0$ denote the full subcategory of $\Coh(X_0)$ consisting of coherent sheaves on $X_0$ flat over $Y$. 

\begin{thm} \label{thm:gluing1} In the above Setup: \begin{enumerate} \item \label{gluingA} If $M \in \E$, then $M|X_i \in \E_i$ for $i=1,2$ and \be \Tor_1^X(M,X_0) & = & 0. \ee \item \label{gluingB} If $M_i \in \E_i$, then $M_i|X_0 \in \E_0$ and \be \Tor_1^{X_i}(M_i,X_0) & = & 0. \ee \item \label{gluingC} The full subcategories $\E$ and $\E_i$ are exact. \item \label{gluingD} The usual pullback and descent functors yield an equivalence of categories \be \E & = & \E_1 \times_{\E_0} \E_2 \ee identifying $\Hom$ and $\Ext^1$ groups as in Theorem~\ref{thm:descent}\eqref{desc4}. \end{enumerate} \end{thm}

\begin{proof} For \eqref{gluingA}, consider the natural map \bne{pottest} I_{X_i/X} \otimes M & \to & M \ene for $i=0,1,2$, whose injectivity when $i=0$ is equivalent to the desired $\Tor$-vanishing.  Since $X$ and $X_i$ are flat over $Y$, so is $I_{X_i/X}$ and $I_{X_i/X}(y) = I_{X_i(y)/X(y)}$ for each geometric point $y$ of $Y$.  The restriction of \eqref{pottest} to $X(y)$ is therefore the natural map \bne{pottest2} I_{X_i(y)/X(y)} \otimes M(y) & \to & M(y), \ene which is injective when $i=0$ by the assumption \bne{assumptionM} \Tor_1^{X(y)}(M(y),X_0(y)) & = & 0 \ene on $M \in \E$.  Since the diagram $$ \xym{ X_0(y) \ar[r] \ar[d] & X_1(y) \ar[d] \\ X_2(y) \ar[r] & X(y) } $$ is also a pushout diagram of closed embeddings (Theorem~\ref{thm:pushouts}\eqref{pushouts2}), Theorem~\ref{thm:descent}\eqref{desc1} (applied to this fiber diagram using \eqref{assumptionM}) says we also have the vanishings \bne{newvanishings} \Tor_1^{X(y)}(M(y),X_i(y)) & = & 0 \\ \label{2ndvan} \Tor_1^{X_i(y)}(M(y)|X_i(y),X_0(y))  & = & 0. \ene  The vanishings \eqref{newvanishings} says that the maps \eqref{pottest2} are also injective when $i=1,2$.  Since $M \in \E$ is also flat over $Y$ and the map $X \to Y$ and the sheaf $M$ are sufficiently nice, \cite[IV.11.3.7]{EGA} says that the established injectivity of the maps \eqref{pottest2} implies injectivity of all the maps \eqref{pottest} and the flatness of all their cokernels over $Y$.  So the $M|X_i$ are all flat over $Y$ and since \be (M|X_i)(y) & = & M(y)|X_i(y) \ee the vanishing \eqref{2ndvan} shows $M|X_i \in \E_i$ for $i=1,2$.

One proves \eqref{gluingB} by an almost identical consideration of the natural map $$I_{X_0/X_i} \otimes M_i \to M_i,$$ again using \cite[IV.11.3.7]{EGA} and the hypothesized flatness and fiberwise $\Tor$-vanishing on $M_i \in \E_i$.

For \eqref{gluingC}, one first notes that a kernel or extension of objects in $\E_i$ is flat over $Y$ because objects in $\E_i$ are flat over $Y$.  Then one uses the fact that the appropriate kernel or extension stays exact after restricting to the fiber over $y$ to establish the necessary fiberwise $\Tor$-vanishing.

In light of the first three parts, \eqref{gluingD} becomes a standard variant of Theorem~\ref{thm:descent} as discussed in Remark~\ref{rem:descent}. \end{proof}

\begin{rem} \label{rem:gluing1} There are many possible variations of Theorem~\ref{thm:gluing}.  For example, suppose one has some chosen closed subscheme $Z_i(y) \subseteq X_i(y)$ disjoint from $X_0(y)$ for each geometric point $y$.  Then one could add fiberwise $\Tor$-vanishing with $Z_1(y) \coprod Z_2(y)$ into the definition of $\E$ and fiberwise $\Tor$-vanishing with $Z_i(y)$ into the definition of $\E_i$ without changing the conclusion of Theorem~\ref{thm:gluing}. \end{rem}

On completely formal grounds, we can translate Theorem~\ref{thm:gluing} into the language of log flatness and use it to prove results about log quotient spaces.

\noindent {\bf Setup:} We go back to the world of log schemes.  Suppose $f_i : X_i \to Y$ ($i=1,2$) are nodal degenerations (\S\ref{section:semistabledegenerations}) with the same base and the same relative boundary $X_0 \to Y$.  Assume that the map of schemes \be \u{f} = \u{f}_1 \coprod \u{f}_2 : \u{X} := \u{X}_1 \coprod_{\u{X}_0} \u{X}_2 & \to & \u{Y} \ee lifts to a nodal degeneration of log schemes $f : X \to Y'$ (note that we allow the log structure on $\u{Y}$ to change).  Assume furthermore that there exists a cartesian diagram of log schemes  \bne{cartdiagram} \xym{ X \setminus \u{X}_0 \ar[r] \ar[d]_f & (X_1 \setminus \u{X}_0) \coprod (X_2 \setminus \u{X}_0) \ar[d]^{f_1 \coprod f_2} \\ Y' \ar[r] & Y } \ene where the horizontal arrows are the obvious isomorphisms on the level of underlying schemes. Assume $\u{Y}$ is locally noetherian, so that $X$ and the $X_i$ are also locally noetherian (the map of schemes underlying a nodal degeneration is locally of finite presentation as a matter of definitions) as in the setup of Theorem~\ref{thm:gluing1}.

Let $\E \subseteq \Coh(X)$ denote the full subcategory of coherent sheaves on $X$ log flat over $Y'$.  For $i=0,1,2$, let $\E_i \subseteq \Coh(X_i)$ denote the full subcategory of coherent sheaves on $X_i$ log flat over $Y$.  Note that we do not have morphisms of log schemes $X_0 \into X_i$ or $X_i \into X$, but we do have such closed embeddings on the level of underlying schemes, so it makes sense to speak of the restriction $M|X_i$ for a coherent sheaf $M$ on $X$, say.

\begin{thm} \label{thm:gluing2} In the above Setup: \begin{enumerate} \item \label{gluing2A} If $M \in \E$, then $M|X_i \in \E_i$ for $i=1,2$ and \be \Tor_1^X(M,X_0) & = & 0. \ee \item \label{gluing2B} If $M_i \in \E_i$, then $M_i|X_0 \in \E_0$ and \be \Tor_1^{X_i}(M_i,X_0) & = & 0. \ee \item \label{gluing2C} The full subcategories $\E$ and $\E_i$ are exact. \item \label{gluing2D} The usual pullback and descent functors yield an equivalence of categories \be \E & = & \E_1 \times_{\E_0} \E_2 \ee identifying $\Hom$ and $\Ext^1$ groups as in Theorem~\ref{thm:descent}\eqref{desc4}. \end{enumerate} \end{thm}

\begin{proof}  By Corollary~\ref{cor:nodaldegeneration}, a coherent sheaf $M$ on $X$ log flat over $Y'$ (i.e.\ is in $\E$) iff $M$ is flat over $Y$ in the usual sense and \be \Tor_1^{X(y)}(M(y),Z(y)) & = & 0, \ee where $Z(y) \subseteq X(y)$ is the non-strict locus of $f(y') : X(y') \to \{ y' \}$.  (We will write $y'$ to emphasize that the geometric point $\u{y}$ of $\u{Y}=\u{Y}'$ has log structure from $Y'$ as opposed to $Y$, but we won't worry too much about this distinction in any statement that has nothing to do with log structures.)  Since the $f_i$ are nodal degenerations, the local picture of $\u{f}_i$ near $\u{X}_0$ in Proposition~\ref{prop:nodaldegenerations} makes it clear that $\u{X}_0(y)$ is contained in the singular locus of $\u{f}(y) = \u{f}(y')$ and hence it must be contained in the non-strict locus of the log smoth map $f(y')$.  That same local picture makes it clear that any point of $\u{X_0}$ has neighborhoods in the $\u{X}_i$ on which the non-strict locus of $f_i$ is \emph{exactly} $\u{X}_0$.  Since the diagram \eqref{cartdiagram} is cartesian, the non-strict locus of $f$ must be the same as the non-strict locus $f_1 \coprod f_2$ away from $\u{X}_0$.  Hence there are closed subschemes $Z_i(y) \subseteq X_i(y)$, disjoint from $X_0(y)$, such that the non-strict locus of $f(y')$ is $$ X_0(y) \coprod Z_1(y) \coprod Z_2(y)$$ and the non-strict locus of $f_i(y)$ is $$X_0(y) \coprod Z_i(y).$$  So Corollary~\ref{cor:nodaldegeneration} also says that a coherent sheaf $M_i$ on $X_i$ is log flat over $Y$ (i.e.\ is in $\E_i$) iff $M_i$ is flat over $Y$ in the usual sense and \be \Tor_1^{X_i(y)}(M_i(y), X_0(y) \coprod Z_i(y)) & = & 0. \ee  It is now clear that the present theorem is just one of the variations of Theorem~\ref{thm:gluing1} discussed in Remark~\ref{rem:gluing1}.   \end{proof}

\begin{cor} In the setup of the theorem, suppose $E \in \E$ is a fixed coherent sheaf on $X$ log flat over $Y'$.  Set $E_i := E|X_i$.  Then the usual pullback and descent functors define a bijection between quotients of $E$ log flat over $Y'$ and the set of pairs consisting of quotients of the $E_i$ flat over $Y$ which determine the same quotient of $E_0$.  This bijection is compatible with pullback along any $\u{W} \to \u{Y}$ and hence yields an isomorphism of log quotient spaces \be \LQuot(E/X/Y') & = & \LQuot(E_1/X_1/Y) \times_{\Quot(E_0/X_0/Y)} \LQuot(E_2/X_2/Y). \ee  This isomorphism is compatible with the natural obstruction theories on these quotient spaces.  \end{cor}

\begin{proof} The bijection is immediate from the equivalence of categories in the theorem.  The fact that is it compatible with base change results from the fact that the pullback functor $P$ is clearly compatible with base change and the descent functor $D$, which is defined by forming the kernel $$0 \to D(M_1,M_2) \to M_1 \oplus M_2 \to M_0 \to 0, $$ will also commute with base change when $M_0$ is flat over the base, as it is in our situation.  The usual deformation/obstruction theory for quotients is given by $\Hom$ and $\Ext$ from the kernel (which is log flat over $Y'$ when the quotient is log flat over $Y'$ because we assume $E$ is log flat over $Y'$) to the quotient, so the compatibility statement here is the identification of these groups in Theorem~\ref{thm:gluing2}\eqref{gluingD}. \end{proof}

\noindent {\bf Discussion.}  Of course the question now arises:  when is it possible to actually find a lifting $f$ of $\u{f}$ as in the setup of Theorem~\ref{thm:gluing2}?  Recall the setup:  $f_i : X_i \to Y$ ($i=1,2$) are nodal degenerations (\S\ref{section:semistabledegenerations}) with the same base and the same relative boundary $X_0 \to Y$.  Here we will be interested in lifting \be \u{f} = \u{f}_1 \coprod \u{f}_2 : \u{X} := \u{X}_1 \coprod_{\u{X}_0} \u{X}_2 & \to & \u{Y} \ee \emph{in a very particular way}, to a map of fine log schemes $f : X \to Y'$ so that $f$ is a nodal degeneration.  Basically, we will want to know that the log structures on $X$ and $Y'$ are related to those on the $X_i$ and $Y$ in a reasonable manner.  That is, we want our $f$ to have the following properties:

\noindent {\bf (A1)}  There is a map $\M_Y \to \M_{Y'}$ of log structures on $\u{Y}$.

\noindent {\bf (A2)} There is a log structure $\N$ on $\u{Y}$ locally isomorphic to the one associated to $0 : \NN \to \O_Y$ and a map of log structures $\N \to \M_{Y'}$.

\noindent {\bf (A3)} The maps in {\bf (A1)} and {\bf (A2)} induce a direct sum decomposition \be \M_{Y'} & = & \M_Y \oplus \N . \ee

\noindent {\bf (A4)}  For $i=1,2$, let $\M_{X_i'}$ denote the log structure on $\u{X}_i$ inherited from $X$.  Then there are maps $\M_{X_i} \to \M_{X_i}'$ of log structures on $\u{X}$ under $\u{f}_i^* \M_Y$. (Here $\M_{X_i}$ is of course a log structure under $\u{f}_i^* \M_Y$ via $f_i^\dagger$ and $\M_{X_i'}$ is viewed as a log structure under $\u{f}_i^* : \M_Y$ via the map in {\bf (A1)} and the restriction of $f^\dagger : f^* \M_{Y'} \to \M_X$ to $\u{X}_i$.)

\noindent {\bf (A5)} Away from $\u{X}_0$, $f : X \to Y'$ coincides with the base change of the disjoint union of the maps $f_i : X_i \setminus \u{X}_0  \to Y$ along the map $Y' \to Y$ which is the identity on underlying schemes and the map in {\bf (A1)} on log structures.

\noindent {\bf (A6)} Near any (\'etale) point $x$ of $\u{X}_0 \subseteq \u{X}$, any given charts $a : Q \to \M_Y(Y)$, $m : \NN \to \N(Y)$ for $\N$ and $\M_Y$ near $f(x)$ can be lifted (after possibly shrinking) to charts \bne{ficharts} & \xym@C+20pt{ Q \oplus \NN \ar[r]^-{(f^\dagger_i a,t_i)} & \M_{X_i}(X_i) \\ Q \ar[u]^{(\Id,0)} \ar[r]^-a & \M_Y(Y) \ar[u]_{f_i^\dagger} } \ene for the $f_i : X_i \to Y$ and a chart \bne{keychart} & \xym@C+30pt{ Q \oplus \NN^2 \ar[r]^-{(f^\dagger a,s_1,s_2)} &  \M_X(X) \\ Q \oplus \NN \ar[u]^-{(\Id,\Delta)} \ar[r]^-{(a,m)} & \M_{Y'}(Y) = (\M_Y \oplus \N)(Y) \ar[u]_{f^\dagger} } \ene for $f$ such that the induced maps \bne{theinducedmaps} (\u{f}_i, \alpha_{X_i}(t_i)) : \u{X}_i & \to & \u{Y} \times \u{\AA}^1 \\ \label{smoothq} (\u{f}, \alpha_X(s_1), \alpha_X(s_2)) : \u{X} & \to & \u{Y} \times_{\u{\AA}(Q \oplus \NN)} \u{\AA}(Q \oplus \NN^2) = \u{Y} \times_{\u{\AA}(\NN)} \u{\AA}(\NN^2) \ene are smooth and such that, for $i=1,2$, the restrition map $\M_X(X) \to \M_{X_i'}(X_i) = \M_X|_{\u{X}_i}(X_i)$ takes $s_i \in \M_X(X)$ to the image of $t_i \in \M_{X_i}(X_i)$ under the map $\M_{X_i}(X_i) \to \M_{X_i'}(X_i)$ from {\bf (A4)}. 

\begin{rem} A log structure $\N$ as in {\bf (A2)} is often called a \emph{log point}.  The characteristic monoid of such a log structure is $\ov{\N} = \u{\NN}$.  The data of such a log structure is the same as the data of a line bundle on $\u{Y}$ (one gets an $\O_Y^*$-torsor from such a log structure $\N$ by looking at the sheaf of liftings of $1 \in \ov{\NN} = \ov{\N}$ to $\N$). \end{rem}

\begin{rem} On the level of sheaves of monoids, the direct sum of log structures $\M \oplus \N$ on a scheme $X$ is given by $\M \oplus_{\O_X^*} \N$.  Though the category of log structures does have products, they are badly behaved and do not coincide with sums.  In particular, there is no projection map $\M \oplus \N \to \N$, say.  Formation of direct limits of log structures commutes with formation of characteristic monoids. \end{rem}

\begin{rem} \label{rem:sidescription} Assumption {\bf (A6)} implies that $\alpha_X(s_2)|X_1 \in \O_{X_1}(X_1)$ is zero (and similarly with the roles of $1,2$ reversed).  To see this, first note that commutativity of \eqref{keychart} implies $f^\dagger(m) = s_1+s_2 \in \M_X(X)$.  But $m \in \N(Y)$ maps to zero in $\O_Y(Y)$ by the assumption on $\N$ in {\bf (A2)} and the fact that $m : \NN \to \N(Y)$ is a chart, so, in particular, the image of $f^\dagger(m)$ in $\O_{X_1}(X_1)$ must also be zero, hence $\alpha_X(s_2)|X_1$ must be in the kernel of multiplication by $\alpha_X(s_1)|X_1$.  But the last assumption in {\bf (A6)} implies that $\alpha_X(s_1)|X_1 = \alpha_{X_1}(t_1)$, and $\alpha_{X_1}(t_1) : \O_{X_1} \to \O_{X_1}$ must be injective (i.e.\ $\alpha_{X_1}(t_1)$) must be regular) by smoothness of \eqref{theinducedmaps}.  In terms of the pullback description $\O_X = \O_{X_1} \times_{\O_{X_0}} \O_{X_2}$, this means that we have \be \alpha_X(s_1) & = & (\alpha_{X_1}(t_1),0) \\ \alpha_X(s_2) & = & (0,\alpha_{X_2}(t_2)).  \ee  Since $\alpha_{X_1}(t_1)$ generates the ideal of $\u{X}_0 \subseteq \u{X}_1$ (c.f.\ Remark~\ref{rem:relativeboundary}), the function $\alpha_X(s_1) \in \O_X$ generates the ideal of $\u{X}_2 \subseteq \u{X}$ (c.f.\ Theorem~\ref{thm:pushouts}\eqref{pushouts6}).  \end{rem}

\begin{rem} The map $\u{Y} \to \u{\AA}(\NN) = \AA^1$ used to define the fibered product in \eqref{smoothq} is the one determined by the image of $m \in \N(Y)$ in $\O_Y(Y)$, which is zero (see the above Remark).  The fiber of the other map $\u{\AA}(\Delta)$ used to define the fibered product in \eqref{smoothq} over the origin is the pushout $\AA^1 \coprod_{0} \AA^1$, so in fact we have \be \u{Y} \times_{\u{\AA}(\NN)} \u{\AA}(\NN^2) & = & \u{Y} \times (\AA^1 \coprod_0 \AA^1). \ee The assumptions at the end of {\bf (A6)} and the discussion in the above Remark then imply that the map \eqref{smoothq} is just the one determined by the maps \eqref{theinducedmaps}, the two obvious embeddings \bne{obviousembedding} \u{Y} \times \AA^1 & \into & \u{Y} \times (\AA^1 \coprod_0 \AA^1), \ene  and the universal property of the pushout $\u{X} = \u{X}_1 \coprod_{\u{X}_0} \u{X}_2$.   In other words, the composition of $\u{X}_i \into \u{X}$ and \eqref{smoothq} coincides with the composition of \eqref{theinducedmaps} and the appropriate choice of ``obvious embedding" \eqref{obviousembedding}. \end{rem}

\begin{rem} The map in {\bf (A4)} and the map $\u{f}_i^* \N \to \M_{X_i'}$ obtained by composing $f^*$ of the map in {\bf (A2)} and the restriction of $f^\dagger$ to $\u{X}_i$ do \emph{not} yield a direct sum decomposition \be \M_{X_i'} & = & \M_{X_i} \oplus \u{f}_i^* \N . \ee  If these maps did yield such a direct sum decomposition, then the relative characteristic of $f$ at a point of $\u{X}_0$ would be $\NN$, but the chart \eqref{keychart} shows that this relative characteristic monoid is actually $\ZZ$, as we know must be the case at a relative singular point of a nodal degeneration. \end{rem}

The following discussion may clarify the Setup.  Since a smooth map is \'etale locally the projection from a product with affine space, Proposition~\ref{prop:nodaldegenerations}\eqref{boundarypoint} says that near a point of $\u{X}_0$, the maps $f_1$ and $f_2$ look like $\Spec( \slot \to \slot)$ of diagrams of log rings as below  \be \xym{ Q \oplus \NN \ar[r]^-{(a,z_1)} & A[z_1,z_2,\dots,z_n] \\ Q \ar[u]^{(\Id,0)} \ar[r]^a & A \ar[u] } & \quad & \xym{ Q \oplus \NN \ar[r]^-{(a,w_1)} & A[w_1,z_2,\dots,z_n] \\ Q \ar[u]^{(\Id,0)} \ar[r]^a & A \ar[u] } \ee with $\u{D} = \Spec A[z_2,\dots,z_n]$ cut out in $\u{X}_1$ by $z_1$ and in $\u{X}_2$ by $w_1$.  The \'etale local picture of $\u{f}$ near such a point is hence $\Spec$ of the ring map \be A & \to & A[z_1,\dots,z_n] \times_{A[z_2,\dots,z_n]} A[w_1,z_2,\dots,z_n] = A[w_1,z_1,\dots,z_n]/(z_1w_1).\ee  We want our lifting $f$ of $\u{f}$ to a nodal degeneration to look, in this local picture, like $\Spec (\slot \to \slot)$ of the diagram of log rings: \bne{localpicture} & \xym@C+20pt{ Q \oplus \NN^2 \ar[r]^-{(a,z_1,w_1)} & A[w_1,z_1,\dots,z_n]/(z_1w_1) \\ Q \oplus \NN \ar[u]^{(\Id,\Delta)} \ar[r]^-{(a,0)} & A \ar[u] } \ene  

It will not always be possible to find a global lifting $f$ with all of our desired properties, as we managed to do in our local discussion.  Related issues are addressed in \cite[3.12-14]{O2}.  We will not attempt to address this issue here.  Roughly speaking, the obstruction to finding such a global $f$ with all the properties above should be the class of $N^{\lor}_{X_0/X_1} \otimes N^{\lor}_{X_0/X_2}$ in the relative Picard group $\Pic(X_0/Y)$; the choice of line bundle $L \in \Pic(Y)$ and the choice of an isomorphism \be f_0^* L & \cong & N^{\lor}_{X_0/X_1} \otimes N^{\lor}_{X_0/X_2} \ee should be closely related to the choice of a lifting $f$ of $\u{f}$ with the above properties (the log structure $\N$ should be the one corresponding to $L$, for example).  I have not thought through the details.

\begin{rem} Instead of trying to understand when we can glue two nodal degenerations along a common boundary, we could try to understand when we can ``take apart" a nodal degeneration with a ``universal node." \end{rem}

\begin{rem} In the above discussion we have, for simplicity, discussed the situation where one attempts to glue the entire boundary of one nodal degeneration to the entire boundary of another; of course one may work only with one component of the boundary at a time, or one may glue two pieces of the boundary in the same nodal degeneration, etc.\  It will be clear in what follows that all the action happens locally near the gluing locus; it isn't even necessary to assume anything about our $f_i$ other than that they look like nodal degenerations with boundary $\u{X}_0$ near a chosen $\u{X}_0 \subseteq X_i$. \end{rem}

\section{Modules Over Monoids} \label{section:modules}  This section is a brief review of the theory of modules over monoids, with a special emphasis on flat and free modules, which play a role in the theory of log flatness, especially in \S\ref{section:gradedflatnesscriteria}.  This overlaps considerably with the material in \cite{GM}, but it is included here as well to keep things self-contained.  The theory of flatness is new here.

\begin{defn} \label{defn:module} Let $P$ be a monoid.  Recall the a $P$-\emph{module} is a set $M$ equipped with an \emph{action map} $P \times M \to M$, written $(p,m) \mapsto p+m$, such that $0+m = m$ for all $m \in M$ and $(p_1+p_2)+m = p_1+(p_2+m)$ for all $p_1,p_2 \in P$, $m \in M$.\end{defn}

Alternatively, for a set $M$, the set $\Hom_{\Sets}(M,M)$ is a (non-commutative) monoid under composition and a $P$-module structure on $M$ is a homomorphism of monoids $P \to \Hom_{\Sets}(M,M)$.  If $P$ is a group, then a $P$-module structure on $M$ is an \emph{action} of $P$ on $M$ in the usual sense.  

If $h : Q \to P$ is a monoid homomorphism, $P$ becomes a $Q$-module with the action $q+p := h(q)+p$.  A subset $I \subseteq P$ such that addition takes $I \times P$ into $I$ is called an \emph{ideal}. An ideal is the same thing as a $P$-submodule of $P$.  $P$-modules form a category $\Mod(P)$ where a morphism is a morphism of sets respecting the actions.   The category $\Mod(P)$ has all direct and inverse limits and formation of inverse limits commutes with passing to the underlying set.  Finite products and finite sums \emph{do not coincide} in $\Mod(P)$.  The categorical direct sum (coproduct) of $P$-modules $M_i$ is their set-theoretic disjoint union $\coprod_i M_i$ with the obvious $P$-action respecting this coproduct decomposition.  In particular, the initial object of $\Mod(P)$ is the empty $P$-module $\emptyset$ and the terminal object of $\Mod(P)$ is the one-element set with (necessarily) trivial $P$-action.

If $i \mapsto M_i$ is a \emph{filtered} direct limit system of $P$-modules, its direct limit $M$ is constructed ``in the usual way" by putting a monoid structure on the set-theoretic filtered direct limit.  Despite the example of coproducts and filtered direct limits, the forgetful functor $\Mon \to \Sets$ does \emph{not} commute with general direct limits---it does not commute with coequalizers.

\begin{rem} \label{rem:nozeroobject} The category $\Mod(P)$ does not have a zero object so there are no ``zero morphisms" and it does not make sense to speak of the ``kernel" or ``cokernel" of a $\Mod(P)$ morphism $M \to N$. \end{rem}

\subsection{Flat and free modules} \label{section:freemodules} The forgetful functor from $\Mod(P)$ to sets has a left adjoint called the \emph{free module functor} which takes a set $S$ to the $P$-module $P \times S$ with action $p+(p',s) = (p+p',s)$.  A $P$-module in the essential image of the free module functor is called \emph{free} and a choice of subset $S \subseteq M$ such that the map $P \times S \to M$ given by $(p,s) \mapsto p+s$ is a $P$-module isomorphism (i.e.\ is bijective) is called a \emph{basis for} $M$.  Note that $M$ is free iff $M$ is isomorphic to a categorical direct sum of copies of $P$.

\begin{example} \label{example:groupflatness} If $P=A$ is an abelian group, then an $A$-module (set with $A$-action) $M$ is free iff the $A$-action is free (trivial stabilizers) and a choice of basis is the same thing as the choice of a point in each $A$-orbit.  In general, picking a point in each orbit shows that $M$ is a direct sum of $A$=modules of the form $A/B$, where $B$ is a subgroup of $A$.   Any \end{example}

\begin{defn} \label{defn:flat} A $P$-module $M$ is called \emph{flat} iff $M$ is a filtered direct limit of free $P$-modules. \end{defn}

Here is an indication that this notion of flatness has similar formal properties to the usual flatness notion for modules over rings:

\begin{prop} \label{prop:flatness} Let $P$ be a monoid. \begin{enumerate} \item Any free $P$-module if flat.  \item Any filtered direct limit of flat $P$-modules is flat. \item \label{projectiveimpliesfree} Any direct summand of a free $P$-module is free.  \item Any direct summand of a flat $P$-module is flat. \end{enumerate} \end{prop}

\begin{proof} The first two statements are immediate from the definition of ``flat" above.  For the third statement, suppose $M = M_1 \coprod M_2$ as $P$-modules and $S \subseteq M$ is a basis for $M$.  Then each $s \in S$ is contained in exactly one of the $M_i$ and if $s \in M_i$, then $p+s \in M_i$ for all $p \in P$.  Since every $m \in M$ can be uniquely written $m=s+p$ for $s \in S$, $p \in P$, each $m \in M_i$ can be uniquely written $m = s+p$ with $p \in P$ and $s \in S_i := \{s \in S : s \in S_i \}$, so $M_i$ is free with basis $S_i$ for $i=1,2$.  For the final statement, suppose $M = M_1 \coprod M_2$ (as $P$-modules) and $M$ is flat---so $M$ is the direct limit of a filtered system $j \mapsto M_j$ of free $P$-modules.  For $i \in \{ 1,2 \}$, let $M_{j,i} \subseteq M_j$ be the preimage of $M_i \subseteq M$ under the structure map $M_j \to M$ to the direct limit.  Since the latter structure maps are maps of $P$-modules we have $M_j = M_{j,1} \coprod M_{j,2}$ (as $P$-modules), naturally in $j$, and hence $M_i$ is the filtered direct limit of $j \mapsto M_{j,i}$.  Each $M_{j,i}$ free by the previous part since it is a direct summand of the free module $M_j$. \end{proof}

\begin{example} The $\NN$-module $\ZZ$ is not free, but it is flat because it is the filtered direct limit of the free $\NN$-submodules $$[0,\infty) \subseteq [-1,\infty) \subseteq [-2,\infty) \subseteq \cdots. $$   \end{example}

\begin{example} \label{example:localization} More generally, if $S \subseteq P$ is any submonoid of any monoid $P$, the localization $S^{-1}P$ is a flat $P$-module by the same argument one uses to show that the localization of a ring is a filtered direct limit of copies of the ring.  That is: regard $S$ as a category where the objects are elements of $S$ and where a morphism $u : s \to t$ is an element $u \in S$ such that $s+u = t$.  Composition is given by addition.  This category $S$ is filtered.  Consider the functor $F : S \to \Mod(P)$ defined on objects by $F(s) := P$ for all $s \in S$ and on morphisms by $F(u) := u + \slot : P \to P$.  If we define $f_s : F(s) \to S^{-1} P$ by $p \mapsto p-s$ then the $P$-module diagrams $$ \xym{ F(s) \ar[rd]_{f_s} \ar[rr]^{F(u)} & & F(t) \ar[ld]^{f_t} \\ & S^{-1} P } $$ commute and the induced map $\dirlim F \to S^{-1} P$ is an isomorphism of $P$-modules. \end{example}

There is a ``forgetful functor" from the category $\Mod(\ZZ[P])$ of modules over the ring $\ZZ[P]$ to $P$-modules which takes $M$ to $M$ regarded as a $P$-module with the action $p+m := [p] \cdot M$.  Here $[p]$ is the image of $p$ in $\ZZ[P]$ and $\cdot$ is the scalar multiplication of $\ZZ[P]$ on $M$.  This forgetful functor admits a left adjoint \bne{ZP} \ZZ[ \slot ] : \Mod(P) & \to & \Mod(\ZZ[P]) \\ \nonumber M & \mapsto & \ZZ[M], \ene where $\ZZ[M] = \oplus_{m \in M} \ZZ[m]$ with the action the action given by the unique $\ZZ$-linear extension of $[p] \cdot [m] := [p+m]$.

\begin{rem} Given a morphism of $P$-modules $M \to N$, one can form the kernel or cokernel of the $\ZZ[P]$-module morphism $\ZZ[M] \to \ZZ[N]$ (c.f.\ Remark~\ref{rem:nozeroobject}, but this $\ZZ[P]$-module will not, in general, be in the essential image of \eqref{ZP}.  \end{rem}

\begin{thm} \label{thm:freetofree} The functor \eqref{ZP} takes free (resp.\ flat) $P$-modules to free (resp.\ flat) $\ZZ[P]$-modules.  In fact, if $S$ is a basis for a $P$-module $M$, then $\{ [s] : s \in S \}$ is a basis for the $\ZZ[P]$-module $\ZZ[M]$.  \end{thm}

\begin{proof} If $M$ is a free $P$-module with basis $S$, then we use the adjointness relations mentioned above to establish a natural bijection \be \Hom_{\Mod(\ZZ[P])}(\ZZ[M],N) & = & \Hom_{\Mod(P)}(M,N) \\ & = & \Hom_{\Sets}(S,N). \ee  This shows that \eqref{ZP} takes free modules to free modules.  The functor \eqref{ZP} is a left adjoint so it commutes with direct limits.  In particular it commutes with filtered direct limits, so it also takes flat modules to flat modules because a filtered direct limit of free modules over a ring is flat. \end{proof}

\begin{cor} \label{cor:AGtoAHflat} Let $h : G \to H$ be an injective map of groups.  Then for any ring $A$, the map of group algebras $A[h]: A[G] \to A[H]$ makes $A[H]$ a free $A[G]$-module of rank $H/G$.  In particular, $A[h]$ is faithfully flat. \end{cor}

\begin{proof}  Since $h$ is injective, $H$ is a free $G$-module of rank $H/G$ (c.f.\ Example~\ref{example:groupflatness}) so the result is immediate from the theorem. \end{proof}

\subsection{Finiteness} \label{section:finiteness}

\begin{prop} \label{prop:finiteness} For a monoid $P$ and a $P$-module $M$, the following are equivalent: \begin{enumerate} \item \label{finitegen1} There exists a finite subset $S \subseteq M$ such that the map of $P$-modules $P \times S \to M$ defined by $(p,s) \mapsto p+s$ is surjective.  \item $\ZZ[M]$ is a finitely generated $\ZZ[P]$-module. \item \label{finitegen3} There exists a non-zero ring $A$ for which $A[M]$ is a finitely generated $A[P]$-module. \end{enumerate} \end{prop}

\begin{proof} The only conceivably nontrivial implication is \eqref{finitegen3} implies \eqref{finitegen1}.  Suppose $g_1,\dots,g_n$ generate $A[M]$ as an $A[P]$-module.  Write $g_i = \sum_{m \in M} a_m^i [m]$ where the set $S(i) := \{ m \in M : a_m^i \neq 0 \}$ is finite.  To see that $S := \cup_{i=1}^n S(i)$ satisfies the condition in \eqref{finitegen1}, fix some $m \in M$ and use the fact that the $g_i$ generate $A[M]$ to write $[m] = \sum_{i=1}^n b_i g_i$ for $b_i \in A[P]$.  The fact that the coefficient of $[m]$ is $1 \neq 0$ when we expand out the sum on the right in particular means that there must be some $s \in S$ and some $p \in P$ for which $p+s = m$. \end{proof}

\begin{defn} A $P$-module $M$ is called \emph{finitely generated} iff it satisfies the equivalent conditions of Proposition~\ref{prop:finiteness}. A subset $S \subseteq M$ as in \eqref{finitegen1} is called a \emph{set of generators} for $M$. \end{defn}

\begin{cor} \label{cor:finiteness} If $P$ is a finitely generated monoid, $M$ is a finitely generated $P$-module and $N$ is a $P$-submodule of $M$, then $N$ is finitely generated. \end{cor}

\begin{proof} Since $\ZZ[M]$ is a finitely generated module over the noetherian ring $\ZZ[P]$ and $\ZZ[N]$ is a $\ZZ[P]$-submodule of it, $\ZZ[N]$ is also a finitely generated $\ZZ[P]$-module, hence $N$ is a finitely generated $P$-module. \end{proof}

\subsection{Tensor product of modules} \label{section:tensorproduct}  Let $P$ be a monoid and let $M,N,T$ be $P$-modules.  A function $f : M \times N \to T$ is called $P$-\emph{bilinear} iff \bne{bilinear} f(p + m, n) & = & p + f(m,n) \\ \nonumber f(m, p + n) & = & p + f(m,n) \ene for every $m \in M$, $n \in N$, $p \in P$.  Let $\Bilin_P(M \times N,T)$ denote the set of $P$-bilinear maps from $M \times N$ to $T$.  

\begin{prop} \label{prop:tensorproduct} For any $M,N \in \Mod(P)$, there is a $P$ module $M \otimes_P N$, unique up to unique isomorphism, with the following universal property:  There is a $P$-bilinear map $\tau : M \times N \to M \otimes_P N$ such that any $P$-bilinear map $f : M \times N \to T$ factors uniquely as $\ov{f} \tau$ for a $P$ module map $\ov{f} : M \otimes_P N \to T$. \end{prop}

\begin{proof}  The uniqueness argument is standard.  For existence, define $M \otimes_P N$ to be the quotient of $M \times N$ by the smallest equivalence relation $\sim$ enjoying the following two properties: \begin{enumerate} \item \label{sim1} $(p + m,n) \sim (m, p + n)$ for every $p \in P$, $m \in M$, $n \in N$. \item \label{sim2} If $(m_1,n_1) \sim (m_2,n_2)$ for some $m_i \in M$, $n_i \in N$, then $(p + m_1,n_1) \sim (p + m_2,n_2)$ for every $p \in P$. \end{enumerate}  For $(m,n) \in M \times N$, let $m \otimes n$ denote the image of $(m,n)$ in $M \otimes_P N$.  Regard $M \otimes_P N$ as a $P$ module using the action $p + (m \otimes n) := (p + m) \otimes n$.  This is well-defined since $\sim$ satisfies \eqref{sim2} and clearly satisfies the requisite property \be (p_1+p_2) + (m \otimes n) & = & p_1 + (p_2 + (m \otimes n)) \ee for an action.  If we define $\tau : M \times N \to M \otimes_P N$ by $\tau(m,n) := m \otimes n$, then $\tau$ is $P$-bilinear because $\sim$ satisfies \eqref{sim1}.

Suppose $f : M \times N \to T$ is $P$-bilinear.  Define an equivalence relation $\cong$ on $M \times N$ by declaring $(m_1,n_1) \cong (m_2,n_2)$ iff $f(m_1,n_1) = f(m_2,n_2)$.  It is clear from bilinearity that $\cong$ satisfies \eqref{sim1} and \eqref{sim2}, so, since $\sim$ is the smallest equivalence relation satisfying these properties, we have $$(m_1,n_1) \sim (m_2,n_2) \implies f(m_1,n_1) = f(m_2,n_2) $$ and we can therefore define a function $\ov{f} : M \otimes_P N \to T$ by $\ov{f}(m \otimes n) := f(m,n)$.  It is clear that this $\ov{f}$ is a $P$-module map and that $f = \ov{f} \tau$.  The uniqueness of $\ov{f}$ is automatic because $\tau$ is surjective (this is one place where the tensor product of $P$-modules is a little easier than the tensor product of modules over a ring). \end{proof}

\begin{defn} \label{defn:tensorproduct} The module $M \otimes_P N$ of Proposition~\ref{prop:tensorproduct} is called the \emph{tensor product} of the $P$-modules $M$ and $N$. \end{defn}  

It is clear from the universal property of $M \otimes_P N$ that formation of $M \otimes_P N$ is functorial in both $M$ and $N$, so we can consider, for example, the functor \be \slot \otimes_P N : \Mod(P) & \to & \Mod(P) .\ee  Let $i \mapsto M_i$ be a direct limit system of $P$-modules and let $m_i \mapsto \ov{m}_i$ denote the natural map from $M_i$ to the direct limit.  It is easy to check that the natural map \bne{bilin2} \Bilin_P( (\dirlim M_i ) \times N, T) & = & \invlim \Bilin_P(M_i \times N, T) \\ \nonumber g & \mapsto & ((m_i,n) \mapsto g(\ov{m}_i,n)) \ene is an isomorphism with inverse $$(f_i) \mapsto ((\ov{m}_i,n) \mapsto f_i(m_i,n)),$$ and it follows formally from this using the universal property of tensor product that 

\begin{lem} \label{lem:tensorproductcommuteswithdirectlimits} $ \slot \otimes_P N$ preserves direct limits. \end{lem}

In particular, $M \otimes_P \slot$ preserves coproducts (direct sums), and it is clear that $M \otimes_P P = M$, so \bne{tensorwithfreemodule} & M \otimes_P (\coprod_S P) = \coprod_S (M \otimes_P P) = M \times S. \ene

One can also prove Lemma~\ref{lem:tensorproductcommuteswithdirectlimits} by checking that $ \slot \otimes_P N$ is left adjoint to \be \Hom_P(N, \slot) : \Mod(P) & \to & \Mod(P). \ee  That is, we have a bijection \bne{tensorHom} \Hom_P(M \otimes_P N, T) & = & \Hom_P(M, \Hom_P(N,T)) \ene natural in $M,N,T$.

If $T$ is a $\ZZ[P]$ module, then it is clear that \bne{bilin} \Bilin_P(M \times N, T) & = & \Bilin_{\ZZ[P]}(\ZZ[M] \times \ZZ[N],T) , \ene where, on the left, $T$ is regarded as a $P$ module via the forgetful functor and the right side is the set of bilinear maps of $\ZZ[P]$ modules in the usual sense.  Using the universal properties of tensor products (for modules over rings and monoids), formula \eqref{bilin}, and the adjointness property of \eqref{ZP}, we obtain a natural isomorphism of $\ZZ[P]$ modules \bne{tensorformula} \ZZ[ M \otimes_P N ] & = & \ZZ[M] \otimes_{\ZZ[P]} \ZZ[N] \ene by showing that both sides have the same maps to any $\ZZ[P]$ module $T$.

\begin{prop} \label{prop:flatness} If $N$ is a flat $P$-module then the functor $\slot \otimes_P N$ commutes with equalizers. \end{prop}

\begin{proof} We first show that $\slot \otimes_P N$ commutes with equalizers when $N$ is free.  If we choose a basis $S$ for $N$, then the functor $\slot \otimes_P N$ is identified with the functor $M \mapsto M \times S$.  This functor commutes with equalizers (notice that this is just a trivial statement about inverse limits of sets).  If follows formally that a filtered direct limit of free modules preserves equalizers because filtered direct limits of sets commute with finite inverse limits. \end{proof}

\begin{rem} \label{rem:flatness} Even when $N$ is free, the functor $\slot \otimes_P N$ need not commute with \emph{products}.  Indeed, if $S$ is a basis for $N$, then $\slot \otimes_P N$ is identified with $M \mapsto M \times S$ and this functor certainly does not commute with products: \be (M \times S) \times (M \times S) & \neq & (M \times M) \times S \ee in general!  The converse of Proposition~\ref{prop:flatness} (``Lazard's Theorem") is probably true, but I have not thought through the details. \end{rem}

\subsection{Base change} 

\begin{defn} If $h : Q \to P$ is a monoid homomorphism and $M$ is a $Q$-module, then $M \otimes_Q P$ has a natural $P$-module structure via the action $p+(m \otimes p') := m \otimes (p+p')$.  The $P$-module $M \otimes_Q P$ is called the \emph{base change} or \emph{extension of scalars} of $M$ along $h$. \end{defn}

\begin{example} \label{example:sharpeningmodule} Let $P$ be a monoid, $M$ a $P$-module, $\ov{P} := P/P^*$ the sharpening of $P$.  Then $\ov{M} := M \otimes_P \ov{P}$ is just the quotient of $M$ by the equivalence relation $\sim$ where $m_1 \sim m_2$ iff there if a $u \in P^*$ such that $u \cdot m_1 = m_2$.  (One can check directly that this has the right universal property.)  The image of $m \in M$ in $\ov{M}$ is often denoted $\ov{m}$. \end{example}

The extension of scalars functor is left adjoint to the restriction of scalars functor $\Mod(P) \to \Mod(Q)$.  That is, we have \bne{extensionrestrictionadjointness} \Hom_P(M \otimes_Q P, N) & = & \Hom_Q(M,N), \ene where, on the right, the $P$ module $N$ is regarded as a $Q$ module by restriction of scalars along $h : Q \to P$.  Using \eqref{extensionrestrictionadjointness} and Lemma~\ref{lem:tensorproductcommuteswithdirectlimits} we obtain:

\begin{lem} \label{lem:basechange} Base change takes free (resp. flat) $Q$-modules to free (resp.\ flat) $P$-modules. \end{lem}

Similarly, if $P_1,P_2$ are monoids under a monoid $Q$, then $P_1 \otimes_Q P_2$ has a natural monoid structure with addition defined by \be p_1 \otimes p_2 + p_1' \otimes p_2' & := & (p_1+p_1') \otimes (p_2+p_2'). \ee There are also monoid homomorphisms $P_i \to P_1 \otimes_Q P_2$ given by $p_1 \mapsto p_1 \otimes 0$ and $p_2 \mapsto 0 \otimes p_2$ when $i=1,2$ respectively.  It is easy to see that the diagram of monoids $$ \xym{ Q \ar[r] \ar[d] & P_1 \ar[d] \\ P_2 \ar[r] & P_1 \otimes_Q P_2 } $$ is a pushout diagram, so one might also denote the monoid $P_1 \otimes_Q P_2$ by $P_1 \oplus_Q P_2$.

\subsection{Flat modules over integral monoids}  The purpose of this section is to characterize flat modules over integral monoids.  Recall that a monoid $P$ is called \emph{integral} iff it satisfies any/all of the following equivalent conditions: \begin{enumerate} \item $P \to P^{\rm gp}$ is injective \item $P$ is (isomorphic to) a submonoid of a group \item for all $p,p',p'' \in P$, the equality $p+p'=p+p''$ implies $p'=p''$.\end{enumerate}  The results of this section are related to \cite[4.1]{KK}.

\begin{defn} \label{defn:torsionfreeandcomparable}  Let $P$ be an integral monoid.  A $P$-module $M$ is called \emph{torsion-free} iff, for all $p,p' \in P$, $m\in M$, the equality $p+m = p'+m$ in $M$ implies $p=p'$.  $M$ is called \emph{comparable} iff, for all $p_1,p_2 \in P$, $m_1,m_2 \in M$, the equality $p_1+m_1=p_2+m_2$ in $M$ implies the existence of $m \in M$ and $\rho_1,\rho_2 \in P$ such that $\rho_1+m=m_1$ and $\rho_2+m=m_2$. \end{defn}

\begin{rem} The above definitions would make sense even when $P$ is not integral, but they are probably useless in that setting. \end{rem}

\begin{example} \label{example:torsionfree} If $h : Q \to P$ is a map of integral monoids, then it is clear that $P$ is torsion-free as a $Q$-module iff $h$ is injective. \end{example}

\begin{lem} \label{lem:flatPmodules} Let $P$ be an integral monoid. \begin{enumerate} \item \label{freeimplies} Any free $P$-module is torsion-free and comparable.  \item \label{filtereddirlim} A filtered direct limit of torsion-free (resp.\ comparable) $P$-modules is torsion-free (resp.\ comparable). \end{enumerate} \end{lem}

\begin{proof} \eqref{freeimplies} Let $M$ be a free $P$-module with basis $S \subseteq M$.  Suppose $p+m=p'+m$ in $M$.  Since $S$ is a basis we can write $m = p''+s$ for some $s \in S$, $p'' \in P$ and we then have $p+p''+s=p'+p''+s$ in $M$, so since $S$ is a basis we have $p+p'' = p'+p''$ in $P$, hence $p=p'$ because $P$ is integral.  This shows that $M$ is torsion-free.  For ``$M$ comparable," suppose $p_1+m_1=p_2+m_2$.  Since $S$ is a basis, we can write $m_1 = \rho_1+s$, $m_2 = \rho_2+s'$ for some $\rho_1,\rho_2 \in P$, $s,s' \in S$.  The original equality then reads $p_1+\rho_1 + s = p_2+\rho_2+s'$, but $S$ is a basis so this implies $s=s'$ (and $p_1+\rho_1=p_2+\rho_2$, which is also clear from torsion-freeness). 

\eqref{filtereddirlim} is trivial.  The point is that, if one has an equality such as $p_1+m_1=p_2+m_2$ (involving only finitely many elements of $M$) in a filtered direct limit $M$ of $P$-modules $M_i$, then there is a large enough $i$ so that $m_1$ and $m_2$ are in the image of $M_i \to M$ and we already have the analogous equality in $M_i$. \end{proof}

\begin{thm} \label{thm:integralflatness} Let $P$ be an integral monoid.  A $P$-module is flat iff it is torsion-free and comparable. \end{thm}

\begin{proof}  The implication $\implies$ is clear from the definition of ``flat $P$-module" (Definition~\ref{defn:flat}) and Lemma~\ref{lem:flatPmodules}.  Now suppose $M$ is torsion-free and comparable and we want to show $M$ is flat.  Define an equivalence relation $\sim$ on (the set underlying) $M$ by saying $m_1 \sim m_2$ iff there are $p_1,p_2 \in P$ such that $p_1+m_1 = p_2 +m_2$.  (It is trivial to check that this is an equivalence relation; none of the hypotheses on $P$ are needed.)  Let $S \subseteq M$ be a subset containing exactly one element from each $\sim$-equivalence class.  Let $M_S \subseteq M$ be the $P$-submodule generated by $S$.  I claim $M_S$ is free.  It suffices to show that the map $P \times S \to M$ given by $(p,s) \mapsto p+s$ is injective, so suppose $p_1+s_1 = p_2 + s_2$.  Then $s_1=s_2$ because $s_1 \sim s_2$ and $S$ contains only one representative from each equivalence class, and then $p_1=p_2$ since $M$ is torsion-free.  To complete the proof it now suffices to show that the $M_S$ are a filtered collection of submodules of $M$ (it is clear that their union is $M$).  Suppose $S_1,S_2 \subseteq M$ each contain one element from each equivalence class.  Pick some equivalence class $[m]$ and let $s_1 = s_1([m]) \in S_1$, $s_2 = s_2([m]) \in S_2$ be the unique representatives of this equivalence class in $S_1$, $S_2$ (respectively).  Then we can write $p_1+s_1=p_2+s_2$ for some $p_1,p_2 \in P$.  Since $M$ is comparable, this means we can find some $t = t([m]) \in M$ and $\rho_1,\rho_2 \in P$ such that $\rho_1+t=s_1$ and $\rho_2+t=s_2$.  Let $T \subseteq M$ be the set of all the $t([m])$ constructed in this manner, so that $T$ manifestly contains one element from each $\sim$-equivalence class.  Since $\rho_1+t=s_1$ and $\rho_2+t=s_2$, the $P$-submodule of $M$ generated by $T$ contains $s_1$ and $s_2$ and this is true for each $\sim$-equivalence class, so $M_{S_1}, M_{S_2} \subseteq M_T$ as desired. \end{proof}

\begin{cor} \label{cor:flatnessoveragroup} If $G$ is a group, then the following conditions are equivalent for a $G$-module $M$ (set with $G$-action): \begin{enumerate} \item $M$ is a free $G$-module. \item The action of $G$ on $M$ is a free action. \item $M$ is a torsion-free $G$-module. \item $M$ is a flat $G$-module. \end{enumerate} \end{cor}

\begin{proof}  The only statement that is perhaps not obvious from the definitions or the discussion of Example~\ref{example:groupflatness} is the fact that a flat $G$-module is torsion-free, which is clear from the theorem. \end{proof}

\begin{cor} \label{cor:flatiffintegral} Let $h : Q \to P$ be a map of integral monoids.  The following are equivalent: \begin{enumerate} \item \label{hflat} $P$ is flat as a $Q$-module. \item $h$ is injective and $P$ is a comparable $Q$-module.  \item \label{Zhflat} $\ZZ[P]$ is flat over $\ZZ[Q]$. \end{enumerate} \end{cor}

\begin{proof} The equivalence of the first two statements is immediate from the theorem (c.f.\ Example~\ref{example:torsionfree}).  The equivalence of the last two statements is a foundational result of Kato---see \cite[4.1]{KK}. \end{proof}

Note that Theorem~\ref{thm:freetofree} shows that \eqref{hflat}$\implies$\eqref{Zhflat} in Corollary~\ref{cor:flatiffintegral} even without the assumption that $Q$ and $P$ are integral.  I don't know whether one has \eqref{Zhflat}$\implies$\eqref{hflat} in general.

\begin{lem} \label{lem:flatnessandsharpening} Let $P$ be an integral monoid.  A $P$-module $M$ is flat iff $M$ is flat as a $P^*$-module and $\ov{M}$ is a flat $\ov{P}$-module.  If $h : Q \to P$ is an injective map of integral monoids such that $\ov{h} : \ov{Q} \to \ov{P}$ is flat, then $h$ is flat. \end{lem}

\begin{proof} We will use the criterion for flatness in Theorem~\ref{thm:integralflatness}.  Note that $\ov{P}$ is also integral.  Suppose $\ov{M}$ is flat over $\ov{P}$ and $M$ is flat over $P^*$ and let us prove $M$ is flat over $P$.  For ``torsion-free," if $p_1+m = p_2+m$ in $M$, then $\ov{p}_1+\ov{m} = \ov{p}_2+\ov{m}$ in $\ov{M}$, so since $\ov{M}$ is torsion-free $p_2=u+p_1$ for some $u \in P^*$, but we must have $u=0$ because the original equation says $p_1+m = u+p_1+m$ and $M$ is flat over $P^*$.  For ``comparable," suppose $p_1+m_1=p_2+m_2$ in $M$.  Then $\ov{p}_1+\ov{m}_1 = \ov{p}_2+\ov{m}_2$ in $\ov{M}$ so since $\ov{M}$ is flat hence comparable, there are $\rho_1', \rho_2' \in P$ and $m \in M$ such that $\ov{\rho}_1' + \ov{m} = \ov{m}_1$ and $\ov{\rho}_2' + \ov{m} = \ov{m}_2$ in $\ov{M}$, hence there are $u_1,u_2 \in P^*$ such that $u_1+\rho_1'+m = m_1$ and $u_2+\rho_2'+m=m_2$ in $M$.  Taking $\rho_i := \rho_i'+u_i$, we are done.  The other implication is clear from stability of flatness under base change (Lemma~\ref{lem:basechange}) and the fact that if $M$ is torsion-free then clearly its restriction to $P^*$ is torsion-free, which is the same as flat because $P^*$ is a group (Corollary~\ref{cor:flatnessoveragroup}). 

For the second statement, the only issue is to prove that $P$ is comparable as a $Q$-module (c.f.\ Example~\ref{example:torsionfree}).  Say \bne{qandp} q_1+p_1 & = & q_2+p_2 \ene in $P$ for some $q_1,q_2 \in Q$, $p_1,p_2 \in P$.  We need to find $r_1,r_2 \in Q$ and $p \in P$ so that $p_1=p+r_1$ and $p_2 = p+r_2$.  By looking and the image of \eqref{qandp} in $\ov{P}$ and using flatness of $\ov{h}$ (more precisely, using comparability of $\ov{P}$ and an $\ov{Q}$-module), we can find $r_1,r_2' \in Q$, $p' \in P$, and $u_1,u_2 \in P^*$ such that \bne{ovrandp} p_1 & = & p' + u_1 + r_1 \\ \nonumber p_2 & = & p' + u_2 + r_2' \ene in $P$.  Set $p := p'+u_1$, $u := u_2-u_1$, so we can rewrite \eqref{ovrandp} as \bne{randp} p_1 & = & p+r_1 \\ p_2 & = & p + u + r_2'. \ene If we can show that $u \in P^*$ is in fact in $Q^*$, then we can finish by taking $r_2 := u + r_2'$.  To do this, we plug the formulas \eqref{randp} for $p_1$ and $p_2$ back into \eqref{qandp} and use integrality of $P$ to cancel a $p$ to find $q_1+r_1 = q_2+u+r_2$.  But this means that $\ov{q}_1+\ov{r}_1 \in \ov{Q}$ and $\ov{q}_2+\ov{r}_2 \in \ov{Q}$ have the same image in $\ov{P}$, so by injectivity of $\ov{h}$, we have to have $q_1+r_1 = q_2+v+r_2$ for some $v \in Q^*$, but since $h$ is injective (and $P$ is integral), we must have $u=v$.    \end{proof}

\begin{rem} The hypothesis that $M$ is flat over $P^*$ cannot be dropped in the above lemma.  Any group action which is transitive but not free is a counterexample. \end{rem}

\begin{rem} I think Ogus uses ``quasi-integral" to mean that a $P$-module is flat as a module over $P^*$. \end{rem}

\begin{lem} \label{lem:sharpeningflatness} Suppose $h : Q \to P$ is a map of integral monoids inducing an isomorphism $h^* : Q^* \to P^*$ on units.  Then $P$ is torsion-free (resp.\ comparable, flat) as a $Q$-module iff $\ov{P}$ is torsion-free (resp.\ comparable, flat) as an $\ov{Q}$-module.  \end{lem}

\begin{proof} This is straightforward from the definitions. \end{proof}

\subsection{Monoidal Quillen-Suslin} Recall that the Serre Conjecture (Quillen-Suslin Theorem) says that every finitely generated flat module over a polynomial ring is free.  Here we will prove a version of this result for modules over fine monoids.

For a $P$-module $M$ and $m,m' \in M$ declare $m < m'$ iff there is a $p \in P \setminus P^*$ such that $p+m = m'$.

\begin{lem} \label{lem:DCC} Let $P$ be a fine monoid, $M$ a finitely generated, torsion-free $P$-module.  Then there does not exist an infinite sequence $m_1,m_2,\dots$ of elements of $M$ with $ \cdots < m_2 < m_1.$ \end{lem}

\begin{proof} Suppose, toward a contradiction, that $m_1,m_2,\dots$ is such a sequence.  Then we have a sequence $$ P + m_1 \subset P+m_2 \subset \cdots$$ of $P$-submodules of $M$, which I claim is \emph{strictly} increasing.  This will yield a contradiction because the union of this sequence is clearly a $P$-submodule of $M$, hence is a finitely generated $P$-module by Corollary~\ref{cor:finiteness}.  To prove the claim it suffices to show that $m_{n+1} \notin P+m_n$.  Suppose $m_{n+1} = p + m_n$ for some $p \in P$.  We also have $m_n = p' + m_{n+1}$ for some $p' \in P \setminus P^*$ since $m_{n+1} < m_n$.  We then find that $0 + m_{n+1} = p + p' + m_{n+1}$, hence $p+p' = 0$ because $M$ is torsion-free.  But then $p' \in P^*$, a contradiction. \end{proof}

\begin{thm} \label{thm:monoidalQuillenSuslin} Let $P$ be a fine monoid, $M$ a finitely generated, flat $P$-module.  Then $M$ is a free $P$-module with a finite basis $S \subseteq M$. \end{thm}

\begin{proof}  By Theorem~\ref{thm:integralflatness} $M$ is torsion-free and comparable.  Suppose $S'' \subseteq M$ is any finite subset of $M$ generating $M$ as a $P$-module and $s \in S''$.  By Lemma~\ref{lem:DCC} we can find some $s' \in M$ with $s' \leq s$ such that $s'$ is $<$-\emph{minimal} (i.e.\ there is no $m \in M$ with $m < s'$).  Then $S' := \{ s' : s \in S'' \}$ is a finite subset of $M$ generating $M$ and consisting solely of $<$-minimal elements.  Choose a subset $S \subseteq M$ of minimal cardinality among all subsets $S' \subseteq M$ with these properties.  I claim $S$ is a $P$-basis for $M$.  Since $S$ generates $M$, the map of sets $P \times S \to M$ given by $(p,s) \mapsto p+s$ is surjective; the issue is injectivity.  Suppose $p_1 + s_1 = p_2 + s_2$ for some $p_1,p_2 \in P$, $s_1+s_2 \in S$.  We want to show that $p_1=p_2$ and $s_1=s_2$.  If $s_1=s_2$ then $p_1=p_2$ since $M$ is torsion-free, so assume now that $s_1 \neq s_2$.  Since $M$ is comparable, we can find $m \in M$ and $\rho_1, \rho_2 \in P$ such that $\rho_1+m = s_1$, $\rho_2 + m =s_2$.  Now suppose $p+m' = m$ for some $p \in P$, $m' \in M$.  Then we find $(\rho_1 + p) +m' = \rho_1 + m = s_1$, hence $\rho_1+p \in P^*$ since $s_1$ is $<$-minimal, hence $p \in P^*$.  This proves that $m$ is $<$-minimal.  Since $s_1$ and $s_2$ are distinct and both are in the $P$-submodule of $M$ generated by $m$, the set $S' := \{ m \} \cup (S \setminus \{ s_1,s_2 \})$ generates $M$, consists solely of $<$-minimal elements, and has smaller cardinality than $S$, a contradiction. \end{proof}

\subsection{Free morphisms of monoids}  A monoid homomorphism $h : Q \to P$ will be called \emph{free} iff $P$ is free as a $Q$-module---i.e.\ there is a sub\emph{set} $S \subseteq P$ such that $(s,q) \mapsto s+h(Q)$ defines a bijection of sets $S \times Q \to P$.  In particular, this requires $h$ to be monic, so we will often drop $h$ from the notation.

Such a set $S$ will be called a \emph{basis (for $h$)}.  If $S$ any basis, then we can write $0 = s_0 + q_0$ for some $s_0 \in S$ and some $s_0 \in S$, so $s_0$ and $q$ are units and we can then replace $s_0$ by $0$ to obtain a new basis, thus we can---and will---assume that any basis contains $0$.

We again emphasize that a basis is not required to be a submonoid.  The notion of ``basis" is for $Q$-\emph{modules}, not monoids under $Q$.  If $q : P \to P/Q$ is the quotient of $h$ and $S$ is a basis for $h$, then $q|S : S \to P/Q$ is a bijection of sets.  One can find a basis for $h$ which is a sub\emph{monoid} iff $P$ splits as $Q \oplus P/Q$.

Notice that the free monoid homomorphism $h$ can be recovered from the functions $\alpha : S \times S \to S$ and $\beta : S \times S \to Q$ defined by \be s+t & = & \alpha(s,t) + \beta(s,t). \ee  The functions $\alpha$ and $\beta$ (we will call them the \emph{structure maps}) defined in this way clearly satisfy the conditions \bne{conditions} \alpha(s,t) & = & \alpha(t,s) \\ \nonumber \beta(s,t) & = & \beta(t,s) \\ \nonumber \alpha(s,0) & = & s \\ \nonumber \beta(s,0) & = & 0 \\ \nonumber \alpha(\alpha(r,s),t) & = & \alpha(r,\alpha(s,t)) \\ \nonumber \beta(\beta(r,s),t) & = & \beta(r,\beta(s,t)) \ene for all $r,s,t \in S$.  Conversely, given structure maps $\alpha$ and $\beta$ satisfying these conditions, one can define a monoid structure (addition law) on $S \times Q$ by setting \be (s,q)+(s',q') & := & (\alpha(s,s'),\beta(s,s')+q) \ee and a free morphism $Q \to S \times Q$ by $q \mapsto (0,q)$.  The constructions are inverse in the sense that, if $\alpha$ and $\beta$ are constructed from a free map $h : Q \to P$ with basis $S \subseteq P$, then the map $S \times Q \to P$ given by $(s,q) \mapsto s+q$ is an isomorphism of monoids under $Q$ when $S \times Q$ is given the monoid structure constructed from $\alpha,\beta$ as above.

\begin{rem} The data $\alpha,\beta$ and conditions \eqref{conditions} look like a ``$Q$-deformation" of the requirement that $\alpha$ define a monoid structure on $S$. \end{rem}

\begin{defn} \label{defn:product} If $h_i : Q_i \to P_i$ are maps of monoids, then their \emph{product} $h$ is the map of monoids from $Q := \prod_i Q_i$ to $P := \prod_i P_i$ defined by $h(q)_i := h_i(q_i)$. \end{defn}

\begin{lem} \label{lem:freemorphisms} \begin{enumerate} \item \label{product} A product of free morphisms is free. \item \label{composition} Free morphisms are closed under composition. \item \label{pushout} Free morphisms are closed under pushout. \item \label{group} Any injective map from a group to an integral monoid is free. \item \label{integralimpliesfree} Any injective integral morphism between fine, sharp monoids is free. \end{enumerate} \end{lem}

\begin{proof} For \eqref{product} just note that if $h_i : Q_i \to P_i$ is free with basis $S_i \subseteq P_i$, then the product of the $h_i$ is free with basis $\prod_i S_i \subseteq \prod_i P_i$. For \eqref{composition}, if $h : Q \to P$ is free with basis $S \subseteq P$ and $g : P \to R$ is free with basis $T$, then $gh$ is free with basis $g(S) \times T$: the map $S \times T \times Q \to R$ given by $(s,t,q) \mapsto g(s)+t+gh(q)$ is a bijection of sets.  For \eqref{pushout}, we could quote Lemma~\ref{lem:basechange}, but let us give a very concrete proof:  Suppose $h : Q \to P$ is free with basis $S$, so we can identity $P$ as a monoid under $Q$ with $S \times Q$ with the monoid structure discussed above.  If $g : Q \to Q'$ is an arbitrary monoid homomorphism, then the set $S \times Q'$ becomes a free monoid under $Q'$ by giving it the addition law defined as above using the structure maps $\alpha, g \beta$.  One easily checks that the diagram $$ \xym{ Q \ar[r] \ar[d]_g & S \times Q \ar[d]^{(s,q) \mapsto (s,g(q))} \\ Q' \ar[r] & S \times Q' } $$ is a pushout diagram in monoids.  For \eqref{group}, if $h : G \to P$ is injective and $P$ is integral, then the $G$-action on the set $P$ is a free action and a basis $S$ can be obtained by choosing a point from each $G$-orbit.  The statement \eqref{integralimpliesfree} is the \emph{Integral Splitting Lemma} of F.~Kato: see \cite[\S1]{FK}, or my slightly different proof in \cite[Lemma~7]{G1}.  A basis for $h$ can be constructed as follows: Call an element $p \in P$ \emph{primitive (for $h$)} iff whenever we can write $p = p' + h(q)$ for $p' \in P$, $q \in Q$, we must have $q = 0$.  One can show that the primitive elements of $P$ form a basis. \end{proof}

\begin{example} \label{example:freemonoidhomomorphisms} For any monoids $P,Q$, the monoid homomorphism $Q \to Q \oplus P$ given by $q \mapsto (q,0)$ is free with basis $\{ 0 \} \times P$. \end{example}

\subsection{Partition morphisms}  \label{section:partitionmorphisms} In many ``real world" applications of logarithmic geometry (particularly in degeneration theory), the monoid homomorphisms involved are fairly simple.  In this section we introduce some typical types of monoid homomorphisms that are common in applications and establish their basic properties.

\begin{defn} \emph{Partition morphisms} (resp.\ \emph{partition morphisms with boundary}) are the smallest class of morphisms of fine monoids closed under composition, pushout, and finite products (Definition~\ref{defn:product}) and containing the small diagonal maps $\Delta : \NN \to \NN^m$ for all $m \geq 1$ (as well as all isomorphisms) (resp.\ as well as the map $0 \to \NN$). \end{defn}

It is clear from this definition that every partition morphism is, in particular, a partition morphism with boundary.

\begin{defn} \label{defn:vertical} A map of monoids is called \emph{vertical} iff its cokernel is a group. \end{defn}

\begin{prop} \label{prop:diagonalmaps} For every $m \geq 1$, the diagonal map $\Delta : \NN \to \NN^m$ is free with basis \be S & = & \{ s \in \NN^m : s_i=0 {\rm \; for \; some \; } i \} \ee and $\Delta$ is vertical with cokernel $\ZZ^{m-1}$. \end{prop}

\begin{proof} Given $p \in \NN^m$, let $q := \min \{ p_1,\dots,p_m \}$.  Then we have $p  =  \Delta(q) + s$ where $s = (p_1-q,\dots,p_m-q)$ and $s \in S$ because $q=p_i$ for some $i$, so one of the coordinates of $s$ is zero.  If $p = \Delta(q')+s'$ is another such expression, then since each coordinate of $s'$ is $\geq 0$ and at least one coordinate of $s'$ is zero, the minimum of any of the coordinates of $p$ must be $q'$.  That is, $q=q'$.  But then we have $\Delta(q)+s = \Delta(q)+s'$, so $s=s'$.  Clearly the cokernel of $\Delta^{\rm gp}$ is $\ZZ^{m-1}$, so we need only show that $\Cok \Delta$ is a group.  It suffices to show that for any $a \in \NN^m$, there is $b \in \NN^m$ such that $a+b \in \Delta(\NN) \subseteq \NN^m$.  To arrange this, we need only take any integer $t$ at least as large as any of the $a_i$ and set $b := (t-a_1,t-a_2,\dots,t-a_m)$. \end{proof}

\begin{lem} \label{lem:verticalmorphisms} \begin{enumerate} \item \label{verticalproducts} A product of vertical morphisms is vertical. \item \label{verticalpushouts} Any pushout of a vertical morphism is vertical. \item \label{verticalcomposition} A composition of vertical morphisms is vertical. \end{enumerate} \end{lem}

\begin{proof} For \eqref{verticalproducts} just note that the cokernel of a product is the product of the cokernels.  For \eqref{verticalpushouts} just note that the cokernel of any pushout of $h$ is equal to the coknerel of $h$ (direct limits commute amongst themselves).  For \eqref{verticalcomposition}, consider any monoid homomorphisms $P \to Q$ and $Q \to R$.  We have a pushout diagram as below: \bne{pushoutdiagramZ} & \xym{ P \ar[r] \ar[d] & Q \ar[r] \ar[d] & R \ar[d] \\ 0 \ar[r] & Q/P \ar[r] \ar[d] & R/P \ar[d] \\ & 0 \ar[r] & R/Q } \ene If $P \to Q$ and $Q \to R$ are vertical, then $Q/P$ is a group and the cokernel of $Q/P \to R/P$ is a group, so to show $P \to R$ is vertical (i.e.\ $R/P$ is a group) we reduce to the trivial lemma below. \end{proof}

\begin{lem} Let $Q \to P$ be a monoid homomorphism.  Suppose $Q$ and $P/Q$ are groups.  Then $P$ is a group. \end{lem}

\begin{proof} Trivial exercise. \end{proof}

\begin{prop} \label{prop:partitionmorphismsarefree} Let $h : Q \to P$ be a partition morphism with boundary.  Then $h$ is free and $(Q/P)^{\rm gp}$ is a finitely generated free abelian group.  If $h$ is a partition morphism, then $h$ is vertical $(Q/P = (Q/P)^{\rm gp})$. \end{prop}

\begin{proof} To see that $h$ is free, just note that the diagonal maps $\Delta : \NN \to \NN^m$ are free (Proposition~\ref{prop:diagonalmaps}), $0 \to \NN$ is free, and free maps are closed under composition, pushout, and products (Lemma~\ref{lem:freemorphisms}).  The argument for verticality of partition maps goes the same way (replace Lemma~\ref{lem:freemorphisms} with Lemma~\ref{lem:verticalmorphisms}).  To see that $(Q/P)^{\rm gp}$ is a finitely generated free abelian group, we first note that this is true when $Q \to P$ is one of the diagonal maps $\Delta$ or $0 : 0 \to \NN$, then we check that the property of having $(Q/P)^{\rm gp}$ a finitely generated free abelian group is a property of monoid homomorphisms $h : Q  \to P$ which is stable under pushout and finite products (this is trivial) and composition.  For stability under composition, we argue as in the proof of Lemma~\ref{lem:verticalmorphisms}\eqref{verticalcomposition}, noting that the groupification of \eqref{pushoutdiagramZ} stays a pushout diagram (groupification preserves direct limits) and an extension of a finitely generated free abelian group by a finitely generated free abelian group is again a finitely generated free abelian group.  \end{proof}

Partition morphisms with boundary arise in nature as follows.  Suppose $X$ is a smooth variety with log structure from a simple normal crossings divisor $D = D_1 \cup \cdots \cup D_n$ and $Y$ is a smooth variety with log structure from a simple normal crossings divisor $E = E_1 \cup \cdots \cup E_m$.  Suppose $f : X \to Y$ is a map of varieties such that for each $i$ we have $f^*E_i = \sum_j a_i^j D_j$ where each $a^i_j$ is zero or one.  Then $f : X \to Y$ is a log smooth morphism such that each map $\ov{f}^\dagger_x : \ov{\M}_{Y,f(x)} \to \ov{\M}_{X,x}$ is a partition morphism with boundary (between free monoids).

\section{Flatness} \label{section:flatness}  This section is a collection of generalities concerning flatness in the context of algebraic geometry needed elsewhere in the text.  For general background, see \cite[IV.2]{EGA}, \cite[IV.6]{EGA}, \cite[IV.11]{EGA}, \cite[IV]{SGA1}, \cite[Chapters 2 and 8]{Mat}, \cite{RG}, \cite[\S\S7.94, 7.121-123, 34]{SP}.

\subsection{Flatness over stacks} \label{section:flatnessoverstacks}  Let $f : X \to Y$ be a morphism of algebraic stacks, $M$ an $\O_X$-module.  We briefly recall here how one defines ``$M$ is flat over $Y$."  In fact we start by recalling how one defines this for schemes.

\begin{defn} Let $f : X \to Y$ be a map of ringed spaces, $x$ a point of $X$ with image $y := f(x) \in Y$.  An $\O_X$-module $M$ is called \emph{flat over} $Y$ \emph{at} $x$ iff $F_x$ is flat when regarded as an $\O_{Y,y}$-module via restriction of scalars along $f_x : \O_{Y,y} \to \O_{X,x}$.  $M$ is called \emph{flat over} $Y$ iff it satisfies the equivalent conditions of the following: \end{defn}

\begin{lem} \label{lem:flatnessdefinition} Let $f : X \to Y$ be a map of ringed spaces, $M$ an $\O_X$-module.  The following are equivalent: \begin{enumerate} \item \label{exactA} The functor $ \slot \otimes_{f^{-1} \O_Y} M : \Mod(f^{-1} \O_Y)  \to  \Mod(\O_X) $ is exact.  \item \label{exactB} The functor $ f^* \slot \otimes M = f^{-1} \slot \otimes_{f^{-1} \O_Y} M : \Mod(\O_Y)  \to  \Mod(\O_X) $ is exact.  \item \label{exactC} $M$ is flat over $Y$ at $x$ for every $x \in X$. \end{enumerate}  If $Y$ is a scheme which is either quasi-separated or locally noetherian, then these conditions are also equivalent to \begin{enumerate} \setcounter{enumi}{3} \item \label{exactQco} The functor $f^* \slot \otimes M : \Qco(Y) \to \Mod(\O_X)$ is exact. \end{enumerate}  \end{lem}

\begin{proof} \eqref{exactA} implies \eqref{exactB} because the functor in \eqref{exactB} is the composition of the functor in \eqref{exactA} and the exact functor $ f^{-1} : \Mod(\O_Y)  \to  \Mod(f^{-1} \O_Y). $  \eqref{exactB} implies \eqref{exactC} because there is an isomorphism \be N \otimes_{\O_{Y,y}} M_x & = & (f^{-1}(y_* N) \otimes_{f^{-1} \O_Y} M)_x \ee natural in $N \in \Mod(\O_{Y,y})$ and the right hand side is clearly exact under hypothesis \eqref{exactB}.  \eqref{exactC} implies \eqref{exactA} because \be (N \otimes_{f^{-1} \O_Y} M)_x & = & N_y \otimes_{\O_{Y,y}} M_x \ee and the exactness in \eqref{exactA} can be checked after composing with stalks at each point of $x$. 

Now suppose $Y$ is a scheme satisfying the indicated hypotheses.  Obviously \eqref{exactB} implies \eqref{exactQco}, so to add \eqref{exactQco} onto the list of equivalent conditions we need only show that \eqref{exactQco} implies \eqref{exactC}.  This is easy once we show that for every $y \in Y$ and every $\O_{Y,y}$-module $M$, there exists a quasi-coherent sheaf $\F$ on $Y$ with stalk $\F_y \cong M$.  Take an affine open neighborhood $U = \Spec A$ of $y$ in $Y$.  Let $f : U \into Y$ be the inclusion.  We can easily find such a quasi-coherent sheaf $\F$ on $U$ simply by taking $\F = M^{\sim}$, regarding $M$ as an $A$-module via the natural map $A \to \O_{Y,y} = A_y$, and we could then obtain the desired $\F$ on $Y$ by using the pushforward $f_* \F$, provided we know this pushforward is quasi-coherent.  This is \cite[II.5.8(c)]{H} provided we know the map $f$ is separated and quasi-compact.  Like any monomorphism $f$ is certainly separated because its diagonal is an isomorphism.  If $Y$ is quasi-separated then since $U$ is affine it is easy to see that $f$ is quasi-compact.  If $Y$ is locally noetherian, then we can take $A$ noetherian, hence $U$ is a noetherian space and $f$ is again easily seen to be quasi-compact because any open subspace of a noetherian space is quasi-compact.        \end{proof}

\begin{rem} If $f : X \to Y$ is a map of ringed topoi, one should define flatness using the first condition in Lemma~\ref{lem:flatnessdefinition}.  This condition is implied by the third condition provided $X$ has enough points.  In this generality, \eqref{exactB} implies \eqref{exactC} is not so clear to me.  In the above proof we used the fact that the adjunction morphism $y^{-1} y_* \to \Id$ is an isomorphism for a point $y$ of the topological space $Y$.  This map is not generally an isomorphism for an arbitrary point $y=(y^{-1},y_*)$ of an arbitrary topos $Y$ (in the sense of \cite[IV.6.1]{SGA4}).  For example, if $Y = BG$ is the classifying topos of a group $G$, then the only point of $Y$ is $y=(y^{-1},y_*)$ where $y^{-1} : BG \to \Sets$ is ``forget the $G$-action" \cite[IV.7.2]{SGA4} and $y_* : \Sets \to BG$ is given by $y_*T := \Hom_{\Sets}(G,T)$ where $y_*T$ is regarded as a $G$-set via the action $(g \cdot \gamma)(h) := \gamma(hg).$  The adjunction morphism $y^{-1} y_* \to \Id$ is given by $\gamma \mapsto \gamma(1)$; this is not an isomorphism unless $G$ is the trivial group.  \end{rem}

It is straightforward to check that flatness is ``stable under base change" in the following sense:  Given a cartesian diagram of ringed spaces or locally ringed spaces \bne{LRScartesian} & \xym{ X' \ar[r]^f \ar[d] & X \ar[d] \\ Y' \ar[r] & Y } \ene and an $\O_X$-module $F$ flat over $Y$, the $\O_{X'}$-module $F' := f^* F$ is flat over $Y'$.  This is clear for ringed spaces from the usual stability of flatness under base change and the fact that the stalk of $\O_{X'}$ at a point of $X'$ is the tensor product of the corresponding stalks of $\O_{Y'}$ and $\O_{X}$ over $\O_{Y}$.  The statement for locally ringed spaces then follows from the fact that the comparison map from the locally ringed space fibered product to the ringed space fibered product is flat.

\begin{defn} \label{defn:flatcover} A morphism $f : X \to Y$ of locally ringed spaces is called a \emph{flat cover} iff, for every $y \in Y$, there is an $x \in f^{-1}(y)$ such that $f_x : \O_{Y,y} \to \O_{X,x}$ is flat (equivalently faithfully flat since $f_x$ is a local map of local rings). \end{defn}

It is straightforward to see that flat covers are stable under composition and base change in locally ringed spaces.  It is also straightforward to see that in the cartesian $\LRS$ diagram \eqref{LRScartesian} where $Y' \to Y$ is a flat cover, an $\O_X$-module $M$ is flat over $Y$ iff the $\O_{X'}$-module $M' := f^*M$ is flat over $Y'$.

The finiteness hypotheses one often places on algebraic stacks are irrelevant here, so in this section, an \emph{algebraic stack} is a stack $X$ on schemes (or schemes over some base) in the \'etale topology with representable diagonal such that there is a scheme $X'$ and a (necessarily representable) flat cover $X' \to X$.

\begin{defn} \label{defn:flatnessoverstacks} Suppose $Y$ is an algebraic stack (in the above sense), $X$ is a scheme, $M$ is an $\O_X$-module, and $f : X \to Y$ is a map of algebraic stacks.  Choose a flat cover $Y' \to Y$.  We say that $M$ is \emph{flat over} $Y$ iff the pullback $M'$ of $M$ to $X' := X \times_Y Y'$ is flat over $Y'$.  \end{defn}

To see that this is a ``good definition," one checks that: \begin{enumerate} \item It is independent of the choice of flat cover $Y' \to Y$. \item It reduces to Definition~\ref{defn:flat} if $f$ is a map of schemes. \item \label{basechange} It has the same ``stability under base change" as the usual notion for locally ringed spaces.  \end{enumerate} Suppose $Y'' \to Y$ is another flat cover.  Then we can set $Y''' := Y' \times_Y Y''$ and note that $Y''' \to Y'$ and $Y''' \to Y''$ are flat covers, then note that flatness of $M'$ over $Y'$ is equivalent to flatness of $M''$ over $Y''$ because both are equivalent to flatness of $M'''$ over $Y'''$.  It is then clear that this notion reduces to the usual notion of flatness if $Y$  is a scheme, for then one can take $Y'=Y$, and it is similarly straightforward to establish \eqref{basechange}.

One can also allow $X$ to be an algebraic stack, then define flatness using a flat cover of $X'$ provided one has a reasonable notion of ``$\O_X$-module".  Since this is not needed in the present paper we leave it to the reader.  We have also assumed here, for simplicity, that ``representable" means ``representable by schemes," but one can make sense of this for algebraic stacks where the diagonal is only representable by algebraic spaces.

\subsection{Fppf stalks artifice}  One can alternatively define flatness for maps of algebraic stacks by using fppf stalks.  In this section, an \emph{algebraic stack} is a stack $X$ on schemes (or schemes over some base) in the \'etale topology with representable diagonal such that there is a scheme $X'$ and a (necessarily representable) fppf cover $X' \to X$.

Let $k$ be an algebraically closed field.  Any fppf cover of $\Spec k$ has a section, hence the fppf topos of $\Spec k$ is equivalent to the category of sets (via the global section functor).  Consequently any geometric point $\ov{x} : \Spec k \to X$ of a scheme $X$ yields a point of the fppf topos of $X$; the image of this point in the Zariski topos is just the point $x$ of $X$ given by the image of the map $\ov{x}$.  Every Zariski point of $X$ hence underlies some fppf point of $X$.  The fppf local ring $\O_{X,\ov{x}}^{fppf}$ of $X$ at $\ov{x}$ is the (filtered) direct limit of the local rings $\O_{U,u}$ where $(U,u)$ runs over neighborhoods of $\ov{x}$ in the fppf site of $X$ (fppf maps $U \to X$ equipped with a lift $u : \Spec k \to U$ of $\ov{x}$).  The map $\O_{X,x} \to \O_{X,\ov{x}}^{fppf}$ is hence a filtered direct limit of flat local maps of local rings, hence is itself a flat local map of local rings, hence it is faithfully flat.  If $X \to Y$ is a map of schemes and $\ov{x}$ is a geometric point of $X$ with image $\ov{y}$ in $Y$, then there is a commutative diagram of rings \bne{fppfstalks} & \xym{ \O_{X,x} \ar[r] &  \O_{X,\ov{x}}^{fppf} \\ \O_{Y,y} \ar[r] \ar[u] & \O_{Y,\ov{y}}^{fppf} \ar[u] } \ene where the horizontal arrows are faithfully flat and the map $$\O_{Y,\ov{y}}^{fppf} \otimes_{\O_{Y,y}} \O_{X,x} \to \O_{X,\ov{x}}^{fppf}$$ is flat.

\begin{lem} \label{lem:flatness2} Suppose $A \to A'$ is a flat ring map, $B \to B'$ is a faithfully flat ring map, and $$ \xym{ B \ar[r] & B' \\ A \ar[u] \ar[r] & A' \ar[u]  } $$ is a commutative diagram of rings.  If $M$ is a $B$-module such that $M \otimes_B B'$ is flat over $A'$, then $M$ is flat over $A$.  If the natural map $B \otimes_A A' \to B'$ is flat then the converse holds. \end{lem}

\begin{proof} We want to check that $N \mapsto N \otimes_A M$ is an exact functor in $N \in \Mod(A)$.  Since $B \to B'$ is faithfully flat, it suffices to check that $N \mapsto (N \otimes_A M) \otimes_B B'$ is an exact functor.  But \be (N \otimes_A M) \otimes_B B' & = & (N \otimes_A A') \otimes_{A'} (M \otimes_B B') \ee is exact since $A \to A'$ is flat and $M \otimes_B B'$ is flat over $A'$.  Now suppose the natural map from $C := B \otimes_A A'$ to $B$ is flat and $M$ is flat over $A$; we want to show $M \otimes_B B'$ is flat over $A'$.  Stability of flatness under base change implies that the $C$-module $M' := M \otimes_A A'$ is flat over $A'$ and then $M \otimes_B B' = M' \otimes_C B'$ is flat over $A'$ since $C \to B'$ is flat. \end{proof}

\begin{lem} Let $f : X \to Y$ be a map of schemes, $M$ an $\O_X$-module.  The following are equivalent: \begin{enumerate} \item $M$ is flat over $Y$. \item For every geometric point $\ov{x}$ of $X$, the fppf stalk $M_{\ov{x}} := M_x \otimes_{\O_{X,x}} \O_{X,\ov{x}}^{fppf}$ is flat over $\O_{Y,\ov{y}}^{fppf}$. \end{enumerate} \end{lem}

\begin{proof} Apply Lemma~\ref{lem:flatness2} to the diagram \eqref{fppfstalks}. \end{proof}

If $\ov{x}$ is a geometric point of an algebraic stack $X$, then one can define the fppf local ring $\O_{X,\ov{x}}^{fppf}$ of $X$ at $\ov{x}$ by using a fixed fppf cover $X' \to X$ and a chosen lift $\ov{x}'$ of $\ov{x}$ to $X'$; one then checks easily that $\O_{X,\ov{x}}^{fppf} := \O_{X',\ov{x}'}^{fppf}$ does not depend on these choices.  One can then define flatness using the second criterion in the lemma above.  It is straightforward to see that this notion of flatness coincides with the one from \S\ref{section:flatnessoverstacks}.

\subsection{Flatness and \'etale maps}  The main result of this section (Lemma~\ref{lem:etaleflatness}) can also be found in the \emph{Stacks Project} \cite[Lemma~34.3.3]{SP}\footnote{The hypothesis ``quasi-coherent" in the statement there is not necessary, as is made clear in Lemma~34.3.4.}.  The proof given here is different and (I think) simpler.

\begin{lem} \label{lem:etaleflatness1} Let $A \to B$ be a ring homomorphism.  Let $C := B \otimes_A B$.  Regard $B$ as a $C$-algebra via the ring homomorphism $m : C \to B$ given by $b_1 \otimes b_2 \mapsto b_1b_2$.  For any $B$-modules $M,N$, we have an isomorphism of $B$-modules \be M \otimes_B N & \to & (M \otimes_A N) \otimes_C B \\ m \otimes n & \mapsto & m \otimes n \otimes 1 \ee natural in $M,N$. \end{lem}

\begin{proof} The inverse of the map in question is given by $m \otimes n \otimes b \mapsto bm \otimes n$.  The only issue is to check that these maps are well-defined (check bilinearity), which is straightforward. \end{proof}

\begin{lem} \label{lem:etaleflatness2} Suppose $$ \xym{ A \ar[r]^f  \ar[d]_f & B \ar[d] \ar@/^1pc/[rdd]^{=} \\ B \ar[r] \ar@/_1pc/[rrd]_{=} & D \ar[rd] \\ & & B}$$ is a commutative diagram of flat ring homomorphisms and the induced map $C := B \otimes_A B \to D$ is also flat.  Then a $B$-module $M$ is flat iff it is flat as an $A$-module. \end{lem}

\begin{proof} The ``only if" is standard and relies only on the flatness of $f$, so the issue is to prove $M$ is a flat $B$-module when it is flat as an $A$-module.  The hypotheses ensure that the multiplication map $m : C \to B$ factors as a composition $C \to D \to B$ of flat maps, hence is flat.  But then Lemma~\ref{lem:etaleflatness1} shows that the functor $M \otimes_ B \slot$ can be written as a composition of the functors $ M \otimes_A \slot$ and $\slot \otimes_C B$, both of which are exact since $M$ is flat over $A$ and $B$ is flat over $C$. \end{proof}

\begin{lem} \label{lem:etaleflatness} Let $X \to Y \to Z$ be maps of schemes with $Y \to Z$ \'etale.  Then an $\O_X$-module $F$ is flat over $Y$ iff it is flat over $Z$. \end{lem}

\begin{proof} The ``only if" is standard and relies only on the fact that the \'etale map $Y \to Z$ is flat.  For the ``if," first note that the diagonal map $\Delta : Y \to Y \times_Z Y =: W$ is \'etale because it is a map of \'etale $Y$-schemes [SGA1 I.4.8].  Let $x$ be a point of $X$, $y$ (resp.\ $z$) its image in $Y$ (resp.\ $Z$).  We want to show $F_x$ is a flat $\O_{Y,y}$ module assuming it is a flat $\O_{Z,z}$-module.  To do this, we just note that the commutative diagram of rings $$ \xym{ \O_{Z,y} \ar[r] \ar[d] & \O_{Y,y} \ar[d] \ar@/^1pc/[rdd]^{=} \\ \O_{Y,y} \ar@/_1pc/[rrd]_{=} \ar[r] & \O_{W,\Delta(y)} \ar[rd]^{\Delta_y} \\ & & \O_{Y,y} } $$ satisfies the hypotheses of Lemma~\ref{lem:etaleflatness2}: the map $\Delta_y$ is flat because $\Delta$ is \'etale, hence flat, and the map $$ \O_{Y,y} \otimes_{\O_{Z,z}} \O_{Y,y} \to \O_{W,\Delta(y)} $$ is flat because it is a localization by basic structure theory of inverse limits of schemes.  \end{proof}

\begin{rem} Taking $Z = \Spec \CC$, $Y=\AA^1_{\CC}$ (resp.\ the first infinitesimal neighborhood of the origin in $\AA^1_{\CC}$) $X=\Spec \CC$ regarded as the origin in $Y$, $F = \O_X$ we see that ``\'etale" cannot be weakend to ``smooth" (resp.\ ``finite flat") in Lemma~\ref{lem:etaleflatness} even when the maps are maps of finite-type $\CC$-schemes. \end{rem}

\subsection{Fiberwise flatness criteria}  For the sake of clarity we here recall the ``crit\`ere de platitude par fibres," refering mostly to \cite[7.121-122]{SP} for proofs.  The statement we want (Lemma~\ref{lem:fiberwiseflatness}) is basically \cite[7.122.8]{SP}, but the variant here is not explicitly stated there, so we will sketch the proof. 

\begin{lem} \label{lem:Matsumuraflatness} Let $(A,\m,k) \to (B,\n,l)$ be a local map of local rings.  Assume $B$ is noetherian and $M$ is finitely generated as a $B$-module.  Then $M$ is flat over $A$ iff $\Tor_1^A(M,k)=0$. \end{lem}

\begin{proof} This is \cite[7.94.7]{SP}---or apply \cite[20.C Theorem~49]{Mat}, using the criterion (3').  Note that $M$, as an $A$-module, is ``idealwise separated" for $\m$, as discussed in \cite[Page 145, Example~1]{Mat}.\footnote{Although this is clearly the main intended use of \cite[20.C]{Mat}, this is never quite made as explicit as one might expect.} \end{proof}

In what follows we need to make some ``noetherian approximation" arguments.  Let $\C$ denote the category whose objects are pairs $(A \to B,M)$ where $A \to B$ is a local map of local rings and $M$ is a $B$-module.  A $\C$-morphism from $(A \to B,M)$ to $(A' \to B',M')$ is a commutative square of local maps of local rings $$ \xym{ A \ar[r] \ar[d] & B \ar[d] \\ A' \ar[r] & B' } $$ together with a $B$-module map $M \to M'$ (regarding $M'$ as a $B$-module via the ring map $B \to B'$).

\begin{lem} \label{lem:noetherianapproximation} Let $A \to B$ be a local homomorphism of local rings, $M$ a finitely presented $B$-module.  Assume $B$ is of essentially finite presentation over $A$.  Then there exists a filtered partially ordered set $\Lambda$ and a functor $(A_\lambda \to B_{\lambda},M_\lambda)$ from $\Lambda$ to $\C$ with the following properties: \begin{enumerate} \item The direct limit of $(A_\lambda \to B_{\lambda},M_\lambda)$ is $(A \to B,M)$. \item For each $\lambda \in \Lambda$, the local rings $A_\lambda$ and $B_\lambda$ are noetherian.\footnote{One can even arrange that $A_\lambda$ is essentially of finite type over $\ZZ$ and $B_\lambda$ is essentially of finite type over $A_\lambda$, but these extra hypotheses are not necessary for anything that follows.}  \item Each $M_{\lambda}$ is a finitely generated $B_\lambda$-module. \item For each $\lambda \leq \mu$, the map $A_\lambda \otimes_{B_\lambda} A_\mu \to B_{\mu}$ presents $B_\mu$ as the localization of $A_\lambda \otimes_{B_\lambda} A_\mu$ at a prime ideal. \item For each $\lambda \leq \mu$, the map $M_\lambda \otimes_{B_\lambda} B_\mu \to M_\mu$ is an isomorphism. \end{enumerate}  Furthermore, if $M$ is flat over $A$ then we can also arrange that this direct limit system has the property: \begin{enumerate} \setcounter{enumi}{5} \item $M_\lambda$ is flat over $A_\lambda$ for each $\lambda \in \Lambda$.\footnote{In fact, when $M$ is flat over $A$, in \emph{any} direct limit system satisfying the first five properties, $M_\lambda$ will be flat over $A_\lambda$ for all sufficiently large $\lambda$.} \end{enumerate} \end{lem}

\begin{proof} See \cite[7.121.11]{SP} and \cite[7.122.3]{SP}. \end{proof}

\begin{lem} \label{lem:fiberwiseflatness} Let $A \to B \to C$ be local map of local rings, $M$ a $C$-module.  Suppose $M$ is flat over $A$ and assume at least one of the following holds: \begin{enumerate} \item $B$ and $C$ are noetherian and $M$ is finitely generated as a $C$-module.  \item $A \to B$ and $A \to C$ are of essentially finite presentation and $M$ is finitely presented as a $C$-module.  \end{enumerate} Let $k$ be the residue field of $A$.  Then $M$ is flat over $B$ iff $M \otimes_A k$ is flat over $B \otimes_A k$. \end{lem}

\begin{proof} The implication $\implies$ is clear from stability of flatness under base change because $M \otimes_A k = M \otimes_B (B \otimes_A k)$ and requires none of the finiteness hypotheses.  For the other implication under the noetherian assumption: Set $\ov{B} := B \otimes_A k$.  Let $K$ be the residue field of $B$.  By Lemma~\ref{lem:Matsumuraflatness} it suffices to show $\Tor_1^B(M,K)=0$.  Since $N \otimes_B K = (N \otimes_A k) \otimes_{\ov{B}} K$ for any $C$-module $N$, it suffices to establish the vanishings \be \Tor_1^A(M,k) & = & 0 \\ \Tor_1^{\ov{B}}(M \otimes_A k,K) & = & 0, \ee which are clear from the hypotheses.

Under the other finiteness hypotheses, we find (by a variant of Lemma~\ref{lem:noetherianapproximation} as in the proof of \cite[7.122.8]{SP}) a filtered poset $\Lambda$ and a direct limit system $$(A_\lambda \to B_\lambda \to C_\lambda, M_\lambda)$$ indexed by $\Lambda$ in a category analogous to $\C$ above satisfying: \begin{enumerate} \item The direct limit of $(A_\lambda \to B_\lambda \to C_\lambda, M_\lambda)$ is $(A \to B \to C,M)$.  \item The maps $A_\lambda \to B_\lambda \to C_\lambda$ are local maps of local noetherian rings.  \item The direct limit system $(A_{\lambda} \to C_{\lambda},M_\lambda)$ in $\C$ satisfies all six properties of Lemma~\ref{lem:noetherianapproximation}.  \item Letting $k_\lambda$ denote the residue field of $A_{\lambda}$, the direct limit system $$(B_\lambda \otimes_{A_\lambda} k_\lambda \to C_\lambda \otimes_{A_\lambda} k_\lambda, M_\lambda \otimes_{A_\lambda} k_\lambda)$$ in $\C$ also satisfies all six properties of Lemma~\ref{lem:noetherianapproximation}. \end{enumerate}  Then by the noetherian case that we just handled, each $M_\lambda$ is flat over $B_\lambda$, hence the filtered direct limit $M$ is flat over $B$. \end{proof}

\begin{thm} \label{thm:fiberwiseflatness} Let $X \to Y \to Z$ be maps of schemes, $M$ an $\O_X$-module.  Suppose $M$ is flat over $Z$.  Assume at least one of the following holds: \begin{enumerate} \item $X$ and $Y$ are locally noetherian and each stalk $M_x$ is a finitely generated $\O_{X,x}$-module. \item \label{w} $X \to Z$ and $Y \to Z$ are of locally finite presentation and $M$ is of locally finite presentation as an $\O_X$-module. \item \label{ww} For each $x \in X$ (with image $y \in Y$, $z \in Z$), the local maps of local rings $\O_{Z,z} \to \O_{X,x}$ and $\O_{Z,z} \to \O_{Y,y}$ are of essentially finite presentation and $M_x$ is finitely presented as an $\O_{X,x}$-module. \end{enumerate} Then $M$ is flat over $Y$ iff $M|_{X_z}$ is flat over $Y_z$ for each point $z \in Z$, or, equivalently, for each geometric point $\ov{z}$ of $Z$. \end{thm}

\begin{proof} The implication $\implies$ is just stability of flatness under base change.  For the other implication, consider a point $x \in X$ with images $y,z$ in $Y,Z$.  We need to prove that $M_x$ is flat over $\O_{Y,y}$.  The fibers $X_z$ and $Y_z$ are the same as the ones calculated in ringed spaces so we have a pushout diagram of local maps of local rings $$ \xym{ \O_{Z,z} \ar[d] \ar[r] & \O_{Y,y} \ar[r] \ar[d] & \O_{X,x} \ar[d] \\ k(z) \ar[r] & \O_{Y_z,y} \ar[r] & \O_{X_z,x} } $$ and the assumption that $M|_{X_z}$ is flat over $Y_z$ implies that the stalk $(M|_{X_z})_x$ is flat over $\O_{Y_z,y}$---i.e.\ $M_x \otimes_{\O_{Z,z}} k(z)$ is flat over $\O_{Y,y} \otimes_{\O_{Z,z}} k(z)$.  The result then follows by applying Lemma~\ref{lem:fiberwiseflatness} to $\O_{Z,z} \to \O_{Y,y} \to \O_{X,x}$ and $M_x$  (note that the assumptions in \eqref{w} implies the ones in \eqref{ww}).  It is straightforward to replace ``points" with ``geometric points" because any field extension is faithfully flat. \end{proof}

Here is another variant:

\begin{thm} \label{thm:fiberwiseflatness2}  Let $f : X \to Y$, $h : Y \to Z$ be morphisms of schemes with composition $g : X \to Z$ and let $M$ be a quasi-coherent sheaf on $X$.  Assume that $g$, $h$, and $M$ are of locally finite presentation and that $h$ is flat.  Then for a point $x \in X$ with images $y := f(x)$, $z := g(x)$ the following conditions are equivalent: \begin{enumerate} \item $M$ is flat over $Z$ at $x$ and $M|X_z$ is flat over $Y_z$ at $x$. \item $M$ is flat over $Y$ at $x$. \end{enumerate}  The set $U$ of $x \in X$ satisfying these equivalent conditions is open in $X$.  Assume, furthermore, that $M$ is flat over $Z$ and $\Supp M$ is proper over $Z$.  Then $V := Z \setminus g(X \setminus U)$ is open in $Z$ and is the terminal object in the category of $Z$ schemes $Z'$ for which the pullback $M'$ of $M$ to $X' := X \times_Z Z'$ is flat over $Y' = Y \times_Z Z'$ \end{thm}

\begin{proof}  Up to (but not including) the last two sentences this is \cite[IV.11.3.10]{EGA} (note that the hypothesis $\F_x \neq 0$ there is irrelevant since we assume $h$ is flat, hence both (a) and (b) there hold trivially if $\F_x = 0$).  Clearly $U$ contains the complement of the support of $M$, so we can write $V = Z \setminus g(\Supp M \setminus U)$, which makes it clear that $V$ is open because $g|\Supp M$ is a closed map on topological spaces by the assumption that $g$ is proper.  Clearly $M | g^{-1}(V)$ is flat over $h^{-1}(V) \subseteq Y$ because $g^{-1}(V) \subseteq U$.  For the final statement, we need to show that if $t : Z' \to Z$ is such that $M'$ is flat over $Y'$, then $t : Z' \to Z$ factors (necessarily uniquely) through $V \subseteq Z$.  It suffices to show that $t$ factors through $V$ on the level of topological spaces.  Suppose it doesn't.  Then there is some $z' \in Z'$ such that $z := t(z') \notin V$.  This means we can write $z = g(x)$ for some $x \in X \setminus U$.  Since $M'$ is flat over $Y'$, $M'|X'_{z'}$ is flat over $Y'_{z'}$ by stability of flatness under base change.  But $X'_{z'} \to Y'_{z'}$ is a faithfully flat base change of $X_z \to Y_z$ (along $\Spec$ of the field extension $k(z) \into k(z')$), so this implies $M|X_z$ is flat over $M|Y_z$.  But we assume $M$ is flat over $Z$, so by the first part of the theorem this implies that $M$ is flat over $Y$ at $x$---i.e.\ $x \in U$, a contradiction. \end{proof}

\begin{rem} \label{rem:fiberwiseflatness} Theorem~\ref{thm:fiberwiseflatness} also holds whenever each of $X$, $Y$, $Z$ is an algebraic stack with representable diagonal admitting an fppf cover by a scheme.  First one uses the fact that the hypotheses are fppf local in nature to even define these concepts for stacks (for example ``locally noetherian" is fppf local in nature: one implication is the Hilbert Basis Theorem and for the other implication one checks the ascending chain condition for ideals, say) to even define the concepts.  Basically one then applies the theorem for schemes to the top row of a diagram \bne{stackdig} & \xym{ X' \ar[r] \ar[d] & Y' \ar[r] \ar[d] & Z' \ar[d] \\ X \ar[r] & Y \ar[r] & Z } \ene where $Z'$ is an fppf cover of $Z$ by a scheme, $Y' \to Y \times_Z Z'$ is an fppf cover by a scheme, and $X' \to X \times_Y Y'$ is an fppf cover by a scheme.  All the vertical arrows in \eqref{stackdig} are then fppf covers by schemes and one translates the hypotheses and conclusions in Theorem~\ref{thm:fiberwiseflatness} back and forth between the top row and bottom row, which is not particularly hard because most of the hypotheses for the bottom row are basically defined by saying that they hold for the top row... (Really one first bootstraps up from schemes to algebraic spaces by this discussion, then from there to stacks.) \end{rem}

\section{Graded Modules} \label{section:gradedmodules}

\subsection{Graded rings} Let $G$ be an abelian group.  A $G$-\emph{grading} on a ring $A$ is a direct sum decomposition $A = \oplus_{g \in G} A_g$ (as an additive abelian group) such that the multiplication for $A$ takes $A_g \times A_h$ into $A_{g+h}$.  In particular, $A_0 \subseteq A$ is a subring of $A$.  An element $a \in A_g \subseteq A$ is called \emph{homogeneous of degree} $g$ and we sometimes write $|a|=g$ to indicate that $a$ is homogeneous of degree $g$.   A ring $A$ equipped with a $G$-grading is called a $G$-\emph{graded ring}.  The \emph{support} of a $G$-graded ring is the submonoid of $G$ generated by the set of $g \in G$ such that $A_g \neq 0$.  

A \emph{morphism} of $G$-graded rings $f : A \to B$ is a ring homomorphism such that $f$ takes $A_g \subseteq A$ into $B_g \subseteq B$ for every $g \in G$.  Let $\An(G)$ denote the category of $G$-graded rings.  The functor \bne{A0} \An(G) & \to & \An \\ \nonumber A & \mapsto & A_0 \ene has a left adjoint \bne{degreezero} \An & \to & \An(G) \\ \nonumber A & \to & A \ene given by viewing a ring $A$ as a $G$-graded ring supported in degree zero.  That is, we have a natural bijection \bne{A0degreezeroadjunction} \Hom_{\An(G)}(A,B) & = & \Hom_{\An}(A,B_0) \ene for each ring $A$ and each $G$-graded ring $B$.

A \emph{graded ring} is a pair $(G,A)$ consisting of an abelian group $G$ and a $G$-graded ring $A$.  A morphism of graded rings $(\gamma,f) : (G,A)  \to  (H,B)$ is a pair consisting of a group homomorphism $\gamma : G \to H$ and a ring homomorphism $f : A \to B$ such that $f$ takes $A_g \subseteq A$ into $B_{\gamma(g)} \subseteq B$ for every $g \in G$.  Equivalently, $f$ is a morphism of $H$-graded rings when $A$ is regarded as an $H$-graded ring via the decomposition $A = \oplus_{h \in H} ( \oplus_{g \in \gamma^{-1}(h)} A_g).$  Let $\GrAn$ denote the category of graded rings.

Every ring $A$ can be viewed as a $0$-graded ring.  This defines a functor \bne{0A} \An & \to & \GrAn \\ \nonumber A & \mapsto & (0,A) . \ene The functor \eqref{0A} is right adjoint to the functor \bne{pi2} \pi_2 : \GrAn & \to & \An \\ \nonumber (G,A) & \mapsto & A \ene and left adjoint to the functor \bne{GrA0} \GrAn & \to & \An \\ \nonumber (G,A) & \mapsto & A_0. \ene  That is, we have natural bijections \bne{0Arightadjointtopi2} \Hom_{\An}(A,B) & = & \Hom_{\GrAn}((G,A),(0,B)) \\ \Hom_{\GrAn}((0,A),(G,B)) & = & \Hom_{\An}(A,B_0). \ene

\subsection{Monoids to graded rings} \label{section:monoidstogradedrings} Let $P$ be a monoid.  The ring $\ZZ[P]$ is naturally equipped with a grading by $P^{\rm gp}$ by setting $(\ZZ[P])_g$ equal to the abelian group of formal sums $\sum_p a_p [p]$ where $p \in P$ runs over the preimage of $g$ under $P \to P^{\rm gp}$ and all but finitely many $a_p$ are zero.  This defines a functor \bne{ZPGrAn} \Mon & \to & \GrAn \\ \nonumber P & \mapsto & (P^{\rm gp},\ZZ[P]) \ene which factors the functor $\ZZ[ \slot ]$ in \eqref{ZP} through the forgetful functor $\pi_2$ in \eqref{pi2}.  Obviously we can replace $\ZZ$ with any ring $A$, regarding $A[P]$ as a $P^{\rm gp}$-graded $A$-algebra.  Whenever we speak of $A[P]$ as a graded ring, it is understood to be graded by $P^{\rm gp}$ in this manner.

The functor \eqref{ZPGrAn} preserves direct limits because $P \mapsto P^{\rm gp}$ preserves direct limits and $\ZZ[ \slot ]$ preserves direct limits.  Pushouts in $\GrAn$ are discussed briefly in \S\ref{section:gradedmodules2}.

\noindent {\bf Warning:} If $P$ is not integral, then an element like $[p]+[q]$ in $\ZZ[P]$ can be ``homogeneous" even if $p \neq q$ because $p,q$ may have the same image in $P^{\rm gp}$.  However, when $P$ is integral, $\ZZ[P] = \oplus_{p} \ZZ[p]$ \emph{is} the $P^{\rm gp}$-grading of $\ZZ[P]$ (suppressing notation for $P \into P^{\rm gp}$) and a homogeneous element of $\ZZ[P]$ is one of the form $a[p]$ for $a \in \ZZ$, $p \in P$.

\subsection{Graded modules} \label{section:gradedmodules2} For a graded ring $(G,A)$, a $(G,A)$-module (also called a $G$-\emph{graded} $A$-\emph{module} or just a \emph{graded} $A$-\emph{module}) is an $A$-module $M$ equipped with a direct sum decomposition $M = \oplus_{g \in G} M_g$ (as an additive abelian group) such that scalar multiplication $A \times M \to M$ takes $A_g \times M_h$ into $M_{g+h}$ for each $g,h \in G$.  An element $m \in M_g \subseteq M$ is called \emph{homogeneous of degree} $g$.  A \emph{morphism} of $(G,A)$-modules is a morphism of $A$-modules compatible with the direct sum decompositions.  Let $\Mod(G,A)$ denote the category of $(G,A)$-modules.

\noindent {\bf Warning:}  For $(G,A)$-modules $M,N$, the set \be \Hom_{G,A}(M,N) & := & \Hom_{\Mod(G,A)}(M,N) \ee does not have any natural $A$-module structure, though it does have a natural $A_0$-module structure.

The category $\Mod(G,A)$ is an abelian category.  The kernel (resp.\ cokernel) of a $\Mod(G,A)$-morphism $f : M \to N$ is just its kernel (resp.\ cokernel) as a map of $A$-modules equipped with the evident grading obtained from the fact that kernels and cokernels of abelian groups commute with direct sums (so, for example, $\Ker( f : M \to N) = \oplus_g \Ker(f_g : M_g \to N_g)$ defines the grading on the $(G,A)$-module $\Ker f$).  The forgetful functor \bne{ModGAtoModA} \Mod(G,A) & \to & \Mod(A) \ene is faithful but not full and is ``faithfully exact" in the sense that a sequence of $(G,A)$-modules $$ 0 \to M' \to M \to M'' \to 0$$ in exact iff its image under \eqref{ModGAtoModA} is exact (indeed, both exactness conditions are equivalent to exactness of the underlying sequence of abelian groups).

For $h \in G$, we have a shift functor \be \Mod(G,A) & \to & \Mod(G,A) \\ M & \mapsto & M \{ h \} \ee which is the identity on the underlying module, but shifts the grading according to the rule $(M \{ h \})_g := M_{g+h}$.  This shift functor is an isomorphism of categories with inverse $\{ - h \}$.  

\begin{prop} \label{prop:enoughprojectives} The abelian category $\Mod(G,A)$ has enough projectives and injectives.  The image of a projective $(G,A)$-module under the forgetful functor \eqref{ModGAtoModA} is a projective $A$-module. \end{prop}

\begin{proof}  The first statement holds by general nonsense \cite[1.10]{Tohoku} because $\Mod(G,A)$ has all direct and inverse limits and these limits are well-behaved because they always coincide with the corresponding limits of abelian groups on the level of underlying abelian groups.  For the second statement, define a $(G,A)$-module to be \emph{free} iff it is a direct sum of shifts of copies of $A$.  It is clear that a free $(G,A)$-module is projective and that any $(G,A)$-module is a quotient of a free $(G,A)$-module, so this gives another proof of the existence of enough projectives and it shows that any projective $(G,A)$-module is a direct summand of a free $(G,A)$-module.  But the image of a free $(G,A)$-module under \eqref{ModGAtoModA} is clearly a free $A$-module, so the image of any projective $(G,A)$-module under \eqref{ModGAtoModA} is a summand of a free $A$-module and is hence a projective $A$-module. \end{proof}

\subsection{Graded tensor product} \label{section:gradedtensorproduct} If $M$ and $N$ are $(G,A)$-modules, the usual $A$-module tensor product $M \otimes_A N$ has a natural $(G,A)$-module structure given by the grading \be (M \otimes_A N)_g & := & \sum_{g_1+g_2=g} M_{g_1} \otimes_{\ZZ} N_{g_2} . \ee Note that this sum of $\ZZ$-modules (abelian groups) is hardly ever a \emph{direct} sum and we confuse $m \otimes n \in M_{g_1} \otimes_{\ZZ} N_{g_2}$ with its image in $M \otimes_A N$.  It takes a minute to see that \be M \otimes_A N & = & \bigoplus_{g \in G} (M \otimes_A N)_g. \ee  (We refer the reader to \cite[II.11.5]{Bourbakialgebra} for more on the basic notions of graded modules.)  

If $A \to B$ is a morphism of $G$-graded rings then $B$ becomes a $(G,A)$-module in an obvious manner.  Indeed, any $(G,B)$-module $M$ can be viewed as a $(G,A)$-module by \emph{restriction of scalars} in the usual way.  If $M$ is a $(G,A)$-module, the tensor product $M \otimes_A B$, as described above, has a natural $(G,B)$-module structure and, as such, is called the \emph{extension of scalars} of $M$ (along $A \to B$).  Extension of scalars is left adjoint to restriction of scalars as usual: \bne{extrest} \Hom_{G,B}(M \otimes_A B, N) & = & \Hom_{G,A}(M,N) . \ene

Now suppose $(\gamma,f) : (G,A) \to (H,B)$ is a morphism of graded rings and $N$ is an $(H,B)$-module.  Then the \emph{restriction of scalars} of $N$ is the $(G,A)$-module $N_A$ with decomposition given by $(N_A)_g := M_{\gamma(g)}$ and scalar multiplication given by $a \cdot m := f(a) m \in (M_A)_{g+g'}$ for homogeneous elements $a \in A_g$, $m \in (M_A)_{g'}$.  (This makes sense because $f(a) \in B_{\gamma(g)}$ so $f(a)m \in M_{\gamma(g+g')}$.)  Restriction of scalars defines a functor \bne{restscalars} \Mod(H,B) & \to & \Mod(G,A) \\ \nonumber N & \mapsto & N_A. \ene

\noindent {\bf Warning:} Unless $\gamma$ is an isomorphism, $N_A$ will \emph{not} generally coincide with the ``usual" restriction of scalars of $N$ along $A \to B$.  Indeed, $N_A$ will not generally be isomorphic as an abelian group to $N$ and $N_A$ will not even have any reasonable $B$-module structure.

There is also an extension of scalars \bne{extscalars} \slot \otimes_A B : \Mod(G,A) & \to & \Mod(H,B) \\ \nonumber M & \mapsto & M \otimes_A B, \ene which agrees with the usual extension of scalars on the level of $B$-modules.  We equip $M \otimes_A B$ with the grading \be (M \otimes_A B)_h & := & \sum_{\gamma(g) + h' = h} M_g \otimes_{\ZZ} B_{h'}. \ee  More generally, for $N \in \Mod(H,B)$ we have an extension of scalars functor \bne{slototimesN} \slot \otimes_A N : \Mod(G,A) & \to & \Mod(H,B) \\ \nonumber M & \mapsto & M \otimes_A N \ene which agrees with the usual tensor product on the level of underlying modules when $N$ is viewed as an $A$-module via the usual (ungraded) restriction of scalars.  We equip $M \otimes_A N$ with the grading \be (M \otimes_A N)_h & := & \sum_{\gamma(g) + h' = h} M_g \otimes_{\ZZ} N_{h'}. \ee

The extension of scalars functors \eqref{extscalars} have the usual transitivity property for a composition $$(G,A) \to (H,B) \to (K,C)$$ of maps of graded rings and the usual formula \be (M \otimes_A B) \otimes_B N & = & M \otimes_A N \ee for $M \in \Mod(G,A)$, $N \in \Mod(H,B)$ relating the functors \eqref{extscalars}, \eqref{slototimesN}, and the graded tensor product of $(H,B)$-modules holds in the graded setting.

Extension of scalars \eqref{extscalars} is left adjoint to restriction of scalars \eqref{restscalars}.  The adjunction isomorphism \bne{extrest2} \Hom_{H,B}(M \otimes_A B, N) & = & \Hom_{G,A}(M,N_A) \\ \nonumber k & \mapsto & (m \mapsto k(m \otimes 1)) \ene for $M \in \Mod(G,A)$, $N \in \Mod(H,B)$ requires some explanation.  If $k : M \otimes_A B \to N$ is an $(H,B)$-module morphism, then for a homogeneous element $m \in M_g$, $m \otimes 1 \in M \otimes_A B$ is homogeneous of degree $\gamma(g)$, so $k$ takes it into $N_{\gamma(g)} = (N_A)_g$, so ``$k(m \otimes 1)$" above is understood to lie in $(N_A)_g$.  That is, $m \mapsto k(m \otimes 1)$ is abuse of notation for  the sum over $g \in G$ of the maps \be M_g & \to & (N_A)_g = N_{\gamma(g)} \\ m & \mapsto & k(m \otimes 1). \ee  The inverse of \eqref{extrest2} takes an $(G,A)$-module map $l : M \to N_A$ to the $(H,B)$ module map $M \otimes_A B \to N$ which might abusively be written $m \otimes b \mapsto b l(m)$.  Really we write $l$ as the sum of maps $l_g : M_g \to (N_A)_g = N_{\gamma(g)}$ over $g \in G$ and for $m \in M_g$, $b \in B_h$ we let $l(m \otimes b) := b l_g(m)$, then we extend this recipe $\ZZ$-linearly.

\noindent {\bf Warning:}  The functors \eqref{slototimesN}, restriction of scalars \eqref{restscalars}, and the tensor product for $(G,A)$-modules are \emph{not} related in the way one might expect from the ungraded case.  For $M \in \Mod(G,A)$, $N \in \Mod(H,B)$, the $(H,B)$-module $M \otimes_A N$ (the image of $M$ under \eqref{slototimesN}) does \emph{not} coincide with the tensor product $M \otimes_A N_A$.  Indeed, the latter tensor product does not even carry any natural $B$-module structure.  See \S\ref{section:importantspecialcase} for further discussion.

\begin{lem} \label{lem:restrictionofscalarsexact} The restriction of scalars functor \eqref{restscalars} is exact. \end{lem}

\begin{proof} This follows easily from the fact that we can check exactness on the level of underlying abelian groups and the fact that a direct sum of sequences of abelian groups \bne{seqi} 0 \to A_i' \to A_i \to A_i'' \to 0 \ene is exact iff each \eqref{seqi} is exact. \end{proof}

\begin{prop} \label{prop:tensortakesprojtoproj} Extension of scalars \eqref{slototimesN} takes projective $(G,A)$-modules to projective $(H,B)$-modules. \end{prop}

\begin{proof} This follows formally from the fact that \eqref{slototimesN} has an exact right adjoint (Lemma~\ref{lem:restrictionofscalarsexact}). \end{proof}

\begin{rem} Proposition~\ref{prop:tensortakesprojtoproj} yields an alternative proof of the fact that the forgetful functor \bne{forgetG} \Mod(G,A) & \to & \Mod(A) \ene takes projectives to projectives (Proposition~\ref{prop:enoughprojectives}).  Indeed, the forgetful functor \eqref{forgetG} \emph{is} the extension of scalars functor for the map of graded rings $(0,\Id) : (G,A) \to (0,A)$. \end{rem}

\subsection{Important special case}  \label{section:importantspecialcase}  Here we consider the general constructions of \S\ref{section:gradedtensorproduct} in the important special case of a map of graded rings $(A,G) \to (B,H)$ where $H=0$, so that $\gamma : G \to H$ is also zero.  By the adjunction \eqref{0Arightadjointtopi2}, such a map of graded rings is the same thing as a map of rings $A \to B$ in the usual sense (ignoring the grading on $A$).  For any ring $B$, we have an obvious isomorphism of categories \bne{isocat0B} \Mod(B) & = & \Mod(0,B) \ene taking a $B$-module $N$ to $N$ with the only possible grading by the trivial group: $N = N_0$.  Suppressing this isomorphism, the restriction of scalars \eqref{restscalars}, extension of scalars \eqref{slototimesN}, and adjunction isomorphism \eqref{extrest2} can be viewed as functors \bne{restscalars2} \Mod(B) & \to & \Mod(G,A) \\ \nonumber N & \mapsto & N_A \\ \label{slototimesN2} \slot \otimes_A N : \Mod(G,A) & \to & \Mod(B)  \ene and a natural bijection \bne{extrest3} \Hom_B(M \otimes_A B,N) & = & \Hom_{G,A}(M,N_A). \ene  

For $N \in \Mod(B)$, the $(G,A)$-module $N_A$ is equipped with the grading $N_A = \oplus_{g \in G} N$.  This direct sum decomposition of abelian groups is not a direct sum decomposition of $(G,A)$-modules and is in fact not even a direct sum decomposition of $A$-modules because scalar multiplication ``mixes the summands".  Although $N_A$ has a natural $B$-module structure (making $N_A = \oplus_{g \in G} N$ a direct sum of $B$-modules), the $A$-module structure on $N_A$ underlying the $(G,A)$-module structure on $N_A$ is not the same as the $A$-module structure on $N_A$ obtained via restriction of scalars along $A \to B$ and the aforementioned $B$-module structure on $N_A$.  Indeed, there is no reason to believe that the former $A$-module structure on $N_A$ is even in the essential image of the usual restriction of scalars $\Mod(B) \to \Mod(A)$ and consequently there is no reason to believe that $M \otimes_A N_A$ has any reasonable $B$-module structure when $M$ is an $A$-module or graded $A$-module.

On the other hand, if $M$ is a graded $A$-module, then $M \otimes_A N$ (the image of $M$ under \eqref{slototimesN}) is just the usual tensor product $M \otimes_A N$ regarded as a $(0,B)$-module using the only possible grading.  In other words, the functor \eqref{slototimesN2} is just the restriction of the usual tensor product \be \slot \otimes_A N : \Mod(A) & \to & \Mod(B) \ee to the (faithful but not full) subcategory $\Mod(G,A) \subseteq \Mod(A)$.

The adjunction isomorphism \eqref{extrest3} can be described explicitly as follows:  Given a morphism of $B$-modules $f : M \otimes_A B \to N$, we obtain a corresponding morphism of $A$-modules also abusively denoted $f : M \to N$ via the usual adjunction isomorphism \bne{extrestusual} \Hom_B(M \otimes_A B,N) & = & \Hom_A(M,N), \ene where $N$ now denotes the (ungraded) $A$-module obtained by viewing $N \in \Mod(B)$ as an $A$-module via the usual restriction of scalars.  The isomorphism \eqref{extrest3} is the composition of the isomorphism \eqref{extrestusual} and the isomorphism \bne{taut} \Hom_A(M,N) & = & \Hom_{G,A}(M,N_A) \ene taking $f : M \to N$ to the map $M \to N_A$ given in degree $g$ by $f|M_g : M_g \to (N_A)_g = N$.  The inverse of the latter isomorphism takes a $(G,A)$-module morphism $k : M \to N_A$ to the $A$-module morphism $M \to N$ defined by $k(\sum_g m_g) := \sum_g k(m_g)$.  In other words, the composition of $k \to N_A$ and the natural map of abelian groups $\sum_g \Id : N_A \to N$ (this latter map of abelian groups is in fact a morphism of $A$-modules).

\subsection{The case of group algebras} \label{section:groupalgebras} Let $A$ be a ring, $G$ an abelian group, $A[G]$ the group algebra over $A$ on $G$, viewed as a ring graded by $G$ in the obvious manner.  We have a morphism of graded rings $(0,A) \to (G,A[G])$ taking $A$ onto $A[G]_0 = A$ on the level of rings.

\begin{prop} \label{prop:ModGAG} For a ring $A$ and an abelain group $G$, extension of scalars \be \slot \otimes_A A[G] : \Mod(A) & \to & \Mod(G,A) \ee is an equivalence of abelian categories with inverse given by taking a $(G,A[G])$-module $M$ to $M_0$. \end{prop}

\begin{proof}  For an $A$-module $M$, we clearly have a natural isomorphism $(M \otimes_A A[G])_0 = M$ of $A$-modules.  For a $(G,A[G])$-module $N$, we have a natural isomorphism of $(G,A[G])$-modules \be N_0 \otimes_A A[G] & \to & N \\ n_0 \otimes a[g] & \mapsto & a[g]n_0 \ee with inverse given by the map $N \to N_0 \otimes_A A[G]$ taking a homogeneous element $n \in N_g$ to $[g^{-1}]n \otimes [g] \in (N_0 \otimes_A A[G])_g$.  \end{proof}

\begin{cor} \label{cor:gradedflatnessoverAG} Let $A$ be a ring, $G$ an abelian group, $A[G]$ the group algebra graded by $G$ as usual, $(G,A[G]) \to (H,B)$ a map of graded rings.  Then an $(H,B)$-module $N$ is graded flat over $(G,A)$ iff $N$ is flat as an $A$-module in the usual ungraded sense. \end{cor}

\begin{proof} Via the equivalence in Proposition~\ref{prop:ModGAG}, the graded extension of scalars functor whose exactness defines ``$N$ is graded flat over $(G,A)$" is identified, after composing with the faithfully exact forgetful functor $\Mod(H,B) \to \Mod(B)$, with the usual extension of scalars \be \slot \otimes_A N : \Mod(A) & \to & \Mod(B). \ee \end{proof}

\begin{cor} \label{cor:ModAG} For any ring $A$ and any map $\gamma: G \to H$ of abelian groups, extension of scalars \be \Mod(G,A[G]) & \to & \Mod(H,A[H]) \ee is an equivalence of categories. \end{cor}

\begin{proof} The functor in question sits in a commutative triangle with the equivalences from $\Mod(A)$ of Proposition~\ref{prop:ModGAG} and is hence an equivalence by ``two-out-of-three". \end{proof}

\subsection{Homogeneous ideals} \label{section:homogeneousideals} Let $A$ be a $G$-graded ring.  A \emph{homogeneous ideal} in $A$ is a $(G,A)$-submodule $I \subseteq A$ of $A$.  A homogeneous ideal is, in particular, an ideal of $A$ in the usual sense, and an ideal of $A$ is homogeneous iff it is generated by its homogeneous elements.  An ideal $I \subseteq A$ is homogeneous iff $a = \sum_g a_g \in I$ implies each homogeneous component $a_g$ of $a$ is also in $I$.  

If $I \subseteq A$ is a homogeneous ideal, then the quotient $A/I = \oplus_g A_g / I_g$ is a $G$-graded ring and the natural map $A \to A/I$ is a map of $G$-graded rings.

\begin{defn} \label{defn:semiprime} A homogeneous ideal $I$ in a graded ring $A$ is called \emph{prime} (resp. \emph{semiprime}) iff $ab \in I$ implies at least one of $a,b$ is in $I$ for all $a,b \in A$ (resp.\ whenever at least one of $a,b \in A$ is homogeneous).  We often say that an ideal $I$ of a graded ring is ``semiprime" to mean that it is homogeneous and semiprime. \end{defn}

Whenever we want to check that a homogeneous ideal is semiprime, we will use the following criterion:

\begin{lem} \label{lem:semiprime} A homogeneous ideal $I$ is semiprime iff $ab \in I$ implies at least one of $a,b$ is in $I$ for all homogeneous $a,b \in A$. \end{lem}

\begin{proof} Suppose $ab \in I$ and, say, $a$ is homogeneous.  We must prove that $a$ or $b$ is in $I$.  If $a \in I$ we're done so assume now that $a \notin I$.  Write $b = \sum_g b_g$ as a (finite) sum of homogeneous elements.  Then since $a$ is homogeneous, $ab = \sum_g ab_g$ is the decomposition of $ab$ into homogeneous components, so, since $I$ is homogeneous, each $ab_g$ is $I$, and since $a \notin I$ the hypothesis on $I$ ensures that $b_g \in I$ for each $g$, hence $b \in I$. \end{proof}

Evidently a homogeneous ideal is prime iff it is prime in the usual sense and a homogeneous prime ideal is also clearly semiprime.  A homogeneous ideal $I$ in a $G$-graded ring $A$ is prime (resp.\ semiprime) iff the $G$-graded quotient ring $A/I$ has no (nontrivial) zero divisors (resp.\ no \emph{homogeneous} zero divisors).

\begin{rem} \label{rem:semiprimeideals} In my opinion, ``semiprime ideal" is really the ``correct" analog of a ``prime ideal" in the graded setting.  I would have preferred to use ``prime" for what I ended up calling ``semiprime" and something like ``very prime" for what I ended up calling ``prime," but the terminology used here is so firmly entrenched that my preferred terminology would probably cause confusion.  The point is that almost all of the usual constructions for ungraded modules will go through in the ungraded setting if one replaces prime ideals with semiprime ideals.  The same constructions will not go through replacing prime ideals with homogeneous prime ideals and in general homogeneous prime ideals are not very useful except when they happen to coincide with semiprime ideals. \end{rem}

\begin{defn} \label{defn:totalordering} Let $G$ be an abelian group.  A \emph{total ordering} of $G$ is a total ordering of the set $G$ which is compatible with the group structure in the sense that for all $a,b,c,d \in G$ with $a \leq b$ and $c \leq d$ we have $a+c \leq b+d$ with equality iff $a=b$ and $c=d$. \end{defn}

The next two lemmas are basically taken from \cite[II.11]{Bourbakialgebra}.

\begin{lem} \label{lem:totalordering} An abelian group $G$ admits a total ordering iff it is torsion-free. \end{lem}

\begin{proof}  Suppose $G$ is totally ordered and $g \in G$ is torsion; let us show that $g=0$.  After possible replacing $g$ with $-g$, we can assume $g \geq 0$.  If $g=0$ we're done so we can assume $g > 0$.  Since $g$ is torsion there is a positive integer $n$ with $ng=0$.  But then the properties of a total ordering imply by induction on $n$ that $ng > 0$, which contradicts $ng=0$.  Now suppose $G$ is torsion-free and let us construct a total order on $G$.  Since $G$ is a torsion free it is a subgroup of a $\QQ$ vector space $V$, so it suffices to construct a total order on $V$.  This can be easily done by choosing a basis $B$ for $V$ and a total ordering of $B$, then ordering $V \cong \oplus_B \QQ$ via the usual total ordering of $\QQ$ and the lexicographic ordering with respect to the ordering of $B$. \end{proof}

\begin{lem} \label{lem:semiprimetorsionfreegrading} Let $A$ be a $G$-graded ring, $I \subseteq A$ a homogeneous ideal.  Suppose $G$ is torsion-free.  Then $I$ is prime iff it is semiprime. \end{lem}

\begin{proof} Suppose $I$ is semiprime and $ab \in I$ for some $a,b \in I$; we claim that $a$ or $b$ is in $I$.  Let $a = \sum_g a_g$, $b = \sum_g b_g$ be the decompositions into homogeneous elements (all but finitely many $a_g$ and $b_g$ are zero).  We prove the claim by induction on the ordered pair $(M,N)$ of natural numbers, where $M$ (resp. $N$) is the number of nonzero $a_g$ (resp.\ $b_g$).  If $M$ or $N$ is zero, then $a$ or $b$ is zero hence in $I$, so we can assume now that $M,N > 0$ and that the claim is know for smaller ordered pairs (in the lexicographic ordering, say).  By Lemma~\ref{lem:totalordering} we can pick a total ordering on $G$.  Let $g \in G$ (resp.\ $h \in G$) be the maximum element with respect to this ordering so that $a_g$ (resp.\ $b_h$) is non-zero. By this maximality, the homogeneous degree $g+h$ component of $ab$ is $a_g b_h$, so, since $I$ is homogeneous, this component is also in $I$, so since $I$ is semiprime, either $a_g$ or $b_h$ is in $I$.  Say $a_g \in I$ (the argument when $b_h \in I$ is similar).  Then $(a-a_g)b = ab-a_gb$ is also in $I$ so by induction we know that either $b \in I$ (in which case we're done) or $a-a_g \in I$ (in which case we're also done because $a_g \in I$ hence $a \in I$).  \end{proof}

\begin{rem} \label{rem:homogeneousprimes1} The assumption in the above lemma that $G$ is torsion-free cannot be removed.  For example, let $A = \CC[\ZZ/2\ZZ] \cong \CC[x]/x^2-1$ with the obvious $\ZZ / 2 \ZZ$ grading.  Consider the homogeneous ideal $I = (0)$.  This ideal is not prime because $A$ has zero divisors---for example \be (1[0]+1[1])(1[0]-1[1]) & = & 0, \ee but if $a,b \in A$ are \emph{homogeneous} and $ab=0$, then it is easy to see that either $a$ or $b$ is zero, so $(0)$ is semiprime. \end{rem}

\begin{prop} \label{prop:homogeneousideals} Let $P$ be an integral monoid, $A$ a ring.  We have functions \bne{idealstohomogeneousideals} \{ {\rm \, ideals \; of \; } P \, \} & \to & \{ {\rm \; homogeneous \; ideals \; of \; } A[P] \, \} \\ \nonumber I & \mapsto & A[I] \\ \label{homogeneousidealstoideals} \{ {\rm \, homogeneous \; ideals \; of \, } A[P] \, \} & \to & \{ {\rm \, ideals \; of \; } P \, \} \\ \nonumber J & \mapsto & J \cap P := \{ p \in P : [p] \in J \}.  \ene  \begin{enumerate} \item \label{AIcapPisI} For any ideal $I \subseteq P$ we have $A[I] \cap P = I$.  \item \label{homogeneousideals2} For any homogeneous ideal $J \subseteq A[P]$ we have $A[J \cap P] \subseteq J$ with equality if $A$ is a field.  \item \label{homogeneousideals3} If $A$ is a field then \eqref{idealstohomogeneousideals} and \eqref{homogeneousidealstoideals} are inverse bijections. \item \label{semiprimeimpliesprime} For any ideal $I \subseteq P$, if the homogeneous ideal $A[I]$ is semiprime, then $I$ is prime.  \item \label{primeimpliessemiprime} The converse of \eqref{semiprimeimpliesprime} holds when $A$ is an integral domain.  \item \label{fieldbijection} If $A=k$ is a field, then $I \mapsto k[I]$ establishes a bijection between prime ideals of $P$ and semiprime ideals of $k[P]$. \item \label{lastbijection} If $P^{\rm gp}$ is torsion-free then every semiprime ideal of $A[P]$ is prime and hence if $k$ is a field, the bijection in \eqref{fieldbijection} can be viewed as a bijection between prime ideals of $P$ and homogeneous prime ideals of $k[P]$. \end{enumerate}  \end{prop}

\begin{proof} \eqref{AIcapPisI} For an ideal $I \subseteq P$, the ideal $A[I] = \oplus_{i \in I} A[i]$ is manifestly homogeneous and it is clear that $A[I] \cap P = I$.  For \eqref{homogeneousideals2}, suppose $J \subseteq A[P]$ is a homogeneous ideal.  Clearly $A[J \cap P] \subseteq J$ because every element of $A[J \cap P]$ is a finite sum of multiples of elements of $J$ and $J$ is an ideal (this is all true even without knowing $J$ is homogeneous).  To see that this last containment is an equality when $A$ is a field, the point is that when $J$ is homogeneous, we can check that this inclusion is an equality by checking that any \emph{homogeneous} element of $J$ is in $A[J \cap P]$.  A homogeneous element of $J$ is of the form $a[p]$ for some $p \in P$ and $a \in A$ (here it is important that $P$ be integral---see the Warning in \S\ref{section:monoidstogradedrings}).  If $a=0$, then certainly $0=a[p] \in A[J \cap P]$, so we can assume $a \in A^*$.  But then $a^{-1} a[p] = [p]$ is also in $J$ because $J$ is an ideal, so $p \in J \cap P$ and hence $a[p] \in A[J \cap P]$ as desired.  \eqref{homogeneousideals3} is obvious from \eqref{AIcapPisI} and \eqref{homogeneousideals2}.

For \eqref{semiprimeimpliesprime}, suppose $I \subseteq P$ is an ideal of $P$ such that the homogeneous ideal $A[I]$ of $A[P]$ is semiprime.  To see that $I$ is prime, suppose $p+q \in I$ for some $p,q \in P$.  Then $[p][q] = [p+q] \in A[I]$, so, since $A[I]$ is semiprime and $[p],[q]$ are homogeneous, either $[p]$ or $[q]$ is in $A[I]$ and hence either $p$ or $q$ is in $I$.  

For \eqref{primeimpliessemiprime}, suppose $I \subseteq P$ is a prime ideal and $A$ is an integral domain.  To show that $A[I] \subseteq A[P]$ is semiprime it suffices, by Lemma~\ref{lem:semiprime}, to show that either $a[p]$ or $b[q]$ is in $A[I]$ under the assumption that $a[p] \cdot b[q] = ab[p+q]$ is in $I$.  This is trivial if $a$ or $b$ is zero, so we can assume $a,b \neq 0$.  Since $A$ is a domain, $ab$ is nonzero, so it must be that $p+q \in I$, hence by primeness of $I$ either $p \in I$ (hence $a[p] \in A[I]$) or $q \in I$ (hence $b[q] \in A[I]$). 

For \eqref{fieldbijection} we use \eqref{semiprimeimpliesprime} and \eqref{primeimpliessemiprime} to see that the bijection in \eqref{homogeneousideals3} restricts to a bijection as desired and \eqref{lastbijection} then follows from Lemma~\ref{lem:semiprimetorsionfreegrading}.\end{proof}

\begin{rem} Even if $k=\CC$, $P$ is a finitely generated abelian group, and $I \subseteq P$ is prime, the semiprime ideal $k[I]$ of $k[P]$ need not be prime.  For example, if $P = \ZZ / 2 \ZZ$ and $I = \emptyset$ is the unique prime ideal of $P$, then the homogeneous ideal $\CC[I] = (0)$ in the $\ZZ/2\ZZ$-graded ring $\CC[P]$ is not prime (Remark~\ref{rem:homogeneousprimes1}).  

If $A$ isn't an integral domain, then $A[I]$ is \emph{never} a prime ideal of $A[P]$ when $I \subseteq P$ is a prime ideal: If $ab=0$ in $A$ for $a,b \neq 0$, then $a[0]$ and $b[0]$ are not in $A[I]$ because $0 \in P$ is not in the prime ideal $I$, but $a[0]b[0]=0$ is certainly in $A[I]$. \end{rem}

\noindent {\bf Warning:}  Although we have defined a map $I \mapsto A[I]$ from ideals of $P$ to homogeneous ideals of $A[I]$, there is not in general any reasonable functor from $\Mod(P)$ to $\Mod(P^{\rm gp},\ZZ[P])$.

\subsection{Filtrations}  \label{section:filtrations}  Here we briefly sketch the analog in the graded setting of the theory of associated primes as in \cite[Chapter 3]{Mat}.  Fix a graded ring $(G,A)$ and a $(G,A)$-module $M$.  If $x \in M$ is homogeneous, then \be \Ann x & := & \{ a \in A : ax = 0 \} \ee is a homogeneous ideal of $A$ because if $a = \sum_{g \in G} a_g$ annihilates $x$ then homogeneity of $x$ ensures that each $a_g x$ is zero, so each $a_g \in \Ann x$.

\begin{defn}    A semiprime ideal $I \subseteq A$ is called an \emph{associated semiprime} of $M$ iff the following equivalent conditions hold: \begin{enumerate} \item There exists a homogeneous element $x \in M$ such that $I = \Ann x$. \item $M$ contains a $(G,A)$-submodule isomorphic to a shift $(A/I) \{ g \}$ for some $g \in G$. \end{enumerate} \end{defn}

\begin{lem} \label{lem:associatedsemiprime} Any maximal element $I = \Ann x$ of the set of homogeneous ideals $$ \F  = \{ \Ann x : 0 \neq x \in M \; {\rm homogeneous} \}$$ is an associated semiprime of $M$. \end{lem} 

\begin{proof} C.f.\ \cite[7.B]{Mat}.  Suppose $I = \Ann x$ is maximal and $ab \in I$ (i.e.\ $abx=0$ for homogeneous $a,b \in A$.  We must show that $a$ or $b$ is in $I$---i.e.\ $ax = 0$ or $bx=0$.  If $bx=0$ we're done so suppose $bx \neq 0$.  Note that $a$ is certainly in $\Ann (bx)$ and $\Ann (bx) \in \F$ because $bx$ is homogeneous and nonzero, but the obvious containment $I \subseteq \Ann (bx)$ must be equality by maximality of $I$, so $a \in I$. \end{proof}

\begin{lem} \label{lem:filtrations} Suppose $(G,A)$ is a graded ring with $A$ noetherian and $M \in \Mod(G,A)$ is finitely generated.  Then there is a finite filtration $$0 = M_0 \subseteq M_1 \subseteq \cdots \subseteq M_n = M$$ of $M$ by $(G,A)$-submodules such that each successive quotient $M_i/M_{i-1}$ $(i=1,\dots,n)$ is isomorphic to $(A/\P_i) \{ g_i \}$ for some semiprime ideal $\P_i \subseteq A$ and some $g_i \in G$. \end{lem}

\begin{proof} C.f.\ \cite[7.E]{Mat}.  If $M=0$ we just take $M_0=M_n=0$.  If $M \neq 0$, then by Lemma~\ref{lem:associatedsemiprime} we can find a submodule $M_1 \subseteq M$ isomorphic to $(A/\P_1) \{ g_1 \}$ for some semiprime ideal $\P_1 \subseteq A$, $g_i \in G$.  If $M_1 \neq M$, then we can repeat the same procedure to $M/M_1$ to find $M_2$ and so on.  The process stops at some point because $M$ satisfies the ascending chain condition since an ascending chain of $(G,A)$-submodules is in particular an ascending chain of $A$-submodules and $M$ has the ACC since $A$ is noetherian and $M$ is finitely generated. \end{proof}

\begin{prop} \label{prop:filtrations} Let $M$ be a $(G,A)$-module, $M' \subseteq M$ a $(G,A)$-submodule.  Then there is a filtered partially ordered set $\Lambda$ and a functor $(M_{\lambda}) : \Lambda \to \Mod(G,A)$ such that each $M_\lambda$ is a $(G,A)$-submodule of $M$ containing $M'$ and finitely generated over $M'$.  In particular, every homogeneous ideal of $A$ is a filtered direct limit of finitely generated homogeneous ideals. \end{prop}

\begin{proof} Take $\Lambda$ to be the set of all finite sets $\lambda$ of homogeneous elements of $M$ and let $M_{\lambda}$ be the $(G,A)$-submodule of $M$ generated by $M'$ and the elements of $\lambda$. \end{proof}

\subsection{Graded flatness} \label{section:gradedflatness} Let $A$ be a $G$-graded ring.  For fixed $M \in \Mod(G,A)$, the functor \bne{gradedtensor} M \otimes_A \slot : \Mod(G,A) & \to & \Mod(G,A) \ene is right exact (preserves cokernels):  This follows formally from the fact that the ungraded tensor product preserves direct limits together with the fact that the forgetful functor $\Mod(G,A) \to \Mod(A)$ is faithfully exact and commutes with tensor products.  

\begin{defn} \label{defn:gradedflat1} A $(G,A)$-module $M$ is called \emph{graded flat} (we avoid calling this ``flat" to avoid confusion with the notion of flatness of the underlying $A$-module) iff the functor \eqref{gradedtensor} is exact (i.e.\ left exact). \end{defn}
 
More generally, consider a $\GrAn$-morphism $(\gamma,f) : (G,A) \to (H,B)$ and an $(H,B)$-module $N$.  Recall the extension of scalars functor \eqref{slototimesN} from \S\ref{section:gradedtensorproduct}: \bne{slototimesNagain} \slot \otimes_A N : \Mod(G,A) & \to & \Mod(H,B) \ene This functor preserves direct limits for the same formal reasons that \eqref{gradedtensor} preserves direct limits.  

\begin{defn} \label{defn:gradedflat2} An $(H,B)$-module $N$ is called \emph{graded flat over} $(G,A)$ iff the functor \eqref{slototimesNagain} above is exact (i.e.\ left exact). \end{defn}

\begin{rem} \label{rem:gradedflatness} When $(\gamma,f) = (\Id,\Id)$, the notion of ``graded flat over $(G,A)$" specializes to the notion ``graded flat" in Definition~\ref{defn:gradedflat1}.  If we weren't dealing with graded modules, then we would just formulate the ``flat over $A$" in terms of ``flat" by noting that $N$ is flat over $A$ iff $N$ is flat when regarded as an $A$-module by restriction of scalars.  For graded modules, this does not make sense:  One can certainly ask whether the restriction of scalars $N_A$ is a graded flat $(G,A)$-module in the sense of Definition~\ref{defn:gradedflat1}, but it is not clear that this has any relationship with $N$ being graded flat over $(G,A)$ in the sense above (see the Warnings in \S\ref{section:gradedtensorproduct} and \S\ref{section:importantspecialcase}).  We will not discuss the former flatness notion in the present paper except perhaps when it happens to coincide with the above notion of ``graded flat over $(G,A)$".  It is important to understand that graded flatness is really a notion for maps of graded rings with a module on the codomain and not just a notion for modules on a fixed graded ring. \end{rem}

\begin{defn} \label{defn:Tor} Since $\Mod(G,A)$ has enough projectives (Proposition~\ref{prop:enoughprojectives}), we can form the left-derived functors \be \Tor_i^{G,A}(\slot,N) := L^i ( \slot \otimes_A N) : \Mod(G,A) & \to & \Mod(H,B) \ee of the right exact functors \eqref{slototimesNagain}. \end{defn}

\begin{prop} \label{prop:Tor}  For any $M \in \Mod(G,A)$, the $B$-module underlying the $(H,B)$-module $\Tor_i^{G,A}(M,N)$ coincides with $\Tor_i^A(M,N)$. \end{prop}

\begin{proof} The Grothendieck Spectral Sequence for the composition of \eqref{slototimesNagain} and the forgetful functor $\Mod(H,B) \to \Mod(B)$ degenerates to yield the desired isomorphism because this forgetful functor is exact.  More concretely: you can compute $\Tor_i^{G,A}(M,N)$ by taking a resolution of $N$ by free $(G,A)$-modules, tensoring it over $A$ with $N$ and taking homology.  But those free $(G,A)$-modules are in particular free $A$-modules (and the tensor product is the usual one on underlying modules), so you are just computing $\Tor_i^A(M,N)$ on the level of underlying $B$-modules. \end{proof}

\begin{lem} \label{lem:flatimpliesgradedflat} Let $(G,A) \to (H,B)$ be a $\GrAn$-morphism, $N$ an $(H,B)$-module.  Suppose the underlying $B$-module $N$ is flat over $A$ in the usual ungraded sense.  Then $N$ is graded flat over $(G,A)$. \end{lem}

\begin{proof} This is immediate from the previous proposition or by using the commutative diagram of functors $$ \xym@C+20pt{ \Mod(G,A) \ar[d]_{\rm forget} \ar[r]^-{\slot \otimes_A N} & \Mod(H,B) \ar[d]^{\rm forget} \\ \Mod(A) \ar[r]^-{\slot \otimes_A N} & \Mod(B) } $$ and the faithful exactness of ``forget". \end{proof}

We will be particularly interested in graded flatness in the setting discussed in \S\ref{section:importantspecialcase}.  Recall that we considered a $G$-graded ring $A$ and an arbitrary map of (ungraded) rings $A \to B$, which is the same thing as a $\GrAn$-morphism $(G,A) \to (0,B)$.  Using the natural isomorphism of categories $\Mod(B) = \Mod(0,B)$, we saw that a $B$-module $N$ determines a (right exact) map of abelian categors \bne{slototimesNN} \slot \otimes_A N : \Mod(G,A) & \to & \Mod(B). \ene 

\begin{defn} \label{defn:gradedflat3} A $B$-module $N$ is called \emph{graded flat over} $(G,A)$ iff \eqref{slototimesNN} is exact.  In particular, when $A=B$, an (ungraded) $A$-module $N \in \Mod(A)$ is called \emph{graded flat over} $(G,A)$ (or just \emph{graded flat} if the $G$-grading on $A$ is clear from context) iff \be \slot \otimes_A N : \Mod(G,A) & \to & \Mod(A) \ee is exact. \end{defn}

\begin{rem} There is a slight potential for confusing the notion of ``graded flat" defined parenthetically above with the notion of ``graded flat" in Definition~\ref{defn:gradedflat1}.  In practice there should never be any ambiguity because the former notion is a notion for \emph{graded modules} with the latter is a notion for \emph{ungraded modules}.  In fact, we really only discuss the notion of ``graded flat" in Definition~\ref{defn:gradedflat1} implicitly through the fact that it is a specialization of the notion of ``graded flat over $(G,A)$" in Definition~\ref{defn:gradedflat2}. \end{rem}

\begin{example} \label{example:gradedflatnesswhenGiszero} If $G=0$, then, suppressing the isomorphism $\Mod(G,A) = \Mod(A)$, the functor \eqref{slototimesNN} is just the usual extension of scalars \be \slot \otimes_A N : \Mod(A) & \to & \Mod(B) \ee and hence a $B$-module is graded flat over $(0,A)$ iff it is flat over $A$ in the usual ungraded sense.  For a more general statement, see Corollary~\ref{cor:gradedflatness}. \end{example}

It is in some sense enough to consider the case $A=B$:

\begin{lem} \label{lem:gradedflatnessdefinition} For $(G,A) \in \GrAn$, $A \to B$ a map of rings and $N$ a $B$-module, $N$ is graded flat over $(G,A)$ iff the usual ungraded restriction of scalars $N \in \Mod(A)$ of $N$ is graded flat (over $(G,A)$). \end{lem}

\begin{proof} Use the two commutative diagrams \be  \xym@C+10pt{ \Mod(G,A) \ar[r]^-{\slot \otimes_A N} \ar[d]_{\rm forget} & \Mod(B) \\ \Mod(A) \ar[ru]_{\slot \otimes_A N} } & \quad \quad & \xym@C+10pt{ \Mod(G,A) \ar[r]^-{\slot \otimes_A N} \ar[d]_{\slot \otimes_A N} & \Mod(B) \ar[ld]^{\rm \; \; \; \; restrict \; scalars} \\ \Mod(A) }  \ee and the fact that ``forget" and ``restrict scalars" are exact. \end{proof}

Although it is a very special case of Definition~\ref{defn:gradedflat2}, the notion of graded flatness for a $G$-graded ring $(G,A)$ and an $A$-module in Definition~\ref{defn:gradedflat3} is most important to us as it is closely related to logarithmic flatness.  We must be careful to distinguish the notions ``flat" and ``graded flat".  Clearly a flat $A$-module is a graded-flat $A$-module (Lemma~\ref{lem:flatimpliesgradedflat}), but the converse is not true.

\begin{example} \label{example:gradedflat} Here is a simple example to keep in mind.  Let $k$ be a field.  If we view $k[x]$ as a $\ZZ$-graded $k$-algebra with $|x|=1$, then, as we will see later (Example~\ref{example:gradedflatnessoverkP}), $M$ is graded-flat iff $x$ is $M$-regular.  So, for example, $k[x] / x$ is not graded-flat (or flat) over $k[x]$, while $k[x]/(x-1)$ \emph{is} graded-flat (but not flat) over $k[x]$.  The latter graded flatness is reflected by the geometric fact that the composition of $1 : \Spec k \to \AA^1_k$ and the projection $\AA^1_k \to [\AA^1_k / \GG_m]$ is an open embedding, hence flat. \end{example}

\begin{prop} \label{prop:gradedflatnessandhomogeneousideals} Let $(G,A) \to (H,B)$ be a $\GrAn$-morphism, $N$ a $(H,B)$-module.  Then $N$ is graded flat over $(G,A)$ iff $$ (0 \to I \to A \to A/I \to 0) \otimes_A N $$ is exact in $\Mod(H,B)$ for every finitely generated homogeneous ideal $I \subseteq A$. \end{prop}

\begin{proof} The ``well-known" argument Matsumura refers to \cite[Page 18]{Mat} carries over easily to the graded setting.   Indeed, that argument is carefully written out in the proof of \cite[Theorem~7.7]{Mat2}, and one need only insert the word ``homogeneous" at the obvious points, making use of Proposition~\ref{prop:filtrations} and the commutativity of the diagram of functors \bne{tensorfiltereddirectlimits} & \xym@C+20pt{ \Mod(G,A)^{\Lambda} \ar[d]_{\dirlim} \ar[r]^{\slot \otimes_A N} & \Mod(H,B)^\Lambda \ar[d]^{\dirlim} \\ \Mod(G,A) \ar[r]^-{\slot \otimes_A N} & \Mod(H,B) } \ene for each filtered category $\Lambda$.  \end{proof} 

\begin{prop} \label{prop:gradedflatnessgroupinjective} Consider a $\GrAn$-morphism of the form $(\gamma,\Id) : (S,A) \to (G,A)$ with $\gamma$ injective and an arbitrary $\GrAn$-morphism $(G,A) \to (H,B)$.  Then an $(H,B)$-module $N$ is graded flat over $(G,A)$ iff it is graded flat over $(S,A)$. \end{prop}

\begin{proof}  Since $\gamma$ is injective, an ideal $I \subseteq A$ is homogeneous for the $G$-grading iff it is homogeneous for the $S$-grading, so that $(S,A)$ and $(G,A)$ have the same homogeneous ideals and we may speak unambiguously of a ``homogeneous ideal of $A$."  By Proposition~\ref{prop:gradedflatnessandhomogeneousideals} and the faithful exactness of $\Mod(H,B) \to \Mod(B)$, both graded flatness conditions are equivalent to exactness of the sequence of $B$-modules $$ (0 \to I \to A \to A/I \to 0) \otimes_A N $$ for each homogeneous ideal $I$ of $A$. \end{proof}

\begin{cor} \label{cor:gradedflatness} Let $G$ be a group, $A$ a ring, viewed as a $G$-graded ring supported in degree zero, $A \to B$ a ring map.  Then a $B$-module $N$ is graded flat over $(G,A)$ iff it is flat over $A$ in the usual ungraded sense. \end{cor}

\begin{proof} Apply Proposition~\ref{prop:gradedflatnessgroupinjective} to $(0,A) \to (G,A)$ and note that graded flatness over $(0,A)$ is certainly the same as flatness over $A$ in the usual ungraded sense (Example~\ref{example:gradedflatnesswhenGiszero}).  \end{proof}

\subsection{Graded flatness and base change}  \label{section:gradedflatnessandbasechange} The extent to which graded flatness is ``stable under base change" is somewhat unclear.

\begin{defn} \label{defn:goodrestrictionofscalars} Let $(\gamma,f) : (G,A) \to (H,B)$ be a $\GrAn$-morphism.  A \emph{good restriction of scalars} for $(\gamma,f)$ is a map of abelian categories \bne{goodscalars} R : \Mod(H,B) & \to & \Mod(G,A) \ene agreeing with the usual restriction of scalars on the level of underlying modules---i.e.\ making \bne{goodscalarsproperty} & \xym{ \Mod(H,B) \ar[d] \ar[r]^R & \Mod(G,A) \ar[d] \\ \Mod(B) \ar[r] & \Mod(A) } \ene commute when the bottom horizontal arrow is the usual ungraded restriction of scalars and the vertical arrows are the forgetful functors. \end{defn}

Note that a good restriction of scalars is automatically exact by commutativity of \eqref{goodscalarsproperty}, faithful exactness of $\Mod(G,A) \to \Mod(A)$, exactness of the forgetful functor $\Mod(H,B) \to \Mod(B)$, and exactness of the usual ungraded restriction of scalars.

\begin{example} \label{example:gammaiso} If $\gamma$ is an isomorphism, then the usual graded restriction of scalars functor $N \mapsto N_A$ of \eqref{restscalars} is a good restriction of scalars. \end{example}

\begin{example} \label{example:gammasplitting} Suppose $\gamma : G \to H$ has a section $s : H \to G$ such that $\oplus_{g \in G \setminus s(H)} A_g$ is contained in the kernel of $f : A \to B$.  Then one can define a good restriction of scalars by taking $M$ to $R(M)$, where $R(M)$ is $M$ regarded as an $A$-module equipped with the decomposition where $M_g := M_h$ if $g = s(h)$ and $M_g := 0$ otherwise. \end{example}

\begin{example} View $\ZZ[x,y]$ as a $\ZZ^2$ graded ring with $|x|=(1,0)$, $|y|=(0,1)$ and $\ZZ[x]$ as a $\ZZ$-graded ring with $|x|=1$.  The map $f : \ZZ[x,y] \to \ZZ[x]$ taking $x$ to $x$ and killing $y$ becomes a map of graded rings $(\pi_1,f)$ where $\pi_1 : \ZZ^2 \to \ZZ$ is projection on the first factor.  The section $s : \ZZ \to \ZZ^2$ taking $n$ to $(n,0)$ has the property that $\oplus_{(a,b) \in \ZZ^2 \setminus s(\ZZ)} \ZZ x^a y^b$ is in the kernel of $f$.  The map $(\pi_1,f)$ admits a good restriction of scalars by the previous example. \end{example}

Now we see how good restriction of scalars is related to stability of graded flatness under base change.  First note that the category $\GrAn$ has pushouts.  A pushout diagram takes the form \bne{pushoutinGrAn} & \xym{ (G,A) \ar[d] \ar[r] & (H_1,B_1) \ar[d] \\ (H_2,B_2) \ar[r] & (K,C)  } \ene where $(K,C) = (H_1 \oplus_G H_2,B_1 \otimes_A B_2)$ and $b_1 \otimes b_2 \in C$ has grading $[|b_1|,|b_2|] \in K$ for homogeneous elements $b_1 \in B_1$, $b_2 \in B_2$.  Suppose an $(H_2,B_2)$-module $N$ is graded flat over $(G,A)$ and let us try to prove that its extension of scalars $C \otimes_{B_2} N$ to a $(K,C)$-module is graded flat over $(H_1,B_1)$.  We want to show that the functor \be \slot \otimes_{B_1} (C \otimes_{B_2} N) : \Mod(H_1,B_1) & \to & \Mod(K,C) \ee is exact.  We can check this after composing with the faithfully exact forgetful functor $\Mod(C,K) \to \Mod(C)$.  Since the module underlying a graded tensor product is just the tensor product of the underlying modules, the resulting composition functor \bne{pqr} \slot \otimes_{B_1} (C \otimes_{B_2} N) : \Mod(H_1,B_1) & \to & \Mod(C) \ene is given by \be M & \mapsto & M \otimes_{B_1} (C \otimes_{B_2} N) \\ & = & M \otimes_{B_1} ( \otimes_{B_1} \otimes_A B_2 ) \otimes_{B_2} N \\ & = & M \otimes_A N. \ee  Given a good restriction of scalars $R$ for $(G,A) \to (H_1,B_1)$, we obtain a commutative diagram \bne{tr} & \xym@C+20pt{ \Mod(H_1,B_1) \ar[r]^-{\slot \otimes_A N} \ar[dd]_R & \Mod(C) \ar[d]^{\rm restrict \; scalars} \\ & \Mod(B_2)  \\ \Mod(G,A) \ar[r]^-{ \slot \otimes_A N} & \Mod(H_2,B_2) \ar[u]_{\rm forget} } \ene and we conclude that \eqref{pqr} is exact (and hence that $C \otimes_{B_2} N$ is graded flat over $(H_1,B_1)$) using faithful exactness of ``restrict scalars," exactness of $R$, exactness of $\slot \otimes_A N$ (since  $N$ is graded flat over $(G,A)$), and exactness of ``forget."  We have proved:

\begin{prop} \label{prop:gradedflatnessandbasechange} Graded flatness is stable under pushout along $\GrAn$-morphisms admitting a good restriction of scalars.  In particular it is stable under pushout along $\GrAn$-morphisms $(\gamma,f)$ where $\gamma$ is an isomorphism. \end{prop}

\subsection{Graded flatness criteria} \label{section:gradedflatnesscriteria}

\begin{lem} \label{lem:gradedflatnessoverkP} Let $k$ be a field, $P$ an integral monoid, $k[P]$ the associated $P^{\rm gp}$-graded ring as in \S\ref{section:monoidstogradedrings}, $N \in \Mod(k[P])$ a module over $k[P]$ in the usual ungraded sense.  Then $N$ is graded flat iff \bne{torvanishingyy} \Tor_1^{k[P]}(N,k[P]/k[I]) & = & 0 \ene for each finitely generated ideal $I \subseteq P$.  If $P$ is fine, then it is enough to check \eqref{torvanishingyy} for each prime ideal $I \subseteq P$. \end{lem}

\begin{rem} A finitely generated monoid has only finitely many prime ideals. \end{rem}

\begin{proof} These $\Tor$-vanishing conditions are certainly necessary for graded flatness because $k[I]$ is a homogeneous ideal in $k[P]$ by Proposition~\ref{prop:homogeneousideals} for each ideal $I \subseteq P$, hence a $(P^{\rm gp},k[P])$-module.  For sufficiency, first note that Proposition~\ref{prop:gradedflatnessandhomogeneousideals} reduces us to show that $\Tor_1^{k[P]}(N,k[P]/J)=0$ for every homogeneous ideal $J \subseteq k[P]$.  By Proposition~\ref{prop:homogeneousideals}, $J = k[I]$ for some ideal $I \subseteq P$.  By writing $I$ as a filtered direct limit of finitely generated ideals $I_i$ and commuting $\Tor_1^{k[P]}(N,k[P]/k[I])$ with this filtered direct limit we obtain the first criterion.

Now suppose $P$ is fine.  Then $k[P]$ is noetherian.  To show that $N$ is graded flat it is enough (by Proposition~\ref{prop:gradedflatnessandhomogeneousideals}, say) to check that $\Tor_1^{k[P]}(N,M) = 0$ for every \emph{finitely generated} graded $k[P]$-module $M$.  By taking a filtration $M_\bullet$ of $M$ as in Lemma~\ref{lem:filtrations} and working our way up the filtration using the long exact sequence of $\Tor$'s associated to the short exact sequences $$0 \to M_{i-1} \to M_{i} \to M_{i} / M_{i-1} \to 0,$$ we reduce to proving $\Tor_1^{k[P]}(N,k[P]/\P) = 0$ for every semiprime ideal $\P \subseteq k[P]$.  By Proposition~\ref{prop:homogeneousideals}, each such $\P$ is of the form $k[I]$ for a prime ideal $I \subseteq P$. \end{proof}

\begin{example} \label{example:gradedflatnessoverkP} For example, when $P=\NN$, the only (non-empty) prime ideal of $\NN$ is $(1)$, so Lemma~\ref{lem:gradedflatnessoverkP} says that a $k[\NN] = k[x]$-module $N$ is graded flat iff \be \Tor_1^{k[x]}(N,k[x]/xk[x]) & = & 0.\ee  This is equivalent to saying that $\cdot x : N \to N$ is injective because $x \in k[x]$ is a regular element. \end{example}

\begin{lem} \label{lem:gradedflatness} Let $(\gamma,f) : (G,A) \to (H,B)$ be a $\GrAn$ morphism such that $B \in \Mod(H,B)$ is graded flat over $(G,A)$.  Let $\{ B_e : e \in E \}$ be a set of $H$-graded rings under $(H,B)$.  Suppose that every finitely generated graded $(H,B)$-module admits a finite filtration all of whose successive quotients are obtained by extension of scalars from a $(G,A)$-module or restriction of scalars from an $(H,B_e)$ module for some $e \in E$.  Then a $B$-module $N$ is graded flat (over $(H,B)$) iff $N$ is graded flat over $(G,A)$ and for every $e \in E$, the $B_e$-module $N_e := N \otimes_B B_e$ is graded flat (over $(H,B_e)$) and $\Tor_1^B(N,B_e)=0$. \end{lem}

\begin{proof}  The conditions are necessary for graded flatness (with only the first hypothesis) because graded flatness is stable under base change along $(H,B) \to (H,B_e)$ (Proposition~\ref{prop:gradedflatnessandbasechange}), so we now assume the hypotheses on $N$ and prove that $N$ is graded flat over $(H,B)$.  It suffices to show that \bne{vanishingTT} \Tor_1^{H,B}(M,N) & = & 0 \ene for all finitely generated $(H,B)$-modules $M$ (by Proposition~\ref{prop:gradedflatnessandhomogeneousideals}, say).  By using the hypothesized filtration of $M$ and long exact sequences of $\Tor$'s to work our way up a finite filtration of $M$ we reduce to proving the vanishing \eqref{vanishingTT} in the following two cases (where $M$ might not be finitely generated):

\noindent {\bf 1.} $M=T \otimes_A B$ for some $(G,A)$-module $T$.  Since $B$ is graded flat over $(G,A)$ by hypotheses, the Grothendieck Spectral Sequence relating the derived functors for the composition $$ \xym@C+20pt{ \Mod(G,A) \ar[rd]_{\slot \otimes_A N} \ar[r]^-{\slot \otimes_A B} & \Mod(H,B) \ar[d]^{\slot \otimes B N} \\ & \Mod(B) } $$ evaluated on $T \in \Mod(G,A)$ degenerates to yield \be \Tor_1^{H,B}(T \otimes_A B, N) & = & \Tor_1^{G,A}(T,N),\ee but the right side vanishes since we assume $N$ is graded flat over $(G,A)$.  To check the hypotheses necessary to get the above Grothendieck Spectral Sequence, we use Proposition~\ref{prop:tensortakesprojtoproj}.

\noindent {\bf 2.} $M$ is the restriction of scalars of an $(H,B_e)$-module.  By Proposition~\ref{prop:Tor}, \be \Tor_1^{H,B}(M,N) & = & \Tor_1^B(M,N) \ee on the level of underlying $B$-modules, so it suffices to show that $\Tor_1^B(M,N)=0$.  For this we can consider the exact sequence of low order terms from the Grothendieck Spectral Sequence (evaluated at $N \in \Mod(B)$) relating the derived versions of the abelian category morphisms in the diagram $$ \xym@C+20pt{ \Mod(B) \ar[r]^-{\slot \otimes_B B_e} \ar[rd]_-{\slot \otimes_B M} & \Mod(B_e) \ar[d]^{\slot \otimes_{B_e} M} \\ & \Mod(B_e) } $$ (c.f.\ [SGA1 IV.5.2]) to reduce to establishing the vanishings \be \Tor_1^B(N,B_e) & = & 0 \\ \Tor_1^{B_e}(M,N_e) & = & 0. \ee  The first of these vanishes by hypothesis and the second vanishes because it equals $\Tor_1^{H,B_e}(M,N_e)$ (Proposition~\ref{prop:Tor} again), which vanishes by the hypotheses that $N_e$ is graded flat over $(H,B_e)$.  \end{proof}

\begin{defn} \label{defn:spawningset} Let $h : Q \to P$ be a free morphism of monoids with a chosen basis $S \subseteq P$.  A subset $E \subseteq P$ is called a \emph{spawning set} (for $S$) iff $S$ is contained in the submonoid of $P$ generated by $E$ and the units of $Q$. \end{defn}

\begin{example} If $E \subseteq P$ generates $P$ then $E$ is a spawning set for any basis $S$.  \end{example}

\noindent {\bf Setup:}  Let $h : Q \to P$ be a free morphism of monoids with basis $S \subseteq P$ and spawning set $E \subseteq P$.  Assume $P/Q$ is integral (this holds when $P$ is integral).  Let $A$ be a ring, $Q \to A$ a monoid homomorphism.  Let $B := A \otimes_{\ZZ[Q]} \ZZ[P]$, regarded as an $A$-algebra graded by $(P/Q)^{\rm gp}$.  By abuse of notation, we write $[p]$ for both the element of $\ZZ[P]$ and the element $1 \otimes [p]$ of $B$ corresponding to $p \in P$.  For $e \in E$, set $B_e := B / ([e])$.  The quotient map $B \to B_e$ is a morphism of $(P/Q)^{\rm gp}$-graded rings because $([e])$ is manifestly a homogeneous ideal of $B$ (\S\ref{section:homogeneousideals}).

\begin{thm} \label{thm:gradedflatness} In the above setup: \begin{enumerate} \item \label{free} As an $A$-module, $B$ is free with basis $\{ [s] : s \in S \}$.  \item \label{homogeneousprimes} If $\P \subseteq B$ is a semiprime ideal of $B$ not intersecting $\{ [e] : e \in E \}$, then $\P = \p \otimes_A B$ for a prime ideal $\p \subseteq A$. \end{enumerate} If $A$ is noetherian and $P$ is finitely generated, then a $B$-module $M$ (in the usual sense with no gradings) is graded flat (over $(P^{\rm gp},B)$) iff the following hold: \begin{enumerate} \item $M$ is flat as an $A$-module. \item For each $e \in E$, $\Tor_1^B(M,B_e) = 0$ and $M_e := M \otimes_B B_e$ is a graded flat $B_e$-module.\end{enumerate} \end{thm} 

\begin{proof} For \eqref{free}, note that $Q$ free on $S$ implies $\ZZ[P]$ is a free $\ZZ[Q]$-module with basis $$\{ [s] : s \in S \}$$ (Theorem~\ref{thm:freetofree}), hence the pushout $B$ is also free on $\{ [s] : s \in S \}$.  The image of $a \in A$ in $B$ will be denoted $a[0]$.  To clarify: the grading of $[s] \in B$ is the image of $s$ under the natural bijection $S \to P/Q$ followed by the natural map $(P/Q) \into (P/Q)^{\rm gp}$.  We need to know $P/Q$ is integral to know this last map is injective, otherwise $B$ might have homogeneous elements of the form $a[s]+a'[t]$, which would cause difficulties in our argument.  Luckily we assume $P/Q$ is integral so $B = \oplus_s A[s]$ \emph{is} the decomposition of $B$ into its homogeneous pieces.

For \eqref{homogeneousprimes}, suppose $\P$ is such a homogeneous prime ideal.  Let \be \p & := & \{ a \in A : a[0] \in \P \}.\ee  Clearly $\p$ is a prime ideal of $A$.  Since $B = \oplus_s A[s]$, $\p \otimes_A B = \oplus_s \p[s]$ (as a graded $A$-module).  We certainly have $\p \otimes_A B \subseteq \P$ because $\P$ is an ideal so it is closed under multiplication by each element $[s]$, and $a[0][s] = a[s]$.  The issue is to prove that this containment is an equality and the key point is that $\P$ is homogeneous, so it suffices to check this for \emph{homogeneous} elements of $\P$, which we are going to do by a kind of induction.  Since $E$ spawns $S$, each $s \in S$ can be written in the form \be s & = & u + \sum_{e \in E} a_e e \ee where each $a_e \in \NN$, all but finitely many $a_e$ are zero, and $u \in Q^*$.  Let $N(s) \in \NN$ be the minimum value of $\sum_e a_e$ in any such expression.  Now suppose we have a strict containment \bne{strictcontainment} (\p \otimes_A B)_s = \p[s] & \subsetneq & \P_s \subseteq A[s] \ene in grading $s$, which we can assume is chosen with $N(s)$ minimal.  It cannot be the case that $N(s)=0$ because $N(s) = 0$ implies $s \in Q^* \subseteq Q$ which implies $s=0$ (because of our convention that a basis contains zero) and the very definition of $\p$ ensures that the containment $\p \otimes_A B \subseteq \P$ cannot be strict in degree $0$.  So we can assume that $N(s)>0$.  Choose an expression \be s & = & u + \sum_{e \in E} a_e e \ee of the form mentioned above with $\sum_e a_e = N(s)$.  Since $N(s)>0$, $a_l > 0$ for some $l \in E$, so the element \be t & := & u+ (a_l-1) l + \sum_{e \neq l} a_e e \ee is in $P$.  Since $S$ is a basis we can write $t=s'+q$ for some $s' \in S$, $q \in Q$.  Adding $l$ to both sides we have $s = s' + l +q$.  But we can also write $s'+l = s'' + r$ for some $s'' \in S$, $r \in Q$, and we then have $s = s'' + q + r$, which implies $s = s''$ and $q+r=0$ (because $S$ is a basis), so $q$ and $r= -q$ are units in $Q$.  We have \be s' & = & r+t \\ & = & (r+u) + (a_l-1)l + \sum_{e \neq l} a_e e \ee hence $N(s') < N(s)$ and we have $s  =  s'+l+q.$  The strict containment \eqref{strictcontainment} means we can find an element $a \in A \setminus \p$ such that $a[s] \in \P$.  In the ring $B$ we have $a[s] = a[s'][l][q]$ and $[l] \notin \P$ by assumption (since $l \in E$), and $[q]$ is certainly not in $\P$ because it is a unit, so we have $a[s'] \in \P$ since $\P$ is semiprime.  But this shows that we have strict containment in degree $s'$ as well, contradicting minimality of $N(s)$. 

For the graded flatness criterion we apply Lemma~\ref{lem:gradedflatness} (in the case $G=0$ so that graded flatness over $(G,A)$ is just usual flatness).  To demonstrate the existence of the filtrations necessary in Lemma~\ref{lem:gradedflatness}, we use Lemma~\ref{lem:filtrations} together with the description of the semiprime ideals in $B$ from the first part of the theorem.  Note that, if a semiprime ideal $\P$ of $B$ contains some $[e]$, then $B/\P$ is obtained by restriction of scalars along $B \to B_e = B/([e])$.  \end{proof}

Notice that there are no finiteness hypotheses on $M$ in the above theorem.

For example, in the case where $Q=0$, Theorem~\ref{thm:gradedflatness} yields:

\begin{cor} Let $A$ be a ring, $P$ an integral monoid, $E \subseteq P$ a subset generating $P$.  Then any semiprime ideal in $B := A[P]$ not containing any of the elements $\{ [e] : e \in E \}$ is of the form $\p \otimes_A B$ for a prime ideal $\p \subseteq A$.  For $e \in E$, set $B_e := B/([e])$.  If $A$ is noetherian and $P$ is fine, then a $B$-module $M$ (in the usual sense with no gradings) is graded flat over $B$ iff $M$ is flat over $A$ and, for each $e \in E$, $\Tor_1^B(M,B_e)=0$ (equivalently, $[e]$ is $M$-regular since $[e] \in A[P]$ is regular) and $M \otimes_B B_e$ is a graded flat $B_e$ module. \end{cor}

In particular, if $P=\NN$ and $E = \{ 1 \}$ we obtain:

\begin{cor} Let $A$ be a ring, $B := A[x]$ the $\ZZ$-graded $A$-algebra with $|x|=1$.  Let $\P \subseteq B$ be a homogeneous prime ideal and let $\p := \{ a \in A : ax^0 \in \P \}$.  If $x \in \P$, \be \P & = & \p \oplus Ax \oplus Ax^2 \oplus \cdots \ee  and otherwise \be \P & = & \p \oplus \p x \oplus \p x^2 \oplus  \cdots \\ & = & \p \otimes_A B.\ee  If $A$ is noetherian, then a $B$-module $M$ is graded flat over $B$ iff $M$ is flat over $A$, $x$ is $M$-regular, and $M/xM$ is a flat $A$-module. \end{cor}

\begin{proof} This follows from the previous corollary once we note that being graded flat over the $\ZZ$-graded ring $A = A[x]/(x)$ is the same thing as being flat over $A$ in the usual sense because $A[x]/(x)$ is a $\ZZ$-graded ring supported in degree zero (Corollary~\ref{cor:gradedflatness}). \end{proof}

\begin{cor} \label{cor:gradedflatnessoverA1} Let $k$ be a field, $k[x]$ the $\ZZ$-graded $k$-algebra with $|x|=1$.  For a $k[x]$-module $M$, the following are equivalent: \begin{enumerate} \item $M$ is graded flat. \item $x$ is $M$-regular. \item $x \in k[x]_{(x)}$ is $M_{(x)}$-regular. \end{enumerate} If $M$ is finitely generated, these conditions are equivalent to: \begin{enumerate} \setcounter{enumi}{2} \item \label{Mlocallyfree2} $M$ is locally free near the origin. \end{enumerate} \end{cor}

\begin{proof}  The equivalence of the first two statements is immediate from the previous corollary.  The equivalence of the second two statements is clear since $M$-regularity of $x$ can be checked after localizing at each prime $\p \in \Spec k[x]$, but it holds trivially at any prime $\p$ other than $(x)$ since $x \in k[x]_{\p}$ is a unit.  Since $x$ is the maximal ideal corresponding to the origin, the equivalence of the last two statements results from standard commutative algebra: \cite[18.B, Lemma~4]{Mat} implies $M_{(x)}$ is a flat $k[x]_{(x)}$-module, but it is finitely presented, so it is free, and then it is free in a neighborhood, again by finite presentation. \end{proof}

If, in the Setup, we take $\Delta : \NN \to \NN^2$ as our monoid homomorphism, $A=k$ a field, and $\NN \to k$ given by $1 \mapsto 0$, then $B = k \otimes_{\ZZ[\NN]} \ZZ[\NN^2]$.  As a ring graded by $\ZZ^2/\ZZ = \ZZ$ in the usual way, $B = k[x,y]/(xy)$ with $|x|=1$, $|y|=-1$.  Using Theorem~\ref{thm:gradedflatness}, we can prove: 

\begin{cor} \label{cor:nodalflatness} Let $k$ be a field.  Grade the ring $B = k[x,y]/(xy)$ by $\ZZ$ so that $|x|=1$, $|y|=-1$.  Let $\m := (x,y) \subseteq B$ be the (homogeneous) maximal ideal of the singular point.  For any $B$-module $M$, the following are equivalent: \begin{enumerate} \item \label{Mgradedflat} $M$ is graded flat over $B$. \item \label{Mperfect} $\Tor_1^B(M,B/\m)=0$. \item \label{injectivemap}  The map $M/yM \oplus M/xM \to M$ given by $(m,n) \mapsto xm+yn$ is injective.  \item \label{Mhasnotor} $\Tor_1^B(M,B/xB)=0$, $y$ is $M/xM$-regular, and similarly with the roles of $x,y$ reversed. \item \label{Mhasnotorvariant} $\Tor_1^B(M,B/xB)=0$, $M/xM$ is graded flat over $B/xB=k[y]$, and similarly with the roles of $x$ and $y$ reversed.  \item \label{Mhaslittletor1} $\Tor_1^B(M,B/xB)=0$ and $y$ is $M/xM$-regular. \item \label{Mhaslittletor1variant} $\Tor_1^B(M,B/xB)=0$ and $M/xM$ is graded flat over $B/xB=k[y]$.  \item \label{Mhaslittletor2} $\Tor_1^B(M,B/yB)=0$ and $x$ is $M/yM$-regular. \item \label{Mhaslittletor2variant} $\Tor^1_B(M,B/yB)=0$ and $M/yM$ is graded flat over $B/yB=k[x]$.  \item \label{localversion}  Any/all of \eqref{Mperfect}-\eqref{Mhaslittletor2variant} hold after localizing at $\m$. \end{enumerate} If $M$ is finitely generated, these conditions are also equivalent to: \begin{enumerate} \setcounter{enumi}{10} \item \label{Mlocallyfree} $M$ is locally free near the origin.  \end{enumerate} \end{cor}

\begin{proof} Certainly \eqref{Mgradedflat} implies \eqref{Mperfect} because $\m$ is a homogeneous ideal, so $B/\m$ is a graded $B$-module.  

To see that \eqref{Mperfect} is equivalent to \eqref{injectivemap}, first consider the short exact sequence of $B$-modules \bne{Bmodules} & 0 \to xB \oplus yB \to B \to B/\m \to 0. \ene  The left map in \eqref{Bmodules} is given by $(xa,yb) \mapsto xa+yb$.  There is an isomorphism of $B$-modules $xB \cong B/yB$ given by $xb \mapsto b$, with inverse $b \mapsto xb$; there is a similar isomorphism with the roles of $x$ and $y$ exchanged.  If we use these isomorphisms to write \bne{funnycoincidence} M \otimes_{B} xB = M \otimes_B B/yB = M/yM \ene (and similarly with the roles of $x$ and $y$ reversed), then tensoring \eqref{Bmodules} over $B$ with $M$ yields an exact sequence $$0 \to \Tor_1^B(M,B/\m) \to M/yM \oplus M/xM \to M \to M/\m M \to 0$$ where the left map is the one in \eqref{injectivemap}.  The equivalence of \eqref{Mperfect} and \eqref{injectivemap} is clear from this exact sequence.

To see that \eqref{injectivemap} implies \eqref{Mhasnotor}, first note that the injectivity in \eqref{injectivemap} in particular implies that the map $M/yM \to M$ given by $m \mapsto xm$ must be injective.  We have an exact sequence of $B$-modules \bne{Bmodules2} & 0 \to xB \to B \to B/xB \to 0.\ene  Tensoring \eqref{Bmodules2} over $B$ with $M$ and using \eqref{funnycoincidence}, we obtain an exact sequence $$ 0 \to \Tor_1^B(M,B/xB) \to M/yM \to M \to M/xM \to 0, $$ where the second map is the injective map mentioned a moment ago, so we find $\Tor_1^B(M,B/xB)=0$.  Now let us prove that $y$ is $M/xM$-regular.  If not, there is an $m \in M \setminus xM$ such that $ym = xm'$ for some $m' \in M$.  But then the injective map in \eqref{injectivemap} would kill $(-m',m)$, a contradiction.  All of these arguments can be repeated with the roles of $x$ and $y$ reversed.

The equivalences ``\eqref{Mhasnotor} iff \eqref{Mhasnotorvariant}," ``\eqref{Mhaslittletor1} iff \eqref{Mhaslittletor1variant}," and ``\eqref{Mhaslittletor2} iff \eqref{Mhaslittletor2variant}" are immediate from Corollary~\ref{cor:gradedflatnessoverA1}.

Obviously \eqref{Mhasnotorvariant} implies \eqref{Mhaslittletor1variant} and \eqref{Mhaslittletor2variant}.  But either of the latter conditions implies \eqref{Mgradedflat} by Theorem~\ref{thm:gradedflatness}:  Take $E = \{ e_1,e_2 \} \subseteq \NN^2$ as the spawning set in the Setup.  

We have proved that \eqref{Mgradedflat}-\eqref{Mhaslittletor2variant} are equivalent.  Any of the conditions formulated in terms of $\Tor$-vanishing can be checked after localizing at each prime ideal.  But each of these conditions holds trivially at any prime other than $\m$, thus we can add \eqref{localversion} to our list of equivalent conditions.  

When $M$ is finitely generated, the condition $\Tor_1^{B_{\m}}(M_{\m},B_{\m}/\m B_{\m})$ is equivalent to \eqref{Mlocallyfree} by the same basic commutative algebra used in the proof of Corollary~\ref{cor:gradedflatnessoverA1}.  \end{proof}

\subsection{Stacks perspective} \label{section:stacksperspective} Let $(G,A)$ be a graded ring.  The group scheme $\Spec \ZZ[G]$ (which we will often abusively denote $G$) acts on $X := \Spec A$ via the action map $a : G \times X \to X$ given by $\Spec$ of the ring map \bne{actionmap} \alpha : A & \to & A[G] = A \otimes_{\ZZ} \ZZ[G] \\ \nonumber a_g & \mapsto & a_g[g] \ene (here $a_g \in A_g$ is homogeneous of degree $g$).  In fact one can show that every action of $\Spec \ZZ[G]$ on $\Spec A$ is of this form by recovering a grading on $A$ from the action map $\alpha$ via the formula \bne{actiontograding} A_g & := & \{ a \in A : \alpha(a) = a[g] \}. \ene We will denote the quotient stack (in the \'etale topology) \be \Spec (A/G) & := & [X / G] \ee so that we have a natural map \be \pi : \Spec A & \to & \Spec (A/G) \ee of stacks.  The map $\pi$ is---essentially by definition---a representable, \'etale-locally-trivial principal $G$-bundle.  In particular it is a flat, affine (hence quasi-compact), surjective morphism since $G$ is flat and affine (over $\Spec \ZZ$).  Formation of this quotient stack is contravariantly functorial in $(G,A) \in \GrAn$ because a $\GrAn$-morphism $(\gamma,f) : (G,A) \to (H,B)$ induces a map of schemes $\Spec f : \Spec B \to \Spec A$ which is $\Spec \ZZ[\gamma] : H \to G$ equivariant for the aforementioned action of $H$ on $\Spec B$ and $G$ on $\Spec A$.

\begin{prop} \label{prop:QcoModGA} Let $(G,A)$ be a graded ring.  There is a natural equivalence of categories \be \Mod(G,A) & = & \Qco(\Spec(A/G)) \ee making the following diagram of equivalences ``commute": \bne{qco} & \xym{ \Mod(G,A) \ar[d]_{\rm forget} \ar[r] & \Qco(\Spec(A/G)) \ar[d]^{\pi^*} \\ \Mod(A) \ar[r]^{M \mapsto M^\sim} & \Qco(\Spec A) } \ene \end{prop}

\begin{proof} This is standard; we sketch the details, putting them in a general context.  Since $\pi$ is an fpqc cover, a fundamental theorem of fpqc descent theory says that the functor \bne{fpqcdescent} \pi^* : \Qco(\Spec (A/G)) & \to & \Desc(\Qco(\Spec A),\pi) \\ \nonumber \F & \mapsto & (\pi^*\F,\phi_\tau) \ene taking $\F$ to $\pi^* \F$ equipped with the tautological $\pi$-descent datum $\phi_{\tau}$ is an equivalence between the category of quasi-coherent sheaves on $\Spec(A/G)$ (defined, for example, as in Definition~\ref{defn:QcoX} of \S\ref{section:logquotientspace}) and the category $\Desc(\Qco(\Spec A),\pi)$ of pairs $(\F,\phi)$ consisting of a quasi-coherent sheaf $\F$ on $\Spec A$ and a $\pi$-descent datum $\phi$ for $\F$.  Furthermore, this equivalence of categories identifies $\pi^* : \Qco(\Spec(A/G)) \to \Qco(\Spec A)$ with the functor \be \Desc(\Qco(\Spec A),\pi) & \to & \Qco(\Spec A) \\ (\F,\phi) & \mapsto & \F \ee given by forgetting the descent datum. Since $\pi$ is a principal bundle with structure group $G$, we have a $2$-cartesian diagram of stacks $$ \xym{ G \times \Spec A \ar[r]^-{a} \ar[d]_{p} & \Spec A \ar[d]^{\pi} \\ \Spec A \ar[r]_-{\pi} & \Spec (A/G) } $$ where $p$ is the projection and $a$ is the action, so a $\pi$-descent datum on $\F \in \Qco(\Spec A)$ is an isomorphism $\phi : a^* \F \to p^* \F$ satisfying the usual cocycle condition on the triple overlap \be \Spec A \times_{\Spec (A/G)} \Spec A \times_{\Spec (A/G)} \Spec A & = & \Spec A[G \times G]. \ee  Under the equivalence $M \mapsto M^{\sim}$ from $\Mod(A)$ to $\Qco(\Spec A)$, a $\pi$-descent datum on an $A$-module $M$ is hence an isomorphism \bne{descentdatum} \phi : M \otimes_{A,\alpha} A[G] & \to & M \otimes_{A,p} A[G] \ene of $A[G]$-modules whose various pullbacks make a certain diagram of isomorphisms of $A[G \times G]$-modules commute; here $\alpha$ is the ring map \eqref{actionmap} and $p$ is abusive notation for the ring map $p : A \to A[G]$ given by $p(a) = a[0]$.  

The key point is that the descent datum $\phi$ is the same thing as a structure of $G$-graded $A$-module on $M$.  Given a graded module structure on $M$, we get a descent datum $\phi$ by setting \bne{gradingtodescentdatum} \phi(m_g \otimes b[h]) & := & m_g \otimes b[h+g] \ene for $g,h \in G$, $m_g \in M_g$, $b \in A$.  Given a descent datum $\phi$ we obtain a grading on $M$ by setting \bne{descentdatumtograding} M_g & := & \{ m \in M : \phi(m \otimes 1[0]) = m \otimes 1[p] \}. \ene In order to show that $M = \oplus_g M_g$, we write $\phi(m \otimes 1[0]) = \sum_g m_g \otimes 1[g]$ for unique $m_g \in M$ (all but finitely many zero), and then we similarly write $m_g = \sum_h m_{g,h} \otimes 1[h]$ for unique $m_{g,h} \in M$.  Then we follow $m_g \otimes 1[0] \otimes 1[0,0]$ around the cocycle condition satisfied by $\phi$ to conclude that $m_{g,h} = 0$ for $g \neq h$ and $m_{g,g}=m_g$.  It follows that $m_g \in M_g$ and $m = \sum_g m_g$ is the unique expression for $m$ as a sum over elements in the subgroups $M_g$.   This yields an equivalence of categories \bne{gradingstodescent} \Mod(G,A) & \to & \Desc(\Qco(\Spec A),\pi) \ene whose composition with \eqref{fpqcdescent} is as desired. \end{proof}

\begin{prop} \label{prop:gradedflatiffstacksflat} Let $(\gamma,f) : (G,A) \to (H,B)$ be a map of graded rings, \be \Spec (f/\gamma) : \Spec(B/H) & \to & \Spec(A/G) \ee the associated map of stacks, $N \in \Mod(H,B)$.  Then the corresponding quasi-coherent sheaf ``$N^{\sim}/H$" on $\Spec(B/H)$ is flat over $\Spec (A/G)$ iff $N$ is graded flat over $(G,A)$ as in Definition~\ref{defn:gradedflat2}. \end{prop}

\begin{proof} To say that $N^\sim / H$ is flat over $\Spec (A/G)$ is equivalent to saying that the functor \bne{pullbacktensor} (\Spec f/\gamma)^* \slot \otimes N^{\sim}/H : \Qco(\Spec(A/G)) & \to & \Qco(\Spec(H/B)) \ene is exact (c.f.\ Lemma~\ref{lem:flatnessdefinition}, note that $\Spec(A/G)$ is quasi-separated since its diagonal is affine).  But this functor is identified, under the natural equivalence of categories of Proposition~\ref{prop:QcoModGA}, with the extension of scalars functor \eqref{slototimesN}, which is exact by definition iff $N$ is graded flat over $(G,A)$. \end{proof}

\begin{rem} The functor \bne{SpecAG} \Spec( \slot / \slot) : \GrAn^{\rm op} & \to & {\bf Stacks} \ene does not preserve general inverse limits.  This is why graded flatness has only a limited stability under base change. \end{rem}

\section{Stacks} \label{section:stacks} All technical results on stacks used in the rest of the paper are collected in this section.

\subsection{Definitions}  The following definitions are used throughout:

\begin{defn} \label{defn:separated} A map of schemes $X \to Y$ is called \emph{separated} (resp.\ \emph{locally separated}, \emph{quasi-separated}) iff the diagonal $X \to X \times_Y X$ is a closed embedding (resp.\ is a quasi-compact locally closed immersion, is quasi-compact). \end{defn}

\begin{defn} \label{defn:algebraicspace} Let $Y$ be a scheme.  A sheaf $X$ on $\Sch/Y$ in the \'etale topology is called a \emph{separated} (resp.\ \emph{locally separated}, \emph{quasi-separated}) \emph{algebraic space (of (locally) finite presentation)} over $Y$ iff the following hold: \begin{enumerate} \item The diagonal morphism $\Delta : X \to X \times_Y X$ is representable by closed embeddings (resp.\ quasi-compact locally closed immersions, quasi-compact maps) of schemes. \item There exists a scheme $X'$ (of (locally) finite presentation) over $Y$ and a $\Sch/Y$-morphism $X' \to X$ (necessarily representable by schemes by the first condition) which is \'etale and surjective. \end{enumerate} When used without additional qualification, ``algebraic space" in this paper is understood to mean ``locally separated algebraic space of finite presentation." \end{defn}

\begin{defn} \label{defn:representable} A morphism $X' \to X$ of categories fibered in groupoids over $\Sch/Y$ is called \emph{representable by separated} (resp.\ \emph{locally separated}, \emph{quasi-separated}) \emph{schemes} (resp. \emph{algebraic spaces}) \emph{(of (locally) finite presentation)} iff, for any $Y$-scheme $U$ and any map $U \to X$ of groupoid fibrations over $Y$-schemes, ``the" $2$-fibered product $U \times_X X'$ is equivalent to a separated (resp.\ locally separated, quasi-separated) $Y$-scheme (resp.\ algebraic space) (of (locally) finite presentation) over $U$.  When used without additional qualification ``representable" in this paper is understood to mean ``representable by locally separated algebraic spaces of finite presentation." \end{defn}

This is the meaning of ``representable" that Olsson intends in \cite[3.2]{Ols}, as he makes clear in the proof.

\begin{defn} \label{defn:algebraicstack} Let $Y$ be a scheme.  Unless explicitly mentioned to the contrary, an \emph{algebraic stack} in this paper is a stack $X$ over $\Sch/Y$ in the \'etale topology such that the following hold: \begin{enumerate} \item The diagonal morphism $\Delta : X \to X \times X$ is representable in the sense of Definition~\ref{defn:algebraicspace}.  \item There is a $Y$-scheme $X'$ of locally finite presentation over $Y$ and a morphism $X' \to X$ (necessarily representable by the first condition) which is smooth and surjective. \end{enumerate} \end{defn}

We are following Olsson \cite[1.2]{Ols} with the use of algebraic spaces with locally separated diagonal.  In \cite[4.1]{LM} they require the diagonal to be separated in the definition of an algebraic stack, so the Olsson definition above is more general in that sense.  However, the notion of algebraic stack in \cite[4.1]{LM} doesn't have as many finiteness conditions and their notion of ``algebraic space" is more relaxed.  Any algebraic stack used in practice ought to be an algebraic stack in the sense of Definition~\ref{defn:algebraicstack}.  

\subsection{Representability} \label{section:representability} The purpose of this section is to clear up any confusion about what is meant by ``representable" elsewhere in the text.  Nothing here is difficult.  The proofs will mostly be omitted. 

\begin{lem} \label{lem:faithful} If $F : \C \to \D$ is a faithful functor then the group homomorphism \bne{autgroup} F : \Aut_{\C}(c) & \to & \Aut_{\D}(Fc) \ene is injective for every $c \in \C$.  If $\C$ is a groupoid, the converse holds.  Similarly, if $F$ is full and $\C$ is a groupoid, then \eqref{autgroup} is surjective for every $c \in \C$.  \end{lem}

Below, $\D$ is an arbitrary category and $\CFG / \D$ is the $2$-category of categories fibered in groupoids over $\D$.

\begin{prop} For an object $X$ in $\CFG / \D$, the following are equivalent: \begin{enumerate} \item $X$ is in the essential image of the functor from presheaves on $\D$ to $\CFG / \D$.  In other words, $X$ ``is" a presheaf. \item $\Aut_{X_d}(x)$ is trivial for every $d \in \D$ and every $x \in X_d$.\footnote{Here $X_d$ is the fiber category over $d$ whose objects are objects of $X$ mapped to $d$ via the structure map $X \to \D$ and whose morphisms are $X$-morphisms mapped to $\Id_d$.}  \item The structure map $X \to \D$ is a faithful functor.  \item The diagonal $\Delta : X \to X \times X$ is fully faithful. \end{enumerate} \end{prop}

\begin{prop} \label{prop:representablebypresheaves} Let $f : X \to Y$ be a morphism in $\CFG / \D$.  The following are equivalent: \begin{enumerate} \item $f$ is \emph{representable by presheaves} in the sense that for any presheaf $U$ on $\D$ and any $\CFG / \D$ morphism $U \to Y$, ``the" $2$-fibered product $U \times_{Y} X$ ``is" a presheaf. \item $f$ is a faithful functor. \item $F_d : X_d \to Y_d$ is faithful for all $d \in \D$. \item \label{autinjectivity} $f : \Aut_{X_d}(x) \to \Aut_{Y_d}(fx)$ is an injective group homomorphism for each $d \in \D$ and each $x \in X_d$. \item The diagonal $\Delta : \Aut_{X_d}(x)  \to  \Aut_{(X \times_{Y} X)_d}(x,x,\Id)$ is a surjection of groups for each $d \in \D$ and each $x \in \X_d$. \item The diagonal $\Delta : \Aut_{X_d}(x)  \to  \Aut_{(X \times_{Y} X)_d}(x,x,\Id)$ is an isomorphism of groups for each $d \in \D$ and each $x \in X_d$. \item The diagonal $\Delta_d : X_d \to (X \times_{Y} X)_d$ is fully faithful for each $d \in \D$. \item The diagonal $\Delta : X \to X \times_{Y} X$ is fully faithful. \end{enumerate} \end{prop}

\begin{proof} These equivalences are mostly a matter of definitions; they are proved easily using Lemma~\ref{lem:faithful} and the general nonsense in \cite[\S2]{LM}. \end{proof}

\begin{defn} \label{defn:formallyrepresentable}  A morphism $f : X \to Y$ in $\CFG / \D$ satisfying the equivalent conditions of Proposition~\ref{prop:representablebypresheaves} is called \emph{formally representable}. \end{defn}

An alternative to ``formally representable" would be ``representable by presheaves".

\begin{prop} \label{prop:formallyrepresentableimpliesrepresentable} If $f : X \to Y$ is a formally representable map of locally separated algebraic stacks of finite presentation over a scheme $Y$, then $f$ is representable (Definition~\ref{defn:representable}). \end{prop}

\begin{proof} In \cite[8.1.2]{LM} they prove this for separated algebraic stacks, but one can weaken the ``separated" assumption to ``locally separated." \end{proof}

\begin{lem} \label{lem:stackificationpreservesrepresentability} Suppose $f : X \to Y$ is a formally representable morphism between prestacks (c.f.\ \cite[3.1]{LM}) for some topology $\tau$ on $\D$.  Then the stackification $f^+ : X^+ \to Y^+$ is also formally representable. \end{lem}

\begin{proof}  Check, say, that condition \eqref{autinjectivity} in Proposition~\ref{prop:representablebypresheaves} for $f$ implies the analogous condition for $f^+$ by noting that: \begin{enumerate} \item Equality of automorphisms of objects of the $\tau$ stack $Y^+$ can be checked locally in the $\tau$-topology.  \item Locally in the $\tau$ topology, any object in a fiber category of $X^+$ is the image of an object in the corresponding fiber category of $X$ under the stackification map $X \to X^+$. \item The stackification doesn't change the automorphism group of objects that come from $X$---the stackification maps $X \to X^+$ and $Y \to Y^+$ are fully faithful \cite[3.2.1]{LM}. \end{enumerate} \end{proof}

\subsection{Formal smoothness} \label{section:formalsmoothness} Consider a $2$-commutative diagram \bne{diatolift} & \xym{ T \ar[d] \ar[r] & X \ar[d]^f \\ T' \ar@{.>}[ru] \ar[r] & Y } \ene in a $2$-category $\C$.  The datum of such a diagram includes a fixed choice of homotopy (invertible $2$-morphism) $\eta : fa \to bi$.  By a \emph{lift} in such a diagram we mean a morphism (i.e.\ a ``$1$-morphism") $l : T' \to X$ together with homotopies $\alpha : a \to li$ and $\beta : fl \to b$ such that $\eta = (\beta * i )(f * \alpha)$.  Lifts form a category (in fact a groupoid) where a morphism from $l=(l,\alpha,\beta)$ to $l' = (l',\alpha',\beta')$ is a homotopy $\gamma : l \to l'$ compatible with the $\alpha$'s and $\beta$'s.  If $\C = \CFG / \D$ for some category $\D$ and $f$ is formally representable, then one can check that there is at most one morphism between any two lifts, so the groupoid of lifts is ``setlike".  This is the only case we will consider.  When we say that there is a \emph{unique} lift in such a $2$-commutative diagram we mean ``unique up to (necessarily unique) isomorphism in the category of lifts".

Let $\D$ denote the category of schemes (or schemes over some base) and let $I$ denote the class of square-zero closed embeddings in $\D$.

\begin{defn}  \label{defn:formallysmooth} A formally representable $\CFG / \D$ morphism $f : X \to Y$ is called \emph{formally smooth} (resp.\ \emph{formally \'etale}) iff, in any $2$-commutative diagram \eqref{diatolift} with $i \in I$ there exists a lift locally on $T'$ in the \'etale topology (resp.\ and given any two lifts in any such diagram, there exists, \'etale locally on $T'$, a (necessarily unique) isomorphism between them). \end{defn}

In particular, note that, in checking whether $f$ is formally smooth or \'etale we can always restrict our attention to the case where $T \into T'$ is a square-zero closed embedding of \emph{affine} schemes.

\begin{rem} It is clear from the local nature of the definition that a formally representable map of prestacks $f : X \to Y$ in the \'etale topology is formally \'etale / formally smooth iff its stackification in the \'etale topology is formally \'etale / formally smooth.  Furthermore, if $f$ is a formally representable map of \emph{stacks} in the \'etale topology, formal \'etaleness of $f$ is equivalent to the existence of a \emph{unique} completion in any diagram \eqref{diatolift} with $i \in I$.  (We will give a more general argument below.) \end{rem}

\begin{rem} \label{rem:formalsmoothness} If $f : X \to Y$ is representable by schemes (resp.\ algebraic spaces), then it is clear that $f$ formally smooth (resp.\ formally \'etale) in the sense of Definition~\ref{defn:formallysmooth} implies representability of $f$ by formally smooth (resp.\ formally \'etale) maps of schemes (resp.\ algebraic spaces).  This is simply because the notion of formal smoothness / \'etaleness above is clearly stable under $2$-base change and specializes to the usual notion for schemes or algebraic spaces.  In particular, if $f$ is representable by maps of schemes (or algebraic spaces) of locally finite presentation, then $f$ formally smooth (resp.\ formally \'etale) in the above sense implies that $f$ is smooth (resp.\ \'etale) in the usual sense. \end{rem}

For simplicity, let us now suppose $f : X \to Y$ is a map of \emph{schemes}.  I claim that $f$ is formally \'etale iff the following two conditions hold: \begin{enumerate} \item \label{localuniqueness} For any two lifts $l,l'$ in any diagram \eqref{diatolift}, there is an fppf cover $d : S' \to T'$ such that $ld=l'd$. \item \label{localexistence} For any such diagram there is an fppf cover $d : S' \to T$ such that there exists a lift in the diagram \bne{bigdiatolift} & \xym{ S \ar[r]^{\pi_2} \ar[d]_{\pi_1} & T \ar[d] \ar[r] & X \ar[d]^f \\ S' \ar@{.>}[rru] \ar[r]^d & T' \ar[r] & Y } \ene where $S := S' \times_{T'} T$. \end{enumerate}  (We may of course also assume that $T'$ is affine.)  Clearly the existence of unique lifts in all diagrams \eqref{diatolift} implies these two conditions because square zero closed embeddings are closed under base change, so the big square in \eqref{bigdiatolift} is again a diagram of the form \eqref{diatolift}.  Conversely, suppose the two conditions hold and we want to produce a lift.  Using the second condition we find an fppf cover $d : S' \to T'$ and a lift $l' : S' \to X$ as indicated in \eqref{bigdiatolift}.  We next check that this lift descends to $T'$ as follows: Let $R' := R' \times_{S'} R'$, $R := R' \times_{S'} S$ and let $\pi_1,\pi_2 : R' \rightrightarrows S'$ be the two projections.  Then $l' \pi_1$ and $l' \pi_2$ both furnish lifts in the big square of \bne{biggerdiatolift} & \xym{ R \ar@<.5ex>[r] \ar@<-.5ex>[r] \ar[d] & S \ar[r]^{\pi_2} \ar[d]_{\pi_1} & T \ar[d] \ar[r] & X \ar[d]^f \\ R' \ar@<.5ex>[r] \ar@<-.5ex>[r] \ar@{.>}[rrru] & S'  \ar[r]^d & T' \ar[r] & Y } \ene (the meaning of ``big square" doesn't depend on the choices of parallel arrows) and again this big square is of the form \eqref{diatolift} so the first condition implies that there is an fppf cover $e : Q' \to R'$ so that $\pi_1 l' e = \pi_2 l' e : Q' \to X$.  But $e$ is an fppf cover, so this implies $\pi_1 l' = \pi_2 l'$.  But $d$ is an fppf cover, so this implies $l' = ld$ for some (necessarily unique) $l : T' \to X$.  Now we check that this $l$ furnishes a lift in the original diagram \eqref{diatolift}.  To check that the lower triangle commutes we use the fact that $d$ is an fppf cover, so we can check equality after precomposing with $d$, which holds exactly because the lower triangle in \eqref{bigdiatolift} commutes.  To check that the upper triangle commutes it is enough to check after first applying the fppf cover $S \to T$ (a base change of $S' \to T'$) where we reduce to commutativity of the upper square in \eqref{bigdiatolift}.

Extracting the hypotheses that were actually necessary above, we have proved:

\begin{lem}  Let $I$ be a class of maps in a category $\C$ stable under base change and let $f : X \to Y$ be a map of sheaves on $\C$ in some topology $\tau$.  Then the \emph{unique} \emph{RLP of} $f$ \emph{with respect to} $I$ (i.e.\ the existence of a unique lift in every diagram \eqref{diatolift} with $i \in I$) is equivalent to the local uniqueness of such lifts in the $\tau$ topology plus the local existence of a lift in the $\tau$-topology (i.e.\ conditions \eqref{localuniqueness} and \eqref{localexistence} above with ``fppf" replaced by ``$\tau$"). \end{lem}

The advantage of the formulation in terms of local uniqueness and local existence of lifts is that the latter is clearly stable under sheafification:

\begin{lem} Let $I$ be a class of maps in a category $\C$ stable under base change and let $f : X \to Y$ be a map of presheaves on $\C$.  Let $\tau$ be a topology on $\C$.  Then the following are equivalent: \begin{enumerate} \item The sheafification $f^+$ of $f$ in the $\tau$ topology has the unique RLP with respect to $I$. \item $f$ has the local existence and uniqueness RLPs with respect to $I$ with respect to the topology $\tau$.  \end{enumerate} \end{lem}

The same argument (plus perhaps a little bookkeeping to keep track of homotopies) makes sense for representable maps of prestacks.  We will state the version relevant for formal \'etaleness:

\begin{lem} Let $\D$ be the category of schemes (or schemes over some base) equipped with some topology $\tau$.  Suppose $f : X \to Y$ is a formally representable map of stacks over $\D$ in a topology $\tau$ at least as fine as the \'etale topology.  Then to check that $f$ is formally \'etale, it suffices to check that \begin{enumerate} \item In any $2$-commutative diagram as in \eqref{diatolift} with $i$ a square-zero closed embedding of affine schemes, there exists a lift locally on $T'$ in the $\tau$ topology, and \item given two lifts in such a diagram, there exists an isomorphism between them locally on $T'$ in the $\tau$ topology. \end{enumerate} \end{lem}

We will use the following variant:

\begin{lem} \label{lem:formaletalenesscriterion} Let $\D$ be the category of schemes (or schemes over some base).  Suppose $f : X \to Y$ is a formally representable map of prestacks over $\D$ in the fppf topology such that the stackification $f^+$ of $f$ in the \'etale topology is a map of algebraic stacks.  Then to prove that $f^+$ is formally \'etale, it suffices to check that $f$ satisfies the following two conditions: \begin{enumerate} \item In any $2$-commutative diagram as in \eqref{diatolift} with $i$ a square-zero closed embedding of affine schemes, there exists a lift locally on $T'$ in the fppf topology. \item Given any two lifts in any such diagram, there exists an isomorphism between them locally on $T'$ in the fppf topology. \end{enumerate} \end{lem}

\begin{proof} Let $f^{++}$ denote the stackification of $f$ in the fppf topology.  Since $f$ is an fppf prestack it is \emph{a fortiori} an \'etale prestack, so both $f^+$ and $f^{++}$ are formally representable by Lemma~\ref{lem:stackificationpreservesrepresentability}, so it makes sense to ask whether they are formally \'etale.  The fppf local nature of the conditions on $f$ in the statement of the lemma makes it clear that they hold for $f^{++}$ iff they hold for $f$.  By the previous lemma, if these conditions hold for $f^{++}$, then $f^{++}$ is formally \'etale, so we have proved that the conditions imply formal \'etaleness of $f^{++}$.  But $f^{+} = f^{++}$ because an algebraic stack is a stack in the fppf topology \cite[10.7]{LM}. \end{proof}

\subsection{Lifting up to homotopy}  \label{section:liftinglemmas} The following technical results will be used in the proofs of Theorem~\ref{thm:AAAhetale} and Theorem~\ref{thm:LhtoLogAQetale}.  See Remark~\ref{rem:HomotopyLifting} for another application.

\begin{lem} \label{lem:lifting} {\bf (Lifting Lemma)} Consider a commutative diagram of monoids $$ \xym{ P \ar[r]^b & M \ar[r] & A \\ Q \ar[u]^h \ar[r]_a & M' \ar[u]_i  \ar[r] & A' \ar[u] }$$ where the right square is a strict map of integral log rings, $h : Q \into P$ is an injective map of fine monoids, and $A' \to A$ is a surjective ring map with square zero kernel $I$.  Let $q : P \to P/Q$ denote the cokernel of $h$.  Then there exist: \begin{enumerate} \item a finite faithfully flat ring map $A' \to B'$ \item a group homomorphism $\alpha : (P/Q)^{\rm gp} \to B^*$, where $B := B' \otimes_{A'} A$, and \item a lift in the diagram $$ \xym{ P \ar[r]^{\alpha q \cdot b} \ar@{.>}[rd]^l & N \ar[r] & B \\ Q \ar[u]^h \ar[r]_{ a} & N' \ar[u]_i \ar[r] & B' \ar[u] }$$ where the right square is the map of log rings obtained from the original map of log rings by pushing out along $A' \to B'$ and taking associated log structures and $$ (\alpha q \cdot b)(p)  =  [b(p),\alpha(q(p))] \in N = M \oplus_{A^*} B^*.$$ \end{enumerate}  If $(P/Q)^{\rm gp}$ is torsion-free or the order of the torsion subgroup of $(P/Q)^{\rm gp}$ is invertible in $A$, then we can take $A' = B'$ and $\alpha=1$.  \end{lem}

\begin{proof} Since the map of log rings is strict, there is an ``exact sequence" of monoids $$ \xym{ 0 \ar[r] & I \ar[r] & M' \ar[r]^i & M \ar[r] & 0 } $$ where the left map is $i \mapsto (1+i) \in (A')^* \subseteq M'$ and $I$ is regarded as an abelian group under addition.  Groupifying this yields a diagram of groups with exact row \bne{exrow} & \xym{ 0 \ar[r] & I \ar[r] & (M')^{\rm gp} \ar[r]^i & M^{\rm gp} \ar[r] & 0 \\ &  & Q^{\rm gp} \ar[r] \ar[u]^a & P^{\rm gp} \ar@{.>}[lu]_l \ar[u]_b } \ene and one argues using integrality of all monoids involved (and the fact that $I$ is a group) that lifting as in the statement of the lemma (with $A'=B'$) is the same thing as lifting as indicated in this diagram.  If $(P/Q)^{\rm gp}$ is torsion free then it is free because $Q$ and $P$ are finitely generated, so one can choose a splitting $P^{\rm gp} = Q^{\rm gp} \oplus (P/Q)^{\rm gp}$ and a lift is easily constructed.  In general, we first let $G \subseteq P^{\rm gp}$ be the preimage of the free summand of \bne{PmodQ} (P/Q)^{\rm gp} & \cong & \ZZ^n \oplus \ZZ / n_1 \ZZ \oplus \cdots \oplus \ZZ / n_k \ZZ \ene under the quotient map $P^{\rm gp} \to (Q/P)^{\rm gp}$.  If we replace $P^{\rm gp}$ with $G$ and $b$ with $b|_G$, then we can find a lift $l : G \to (M')^{\rm gp}$ by the same argument because $G/Q^{\rm gp}$ is free.  Then by replacing $Q^{\rm gp}$ and $a$ by $G$ and $l$ respectively, we reduce to treating the case where $(P/Q)^{\rm gp}$ is torsion---i.e. $n=0$ in \eqref{PmodQ}.  

Choose $p_1,\dots,p_k \in P^{\rm gp}$ so that $p_i$ maps to a generator of the $i^{\rm th}$ summand of $(P/Q)^{\rm gp}$.  Then $n_ip_i \in Q$.  To produce a lift $l$ in \eqref{exrow} we must produce elements $l(p_1),\dots,l(p_k) \in (M')^{\rm gp}$ lying over $b(p_1),\dots,b(p_k) \in M^{\rm gp}$ (respectively) and satisfying $l(p_i)^{n_j} = a(n_jp_j)$ (we will write the binary operations on the monoids $M$ and $M'$ multiplicatively, and those of $Q$, $P$ additively).

Choose $m_j \in (M')^{\rm gp}$ such that $i(m_j) = b(p_j)$ in $M^{\rm gp}$ for each $j \in \{ 1, \dots, k \}$.  Then we have \be i((m_j)^{n_j}) & = & i(a(n_j p_j)) \ee so there is $i_j \in I$ such that \be (1+i_j) m_j^{n_j} & = & a(n_j p_j). \ee  If the order $n_1 \cdots n_k$ of $(Q/P)^{\rm gp}$ is invertible in $A$, then since $I$ is an $A$-module, there are $i_j' \in I$ such that $n_j i_j' = i_j$.  In this case we can take $l(p_i) := (1+i_j')m_j$ and obtain the desired lift because $1+i_j' \in (A')^*$ lies over $1 \in A^*$ so $(1+i_j')m_j$ lies over $b(p_j)$ and we have $$((1+i_j')m_j)^{n_j} = (1+n_ji_j')m_j^{n_j} = (1+i_j)m_j^{n_j} = a(n_j p_j)$$ because $I$ is square zero.  In general we set $u_j := 1 + i_j \in (A')^*$ and we consider the finite faithfully flat ring map \be A' & \to & B' := A'[x_1,\dots,x_k]/(x_1^{n_1} - u_1,\dots,x_k^{n_k}-u_k) \ee which pushes out along $A' \to A$ to \be A & \to & B := A[x_1,\dots,x_k]/(x_1^{n_1}-1,\dots,x_k^{n_k}-1) \ee because $u_j \in A'$ lies over $1 \in A$.  Since each $x_j$ is an $n_j^{\rm th}$ root of unity in $B$, we have a group homomorphism $\alpha : (P/Q)^{\rm gp} \to B^*$ taking the generator of the $j^{\rm th}$ cyclic summand to $x_j$.  

Set $I' := I \otimes_{A'} B'$.  Since the faithfully flat ring maps $A' \to B'$ and $A \to B$ are local (anything mapping to a unit was already a unit), note that we have \be N' & = & M' \oplus_{(A')^*} (B')^* \\ N & = & M \oplus_{A^*} B^* \\ (N')^{\rm gp} & = & (M')^{\rm gp} \oplus_{(A')^*} (B')^* \\ N^{\rm gp} & = & M^{\rm gp} \oplus_{A^*} B^*. \ee  We can construct a lift in \bne{exrow2} & \xym{ 0 \ar[r] & I' \ar[r] & (N')^{\rm gp} \ar[r] & N^{\rm gp} \ar[r] & 0 \\ &  & Q^{\rm gp} \ar[r] \ar[u]^{[a,1]} & P^{\rm gp} \ar@{.>}[lu]_l \ar[u]_{\alpha q \cdot b } } \ene by setting $l(p_j) = [m_j,x_j]$ because $$n_j[m_j,x_j] = [m_j^{n_j},x_j^{n_j}] = [u_j m_j^{n_j},1] = [a(n_j p_j),1]$$ and $[m_j,x_j] \in (N')^{\rm gp}$ lies over $[b(p_j),x_j]=[b(p_j),\alpha(q(p_j))] \in N^{\rm gp}$. \end{proof}

The above lifting lemma is really just a shadow of a more general statement (Theorem~\ref{thm:homotopylifting} below).  One should think of $\alpha q : P^{\rm gp} \to B^*$ in Lemma~\ref{lem:lifting} as a \emph{homotopy} from $b$ to $il$.  Starting with a diagram commuting on the nose, we produced a \emph{lift up to homotopy}.  The point is that, if we start with a diagram \emph{commuting up to homotopy} then we can find a \emph{lift up to homotopy}, even without assuming $h : Q \to P$ is injective.

\begin{thm} \label{thm:homotopylifting}  {\bf (Homotopy Lifting)}  Consider a diagram of monoids $$ \xym{ P \ar[r]^b \ar@{.>}[r] & M \ar[r] & A \\ Q \ar[u]^h \ar[r]_a & M' \ar[u]_i  \ar[r] & A' \ar[u] }$$ where the right square is a strict map of integral log rings, $h : Q \to P$ is a map of fine monoids, $A' \to A$ is a surjective ring map with square zero kernel $I$, and the left square \emph{commutes up to homotopy} in the sense that there is a fixed group homomorphism $\eta : Q^{\rm gp} \to A^*$ such that \bne{homotopy} \eta \cdot bh & = & ia.\ene  Then, after possibly replacing \be \xym{ M \ar[r] & A \\ M' \ar[u]^i  \ar[r] & A' \ar[u] } & \quad {\rm \emph{with}} \quad &  \xym{ N \ar[r] & B \\ N' \ar[u]^{``i"}  \ar[r] & B' \ar[u] }  \ee for some finite faithfully flat ring map $A' \to B'$ as in Lemma~\ref{lem:lifting}, there is a \emph{lift up to homotopy}: that is, a triple $(l,\alpha,\beta)$ consisting of a monoid homomorphism $l : P \to M'$, a group homomorphism $\alpha : P^{\rm gp} \to A^*$, and a group homomorphism $\beta : Q^{\rm gp} \to (A')^*$ such that the following hold: \bne{liftuptohomotopy} \alpha \cdot b & = & il \\ \beta \cdot lh & = & a  \\ \eta & = & i \beta \cdot \alpha h. \ene  If $h$ is injective then we can take $\beta=1$.  Furthermore, given two lifts up to homotopy $(l',\alpha',\beta')$, $(l,\alpha,\beta)$ there is a unique group homomorphism $\gamma : P^{\rm gp} \to (A')^*$ satisfying the conditions: \bne{homotopybetweenlifts} \gamma \cdot l' & = & l \\ i \gamma \cdot \alpha' & = & \alpha \\ \beta' & = & \beta \cdot \gamma h. \ene \end{thm}

\begin{proof} We will address the ``furthermore" at the very end of the proof.  After groupifying we obtain a diagram of abelian groups $$ \xym{ 0 \ar[r] & I \ar[r] & (M')^{\rm gp} \ar[r]^i & M^{\rm gp} \ar[r] & 0 \\ & & Q^{\rm gp} \ar[u]^a \ar[r]^h & \ar@{.>}[lu] P^{\rm gp} \ar[u]_b } $$ where the row is exact and the square commutes up to the homotopy $\eta$.  One argues much as in the previous proof that it is enough to find a lift up to homotopy in this diagram of abelian groups.  We can factor $Q^{\rm gp} \to P^{\rm gp}$ as $$Q^{\rm gp} \to G \to H \to P^{\rm gp}$$ where the first map is surjective, the second is injective with torsion-free cokernel, and the third is injective with torsion cokernel.  By successively lifting (from left to right and up to homotopy, taking $\beta=1$ whenever possible) in the diagram  $$ \xym{ (M')^{\rm gp} \ar[rrr]^i & & & M^{\rm gp} \\ Q^{\rm gp} \ar[u]^a \ar[r] & G \ar@{.>}[lu] \ar[r] & H \ar@{.>}[llu] \ar[r] & \ar@{.>}[lllu] P^{\rm gp} \ar[u]_b } $$ we reduce to treating the Cases~1-3 below.  

\noindent {\bf Case 1.} $h^{\rm gp}$ is surjective.  Fix some isomorphism \be P^{\rm gp} & = & \ZZ^n \oplus \ZZ/n_1 \ZZ \oplus \cdots \oplus \ZZ/n_k \ZZ. \ee  For $j \in \{ 1, \dots, n \}$, let $e_j \in P^{\rm gp}$ be a generator of the $j^{\rm th}$ summand of $\ZZ^n \subseteq P^{\rm gp}$ and for $j \in \{ 1, \dots, k \}$, let $p_j$ be a generator of $\ZZ / n_j \ZZ \subseteq P^{\rm gp}$.  Choose $r_j \in Q^{\rm gp}$ such that $h(r_j) = e_j$ and $q_j \in Q^{\rm gp}$ such that $h(q_j) = p_j$.  Choose, for $j \in \{ 1, \dots, k \}$, $m_j \in (M')^{\rm gp}$ such that $i(m_j)=b(p_j)$.  Since $n_jp_j = 0$ in $P^{\rm gp}$, $b(p_j)$ must lie in $A^* \subseteq M^{\rm gp}$ and must in fact be an $n_j^{\rm th}$ root of unity, so we have $m_j^{n_j} = 1+i_j$ for some $i_j \in I$.  In particular $m_j$ is in $(A')^*$.  If $n_j$ is invertible in $A$, set $x_j := 1+i_j/n_j$, otherwise adjoin $x_j$ to $A'$ as a formal variable subject to the relation $x_j^{n_j} = 1 + i_j$ (i.e.\ replace $A'$ by a finite faithfully flat extension $A' \to B'$).  Either way we arrange that $x_j^{n_j} = m_j^{n_j}$, so $m_j / x_j$ is an $n^{\rm th}$ root of unity in $(A')^*$.  We can then define: \be l : P^{\rm gp} & \to & (M')^{\rm gp} \\ l(e_j) & := & a(r_j) \\ l(p_j) & := & m_j / x_j \quad \in (A')^* \subseteq (M')^{\rm gp} . \ee

I claim that for every $q \in Q^{\rm gp}$ there is a (necessarily unique) $\beta(q) \in (A')^*$ such that \bne{betaSSS} \beta(q) lh(q) & = & a(q). \ene  First of all, it is clear that the set of $q$ for which such a $\beta(q)$ exists is a subgroup $S$ of $Q^{\rm gp}$ and $q \mapsto \beta(q)$ is a homomorphism $\beta : S  \to (A')^*$, so the claim is that $S = Q^{\rm gp}$.  Since $Q^{\rm gp}$ is generated by $\Ker h^{\rm gp}$ (also abusively called $\Ker h$), the $r_j$, and the $q_j$, it suffices to check that these are contained in $S$.  The assertion that $\Ker h \subseteq S$ is equivalent to the assertion that $a(q) \in (A')^* \subseteq (M')^{\rm gp}$ for $q \in \Ker h$.  Since the nilpotent thickening $A' \to A$ is local, it suffices to show that $ia(q) \in A^*$.  But $ia(q) = \eta(q)$ by \eqref{homotopy}, so this is clear.  It is clear that $r_j \in S$ (we can take $\beta(r_j)=1$).  To see that $q_j \in S$, we note that $lh(q_j) = l(p_j) \in (A')^*$, so it is equivalent to check that $a(q_j) \in (A')^*$.  It suffices to check that $ia(q_j) \in A^*$.  But $ia(q_j) = \eta^{-1}(q_j) b(p_j)$ by \eqref{homotopy} and this is in $A^*$ since $\eta(q_j), b(p_j) \in A^*$.  The claim is proven.  We are justified in writing $\beta = a /(lh)$.

We next claim that for every $p \in P^{\rm gp}$ there is a (necessarily unique) $\alpha(p) \in A^*$ such that \bne{alphaSSS} \alpha(p) b(p) & = & il(p) . \ene  The argument is similar to that of the previous paragraph: it suffices to check the existence of the $\alpha(e_j)$ and the $\alpha(p_j)$.  We compute $$il(e_j) = ia(r_j) = \eta(r_j) bh(r_j) = \eta(r_j) b(e_j)$$ using \eqref{homotopy}, so we can take $\alpha(e_j) = \eta(r_j)$.  The existence of $\alpha(p_j)$ is trivial since both $b(p_j)$ and $il(p_j)$ are in $A^*$.  We are justified in writing $\alpha = il / b$.

We finally compute $$ i \beta \cdot \alpha h = i(a/(lh)) \cdot (il/b)h = (ia) / (ilh) \cdot (ilh) / (bh) = (ia)/(bh) = \eta.$$  This completes the construction of $(l,\alpha,\beta)$ satisfying the conditions \eqref{liftuptohomotopy}.

\noindent {\bf Case 2.} $h^{\rm gp}$ is injective with torsion-free cokernel.  Here we can choose some splitting $P^{\rm gp} = Q^{\rm gp} \oplus \ZZ^n$, so that our diagram takes the form $$ \xym{ 0 \ar[r] & I \ar[r] & (M')^{\rm gp} \ar[r]^i & M^{\rm gp} \ar[r] & 0 \\ & & Q^{\rm gp} \ar[u]^a \ar[r] & \ar@{.>}[lu] Q^{\rm gp} \oplus \ZZ^n \ar[u]_{b=(b_1,b_2)} } $$  Since $i$ is surjective we can certainly find some group homomorphism $m : \ZZ^n \to M'$ such that $im = b_2$.  Then one checks easily that \be l := (a,m) : Q^{\rm gp} \oplus \ZZ^n & \to & (M')^{\rm gp} \\ \alpha := (\eta, 1) : Q^{\rm gp} \oplus \ZZ^n & \to & A^* \\ \beta := 1 : Q^{\rm gp} & \to & (A')^* \ee satisfy the conditions \eqref{liftuptohomotopy}.

\noindent {\bf Case 3.} $h^{\rm gp}$ is injective with torsion cokernel.  Fix some isomorphism \be (P/Q)^{\rm gp} & = & \ZZ / n_1 \ZZ \oplus \cdots \oplus \ZZ / n_k \ZZ. \ee  Choose $p_j \in P^{\rm gp}$, $j=1,\dots,k$, mapping to the generators of these cyclic summands under the surjection $P^{\rm gp} \to (P/Q)^{\rm gp}$.  Choose $m_j \in (M')^{\rm gp}$ such that $i(m_j) = b(p_j)$.  Since $n_jp_j \in Q^{\rm gp} \subseteq P^{\rm gp}$, \eqref{homotopy} yields \be \eta(n_jp_j) b(p_j)^{n_j} & = & ia(n_j p_j). \ee  This shows that the images of $m_j^{n_j}$ and $a(n_jp_j)$ under $i$ differ by a unit in $A^*$ so by locality, they already differ by a unit in $(A')^*$, so we can write \be u_j m_j^{n_j} & = & a(n_jp_j) \ee for some $u_j \in (A')^*$ with $i(u_j) = \eta(n_jp_j)$.  Adjoin an $n_j^{\rm th}$ root $x_j$ of $u_j$.  Then we can lift in the diagram of abelian groups $$ \xym{ (M')^{\rm gp} \\ Q^{\rm gp} \ar[u]^a \ar[r]^h & P^{\rm gp} \ar@{.>}[lu]_l } $$ by setting $l(p_j) := x_j m_j$ because $$(x_jm_j)^{n_j} = x_j^{n_j} m_j^{n_j} =  u_j m_j^{n_j} = a(n_j p_j). $$  Similarly, we can lift in the diagram of abelian groups $$ \xym{ A^* \\ Q^{\rm gp} \ar[u]^{\eta} \ar[r]^h & P^{\rm gp} \ar@{.>}[lu]_{\alpha}  } $$ by taking $\alpha(p_j) := i(x_j)$ because $i(x_j)$ is an $n_j^{\rm th}$ root of $\eta(n_j p_j)$ since $x_j^{n_j} = u_j$ and $i(u_j) = \eta(n_jp_j)$.

I claim that $\alpha \cdot b = il$.  To see this, we first prove that this equality holds on the image of $h$ by computing $$ ilh  =  ia = \eta \cdot bh = \alpha h \cdot bh$$ (using \eqref{homotopy}) and we next compute $$ il(p_j) = i(x_j m_j) = i(x_j)i(m_j) = \alpha(p_j) b(p_j).$$  We have proved that $(l,\alpha,\beta=1)$ satisfy the conditions \eqref{liftuptohomotopy}.  

For the ``furthermore," suppose we have two lifts up to homotopy $(l',\alpha',\beta')$ and $(l,\alpha,\beta)$.  I claim that for each $p \in P$ there is a (necessarily unique) $\gamma(p) \in (A')^*$ such that $\gamma(p) \cdot l'(p) = l(p)$ in $M'$.  Clearly if such a $\gamma(p)$ exists for every $p \in P$, then $\gamma$ determines a group homomorphism $\gamma : P^{\rm gp} \to (A')^*$ satisfying the first equality in \eqref{homotopybetweenlifts} and clearly there can be at most one $\gamma$ satisfying that equality (because $M'$ is an integral monoid).  To prove the claim, it suffices, by locality of $A' \to A$ and strictness of $M' \to M$, to check that $il'(p)$ and $il(p)$ ``differ by a unit" in $M$.  But we have $il'(p) = \alpha'(p) b(p)$ and $il(p) = \alpha(p) b(p)$, hence we have $$ \alpha(p) \alpha'(p)^{-1} \cdot il'(p) = il(p).$$  This proves not only the existence of $\gamma$ but also the second equality in \eqref{homotopybetweenlifts}.  The third equality in \eqref{homotopybetweenlifts}, like the second, follows formally from the first: $$ \beta \cdot \gamma h = a / (lh) \cdot \gamma h = a /  (\gamma \cdot l')h) \cdot \gamma h = a / (\gamma h \cdot l'h) \cdot \gamma h = a / (l' h) = \beta'.$$   \end{proof}

\subsection{The groupoid fibrations $\AAA^{\rm pre}(P)$} \label{section:AAApreP} Let $P$ be a monoid.  Let \bne{AAApre} \AAA^{\rm pre}(P) & := & [ \AA(P) / \GG(P) ]^{\rm pre} \ene denote the quotient of $\AA(P)$ by the usual action of $\GG(P)$ taken in the $2$-category of groupoid fibrations over schemes.  This is just the ``free quotient" of $\AA(P)$ by $\GG(P)$.  Explicitly, an object of the category $\AAA^{\rm pre}(P)$ is a pair $(X,x)$ where $X$ is a scheme and $x : P \to \O_X(X)$ is a monoid homomorphism.  A morphism in $\AAA^{\rm pre}(P)$ from $(X',x')$ to $(X,x)$ is a \emph{pair} $(f,u)$ consisting of a morphism of spaces $f : X' \to X$ and a group homomorphism $u : P^{\rm gp} \to \O_{X'}^*(X')$ such that $u \cdot x' = g^* x$.  We have a functor from $\AA(P)$ (really, the category of schemes over $\AA(P)$) to $\AAA^{\rm pre}(P)$ which is the identity on objects and which is given on morphisms by $f \mapsto (f,1)$.  We have a groupoid fibration from $\AAA^{\rm pre}(P)$ to schemes given by $(X,x) \mapsto X$ on objects and by $(f,u) \mapsto f$ on morphisms.  Formation of the groupoid fibration $\AAA^{\rm pre}(P)$ is contravariantly functorial in the monoid $P$.

\begin{prop} \label{prop:AAApreP} The groupoid fibration $\AAA^{\rm pre}(P)$ is a prestack whose stackification in the \'etale topology is $\AAA(P)$.  In fact, for any scheme $X$ and any two objects $(X,x)$ and $(X,x')$ in the fiber category of $\AAA^{\rm pre}(P)$ over $X$, the presheaf of fiber category isomorphisms from $(X,x)$ to $(X,x')$ in $\AAA^{\rm pre}(P)$ is representable by a closed subscheme of $X \times \GG(P)$ (finitely presented over $X$ if $P$ is finitely generated). \end{prop}

\begin{proof} If $f : X' \to X$ is a map of schemes, then an isomorphism in the fiber category of $\AAA^{\rm pre}(P)$ over $X'$ from $(X',f^*x)$ to $(X,f^*x')$ is a group homomorphism $u : P \to \O_{X'}^*(X')$ such that $u \cdot f^* x' = x$.  To give such a $u$ is the same thing as giving a map of $X$-schemes $X' \to Z$, where $Z$ is the closed subscheme of \be X \times \GG(P) = \Spec_X \O_X[P^{\rm gp}] \ee defined by the ideal generated by expressions of the form $[p]x'(p) - x(p)$ for $p \in P$.  (One need only let $p$ range over a set of generators for $P$ to generate this ideal, so if $P$ is finitely generated $Z$ is a finitely presented closed subscheme of the scheme $X \times \GG(P)$ which is in turn finitely presented over $X$.)  The fact that the stackification of $\AAA^{\rm pre}(P)$ in the \'etale topology is $\AAA(P)$ is really just the definition of $\AAA(P)$ or the analog of the fact that you form the quotient sheaf by sheafifying the quotient presheaf.  Alternatively you can just argue that your favorite description of $\AAA(P)$ is the stackification of $\AAA^{\rm pre}(P)$ in the \'etale topology.  For example, if you want to think of $\AAA(P)$ as the stack parameterizing \'etale locally trivial principal $\GG(P)$ bundles $E$ with an equivariant map to $\AA(P)$, then you would note that this stack is the stackification of its full subcategory parameterizing such data where $E$ is the trivial bundle (because every principal bundle with equivariant map is specfified by \'etale descent data on an \'etale cover where the principal bundle is trivial) and this full subcategory is in turn clearly equivalent to $\AAA^{\rm pre}(P)$ as defined above. \end{proof} 

\subsection{$\AAA(h)$ \'etale for $h$ strict}  Recall that a map of monoids $h : Q \to P$ is called \emph{strict} iff the induced map $\ov{h} : \ov{Q} \to \ov{P}$ of sharp monoids is an isomorphism (note $\ov{Q} := Q/Q^*$).  The purpose of this section is to prove the following:

\begin{thm} \label{thm:AAAhetale} Let $h : Q \to P$ be a strict map of fine monoids.  Then the induced map of algebraic stacks $\AAA(h) : \AAA(P) \to \AAA(Q)$ is representable \'etale. \end{thm}

Before giving the proof we note that this result follows from general results of Olsson \cite{Ols} in at least two different ways.  First, one can argue that strictness of $h$ ensures that \bne & \xym{ \AAA(P) \ar[rr]^-{\AAA(h)} \ar[rd] & & \AAA(Q) \ar[ld] \\ & \Log } \ene commutes, where the diagonal arrows are the \'etale maps of Theorem~\ref{thm:etalecover} (equals \cite[5.25]{Ols}), hence $\AAA(h)$ is \'etale by ``two-out-of-three" \cite[I.4.8]{SGA1}.  Alternatively, one quotes \cite[5.23]{Ols} and uses the fact that a strict map of log schemes (or stacks) is log \'etale iff the underlying map of schemes (or stacks) is \'etale.  We have chosen to give a direct proof of Theorem~\ref{thm:AAAhetale} because: \begin{enumerate} \item The result is a simple self-contained statement making no reference to log geometry, so it seems strange to extract it from generalities on log stacks. \item This result is ultimately the only fact we need from Olsson's whole theory of log stacks, so we can thus make the present paper entirely self-contained. \item The proof is enlightening, direct, and elementary, if a little tedious. \end{enumerate}

\begin{lem} \label{lem:strict}  Consider the following conditions for a map of monoids $h : Q \to P$: \begin{enumerate} \item \label{unitpush} The obvious map from $h^* : Q^* \to P^*$ to $h : Q \to P$ is a pushout diagram. \item \label{grouppush} $h$ is a pushout of a map of groups. \item \label{strict} $h$ is strict. \end{enumerate} Then \eqref{unitpush}$\implies$\eqref{grouppush}$\implies$\eqref{strict} and the three conditions are equivalent if $Q$ and $P$ are integral. \end{lem}

\begin{proof} Exercise. \end{proof}

\begin{lem} \label{lem:AAAhrepresentable} Suppose $h : Q \to P$ is a map of finitely generated monoids inducing a surjection $\ov{h} : \ov{Q} \to \ov{P}$ on sharpenings.  Then the induced map of algebraic stacks $\AAA(h) : \AAA(P) \to \AAA(Q)$ is representable. \end{lem}

\begin{proof} By Proposition~\ref{prop:AAApreP}, Proposition~\ref{prop:formallyrepresentableimpliesrepresentable}, and Lemma~\ref{lem:stackificationpreservesrepresentability}, it suffices to prove that the corresponding map of prestacks $$\AAA^{\rm pre}(h) : \AAA^{\rm pre}(P) \to \AAA^{\rm pre}(Q)$$ is formally representable.  We will check the criterion \eqref{autinjectivity} of Proposition~\ref{prop:representablebypresheaves}.  Fix a scheme $X$ and an object $(X,x)$ of $\AAA^{\rm pre}(P)$ over $X$.  We need to prove that the map induced by $\AAA^{\rm pre}(h)$ from automorphisms of $(X,x)$ over $X$ to automorphisms of $(X,xh)$ over $X$ in $\AAA^{\rm pre}(P)$ is injective.  That is, we need to prove that $u = 1$ whenever $u : P^{\rm gp} \to \O_X^*(X)$ is a group homomorphism satisfying $uh^{\rm gp}=1$ and $u \cdot x = x$.  By the universal property of groupification we can prove $u=1$ by proving that $u(p) = 1$ for each $p \in P$ (abusively writing $p$ for its image in $P^{\rm gp}$).  By the hypothesis on $h$ we can write $p = h(q)+v$ for some $q \in Q$, $v \in P^*$.  The condition $u(v)x(v)=x(v)$ implies that $u(v)=1$ because $x(v) \in \O_X^*(X)$ since $v \in P^*$ and then $uh^{\rm gp}=1$ implies $u(h(q))=1$ hence $u(p) = u(h(q))u(v)=1$. \end{proof}  

\noindent \emph {Proof of Theorem~\ref{thm:AAAhetale}.}  Since we assume $P$ and $Q$ are fine, the algebraic stacks in question are of finite presentation over $\Spec \ZZ$ and the map $\AAA(h)$ is representable by Lemma~\ref{lem:AAAhrepresentable}, so it suffices to prove that $\AAA(h)$ is formally \'etale in the sense of Definition~\ref{defn:formallysmooth} in \S\ref{section:formalsmoothness} (c.f.\ Remark~\ref{rem:formalsmoothness}).    The map $\AAA(h)$ is the stackificiation of the map $\AAA^{\rm pre}(h)$ in the \'etale topology, so it suffices to check the lifting conditions in Lemma~\ref{lem:formaletalenesscriterion} for the map $\AAA^{\rm pre}(h)$.  

Let $i : A' \to A$ be a surjection of rings with with square-zero kernel.  Consider a $2$-commutative diagram \bne{2diatolift} & \xym{ \Spec A \ar[d] \ar[r] & \AAA^{\rm pre}(P) \ar[d]^{\AAA^{\rm pre}(h)} \\ \Spec A' \ar@{.>}[ru] \ar[r] & \AAA^{\rm pre}(Q) } \ene of groupoid fibrations.  We must prove that, after possibly passing to an fppf cover of $A'$ (replacing $A' \to A$ with $B' \to B = B' \otimes_{A'} A$ for an fppf ring map $B' \to A'$), there is a lift $(l,\alpha,\beta)$ in \eqref{2diatolift} as in \S\ref{section:formalsmoothness}.  We must also prove that any two such lifts are homotopic by a (necessarily unique) homotopy compatible with the $\alpha$'s and $\beta$'s, at least after passing to an fppf cover of $A'$.

\noindent {\bf Notation:}  We write $h^* : Q^* \to P^*$ for the map on units induced by a map of monoids $h : Q \to P$.  The superscript $*$ has no other meaning having to do with composition of functions.  Juxtaposition is composition, monoids are written multiplicatively, and $f \cdot g : Q \to P$ denotes the product of two monoid homomorphisms $f,g : Q \rightrightarrows P$.   Notation for the map $\iota : Q \to Q^{\rm gp}$ is always suppressed and we often write ``$h$" when we mean ``$h^{\rm gp}$."  If $f : Q^{\rm gp} \to P^*$ is a group homomorphism, and $g : Q \to P$ is a monoid homomorphism, then $f \cdot g$ also abusively denotes the monoid homomorphism $Q \to P$ given by $q \mapsto f(\iota(q))g(q)$.

The $2$-commutative diagram \eqref{2diatolift} is the same thing as a diagram of monoids \bne{monoiddiagramtolift} & \xym{ A & \ar[l]_-b P \\ A' \ar[u]^i & \ar[l]^a  Q \ar[u]_h } \ene with a homotopy $\eta$ connecting the two ways around the square:  that is, a monoid homomorphism $\eta : Q^{\rm gp} \to A^*$ such that \bne{etahomotopy} \eta \cdot bh & = & ia. \ene  The lift $(l,\alpha,\beta)$ is the same thing as a monoid homomorphism $l:P \to A'$, a group homomorphism $\alpha : P^{\rm gp} \to A^*$, and a group homomorphism $\beta: Q^{\rm gp} \to (A')^*$ satisfying the conditions \bne{lift1:1} \alpha \cdot b & = & il \\ \label{lift1:2} \beta \cdot lh & = & a \\ \label{lift1:3} \eta & = & i^* \beta \cdot \alpha h. \ene  We construct $(l,\alpha,\beta)$, check the conditions \eqref{lift1:1}, \eqref{lift1:2}, \eqref{lift1:3}, and prove uniqueness of $(l,\alpha,\beta)$ up to homotopy in several steps.

\noindent {\bf Step 1.} By restricting \eqref{monoiddiagramtolift} to the units we obtain a diagram of monoids (in fact groups) \bne{monoiddiagramtoliftunits} & \xym{ A^* & \ar[l]_-{b^*} P^* \\ (A')^* \ar[u]^{i^*} & \ar[l]^{a^*}  Q \ar[u]_{h^*} } \ene which commutes up to homotopy in the sense that $\eta|_{Q^*} \cdot b^*h^* = i^* a^*$.  By Homotopy Lifting (Theorem~\ref{thm:homotopylifting} for the case of trivial log structures) we can find, after possibly passing to an fppf cover of $A'$, a lift up to homotopy $(l^*,\ov{\alpha},\ov{\beta})$ in \eqref{monoiddiagramtoliftunits}: that is, a monoid homomorphism $l^* : P^* \to (A')^*$, a group homomorphism $\ov{\alpha} : P^* \to A^*$, and a group homomorphism $\ov{\beta} : Q^* \to (A')^*$ such that \bne{lift2:1} \ov{\alpha} \cdot b^* & = & i^* l^* \\ \label{lift2:2} \ov{\beta} \cdot l^* h^* & = & a^* \\ \label{lift2:3} \eta|_{Q^*} & = & i^* \ov{\beta} \cdot \ov{\alpha} h^*. \ene

\noindent {\bf Step 2.} The diagram of monoids (in fact groups) \bne{monoiddiagram2} & \xym{ A^* & \ar[l]_-{\eta} Q^{\rm gp} \\ (A')^* \ar[u]^{i^*} & \ar[l]^{\ov{\beta}}  Q^* \ar@{^(->}[u] }  \ene commutes up to the homotopy $\ov{\alpha}^{-1} h^*$ in the sense that \bne{nexthomotopy} \ov{\alpha}^{-1} h^* \cdot \eta|_{Q^*} & = & i^* \beta \ene by \eqref{lift2:3}.  Since $Q^* \into Q^{\rm gp}$ is injective, Homotopy Lifting (Theorem~\ref{thm:homotopylifting} for the case of trivial log structures using the fact that we can take $\beta=1$ there) yields, after possibly passing to an fppf cover of $A'$, a lift up to homotopy $(\beta,\delta,1)$ in \eqref{monoiddiagram2}: that is, a group homomorphism $\beta : Q^{\rm gp} \to (A')^*$ and a group homomorphism $\delta : Q^{\rm gp} \to A^*$ such that \bne{lift3:1} \delta \cdot \eta & = & i^* \beta \\ \label{lift3:2} \beta|_{Q^*} & = & \ov{\beta} \\ \label{lift3:3} \ov{\alpha}^{-1} h^* & = & \delta|_{Q^*}. \ene

\noindent {\bf Step 3.}  The diagram of monoids \bne{Step3} & \xym@C+15pt{ & P \ar@{.>}[ld]_l & \ar[l] P^* \\ A' & \ar[l]^{\beta^{-1} \cdot a} Q \ar[u]_h & \ar[l] Q^* \ar[u]_{h^*} } \ene admits a completion as indicated (to a strictly commutative diagram) with $l|_{P^*} = l^*$ (as our notation convention would suggest) because the square is a pushout since $h$ is a strict map of fine monoids (Lemma~\ref{lem:strict}) and we compute \be ( \beta^{-1} \cdot a)|_{Q^*} & = & \beta|_{Q^*}^{-1} \cdot a|_{Q^*} \\ & = & \ov{\beta}^{-1} \cdot a^* \\ & = & l^* h^* \ee using \eqref{lift3:2} and \eqref{lift2:2}.  Condition \eqref{lift1:2} holds by commutativity of the triangle in \eqref{Step3}.

\noindent {\bf Step 4.}  The diagram of groups \bne{Step4} & \xym@C+15pt{ & P^{\rm gp} \ar@{.>}[ld]_{\alpha^{-1}} & \ar[l] P^* \\ A' & \ar[l]^{\delta} Q^{\rm gp} \ar[u]_h & \ar[l] Q^* \ar[u]_{h^*} } \ene admits a completion as indicated with $\alpha|_{P^*} = \ov{\alpha}$ because of \eqref{lift3:3} and the fact that the square is a pushout since it is obtained by groupifying the pushout square of \eqref{Step3} and groupification preserves direct limits.  To check \eqref{lift1:3} we compute \be \eta & = & i^* \beta \cdot \delta^{-1} \\ & = & i^* \beta \cdot \alpha h \ee using \eqref{lift3:1} and commutativity of the triangle in \eqref{Step4}.

\noindent {\bf Step 5.} We check \eqref{lift1:1} as follows.  Since the square in \eqref{Step3} is a pushout, it suffices to compute \be ( \alpha \cdot b)|_{P^*} & = & \alpha|_{P^*} \cdot b|_{P^*} \\ & = & \ov{\alpha} \cdot b^* \\ & = & i^* l^* \ee using the equality $\alpha|_{P^*} = \ov{\alpha}$ from Step 4 and \eqref{lift2:1} and \be (\alpha \cdot b)h & = & \alpha h \cdot bh \\ & = & \delta^{-1} \cdot bh \\ & = & \delta^{-1} \cdot \eta^{-1} \cdot ia \\ & = & i^* \beta^{-1} \cdot ia \\ & = & i(\beta^{-1} \cdot a) \\ & = & ilh \ee using commutativity of the triangle in \eqref{Step4}, \eqref{etahomotopy}, \eqref{lift3:1}, and \eqref{lift1:2}.

\noindent {\bf Step 6.}  Now suppose $(l',\alpha',\beta')$ is another lift, so the conditions \eqref{lift1:1}', \eqref{lift1:2}', \eqref{lift1:3}' obtained by replacing $l,\alpha,\beta$ with $l',\alpha',\beta'$ in conditions \eqref{lift1:1}, \eqref{lift1:2}, \eqref{lift1:3} (respectively) are also satisfied.  We must prove that, after possibly passing to an fppf cover of $A'$, there is a homotopy $\gamma$ from $(l',\alpha',\beta')$ to $(l,\alpha,\beta)$ compatible with the $\alpha$'s and $\beta$'s---that is, a group homomorphism $\gamma: P^{\rm gp} \to (A')^*$ satisfying the conditions \bne{h1} \gamma \cdot l' & = & l \\ \label{h2} i \gamma \cdot \alpha' & = & \alpha \\ \label{h3} \beta' & = & \beta \cdot \gamma h. \ene

\noindent {\bf Step 7.} By restricting to units, we obtain two homotopy lifts $((l')^*,\ov{\alpha}',\ov{\beta}')$ and $(l^*,\ov{\alpha},\ov{\beta})$ in the diagram of groups \eqref{monoiddiagramtoliftunits} of Step 1.  So, by the uniqueness up to homotopy in Homotopy Lifting (Theorem~\ref{thm:homotopylifting} for the case of trivial log structures) there is, after possibly passing to an fppf cover of $A'$, a group homomorphism $\ov{\gamma} : P^* \to (A')^*$ satisfying the conditions \bne{h1star} \ov{\gamma} \cdot (l')^* & = & l^* \\ \label{h2star} i^* \ov{\gamma} \cdot \ov{\alpha}' & = & \ov{\alpha} \\ \label{h3star} \ov{\beta}' & = & \ov{\beta} \cdot \ov{\gamma} h^*. \ene

\noindent {\bf Step 8.}  The diagram of groups \bne{Step8} & \xym@C+15pt{ & P^{\rm gp} \ar@{.>}[ld]_{\gamma} & \ar[l] P^* \\ A' & \ar[l]^{\beta' \cdot \beta^{-1}} Q^{\rm gp} \ar[u]_h & \ar[l] Q^* \ar[u]_{h^*} } \ene admits a completion as indicated with $\gamma|_{P^*} = \ov{\gamma}$ because of \eqref{h3star} and the fact that the square is a pushout.  The condition \eqref{h3} is clear from the commutativity of the triangle.

\noindent {\bf Step 9.} To check condition \eqref{h1} we again use the fact that the square of monoids in \eqref{Step3} is a pushout, so it suffices to compute \be (\gamma \cdot l')|_{P^*} & = & \gamma|_{P^*} \cdot l'|_{P^*} \\ & = & \ov{\gamma} \cdot (l')^* \\ & = & l^* \\ & = & l|_{P^*} \ee using \eqref{h1star} and \be (\gamma \cdot l')h & = & \gamma h \cdot l'h \\ & = & \beta' \cdot \beta^{-1} \cdot l' h \\ & = & \beta^{-1} \cdot a \\ & = & lh \ee using the condition \eqref{h3} checked in the previous step, \eqref{lift1:2}', and \eqref{lift1:2}.

\noindent {\bf Step 10.}  To check condition \eqref{h2} we note that the square of groups in \eqref{Step4} (and in \eqref{Step8}) is a pushout, so it suffices to compute \be (i \gamma \cdot \alpha')|_{P^*} & = & i^* \ov{\gamma} \cdot \ov{\alpha}' \\ & = & \ov{\alpha} \\ & = & \alpha|_{P^*} \ee using \eqref{h2star} and \be (i \gamma \cdot \alpha')h & = & i \gamma h \cdot \alpha' h \\ & = & i( \beta' \cdot \beta^{-1}) \cdot \alpha' h \\ & = & i^* \beta' \cdot i^* \beta^{-1} \cdot \alpha'h \\ & = & i^* \beta^{-1} \cdot \eta \\ & = & \alpha h \ee using \eqref{h3}, \eqref{lift1:3}', and \eqref{lift1:3}.

The proof of Theorem~\ref{thm:AAAhetale} is complete. \hfill \qedsymbol

\begin{rem} \label{rem:HomotopyLifting} Theorem~\ref{thm:homotopylifting} can also be used to give a direct proof of Olsson's result \cite[5.23]{Ols} asserting that $\AAA(h)$ is \emph{log} \'etale for \emph{any} map $h : Q \to P$ of fine monoids (when the algebraic stacks $\AAA(P)$ and $\AAA(Q)$ are given the natural log structure).  Indeed, one reduces by the same general nonsense we used to prove Theorem~\ref{thm:AAAhetale} to proving that the corresponding map of log prestacks has the fppf local right lifting property with respect to strict square-zero closed embeddings of affine schemes.  This latter statement unravels exactly to the statement of Homotopy Lifting. \end{rem}

\subsection{$\L(h) \to \Log(\AA(Q))$ \'etale for $h$ monic} \label{section:LhtoLogAQetale} Let $h : Q \to P$ be a map of fine monoids with quotient $q : P \to P/Q$.  The group scheme $\GG(P/Q)$ then acts on $\AA(P)$ through the map $\GG(q) : \GG(P/Q) \to \GG(P)$ and the usual action of $\GG(P)$ on $\AA(P)$.  Let \be \L(h) & := & [  \u{\AA}(P)  /  \GG(P/Q)  ]  \ee be the quotient stack.  Since $\GG(P/Q)$ acts on $\AA(P)$ through automorphisms of log schemes over $\AA(Q)$, there is a natural map of algebraic stacks \bne{LhtoLogAQ} \L(h) & \to & \Log(\AA(Q)) \ene which will be carefully explained in what follows.  The goal of this section is to prove:

\begin{thm} \label{thm:LhtoLogAQetale} Assume $h : Q \into P$ is an injective map of fine monoids.  Then \eqref{LhtoLogAQ} is a representable \'etale map of algebraic stacks. \end{thm}

The proof of Theorem~\ref{thm:LhtoLogAQetale} is similar to that of Theorem~\ref{thm:AAAhetale}.  We first let \be \L(h,t)^{\rm pre} & := & [ \u{\AA}(P) / \GG(P/Q) ]^{\rm pre} \ee be the quotient of $\AA(P)$ by $\GG(P/Q)$ in the $2$-category of groupoid fibrations over schemes (c.f.\ the analogous description of $\AAA^{\rm pre}(P)$ in \S\ref{section:AAApreP}).  Explicitly, an object of $\L(h)^{\rm pre}$ is a pair $(X,x)$ where $X$ is a scheme and $x : P \to \O_X(X)$ is a monoid homomorphism.  A morphism $(f,w) : (X',x') \to (X,x)$ is a pair consisting of a morphism of schemes $f :X' \to X$ and a group homomorphism $w : P/Q \to \O_{X'}^*(X')$ such that $wq \cdot x' = f^*x$.  Note that this latter condition implies that \bne{killQ} f^*xh & = & (wq \cdot x')h  \\ \nonumber & = & (wqh) \cdot (x'h) \\ \nonumber & = & x'h \ene since $qh=1$.  The groupoid fibration from $\L(h,t)^{\rm pre}$ to schemes is given by $(X,x) \mapsto X$ on objects and $(f,w) \mapsto f$ on morphisms.  In fact, \eqref{killQ} shows that we also have a groupoid fibration from $\L(h)^{\rm pre}$ to schemes over $\u{\AA}(Q)$ given by $(X,x) \mapsto (X,xh)$ on objects and $(f,w) \mapsto f$ on morphisms, viewing a scheme over $\u{\AA}(Q)$ as pair $(X,y)$ consisting of a scheme $X$ and a monoid homomorphism $y : Q \to \O_X(X)$.

\begin{prop} \label{prop:Lhtpre} The groupoid fibration $\L(h)^{\rm pre}(P)$ is a prestack whose stackification in the \'etale topology is $\L(h)$.  In fact, for any scheme $X$ and any two objects $(X,x)$ and $(X,x')$ in the fiber category of $\L(h)^{\rm pre}$ over $X$, the presheaf of fiber category isomorphisms from $(X,x)$ to $(X,x')$ in $\L(h)^{\rm pre}$ is representable by a closed subscheme of $X \times \GG(P/Q)$ finitely presented over $X$. \end{prop}

\begin{proof} The proof is essentially the same as that of Proposition~\ref{prop:AAApreP}.  To give an isomorphism from $(X,x)$ to $(X,x')$ over $X$ is to give a group homomorphism $w : (P/Q)^{\rm gp} \to \O_X^*(X)$ such that $wq \cdot x = x'$ as monoid homomorphims $P \to \O_X(X)$.  It is enough to check this last equality on some finite set $p_1,\dots,p_k$ generating $P$, so we see easily that the ``Isom presheaf" in question is represented by the closed subscheme of \be X \times \GG(P/Q) & = & \Spec_X \O_X[(P/Q)^{\rm gp}] \ee defined by the equations $[q(p_i)]x(p_i)-x'(p_i)$ in the ring of global sections of $X \times \GG(P/Q)$. \end{proof}

For a log scheme $X$, a morphism a log schemes $X \to \AA(Q)$ is the same thing as a monoid homomorphism $x : Q \to \M_X(X)$.  We can therefore view the stack $\Log(\AA(Q))$ as the category whose objects are triples $(X,\M_X,x)$ where $X$ is a scheme, $\M_X$ is a fine log structure on $X$, and $x : Q \to \M_X(X)$ is a monoid homomorphism.  A $\Log(\AA(Q))$ morphism $f : (X',\M_{X'},x') \to (X,\M_X,x)$ is a strict morphism of fine log schemes $(f,f^\dagger) : (X',\M_{X'}) \to (X,\M_X)$ such that $f^\dagger x = x'$.  We often refer to $f^\dagger : f^* \M_X \to \M_{X'}$ as an \emph{isomorphism of log structures under} $Q$ in this situation.  There is an obvious forgetful functor from $\Log(\AA(Q))$ to schemes over $\u{\AA}(Q)$ taking $(X,\M_X,x)$ to $(X,\alpha_X x)$.  This is just the usual structure map $\Log(\AA(Q)) \to \u{\AA}(Q)$ for the stack $\Log(\AA(Q))$ of \S\ref{section:Log}.

We now construct a functor \bne{LhpretoLogAQ} \L(h)^{\rm pre} & \to & \Log(\AA(Q)) \ene commuting (strictly) with the functors to schemes over $\u{\AA}(Q)$.  On objects, we map $(X,x)$ to $(X,P^a_x,h)$ where $x^a : P^a_x \to \O_X(X)$ is the (manifestly fine) log structure associated to the prelog structure $x : P \to \O_X(X)$, and $P \to P^a_x(X)$ is the map induced by the canonical map $\underline{P} \to P^a_x$ of prelog structures on $X$, and $h$ is abuse of notation for the monoid homomorphism obtained by precomposing $P \to P^a_x(X)$ with $h : Q \to P$.  Here $\underline{P}$ is the constant sheaf on $X$ associated to $P$ and we are making use of the natural map $P \to \underline{P}(X)$.  Recall that: \begin{enumerate} \item $P^a_x = \underline{P} \oplus_{x^{-1}\O_X^*} \O_X^*$, where $x$ abusively denotes the morphism $\underline{P} \to \O_X$ corresponding to $x :P \to \O_X(X)$. \item The structure map $x^a : P^a_x \to \O_X$ for the log structure $P^a_x$ is given by $[p,u] \mapsto x(p)u$. \item The canonical map $\underline{P} \to P^a_x$ is given by $p \mapsto [p,1]$. \end{enumerate}   Notice that the composition of the map $P \to P^a_x(X)$ induced by the canonical map of prelog structures and the map $x^a : P^a_x(X) \to \O_X(X)$ is the original monoid homomorphism $x$.  The image under \eqref{LhpretoLogAQ} of a morphism $(f,w) : (X',x') \to (X,x)$ is defined as follows.  First note that formation of associated log structures commutes with pullback, so we have a canonical identification $f^* (P_x^a) = P_{f^*x}^a$.  The formula \bne{formulaforwq} \cdot wq : P^a_{f^*x} & \to & P^a_{x'} \\ \nonumber [p,u] & \mapsto & [p,(wq)(p)u] \ene defines a morphism of log structures on $X'$ (i.e.\ respects the structure maps to $\O_X$) because we compute \be ((f^*x)^a)[p,u] & = & f^*x(p)u \\ & = & x'(p)(wq)(p)u \\ & = & (x')^a[p,(wq)(p)u] \ee using $wq \cdot x' = f^*x$.  In fact, \eqref{formulaforwq} is clearly an \emph{iso}morphism of log structures, so it defines a lifting of $f : X' \to X$ to a \emph{strict} morphism of log schemes $(f,\cdot wq)$.  A computation much life \eqref{killQ} (using $qh=1$), shows that in fact \eqref{formulaforwq} is an isomorphism of log structures under $Q$---that is, a morphism in $\Log(\AA(Q))$ lifting $f : X' \to X$.  This completes the construction of \eqref{LhpretoLogAQ}.  We declare \eqref{LhtoLogAQ} to be the stackification of \eqref{LhpretoLogAQ} in the \'etale topology.

\begin{lem} \label{lem:LhtoLogAQrepresentable} For any map $h : Q \to P$ of fine monoids, the morphism \eqref{LhtoLogAQ} is representable by algebraic spaces. \end{lem}

\begin{proof}  The proof is much like that of Lemma~\ref{lem:AAAhrepresentable}.  By Proposition~\ref{prop:Lhtpre}, Proposition~\ref{prop:formallyrepresentableimpliesrepresentable}, and Lemma~\ref{lem:stackificationpreservesrepresentability}, it suffices to prove that the map of prestacks \eqref{LhpretoLogAQ} is formally representable.  We will check the criterion \eqref{autinjectivity} of Proposition~\ref{prop:representablebypresheaves}. 

Fix an object $(X,x)$ of the fiber category of $\L(h)^{\rm pre}$ over a scheme $X$.  We must show that \eqref{LhpretoLogAQ} induces an injection from the group of automorphisms of $(X,x)$ over $X$ to the group of automorphisms of $(X,P^a_x)$ over $X$.  The former group is the subgroup of $\Hom_{\Ab}((P/Q)^{\rm gp},\O_X^*(X))$ consisting of those $w$ for which $wq \cdot x = x$.  The image of such a $w$ under \eqref{LhpretoLogAQ} is the automorphism $\cdot wq : P^a_x \to P^a_x$ of the log structure $P^a_x$ under $Q$.  We can rewrite this automorphism of log structures as $[p,u] \mapsto wq(p) \cdot [p,u]$, where the $\cdot$ is the action of $\O_X^*$ on $P^a_x$.  But $P$ integral certainly implies that $P^a_x$ is a ``quasi-integral" log structure in the sense that the action of $\O_X^*$ on $P^a_x$ is free, so the only way $[p,u] \mapsto wq(p) \cdot [p,u]$ can be the identity map is if $wq(p) = 1$ for all $p \in P$, which happens iff $w=1$ because $q : P \to (P/Q)^{\rm gp}$ is an epimorphism in the category of monoids. \end{proof}

\noindent \emph{Proof of Theorem~\ref{thm:LhtoLogAQetale}.}  Since \eqref{LhtoLogAQ} is a representable map between algebraic stacks of locally finite presentation over $\Spec \ZZ$, we reduce exactly as in the proof of Theorem~\ref{thm:AAAhetale} to checking the lifting conditions in Lemma~\ref{lem:formaletalenesscriterion} for the map of prestacks \eqref{LhpretoLogAQ}.

Consider a $2$-commutative diagram \bne{liftingdiagram} & \xym@C+20pt{ T \ar[d] \ar[r]^-x & \L(h)^{\rm pre} \ar[d]^I{\eqref{LhpretoLogAQ}} \\ T' \ar[r]_-{(\M',a)} \ar@{.>}[ru]^-{x'} & \Log(\AA(Q)) } \ene where $T \to T' = \Spec( i : A' \to A)$ for a surjection of rings $i : A' \to A$ with square-zero kernel.  We must show that, after possibly passing to an fppf cover of $T'$, there is a lift $(x',w,l)$ up to homotopy compatible with the homotopy relating the two ways around the square and that any two such homotopy lifts are homotopic (by a homotopy compatible with the ones making the triangles commute), after possibly passing to an fppf cover of $T'$.

We unravel this as follows.  The top horizontal arrow $x$ in \eqref{liftingdiagram} corresponds to a monoid homomorphism $x : P \to A$, and the bottom horizontal arrow $(\M',a)$ corresponds to a fine log structure $\alpha' : \M' \to \O_{T'}$ on $T'$ together with a monoid homomorphism $a : Q \to M'$, where $M' := \M'(T')$.   Let $\alpha : \M \to \O_T$ be the restriction of $\M'$ to $T$ and let $M := \M(T)$.  We abusively denote the natural map $i : M' \to M$.  We are implicitly given a homotopy $b$ relating the two ways around the square \eqref{liftingdiagram}.  This $b$ is an isomorphism $b : P_x^a \to \M$ of log structures on $T$ under $Q$.  That is, the diagram \bne{D1} & \xym{ P \ar@/^2pc/[rr]^x \ar[r] & P^a_x(T) \ar[r] \ar[d]_{\cong}^b & A \\ Q \ar[u]^h \ar[r]_{i a} & M \ar[ru]_{\alpha} } \ene commutes and hence, suppressing notation for $P \to P^a_x(T)$ we have a commutative diagram \bne{lrdiagram} & \xym{ P \ar[r]^b & M \ar[r]^{\alpha} & A \\ Q \ar[u]^h \ar[r]_a & M' \ar[r]_{\alpha'} \ar[u]_{i} & A' \ar[u]_i } \ene where the square on the right is a strict map of fine log rings.

A lift up to homotopy $(x',w,l)$ in \eqref{liftingdiagram} is a triple where $x' : P' \to A'$ is a monoid homomorphism, $w : (P/Q)^{\rm gp} \to A^*$ is a group homomorphism such that \bne{wcondition} wq \cdot x' & = & x, \ene and $l^a : P^a_{x'} \to \M'$ is an isomorphism of log structures on $T'$ under $Q$ making the following diagram of isomorphisms of log structures on $T$ under $Q$ commute: \bne{lcondition} \xym{ P^a_{ix'} \ar[d]_{ \cdot wq} \ar[r]^{l^a|_T} & \M \\ P^a_x \ar[ru]_b } \ene  The monoid homomorphism $x'$ corresponds to the morphism $x'$ in \eqref{liftingdiagram}, the group homomorphism $w$ satisfying \eqref{wcondition} is the data of a homotopy between the two ways around the upper triangle of \eqref{liftingdiagram}, the isomorphism $l$ is a homotopy between the two ways around the bottom triangle of \eqref{liftingdiagram}, and the commutativity condition \eqref{lcondition} is the condition that $w$ and $l$ should be compatible with the given homotopy $b$ (note that $P^a_{ix'} = i^* P^a_{x'}$).  The condition that $l^a$ is an isomorphism of log structures under $Q$ means the following diagram of monoids commutes: \bne{D2} & \xym{ P \ar@/^2pc/[rr]^{x'} \ar[r] & P^a_{x'}(T') \ar[r] \ar[d]_{\cong}^{l^a} & A' \\ Q \ar[u]^h \ar[r]_{ a} & M' \ar[ru]_{\alpha'} } \ene

The next step is to unravel even further.  I claim that the data $(x',w,l^a)$ as in the above paragraph is the same thing as a pair $(l,w)$ consisting of a monoid homomorphism $l : P \to M'$ and a group homomorphism $w : (P/Q)^{\rm gp} \to A^*$ making \bne{lrdiagram2} & \xym{ P \ar[r]^{wq \cdot b} \ar@{.>}[rd]^l & M \ar[r]^{\alpha} & A \\ Q \ar[u]^h \ar[r]_a & M' \ar[r]_{\alpha'} \ar[u]_{i} & A' \ar[u]_i } \ene commute (i.e.\ making the left square commute).  Given $(x',w,l^a)$ as in the previous paragraph, we obtain the $(l,w)$ as in this paragraph by letting $l$ be the map $l : P \to M'$ given by the composition in \eqref{D2}.   This choice of $l$ clearly makes the lower triangle in \eqref{lrdiagram2} commute, and one sees that the upper triangle commutes using the commutativity of \eqref{lcondition}.  Given $(l,w)$ as in this paragraph, we obtain $(x',w,l^a)$ as in the previous paragraph by setting $x' := \alpha' l$ and letting $l^a : P^a_{x'} \to M'$ be the map on associated log structures induced by $l$.  The commutativity of \eqref{lrdiagram2} ensures the commutativity of \eqref{lcondition}, which ensures that $l^a|_T$ is an isomorphism (by ``two-out-of-three" because $b$ and $\cdot wq$ are isomorphisms), which in turn ensures that $l^a$ itself is an isomorphism because this can be checked on characteristics since all log structures in question are integral and the map on characteristics induced by $l^a$ is the same as the one induced by $l^a|_T$ since $T$ and $T'$ have the same \'etale topos and the characteristic of the pullback is always the inverse image of the characteristic.

To prove the existence of such a pair $(l,w)$ (after possibly passing to an fppf cover of $T'$), we need only apply Lemma~\ref{lem:lifting} to the diagram \eqref{lrdiagram}.  For uniqueness up to homotopy, suppose $(l_1,w_1)$ and $(l_2,w_2)$ are both as in the previous paragraph.  For any $p \in P$, I claim there is a (necessarily unique) unit $v(p) \in (A')^*$ such that $v(p) \cdot l_1(p) = l_2(p)$ in $M'$.  Indeed, it suffices to show that $l_1(p)$ and $l_2(p)$ have the same image in the characteristic of $\M'$ and since $\M' \to \M$ induces an isomorphism on characteristics by strictness, it suffices to show that $i l_1(p)$ and $il_2(p)$ differ by a unit in $M$, and in fact the commutativity conditions \eqref{lrdiagram2} for $(l_1,w_1)$ and $(l_2,w_2)$ ensure that \be (w_2q \cdot  w_1^{-1}q)(p) \cdot l_1i(p) & = & l_2 i (p). \ee  The same commutativity condition shows that $vh=1$, so we can view $v$ as a group homomorphism $v : (P/Q)^{\rm gp} \to (A')^*$ satisfying \bne{vconditions} vq \cdot l_1 & = & l_2 \\ w_1 \cdot i^* v & = & w_2, \ene where $i^* : (A')^* \to A^*$ is the map on units induced by $i : A' \to A$.  Such a $v$ is a homotopy from $l_1$ to $l_2$ compatible with $w_1$ and $w_2$.  The proof of Theorem~\ref{thm:LhtoLogAQetale} is complete. \hfill \qedsymbol

\subsection{An equivalence of stacks} \label{section:equivalence} Fix a homomorphism $h : Q \to P$ of fine monoids, a scheme $\u{Y}$, and a monoid homomorphism $t : Q \to \O_Y(Y)$.  View $Y$ as a fine log scheme by pulling back the log structure on $\AA(Q)$ along $t$, so $t : Y \to \AA(Q)$ is a strict map of fine log schemes.  Define an algebraic stack $\L(\u{Y},h,t)$ by the $2$-cartesian diagram \bne{D10} & \xym{ \L(\u{Y},h,t) \ar[r] \ar[d] & \L(h) \ar[d] \\ \Log(Y) \ar[r]^-{\Log(t)} \ar[d] & \Log(\AA(Q)) \ar[d] \\ \u{Y} \ar[r]^-{\u{t}} & \u{\AA}(Q) } \ene where $\L(h) \to \Log(\AA(Q))$ is the map \eqref{LhtoLogAQ} discussed in \S\ref{section:LhtoLogAQetale}.  Note that the bottom square of \eqref{D10} is cartesian because $t : Y \to \AA(Q)$ is strict (\S\ref{section:limitpreservation}).  Composing $\u{t}$ with the natural map $\u{\AA}(Q) \to \AAA(Q)$ of algebraic stacks, we have a map $\u{Y} \to \AAA(Q)$.  Define $\M(\u{Y},h,t)$ by the $2$-cartesian diagram of algebraic stacks: $$ \xym{ \M(\u{Y},h,t) \ar[r] \ar[d] & \AAA(P) \ar[d]^{\AAA(Q)} \\ \u{Y} \ar[r] & \AAA(Q) } $$  There is a map of algebraic stacks \bne{I} \L(\u{Y},h,t) & \to & \M(\u{Y},h,t) \ene which we will explain carefully in a moment.  The map \eqref{I} fits in a $2$-commutative diagram \bne{LMLog} & \xym{ \L(\u{Y},h,t) \ar[rr] \ar[rd] & & \M(\u{Y},h,t) \ar[ld] \\ & \Log(Y) } \ene where $\M(\u{Y},h,t) \to \Log(Y)$ is Olsson's \'etale map of Theorem~\ref{thm:etalecover}.  The map $\L(\u{Y},h,t) \to \Log(Y)$ is representable (and \'etale if $h$ is injective) since it is a base change of $\L(h) \to \Log(\AA(Q))$, which is representable (Lemma~\ref{lem:LhtoLogAQrepresentable}) and \'etale if $h$ is injective (Theorem~\ref{thm:LhtoLogAQetale}).  We conclude that \eqref{I} is representable \'etale by ``two-out-of-three" \cite[I.4.8]{SGA1} when $h$ is injective.

The purpose of this section is to prove that if $h^{\rm gp}$ is injective with split cokernel sequence, then the representable \'etale map \eqref{I} is in fact an equivalence of stacks.  In fact, a slightly stronger statement is true.

\begin{defn} \label{defn:deformationretract} Let $\C$ be a $2$-category, $i : X \to Y$ a morphism ($1$-morphism) in $\C$.  A morphism $r : Y \to X$ is called a \emph{deformation retract} of $i$ iff $r$ is a retract of $i$ (i.e.\ $ri = \Id_X$) and there is a homotopy (invertible $2$-morphism) $\eta : ir \to \Id_Y$ which ``restricts to the identity on $X$".  This latter condition means that the $2$-morphism $\eta * i : iri=i \to i$ is the identity.  \end{defn}  

If $\C$ is the $2$-category $\CFG / \D$ of categories fibered in groupoids over a category $\D$, then the last condition in the above definition means that the natural transformation $\eta : ir \to \Id_Y$ has $\eta(i(x)) : i(x) \to i(x)$ equal to the identity for each object $x$ of $X$.

The results of this section do not have anything in particular to do with schemes: there is no use of the fppf topology as in the previous sections, and there is no real need to assume $Q$ and $P$ are finitely generated.  The discussion of this section is just ``general nonsense."  For the sake of variety and clarifying generality, let us in fact fix some category $\Esp$ of ``spaces" (take $\Esp$ equal to schemes if you like).  For a monoid $P$, let $\AA(P)$ denote the presheaf $$X \mapsto \Hom_{\Mon}(P,\O_X(X))$$ on spaces.  We are assuming here that each ``space" $X$ comes with a ring $\O_X(X)$ which is contravariantly functorial in $X$.  We won't underline out ``spaces" here because we won't have any notion of ``log space."  Since we want to work throughout in the $2$-category $\CFG / \Esp$, we think of the presheaf $\AA(P)$ as a category whose objects are pairs $(X,x)$ consisting of a space $X$ and a monoid homomorphism $x : P \to \O_X(X)$ and where a morphism $f : (X',x') \to (X,x)$ is a map of spaces $f : X' \to X$ such that $x' = f^* x$.  In any reasonable category of spaces, the presheaf $\AA(P)$ is representable (at least when $P$ is finitely generated), but that is not at all relevant right now.  Let $\GG(P)$ denote the presheaf of (abelian) groups $$ X \mapsto \Hom_{\Ab}(P^{\rm gp},\O_X^*(X)).$$  There is an obvious action of $\GG(P)$ on $\AA(P)$ denoted $(u,x) \mapsto u \cdot x$, where $(u \cdot x)(p) := u(p)x(p)$ for $p \in P$.  This action is equivariant with respect to pullback along a morphism of spaces $f$ in the sense that $f^*(u \cdot x) = f^*u \cdot f^*x$.  We can then form the quotient category \bne{AAA} \AAA(P) & := & [ \AA(P) / \GG(P) ]. \ene Keep in mind that this is just the quotient in $\CFG / \Esp$: we do not yet assume that $\Esp$ has any topology, so one might prefer to call this $\AAA^{\rm pre}(P)$ if one has a topology in mind in which one plans to stackify $\AAA(P)$, but we will not adopt this notation in this section.  The category $\AAA(P)$ has the same objects as the category $\AA(P)$, but in $\AAA(P)$ a morphism from $(X',x')$ to $(X,x)$ is a \emph{pair} $(f,u)$ consisting of a morphism of spaces $f : X' \to X$ and a group homomorphism $u : P^{\rm gp} \to \O_{X'}^*(X')$ such that $u \cdot x' = g^* x$.  We have a functor from $\AA(P)$ to $\AAA(P)$ which is the identity on objects and is given on morphisms by $f \mapsto (f,1)$.  We have a groupoid fibration from $\AAA(P)$ to spaces given by $(X,x) \mapsto X$ on objects and by $(f,u) \mapsto f$ on morphisms.

\noindent {\bf Notation:} We will use the same notation and abuse thereof as in the proof of Theorem~\ref{thm:AAAhetale}. 

Formation of the presheaves $\AA(P)$ and $\GG(P)$ and the quotient category $\AAA(P)$ is contravariantly functorial in $P$.  It is helpful to make this explicit, so let $h : Q \to P$ be a monoid homomorphism.  The we have a $\CFG / \Esp$ morphism \be \AAA(h) : \AAA(P) & \to & \AAA(Q) \ee  given on objects by $(X,x) \mapsto (X,xh)$ and on morphisms by taking $(f,u) : (X',x') \to (X,x)$ to $(f,uh) : (X',x'h) \to (X,xh)$.  (There is an example of the abusive ``$h$" notation here.)  The map $(f,uh)$ is well-defined because we compute \be (uh) \cdot (x'h) & = & (u \cdot x')h \\ & = & xh. \ee 

Much as in the beginning of this section, we now fix a space $Y$, a monoid homomorphism $h : Q \to P$ (for the moment, we make no assumptions at all on $Q$ and $P$), and a monoid homomorphism $t : Q \to \O_Y(Y)$ (a morphism of presheaves from $Y$ to $\AA(Q)$).  We define objects of $\CFG / \Esp$ by \be \L(Y,h,t) & := & [Y \times_{\AA(Q)} \AA(P) \, / \, \GG(P/Q) ] \\ \M(Y,h,t) & := & Y \times_{\AAA(Q)} \AAA(P). \ee

Explicitly: Objects of $\L(Y,h,t)$ are pairs $(f,x)$ where $f : X \to Y$ is a $Y$-space and $x : P \to \O_X(X)$ is a monoid homomorphism such that \bne{A4} f^*t & = & xh.\ene  Of course, this is the same thing as an object of the category $Y \times_{\AA(Q)} \AA(P)$, but in $\L(Y,h,t)$ a \emph{morphism} from $(f',x') \to (f,x)$ is a pair $(g,w)$ where $g : X' \to X$ is a map of $Y$-spaces and $w : (P/Q)^{\rm gp} \to \O_{X'}^*(X')$ is a group homomorphism such that \bne{A5} (wq) \cdot x' & = & g^* x. \ene 

Objects of $\M(Y,h,t)$ are triples $(f,x,u)$ where $f : X \to Y$ is a map of spaces to $Y$, $x : P \to \Gamma_X(X)$ is a monoid homomorphism, and $u : Q^{\rm gp} \to \O_X^*(X)$ is a group homomorphism satisfying \bne{A1} u \cdot f^*t & = & xh. \ene  An $\M(Y,h,t)$ morphism $$(g,v) : (f',x',u') \to (f,x,u)$$ is a pair $(g,v)$ where $g : X' \to X$ is a map of $Y$-spaces (an $\Esp$ morphism satisfying $f' = fg$) and $v : P^{\rm gp} \to \O_{X'}^*(X')$ is a monoid homomorphism satisfying the conditions: \bne{A2} v \cdot x' & = & g^* x \\ \label{A3} g^* u & = & (vh) \cdot u'. \ene  The datum of $u$ satisfying \eqref{A1} is precisely the datum of an isomorphism in the fiber category of $\AAA(Q)$ over the space $X$ from $f^* t$ to $xh$, as in the usual construction of $2$-cartesian products of groupoid fibrations.  The datum of $v$ in the above definition of a morphism is the datum of a morphism in $\AAA(P)$ from $(X',x')$ to $(X,x)$ lying over the morphism of spaces $g : X' \to X$.  The condition \eqref{A3} is precisely the commutativity condition one demands when defining morphisms in the usual construction of $2$-cartesian products of groupoid fibrations.

We define a $\CFG / \Esp$ morphism \bne{Ipre} I : \L(Y,h,t) & \to & \M(Y,h,t) \ene on objects by $(f,x) \mapsto (f,x,1)$.  This is well-defined because the condition \eqref{A4} satisfied by $(f,x)$ implies $(f,x,1)$ satisfies \eqref{A1}.  On morphisms, $I$ takes $(g,w) : (f',x') \to (f,x)$ to $$(g,wq) : (f',x',1) \to (f,x,1).$$  To see that this is well-defined, we first note that \eqref{A5} for $(g,w)$ obviously implies condition \eqref{A2} for $(g,wq)$.  To show that $(g,wq)$ satisfies condition \eqref{A3}, we need to check that $g^*1 = (wqh) \cdot 1$.  This holds because $qh=1$ since $q$ is the quotient of $h$.  It is clear that \eqref{Ipre} is in fact a morphism of categories fibered in groupoids over $Y$ (i.e.\ over $\Esp / Y$).

\begin{thm} \label{thm:equivalence} Assume $h : Q \to P$ is an injective monoid homomorphism (and that $h^{\rm gp}$ is also injective, which is automatic if $P$ is integral) and that the short exact sequence $$ 1 \to Q^{\rm gp} \to P^{\rm gp} \to (P/Q)^{\rm gp} \to 1 $$ splits.  A choice of splitting of this sequence gives rise to a deformation retract (Definition~\ref{defn:deformationretract}) $R$ of $I$ \eqref{Ipre} in $\CFG / Y$.  \end{thm}

\begin{proof}  It is convenient to package the choice of splitting as a pair $(r,s)$ consisting of a group homomorphism $r : P^{\rm gp} \to Q^{\rm gp}$ and a group homomorphism $s : (P/Q)^{\rm gp} \to P^{\rm gp}$ such that: \bne{splitting} rh & = & \Id \\ rs & = & 1 \\ qs & = & \Id \\ \label{identityexpression} hr \cdot sq & = & \Id. \ene 

The map $R$ determined by our splitting $(r,s)$ is given on objects by $(f,x,u) \mapsto (f,u^{-1} r \cdot x)$.  This is well-defined because we compute \be f^* t & = & u^{-1} \cdot xh \\ & = & u^{-1}rh \cdot xh \\ & = & (u^{-1} r \cdot x) h \ee using \eqref{A1} for $(f,x,u)$; this is the required condition \eqref{A4} for the pair $(f,u^{-1} r \cdot x)$.  On morphisms, $R$ takes $$(g,v) : (f',x',u') \to (f,x,u) $$ to $$(g,vs) : (f',(u')^{-1} r \cdot x') \to (f,u^{-1} r \cdot).$$  To check that this makes sense, we must show that $(g,vs)$ satisfies the condition \eqref{A5}.  That is, we need to see that \bne{A5prime} vsq \cdot (u')^{-1}r \cdot x' & = & g^*(u^{-1}r \cdot x). \ene  We first substitute the expression in \eqref{identityexpression} for the identity of $P^{\rm gp}$ into the condition \eqref{A2} satisfied by $(g,v)$ to find \be g^* x & = & v \cdot x' \\ & = & v(hr+sq) \cdot x' \\ & = & vhr \cdot vsq \cdot x'. \ee Now we solve \eqref{A3} for $vh$ and substitute into the above to find \be g^* x & = & ((u')^{-1} \cdot g^*u)r \cdot vsq \cdot x', \ee which is easily rearranged to look like \eqref{A5prime}.

Thus we see that $R$ is a well-defined $\CFG/Y$ morphism.  It is clear that $RI$ is the identity.  It remains to construct a homotopy $\eta : IR \to I$ in $\CFG/Y$ restricting to the identity on $\M(Y,h,t)$.  Such an $\eta$ is a choice of $\M(Y,h,t)$ isomorphism $$ (g, v) : (f,u^{-1}r \cdot x, 1) \to (f,x,u) $$ where $g = \Id$, natural in $(f,x,u)$, and equal to the identity map when $u=1$.  The conditions \eqref{A2}, \eqref{A3} for $(\Id,v)$ to determine such an isomorphism are: \be v \cdot (u^{-1}r \cdot x) & = & x \\ u & = & (vh) \cdot 1. \ee  It is trivial to check that we can arrange this by taking $v = ur$.  \end{proof}

\begin{rem} We leave it to the reader to check that a different choice of splitting yields a different deformation retract $R'$ of $I$ which is homotopic to $R$ relative to $\L(Y,h,t)$. \end{rem}

\begin{cor} \label{cor:equivalence} Suppose $h: Q \to P$ is an injective map of finitely generated monoids such that the cokernel sequence for $h^{\rm gp}$ splits.  Then the map \eqref{I} admits a deformation retract and is hence, in particular, an equivalence of algebraic stacks. \end{cor}

\begin{proof} The map \eqref{I} in question is obtained by stackifying (in the \'etale topology) the map of prestacks of Theorem~\ref{thm:equivalence} (with $\Esp = \Sch$), hence the result is immediate from that theorem in light of the functoriality of stackification. \end{proof}

\begin{rem} The assumption that $Q$ and $P$ are finitely generated in Corollary~\ref{cor:equivalence} is only necessary to ensure that \eqref{I} is a map of \emph{algebraic} stacks in the sense of Definition~\ref{defn:algebraicstack}.  The statement of Corollary~\ref{cor:equivalence} continues to hold, by the same proof, when the words ``finitely generated" and ``algebraic" are deleted (or the finiteness conditions in the definition of ``algebraic stack" are removed). \end{rem}

\begin{cor} \label{cor:equivalence2} Suppose $h: Q \to P$ is an injective map of fine monoids such that $\Cok h^{\rm gp}$ is torsion-free.  Then the map \eqref{I} admits a deformation retract and is hence, in particular, an equivalence of algebraic stacks. \end{cor}

\begin{cor} \label{cor:equivalence3} Suppose $h: Q \to P$ is a partition morphism with boundary.  Then the map \eqref{I} admits a deformation retract and is hence, in particular, an equivalence of algebraic stacks. \end{cor}

\begin{proof} $\Cok h^{\rm gp}$ is torsion-free by Proposition~\ref{prop:partitionmorphismsarefree}. \end{proof}

\begin{example} When $h : Q \to P$ is the diagonal map $\Delta : \NN \into \NN^2$, the quotient of $h=\Delta$ is the map $q : \NN^2 \to \ZZ$ defined by $q(a,b) := a-b$.  The map $\GG(q) : \GG_m \to \GG_m^2$ is given by $t \mapsto (t,t^{-1})$.  Since $\Cok \Delta^{\rm gp} \cong \ZZ$, Corollary~\ref{cor:equivalence2} yields an equivalence of algebraic stacks \be [ (Y \times_{\AA^1} \AA^2) \, / \, \GG_m] & \cong &  Y \times_{\AAA^1} \AAA^2. \ee \end{example}

\subsection{On the \'etale cover} \label{section:ontheetalecover} Let $Y$ be a fine log scheme, $Q \to \M_Y(Y)$ a global chart for $Y$ whose composition with $\alpha_Y : \M_Y(Y) \to \O_X(Y)$ we denote $t : Q \to \O_Y(Y)$.  Let $h : Q \to P$ be a map of fine monoids.  Recall (\S\ref{section:convenientetalecover}) that we set \be \M(\u{Y},h,t) & := & \u{Y} \times_{\AAA(Q)} \AAA(P). \ee  In this section we give a careful description of Olsson's representable \'etale map $\M(\u{Y},h,t) \to \Log(Y)$ from Theorem~\ref{thm:etalecover}.  Combined with the description of $\L(h) \to \Log(\AA(Q))$ in \S\ref{section:LhtoLogAQetale}, this will make the $2$-commutativity of \eqref{LMLog} clear.  

As is our usual \emph{modus operandi}, we will construct a $\CFG/Y$-morphism \bne{fun} \M(\u{Y},h,t)^{\rm pre} & \to & \Log(Y) \ene and declare $\M(\u{Y},h,t) \to \Log(Y)$ to be the associated map of algebraic stacks obtained by stackifying in the \'etale topology.  Here $\M(\u{Y},h,t)^{\rm pre}$ denotes the groupoid fibration described explicitly in the previous section, where it was called $\M(Y,h,t)$.  We will refer to this description in what follows.  We will suppress the isomorphism of log structures $Q_t^a \cong \M_Y$ from our chart $Q \to \M_Y(Y)$ and simply view the log structure on $Y$ as $Q_t^a$.  We view as log scheme as a pair $(X,\M_X)$ consisting of a scheme $X$ and a log structure $\M_X$ on $X$.  We view a morphism of log schemes as a pair $(f,f^\dagger)$ consisting of a morphism $f : X \to Y$ of schemes and a morphism $f^\dagger : f^* \M_Y \to \M_X$ of log structures on $X$  (we don't use the underlining convention in this section).

On objects, \eqref{fun} takes $(f,x,u)$ to $(f, u^{-1} \cdot h : Q_{f^*t}^a \to P^a_x)$.  Here $P^a_x$ is the log structure on $X$ associated to the prelog structure $x : P \to \O_X(X)$, $Q_{f^*t}^a = f^* Q^a_t$ is the pullback of the log structure $Q^a_t$ on $Y$ along the map of schemes $f : X \to Y$ (pullback of prelog structures commutes with formation of associated log structures), and $u^{-1} \cdot h$ is an abuse of notation for the map of log structures on $X$ given by \be u^{-1} \cdot h : Q^a_{f^*t} & \to & P_x^a \\ \nonumber [q, \zeta] & \mapsto & [h(q),u^{-1}(q) \zeta]. \ee  To see that this is a well-defined map of log structures we must check that it respects the structures maps $Q^a_{f^*t} \to \O_X$ and $P_x^a \to \O_X$.  That is, we must check that \be (f^*t)(q) \zeta & = & xh(q) u^{-1}(q) \zeta \ee for each local section $[q,\zeta]$ of $Q^a_{f^*t}$.  This is clear from the condition \eqref{A1} satisfied by $(f,x,u)$.  Our functor \eqref{fun} takes an $\M(\u{Y},h,t)^{\rm pre}$-morphism $$(g,v) : (f',x',u') \to (f,x,u)$$ to the map $$(g,v) : (X',P_{x'}^a) \to (X,P_x^a) $$ of log schemes over $Y$ given by $g : X' \to X$ on spaces and on log structures by \be v : P^a_{g^* x} & \to & P^a_{x'} \\ \nonumber [p,\zeta] & \mapsto & [p,v(p) \zeta]. \ee The fact that this map respects the structure maps to $\O_{X'}$ is immediate from the condition \eqref{A2} on $(g,v)$.  The fact that this is a map of log structures under $(f')^* Q^a_t = Q^a_{(f')^*t}$---i.e.\ the diagram $$ \xym{ & \ar[ld]_{g^*(u^{-1} \cdot h} Q^a_{(f')^*t} \ar[rd]^{(u')^{-1} \cdot h} \\ P^a_{g^*x} \ar[rr]^-v  & & P^a_{x'} } $$ commutes is immediate from condition \eqref{A3} for $(g,v)$.

\end{document}